\documentclass[a4paper]{amsart}

\RequirePackage{amsmath} 
\RequirePackage{amssymb}
\usepackage{amscd,latexsym,amsthm,amsfonts,amssymb,amsmath,amsxtra}
\usepackage[colorlinks=true,linkcolor=blue,urlcolor=blue,citecolor=blue]{hyperref}

\usepackage{color}
\usepackage[all]{xy}
\usepackage[OT2,T1]{fontenc}
\usepackage{bm}
\usepackage{mathtools}
\usepackage{ mathrsfs }
\usepackage{xcolor}
\usepackage{comment}
\usepackage{marginnote}
\usepackage{enumitem}

\DeclareSymbolFont{cyrletters}{OT2}{wncyr}{m}{n}
\DeclareMathSymbol{\Sha}{\mathalpha}{cyrletters}{"58}

\let\Re\undefined
\let\Im\undefined

\def\stacksum#1#2{{\stackrel{{\scriptstyle #1}}
{{\scriptstyle #2}}}}

\DeclareMathOperator{\Re}{Re}
\DeclareMathOperator{\Im}{Im}
\DeclareMathOperator{\Gal}{Gal}

\DeclareMathOperator{\supp}{supp}

\DeclareMathOperator{\GL}{GL}
\DeclareMathOperator{\SL}{SL}
\DeclareMathOperator{\SU}{SU}

\DeclareMathOperator{\SO}{SO}

\DeclareMathOperator{\vol}{vol}
\DeclareMathOperator{\Spec}{Spec}

\DeclareMathOperator{\Aut}{Aut}
\DeclareMathOperator{\rank}{rank}

\DeclareMathOperator{\Frob}{Frob}

\newcommand{\tfn}{\widetilde{\mathfrak{n}}}

\newcommand{\ov}{\overline}
\newcommand{\sym}{{\mathrm{sym}}}

\newcommand{\dLambda}{\frac{(2k-2)(k+2)(k-6)}{3}}

\newcommand{\Nr}{\mathrm{Nr}}
\newcommand{\whZ}{\widehat{\Zz}}

\newcommand{\eps}{\varepsilon}
\newcommand{\vphi}{\varphi}
\newcommand\sumsum{\mathop{\sum\sum}\limits}
\newcommand\sqsqcup{\mathop{\bigsqcup\ \bigsqcup}\limits}
	\newcommand\ccup{\mathop{\bigcup\ \bigcup}\limits}
	
	\newcommand{\bfw}{{\mathbf {w}}}
	\newcommand{\bfp}{{\mathbf {p}}}
	 
	\newcommand{\mcV}{{\mathcal{V}}}
	\newcommand{\mcL}{{\mathcal{L}}}
	\newcommand{\mcD}{{\mathcal{D}}}
	\newcommand{\mcI}{{\mathcal{I}}} \newcommand{\mcE}{{\mathcal{E}}}
	\newcommand{\mcC}{{\mathcal{C}}}
	\newcommand{\mcA}{{\mathcal{A}}}
	\newcommand{\mcAk}{{\mathcal{A}_k}}
	\newcommand{\mcAkn}{{\mathcal{A}^{\mathrm{n}}_k}}
	\newcommand{\mcB}{{\mathcal{B}}}
	\newcommand{\BA}{{\mathbb {A}}}

	 \newcommand{\BH}{{\mathbb {H}}}

	\newcommand{\BQ}{{\mathbb {Q}}} \newcommand{\BR}{{\mathbb {R}}}
	\newcommand{\Rr}{{\mathbb {R}}}\newcommand{\Zz}{{\mathbb {Z}}}
\newcommand{\Nn}{{\mathbb {N}}}

	\newcommand{\Ff}{{\mathbb {F}}}
\newcommand{\Zp}{{\mathbb {Z}_p}}\newcommand{\Zpt}{{\mathbb {Z}^\times_p}}
\newcommand{\mcP}{{\mathcal{P}}}

\newcommand{\mcX}{{\mathcal{X}}}
\newcommand{\Fp}{{\mathbb {F}_p}}

	\newcommand{\Cc}{{\mathbb {C}}}
	\newcommand{\Ct}{{\mathbb {C}}^\times}

	 \newcommand{\BZ}{{\mathbb {Z}}}
	
	\newcommand{\Qq}{\BQ}
	\newcommand{\Hh}{\BH}
		
	\newcommand{\Qp}{\BQ_p}
	\newcommand{\Qv}{\BQ_v}
	\newcommand{\Aa}{\BA}
		\newcommand{\Af}{\BA_{\mathrm{fin}}}
	\newcommand{\Gm}{\mathbb{G}_\mathrm{m}}
	
	\newcommand{\bash}{\backslash}
	
	\newcommand{\kmin}{32}
	\newcommand{\kappamin}{6}

		\newcommand{\mcO}{{\mathcal {O}}}

	\newcommand{\rmB}{{\mathrm {B}}}

	 \newcommand{\mfn}{{\mathfrak{n}}}
	 
	\newcommand{\mfS}{{\mathfrak{S}}}
	\newcommand{\mfp}{{\mathfrak{p}}}

	\newcommand{\omfp}{{\ov{\mathfrak{p}}}}

	\newcommand{\Mod}[1]{\ (\mathrm{mod}\ #1)}

\newcommand{\Res}{\operatorname{Res}}
\newcommand{\tr}{\operatorname{tr}}
\newcommand{\Eis}{\operatorname{Eis}}

\newcommand{\K}{\operatorname{K}}
\newcommand{\sgn}{\operatorname{sgn}}

\newcommand{\Ad}{\operatorname{Ad}}

\newcommand{\diag}{\operatorname{diag}}
\newcommand{\Ind}{\operatorname{Ind}}

\newcommand{\ST}{\operatorname{ST}}

\newcommand{\Id}{\operatorname{Id}}

\newcommand{\fin}{\operatorname{fin}}
\newcommand{\LHS}{\operatorname{LHS}}
\newcommand{\RHS}{\operatorname{RHS}}

\newcommand{\St}{\mathrm{St}}
\newcommand*{\transp}[2][-3mu]{\ensuremath{\mskip1mu\prescript{\smash{\mathrm t\mkern#1}}{}{\mathstrut#2}}}

\newcommand{\mcF}{{\mathcal F}}
\newcommand{\mcR}{{\mathcal R}}
\newcommand{\ra}{\rightarrow}
\newcommand{\mods}[1]{\,(\mathrm{mod}\,{#1})}
\def\peter#1{\langle #1\rangle}
\newcommand{\bfone}{\textbf{1}}

\newcommand{\Npbound}{10^6}
\newcommand{\Nplower}{10^6}
\newcommand{\Nlower}{3}

\begin{document}
	
\theoremstyle{plain}
	\newtheorem{thm}{Theorem}[section]
	
	\newtheorem{cor}[thm]{Corollary}
	\newtheorem{thmy}{Theorem}
	\renewcommand{\thethmy}{\Alph{thmy}}
	\newenvironment{thmx}{\stepcounter{thm}\begin{thmy}}{\end{thmy}}
	\newtheorem{cory}{Corollary}
	\renewcommand{\thecory}{\Alph{cory}}
	\newenvironment{corx}{\stepcounter{thm}\begin{cory}}{\end{cory}}
	\newtheorem*{thma}{Theorem A}
	\newtheorem*{thmD}{Theorem D}
	\newtheorem*{corb}{Corollary B}
	\newtheorem*{corc}{Corollary C}
	\newtheorem*{thmc}{Theorem C}
	\newtheorem{lemma}[thm]{Lemma}  
	\newtheorem{prop}[thm]{Proposition}
	\newtheorem{conj}[thm]{Conjecture}
	 \newtheorem{fact}[thm]{Fact}
	 	\newtheorem{claim}[thm]{Claim}
	
	\theoremstyle{definition}
	\newtheorem{defn}{Definition}[section]
		\newtheorem{example}{Example}[section]

	\theoremstyle{remark}
	
	\newtheorem{remark}{Remark}[section]	
	\numberwithin{equation}{section}

	\title[Bessel periods on $U(2,1) \times U(1,1)$ and non-vanishing]{Bessel Periods on $U(2,1) \times U(1,1)$, Relative Trace Formula 
and Non-Vanishing of Central L-values}

\author{Philippe Michel, Dinakar Ramakrishnan and Liyang Yang}

\address{EPFL-SB-MATH-TAN Station 8\\
	1015 Lausanne, Switzerland}
\email{philippe.michel@epfl.ch}

\address{253-37 Caltech, Pasadena\\
	CA 91125, USA}
\email{dinakar@caltech.edu}

\address{Department of Mathematics, Texas A\&M University, College Station,  TX 77843, USA} 
\email{liyangy@tamu.edu}  
	
\begin{abstract}
In this paper we calculate the asymptotics of the second moment of the Bessel periods associated to certain holomorphic cuspidal representations $(\pi, \pi')$ of $U(2,1) \times U(1,1)$ of regular infinity type (averaged over $\pi$). Using these, we obtain quantitative non-vanishing results for the Rankin-Selberg central $L$-values $L(1/2, \pi \times  \pi')$, which are of degree twelve over $\Qq$, with concomitant difficulty in applying standard methods, especially since we are in a `conductor dropping' situation. We use the relative trace formula, and the orbital integrals are evaluated rather than compared with others. Besides their intrinsic interest, non-vanishing of these critical values also leads, by known results, to deducing certain associated Selmer groups have rank zero.
\end{abstract}
	
\date{\today}%

\maketitle

\begin{flushright}
  \textit{To Herv{\'e} Jacquet}
\end{flushright}
\setcounter{tocdepth}{1}
\renewcommand{\contentsname}{\bf Contents}
\tableofcontents


\section{\bf Introduction}

Let $L(s,\pi)$ be an $L$-function admitting an Euler product factorization
$$L(s,\pi)=\prod_p L_p(s,\pi),\ \Re(s)\gg 1$$ constructed out of some automorphic datum $\{\pi\}$; we assume that $L(s,\pi)$ is analytically normalized, self-dual and even that it  admits analytic continuation to the whole $s$-plane with a functional equation relating $L(s,\pi)$ to $L(1-s,\pi)$ and  root number $\pm 1$. Under such a hypothesis the central value of the finite part $L(1/2,\pi)$ or its derivative $L'(1/2,\pi)$ (depending on the root number) is of great  importance, both from the arithmetic and the analytic points of view. 

A problem of great interest is, given a family  $\mcF$ of such similar automorphic data, to  exhibit some $\pi\in\mcF$ for which $L(1/2,\pi)$ or $L'(1/2,\pi)$  is non zero at least for a large enough family. For instance, such a problem occurs in questions related to the Birch-Swinnerton-Dyer conjecture for some motives; the existence of such non-vanishing can then be used to establish the non-triviality of some Euler system and ultimately conclude the expected value of its arithmetic rank and the finiteness of its Tate-Shafarevitch group. 

A basic strategy to exhibit such non-vanishing is to consider some  sequence of subfamilies $\mcF_N\subset\mcF$ indexed by some parameter $N$, whose size  grows with $N$, and to evaluate, for $N\ra \infty$, some weighted moment of the shape
$$M(\mcF_N)=\sum_{\pi\in\mcF_N}w(\pi)L(1/2,\pi)$$
(for some non-negative weights $(w(\pi))_{\pi\in\mcF_N}$) so as to show that this moment is  non-zero for $N$ large enough.

Some of the earliest examples were given in the (independent) works of Bump-Friedberg-Hoffsten and Murty-Murty \cite{BFH,MurMur} regarding the non-vanishing of  $L(1/2,E\times\chi)$ or its derivative for $E$ a fixed elliptic curve (more generally a modular form) and $\chi$ varying over (odd) quadratic characters; both were motivated by the seminal work of Kolyvagin \cite{Kol}. Another very recent example is the work \cite{radziwi2023nonvanishing} by the third named author and M. Radziwill establishing the non-vanishing of $L(1/2,\pi_4\times\chi)$ for $\pi_4$ a fixed $\GL_4$ automorphic cuspidal representation and $\chi$ varying over (complex) Dirichlet characters; this result has important implications concerning the Birch-Swinnerton-Dyer conjecture for abelian surfaces in connection with the work of Loeffler-Zerbes  and others
\cite{loeffler2023birchswinnertondyer}.

Another example relevant to the present paper is furnished by the family $\mcF_\chi(N)=\{\pi_E \otimes\chi\},$ where $E=\BQ(\sqrt{D})$ is a fixed imaginary quadratic field,  $\chi$ an ideal class character of $E$ and $\pi$ varies over the unitary cuspidal automorphic representation of $\GL(2)/\BQ$ attached to a normalized newform $\varphi$ of level $N$, weight $2$ and trivial character, with $\pi_E$ denoting te base change of $\pi$ to $\GL(2)/E$. Suppose (for simplicity) that $N$ is a prime which is inert in $E$ so that the sign of the functional equation of $L(s,\pi_E)$ is $+1$. When $\chi=1$ is the principal character, the $L$-function $L(s,\pi_E)$ factors as $L(s,\pi)L(s,\pi\otimes\eta)$, where $\eta=\eta_E$ is the quadratic Dirichlet character of $\BQ$ attached to $E$.  In \cite{Duke}, W.~Duke established the non-vanishing of $L(1/2,\pi_E)$ for $\gg N/\log N$ representations $\pi$; later, using the mollification method, Iwaniec and Sarnak established non-vanishing for a positive proportion of such forms \cite{ISIJM}. If $\chi$ is not quadratic, $L(s,\pi_E\times\chi)$ is a Rankin-Selberg $L$-function $L(s,\pi_E\times\pi_\chi)$ where $\pi_\chi={\mathrm I}_\Qq^E(\chi)$ is the (cuspidal) automorphic induction of $\chi$ from $\GL(1)/E$ to $\GL(2)/\Qq$; the existence of a positive proportion of $\pi$ for which $L(1/2,\pi_E\times\pi_\chi)\not=0$ was established by Kowalski and the first named author \cite{KMDMJ}. The evaluation of these moments was based on the Petersson-Kuznetzov formula; in \cite{RR05}, Rogawski and the second named author took a different route and used instead the Relative Trace Formula (RTF) for the pair $(\GL(2)/\Qq, T_\Qq)$ for $T$ the diagonal (split) torus (using the fact that for $\chi$ of order two, $L(s,\pi_E\times\chi)$ is the product of two $\GL(2)$ $L$-functions). A generalization to Hilbert modular forms over a totally real base field $F$ and suitably general idele class characters $\chi$ was achieved by B.~Feigon and D.~Whitehouse in \cite{FeiWhi}, using the RTF for anisotropic pairs $(G,T)$, with $G$ an inner form of $\GL(2)/F$ and $T\simeq \Res_{E/F}\Gm$ a non-split torus attached to  a totally imaginary quadratic extension $E/F$.  In the present paper, we carry out this approach in a higher rank situation (for the base field $\Qq$).


\subsection{First moment for Rankin-Selberg $L$-functions for $U(2,1)\times U(1,1)$}

In this paper, we consider families of $L$-functions attached to automorphic representations $\pi\times\pi'$  on $G \times G'$ for $G\simeq U(2,1)$ (resp. $G'\simeq U(1,1)$) the quasi-split unitary group in three (resp. two) variables, associated to an imaginary quadratic field $E/\Qq$ of conductor $D_E$. These representations admit base changes $\pi_E,\pi'_E$ to $\GL(3)/E$ and $\GL(2)/E$ whose existence is known by Rogawski \cite{Rog} (see also  Flicker \cite{Fli82}).  The $L$-functions we consider are the Rankin-Selberg $L$-functions $L(s,\pi_E\times\pi'_E)$ which for $\Re(s)>1$ admit an Euler product factorization
$$L(s,\pi_E\times\pi'_E)=\prod_{p}L_p(s,\pi_E\times\pi'_E)=\prod_{\mfp|p}\prod_{\mfp}L_\mfp(s,\pi_E\times\pi'_E),\ \Re(s)>1$$
where $\mfp$ runs over the primes of $E$ above $p$.
This $L$-function is completed by an archimedean local factor
$$L_\infty(s,\pi_E\times\pi'_E)=\prod_{w|\infty}L_{w}(s,\pi_E\times\pi'_E)$$
(a product of Gamma functions) and admits a  functional equation of the shape
\begin{equation}
	\label{fcteqn}
	\Lambda(s,\pi_E\times\pi'_E)=\eps(\pi_E\times\pi'_E)C_f(\pi_E\times\pi'_E)^{1/2-s}\Lambda(1-s,\pi_E\times\pi'_E)
\end{equation}
where $$\Lambda(s,\pi_E\times\pi'_E)=L_\infty(s,\pi_E\times\pi'_E)L(s,\pi_E\times\pi'_E)$$ is the "completed" $L$-function, 
 $C_f(\pi_E\times\pi'_E)\geq 1$ is an integer (the arithmetic conductor) and $\eps(\pi_E\times\pi'_E)\in\{\pm 1\}$ is the root number.

\subsubsection{The main assumptions}\label{mainassumptions} We will consider the families for which $\pi'$, the form on the smaller group  is {\it fixed}, while the form in the larger group $\pi$ is varying. 

More precisely (see \S \ref{sec2} for greater details) let $k\geq 0$ be an  integer and let $N,N'\geq 1$ be integers either equal to $1$ or to prime numbers unramified in $E$ . We assume further that 
\begin{itemize}
\item[--] $k$ is even and sufficiently large:
\begin{equation}
	\label{kmin}
	k>\kmin,
\end{equation}
	\item[--]  If $N>1$, then $N$  is \textit{inert} in $E$, and 
	\begin{equation}
		\label{Nmin}
		N'\geq \Nlower
	\end{equation}
	\item[--] If $N'>1$, then $N'$  is \textit{split} in $E$ and 
	\begin{equation}
		\label{Npmin}
		N'\geq \Nplower
	\end{equation}
\item[--] The representation $$\pi'\simeq\pi'_{\infty}\otimes{\bigotimes}'_p\pi'_{p}$$ is a cuspidal automorphic representation of $G'(\mathbb{A})$, with trivial central character, whose archimedean component $\pi'_{\infty}$ is a holomorphic discrete series of weight $k$,  which is unramified at every prime not dividing $N'$ and, if $N'$ is prime, that $\pi'_{N'}$ is the Steinberg representation.

\item[--] The representations $$\pi\simeq\pi_{\infty}\otimes{\bigotimes}'_p\pi_{p}$$ vary over  cuspidal automorphic representations of $G(\mathbb{A})$, with trivial central character whose archimedean component $\pi_{\infty}$ is a holomorphic discrete series of weights $\Lambda=(-2k,k)$ (cf. \cite{Wal76} and below) for the same value of $k$ as above, which is unramified  at every prime not dividing $N$ and, if $N$ is prime, that $\pi_{N}$ is either unramified or the Steinberg representation. \end{itemize} 

We denote by $\mcAk(N)$ the finite set of all such  automorphic representations $\pi$ and we denote by $$\mcAkn(N)\subset \mcAk(N)$$ the subset of those representations which are ramified at $N$: if $N=1$, $$\mcAkn(1)= \mcAk(1)$$ and if $N$ is prime, this is the set of $\pi$ such that $\pi_N$ is the Steinberg representation.

By a version of Weyl's law, one has\footnote{We would like to point out  a seeming discrepancy in \cite[Lemma 9.4]{Wal76} between the formula computing the formal degree of $\pi_\lambda$ and the original (correct) formula from Harish-Chandra \cite[Remark 5.5]{HC}; the former would lead to the asymptotic in the $k$-aspect $|\mcA_k(N)|\asymp k^2N^3$ which is not correct; we are thankful to Paul Nelson for pointing it out.} \begin{equation}
	\label{Weyllaw}
	|\mcAkn(N)|\asymp k^3N^3\hbox{ as }\ k+N\ra\infty.
\end{equation}
\begin{remark}
	\label{rempatho} The conditions \eqref{kmin}, \eqref{Nmin} and \eqref{Npmin} are made either  to avoid pathologies and technical difficulties with small weights or small residual characteristics. They will ensure the absolute convergence and possibly non-vanishing of various local and global integrals in our argument. The conditions \eqref{Npmin} and \eqref{Nmin} can perhaps be improved with more intensive combinatorial  efforts while allowing very small weights  \eqref{kmin} will certainly constitute a major technical challenge.
\end{remark}

\subsubsection{Upper and lower bounds for the first moment}
\label{secmainthm}
Regarding the $L$-function $L(s,\pi_E\times\pi'_E)$, the assumptions above allow us to compute explicitly (see Propositions \ref{propperiodarch} and \ref{conductorthm}) the archimedean factor $L_\infty(s,\pi_E\times\pi'_E)$, the arithmetic conductor
 $C_f(\pi_E\times\pi'_E)$
and the root number which equals 
$$\eps(\pi_E\times\pi'_E)=+1.$$
Moreover the central value $L(1/2,\pi_E\times\pi'_E)$ is then non-negative (see below).

To state our first main result, we also need  the {\em Adjoint} $L$-functions of $\pi_E$ and $\pi'_E$
$$L(s,\Ad,\pi_E),\ L(s,\Ad,\pi'_E)$$
whose holomorphy at $s=1$ is a consequence of Rankin-Selberg theory.

We have (see \S \ref{vino} for the definition of $\asymp$)
\begin{thm}\label{thmA}
	Let the notation and assumptions be as in \S \ref{mainassumptions}. Given  $N'$ there exists $C(N')\geq 1$ such that for any $k,N$ as above ($N$ inert, $k> \kmin$ even) and such that
	\begin{itemize}
		\item $k+N\geq C(N')$ and
		\item either $N=1$ or $N\geq C(N')$,
	\end{itemize} 
	one has
			\begin{equation}\label{firstmomentn}
		\frac{1}{|\mcAkn(N)|}
		\sum_{\pi\in \mcAkn(N)}\frac{L(1/2,\pi_E\times\pi'_E)}{L(1,\Ad,\pi_E)L(1,\Ad,\pi'_E)}
		\asymp 1,	
	\end{equation}
	where the implicit constants depend at most on $E$ and $N'$. 
	
	In particular  for any such $(k,N,N')$ there exists $\pi\in\mcAkn(N)$ such that
	\begin{equation}
		\label{eqnonvanishone}
		L(1/2,\pi_E\times\pi'_E)\not=0.
	\end{equation}
\end{thm}
\begin{remark} Parts of Theorem \ref{thmA}  will likely generalize to $U(n,1)\times U(n-1,1)$  in the {\em level} aspect, with compatible discrete series components at the archimedean places with supercuspidal components specified at some finite places. However, we are focusing on this specific case because we are able to obtain more precise results even as one lets the archimedean component vary. In addition, our pair $(\pi_\infty,\pi'_{\infty})$ is not in general position but exhibits a {\em conductor dropping phenomenon} (see \S \ref{seccond}) which hopefully is of interest.
	 	\end{remark}

 Theorem \ref{thmA} shows that the set of $\pi$'s  for which the corresponding central value $L(1/2, \pi_E\times\pi'_E)$ does not vanish is non-empty. Our next main result quantifies this and shows that the size of this set has polynomial growth as $k+N\ra\infty$.

\begin{thm}\label{thmnonvanishpower}
 	Let the notation and assumptions be as in Theorem \ref{thmA}. 

 There exists an absolute constant $\delta>0$ such that as $k+N\ra\infty$,  we have
\begin{equation}
	\label{nonvanishpower}
	|\{\pi\in \mcAkn(N),\ L(1/2,\pi_E\times\pi'_E)\not=0\}|\gg_{N'} (kN)^{\delta}.
\end{equation}

\end{thm}
\begin{remark}
	See \S \ref{secThm12}  for an explicit value of $\delta$.
\end{remark}	

We would like to stress that \eqref{nonvanishpower} is {\em not} a direct consequence of \eqref{firstmomentn} and of the {\em convexity} bound  $$L(1/2,\pi_E\times\pi'_E)\ll_E (k^8N^4)^{1/4+o(1)}$$ despite the fact that $|\mcAkn(N)|\asymp (kN)^3$ is much larger than the fourth root of the analytic conductor of $L(s,\pi_E\times\pi'_E)$ at $s=1/2$ which is
$\asymp (k^8N^4)^{1/4}=k^2N$.
This is because, at the moment, we cannot exclude the possibility of $L(1,\Ad,\pi_E)$ being extremely small. Instead \eqref{nonvanishpower}  is a consequence of a variation on the {\em proof} of \eqref{firstmomentn} which we now describe.

\subsection{$L$-functions and Bessel periods}

Our proof of Theorem \ref{thmA} follows along the lines of the earlier work of the second author and J. Rogawski \cite{RR05} but with substantially more complicated calculations; it is a consequence of the asymptotic evaluation, using the {\em Relative Trace Formula}, of sums of {\em Bessel periods} of the shape
\begin{align*}
\mathcal{P}(\varphi,\varphi'):=\int_{G'(\mathbb{Q})\backslash G'(\mathbb{A})}\varphi(g')\ov{\vphi'}(g')dg'
\end{align*}
where $\varphi\in\pi$ and $\varphi'\in \pi'$ are  suitable factorable automorphic forms (see \cite{LiuCrelle} for a detailed discussion of these periods).

\subsubsection{From periods to $L$-functions}

The derivation of Theorem \ref{thmA} from \ref{thmB} follows from the validity of the Gan-Gross-Prasad type conjectures for unitary pairs $U(n+1)\times U(n)
$ which predict a relation  between the square of the Bessel period $|\mathcal{P}(\varphi,\varphi')|^2$  and the central $L$-value $L(1/2,\pi_{E}\times\pi'_E)$ (in a precise form due to Ichino-Ikeda). 

These conjectures have now been established  for tempered representations (which is the case in our situation, see the discussion in  \S \ref{sectempered})  due to the work of many people including Beuzart-Plessis, Chaudouard, Liu, Zhang, Zhu, and Zydor; we refer to Beuzart-Plessis' ICM lecture for a complete description of the conjectures and their resolution \cite{RBPICM}. More precisely, Theorem 1.9 of \cite{BPLZZ19} gives 
\begin{multline}\label{eqGGP}
\frac{\big|\mathcal{P}(\varphi,\varphi')\big|^2}{\langle\varphi,\varphi\rangle\langle\varphi',\varphi'\rangle}=\frac{\Lambda(1,\eta)\Lambda_\Qq(2)\Lambda(3,\eta)}{2}\times\\ \frac{\Lambda(1/2,\pi_{E}\times\pi'_E)}{\Lambda(1,\pi_E,\Ad)\Lambda(1,\pi'_E,\Ad)}\cdot \prod_{v}\mathcal{P}_v^{\natural}(\varphi,\varphi'),
\end{multline}
where $\eta=\eta_{E/\Qq}$ is the (fixed) quadratic Legendre character attached to $E/\mathbb{Q}$ and $\Lambda(s,\bullet)$ denotes the {\em completed $L$-function} (including archimedean factor) and $\prod_{v}\mathcal{P}_v^{\natural}(\varphi,\varphi')$ is a finite product of local periods. 

In Section \ref{4.2} we evaluate the local periods $\mathcal{P}_v^{\natural}(\varphi,\varphi')$ explicitly for a specific automorphic form $\vphi'$ (see \S \ref{phipdef}) and for $\vphi$ varying over an orthogonal family $\mcB^{\tfn}_{k}(N)$ of factorable automorphic forms of level $N$ and minimal weights $(-2k,k)$ belonging to the various representations in $\mcA_k(N)$
 (see \S \ref{secspectralexp}  for precise definitions). We show that for such $\vphi'$ and $\vphi$, the local periods are non-negative and that their product satisfies
\begin{equation}\label{m8}
\frac{L_\infty(1/2,\pi_{E}\times\pi'_E)}{L_\infty(1,\pi_E,\Ad)L_\infty(1,\pi'_E,\Ad)}\prod_{v}\mathcal{P}_v^{\natural}(\varphi,\varphi')\asymp_E \frac{1}{kN{N'}^2}
\end{equation}
and by \eqref{eqGGP} one has 
\begin{equation}\label{Lperiodprop}
	\frac{\big|\mathcal{P}(\varphi,\varphi')\big|^2}{\langle \varphi,\varphi\rangle \langle\varphi',\varphi'\rangle} \asymp_E \frac{1}{kN{N'}^2}\frac{L(1/2,\pi_E\times\pi'_E)}{L(1,\pi_E,\Ad)L(1,\pi'_E,\Ad)};
\end{equation}
(where the dependence on $E$ includes the factor $\frac{\Lambda(1,\eta)\Lambda_\Qq(2)\Lambda(3,\eta)}{2}$) the central values $L(1/2,\pi_{E}\times\pi'_E)$ are thus non-negative.

 \begin{remark} Although the knowledge of the full Ichino-Ikeda conjectures seem to provide the needed relationship between periods and $L$-values, these explicit local computations are nevertheless necessary,  first to infer positivity (since the test functions we use  are not of positive type)  and also to make sure that the (positive) constants implicit in the symbols $\asymp_E$ in \eqref{m8} and \eqref{Lperiodprop}, indeed do not depend on the varying parameters $N,N',k$.
 \end{remark}

\subsubsection{Averages of squares of Bessel periods}
With \eqref{Lperiodprop} established, \eqref{firstmomentn} is then consequence (for $\ell=1$) of the following result which evaluates the average of the square of the Bessel periods along the family $\mcB^{\tfn}_{k}(N)$ (see Theorem \ref{firstmomentwithl} for a more precise version): 

   \begin{thm}\label{thmB}
	Let the notation and assumptions be as in \S \ref{mainassumptions}. Let $\varphi'\in\pi'$ be the (fixed) automorphic newform of level $N'$ and minimal weight $k> \kmin$, defined in \S \ref{phipdef} and let $\mcB^{\tfn}_{k}(N)$ be the finite family of automorphic forms defined in \S \ref{secspectralexp}. 
	
	Given $\ell\geq 1$ an integer coprime with $N$ and divisible only by primes inert in $E$, we denote by $\lambda_\vphi(\ell)$ and $\lambda_{\vphi'}(\ell)$ the eigenvalues at $\vphi$ and $\vphi'$ of the Hecke operators $T(\ell)$ and $T'(\ell)$ described in \S \ref{secinertHecke}.	

There is an absolute constant $C\geq 1$ such that for any $\delta>0$ and any quadruple $(k,\ell, N,N')$ satisfying

 either
\begin{equation}
	\label{stablerange}
(\ell N')^2\leq N^{1-\delta},\ N>16, k\geq C(1+1/\delta)
\end{equation}

or \begin{equation}
	\label{stablerange2}
	(\ell N')^2\leq k^{1-\delta},\ N\leq 2^{4k},
\end{equation}
 we have as $k+\ell+N+N'\ra\infty$,
\begin{equation}\label{m1}
\sum_{\varphi\in \mcB^{\tfn}_{k}(N)}\lambda_\vphi(\ell)\frac{\big|\mathcal{P}(\varphi,\varphi')\big|^2}{\langle \varphi,\varphi\rangle \langle\varphi',\varphi'\rangle}=w_E\frac{d_\Lambda}{d_k}(\frac{N}{{N'}})^2\Psi(N)\mfS({N'})\frac{\lambda_{\pi'}(\ell)+o_{\delta,E}(1)}{\ell }.
\end{equation}
Here $w_E$ is the number of units of $E$,
$$\Psi(N)=\prod_{p\mid N}\left(1-\frac{1}p+\frac1{p^{2}}\right),\ \mfS({N'})=\prod_{p|N'}(1-\frac{1}{p^2})^{-1}$$
(possibly equal to $1$ if $N$ or $N'$ is equal to $1$) and 
$$d_\Lambda=\dLambda,\ d_k={k-1}$$
(the formal degrees of $\pi_\infty$ and $\pi_\infty'$ respectively).
\end{thm}

We call the conditions \eqref{stablerange} and \eqref{stablerange2}, on the relative sizes of $N'$, $N$ and $k$, the {\em stable ranges}. As we will see these conditions imply that the error term $o_{\delta,E}(1)$ in \eqref{m1} decays exponentially in $k$ as $k\ra\infty$ or by a positive power of $N$ (with an exponent linear in $k$) as $N\ra\infty$. 

\begin{remark} 
 The first stable range \eqref{stablerange} is reminiscent to the stable range present in \cite{MiRam} and subsequently in \cite{FeiWhi}.
	
\end{remark}

\subsection{Quantitative non-vanishing}

 The lower bound \eqref{nonvanishpower} is an immediate consequence of the following {\em pointwise} upper bound  which we deduced from Theorem \ref{thmB} using the {\em amplification method}.
 
 For any $\pi \in \mcA_{k}(N)$, we denote by $\mcB^{\tilde\mfn}_{k,\pi}(N)\subset \mcB^{\tilde\mfn}_{k}(N)$ the subset of automorphic forms generating $\pi$ (see \S \ref{secspectralexp}). 
 
 \begin{thm}\label{upperboundperiodthm} With the same  notation as above; there exists an absolute constant $\delta>0$ such that for any $\pi\in \mcA_k(N)$ one has
\begin{equation}
	\label{upperboundperiodintro}
	\sum_{\vphi\in\mcB^{\tilde\mfn}_{k,\pi}\!(N)}\frac{\big|\mathcal{P}(\varphi,\varphi')\big|^2}{\langle \varphi,\varphi\rangle \langle\varphi',\varphi'\rangle}\ll_{N'}(kN)^{2-\delta}.
\end{equation}
 \end{thm}
 
 Specializing \eqref{m1} to $\ell=1$, we get 
$$\sum_{\varphi\in \mcB^{\tfn}_{k}(N)}\frac{\big|\mathcal{P}(\varphi,\varphi')\big|^2}{\langle \varphi,\varphi\rangle \langle\varphi',\varphi'\rangle}\asymp_{N'} (kN)^2;$$
this, together with \eqref{upperboundperiodintro} and \eqref{Lperiodprop}, yields Theorem \ref{thmnonvanishpower}.
 
 \begin{remark} As the proof will show, any fixed $\delta$ in the interval $(0,1/82)$ would work. We have not tried to optimize this exponent as it is probably very far from the truth.
 \end{remark}

Let us elaborate on the comment made at the end of \S \ref{secmainthm}. We have the upper bound
\begin{equation}
	\label{upperboudat1}
	L(1,\Ad,\pi_E)\ll_E(kN)^{o(1)};
\end{equation}
for this we use the fact that $\pi_E$ is cuspidal (because $\pi$ is Steinberg at $N$ hence its base change as well) which implies that $L(s,\Ad,\pi_E)$ is holomorphic at $s=1$ and since $\pi$ is tempered, $L(s,\Ad,\pi_E)$ is normally convergent for $\Re s>1$.

 We then observe that \eqref{upperboudat1} together with
 \eqref{upperboundperiodintro} and \eqref{Lperiodprop}  immediately imply the upper bound
$$L(1/2,\pi_E\times\pi'_E)\ll (kN)^{3-\delta+o(1)}.$$
However this bound, is {\em weaker} than the {\em convexity bound} 
\begin{equation}
	\label{convexityboundintro}
	L(1/2,\pi_E\times\pi'_E)\ll_N' (N^4k^{8})^{1/4+o(1)}=(kN)^{o(1)}k^2N,
\end{equation}
which is a consequence of temperedness and of the size of the analytic conductor of $L(s,\pi_E\times\pi'_E)$ at $s=1/2$ which is $\asymp_{N'} k^{8}N^4$ (see \eqref{eqcondbound}).

However we are unable to turn tables and use the, a priori stronger, bound \eqref{convexityboundintro}  to improve Theorem \ref{thmnonvanishpower}. The reason is that  we don't know how to obtain --unconditionally-- a good {\em lower bound} for $L(1,\Ad,\pi_E)$: one would expect that
\begin{equation}
	\label{siegelAdGL3}
	L(1,\Ad,\pi_E)=(kN)^{o(1)}
\end{equation}
 (see \cite{Ram} for a discussion about this problem and \cite{Brumley} for some unconditional, unfortunately not sufficient, lower bounds). If  \eqref{siegelAdGL3} were known it would give, by \eqref{Lperiodprop},
 \begin{equation}
 	\label{convexperiodbound}
 	\frac{\big|\mathcal{P}(\varphi,\varphi')\big|^2}{\langle \varphi,\varphi\rangle \langle\varphi',\varphi'\rangle}\ll_{N'}(kN)^{o(1)}k
 \end{equation}	
 and by \eqref{firstmomentn}
$$|\{\pi\in \mcAkn(N),\ L(1/2,\pi_E\times\pi'_E)\not=0\}|\gg_{N'} (kN)^{o(1)}kN^2\geq |\mcAkn(N)|^{2/3-o(1)}.$$
It would be  interesting  to obtain the "convexity" bound \eqref{convexperiodbound} by a direct geometric analysis of the period integral $\mathcal{P}(\varphi,\varphi')$.

\subsection{Galois representations, Ramanujan \& the Bloch-Kato conjectures}\label{sectempered}
As pointed out earlier, the problem of exhibiting non-vanishing of  central values within certain families of $L$-functions has important applications related to the Birch-Swinnerton-Dyer conjectures. In our case the relevant context is  the {\em Bloch-Kato conjectures}. We explain these connections here along with the fact that the $\pi$ and $\pi'$ we consider are tempered.

\subsubsection*{On temperedness}
It is classical and due to Deligne that for any prime $\ell$ there is an $\ell$-adic Galois representation, $V_\ell(\pi')$, associated to $\pi'$ whose Frobenius eigenvalues at  primes $v\nmid\ell N'$ equal (up to an appropriate twist) the Langlands parameters of $\pi'_v$. As  $V_\ell(\pi')$ occurs in the cohomology of a certain Kuga-Sato modular variety, this implies, by Deligne's Weil II, the purity of the Frobenius ${\mathrm{Frob}}_v$ and that $\pi_v'$ is tempered; varying $v$, that $\pi'$ is tempered everywhere (the Ramanujan-Petersson conjecture).

For $\pi$, which is regular cohomological, the association of  an $\ell$-adic Galois representation $V_\ell(\pi)$ is due to the works of Rogawski, Kottwitz et al, and the proofs are assembled in  \cite{LR} (see Chap. 7 Thms A and B). If $\pi$ is stable, the associated $3$-dimensional Galois representation $V_\ell(\pi)$ occurs in the coho\-mo\-logy in degree $2$ of a modular Picard surface with locally constant coefficients: by the work of Deligne, this implies the purity of the Frobenius ${\mathrm{Frob}}_v,\ v\nmid\ell N$ and eventually the temperedness of $\pi$ at every place.

On the other hand, when $\pi$ is endoscopic, the Galois representation occurring in the cohomology need not be $3$-dimensional; however due to our infinity type, the archimedean parameter forces it to come from a representation $\pi_1\times \xi$ of $U(1,1)\times U(1)$ with $\xi$ unitary and $\pi_1$ in the discrete series at infinity, in fact of the same weight $2k$. So again $\pi$ is tempered because $\pi_1$ and $\xi$ are.

Note that regarding temperedness, $\pi$ being cohomological is not sufficient and we  need to use that $\pi$ occurs in the middle degree cohomology of Picard modular surfaces. By contrast those occurring in degree $1$ are always non-tempered.

\subsubsection*{On the Bloch-Kato conjecture}
Since $\pi$ and $\pi'$ are cuspidal, the representations $V_\ell(\pi)$,  $V_\ell(\pi')$ are irreducible and even absolutely irreducible as neither $\pi$ nor $\pi'$ admits self twists (because of the Steinberg components at $N'$ and $N$). The same holds modulo $\ell$ for $\ell$ large enough.

The tensor product $V_\ell(\pi)\otimes V_\ell(\pi')$ is also absolutely irreducible : again the Steinberg components at the distinct $N$ and $N'$ prevent $\pi$ from being  a twist of the symmetric square of $\pi'$.

To the latter representation is associated a Bloch-Kato Selmer group  $H_f^1(V_\ell(\pi)\otimes V_\ell(\pi')(*))$ (here $(*)$ is a suitable Tate twist depending on $k$) and the Bloch-Kato conjecture predicts that
  if $L(1/2,\pi_E\times\pi'_E)$ does not vanish then this Selmer group is zero.
 
In particular Theorem \ref{thmnonvanishpower} would imply the existence of  infinitely many of these Selmer groups having rank $0$.

Recently Y. Liu, Y. Tian , L. Xiao, W. Zhang and X. Zhu \cite{Zhu} have established the Bloch-Kato conjecture when $\pi$ and $\pi'$ are regular algebraic (as we have here) but cohomological with trivial coefficients whenever $V_\ell(\pi)\otimes V_\ell(\pi')$ is absolutely irreducible and $V_\ell(\pi)$, $V_\ell(\pi')$ are residually irreducible (they also required the presence of Steinberg components --which we have-- and of a supercuspidal component). 

The trivial coefficients condition forces $k$ to be $2$ in our situation which we do not consider to avoid some technical difficulties (that may be serious). We are happy on the other hand to hear from X. Zhu that the methods developed in \cite{Zhu} extend to non-trivial coefficients (moreover without requiring the existence of supercuspidal components) and that such an extension is in the making.

\subsection{Relations to other works on moments of $L$-functions and periods}

 The problem of evaluating moments of $L$-functions involving families of automorphic forms on groups of higher rank is difficult and there are not many positive results. One may think of the work of X. Li \cite{Li11} involving $\GL(3)\times\GL(2)$ Rankin-Selberg $L$-functions as well as the work of Blomer-Khan \cite{BK15} and Qi \cite{Qi20} in a similar context. However a common feature of these works is that the moments are on average over the automorphic forms of the smaller group $\GL_2$. Closer to the spirit of this paper is the work of Blomer-Buttcane \cite{BB} who estimated the fourth moment of standard $L$-functions on average over families of $\GL_3$-automorphic representations with large archimedean parameters and obtain subconvex bounds. Another is the work of Nelson-Venkatesh \cite{NelVen} who build on their substantial development of microlocal calculus on Lie groups, and obtain an asymptotic formula for the first moment of Rankin-Selberg $L$-functions $L(1/2,\pi_E\times\pi'_E)$ associated with pairs of unitary groups $U(n+1)\times U(n)$ (for any $n\geq 2$) on average over families of $U(n)$-automorphic forms with large archimedean parameters in general position. In \cite{Nel}, these methods were expanded further, and Nelson  succeeded in evaluating the first moment above, this time on average over suitable families of $U(n+1)$-automorphic forms; finally in  \cite{Nelsonstandard}, Nelson treated the degenerate case of  {\em split} unitary groups (relative to $E=\Qq\times\Qq$, so that $G\times G'=\GL(n+1)\times\GL(n)$) and with $\pi'$ being an Eisenstein series representation: this gave bounds for the $n$-th moment of standard $\GL(n+1)$ $L$-functions $L(1/2,\pi)$ with $\pi$ having large archimedean parameters in generic position (so as to avoid the "conductor dropping phenomenon").
 
Further in that direction, Jana-Nunes and the third named author independently, have obtained  asymptotic formulae for weighted moments of central values of products of $\GL(n+1)\times\GL(n)$ Rankin-Selberg $L$-functions $L(1/2,\pi\times\pi'_1)\ov{L(1/2,\pi\times\pi'_2)}$ for a pair of possibly varying cuspidal representations $\pi'_1,\pi'_2$ and on average over $\pi$'s when their spectral parameters are in generic position (avoiding the conductor dropping phenomenon)\ \cite{Yan23,yang2023relative,jana2023moments}. Let us point out that the situation we consider here is very non-generic and indeed the conductor of our degree $12$ $L$-functions drop significantly (see \S \ref{seccond}).

\subsection{Idea of Proofs and Structure of the Paper}

The main ingredient of this work is the {\em relative trace formula  of Jacquet and Rallis} for the pair $(G,G')$ for a suitable choice of test functions (described in \S \ref{secfnchoice}). As in \cite{Nel}, our treatment differs from the traditional  uses of the relative trace formula in functoriality (such as \cite{BPLZZ19}) where one compares the geometric sides of two instances of the RTF to deduce consequences for the spectral sides) as we evaluate the geometric side by direct arguments. As pointed out above, such an approach was initiated by Rogawski and the second named author in \cite{RR05} when they rederived, with a delineation of the underlying measure, the non-vanishing result of Duke.

The knowledgeable reader will have noted that the pair $({\rm GL}(2),T)$ is not very far from to  unitary group case $U(1,1)\times U(1)$ and we observe both similarities and discrepancies when passing to the  $U(2,1)\times U(1,1)$ case. A similarity with \cite{RR05} is that the main terms come from the contributions of the identity and unipotent cosets while the regular coset contribution is an error term\footnote{ To be precise, in \cite{RR05}, the identity coset contribution vanishes identically but this is only because the evaluation was the average over a basis of $\GL(2)$ automorphic forms, of the product of two Hecke periods twisted by two {\em distinct} characters; would these two characters have been equal this would have resulted been a main term.}. It is worth noting, however that in the present case, the unipotent coset contribution is {\em significantly smaller} than the main term: by a factor which is at least a positive power of the size of the family $\mcA_k(N)$, while in \cite{RR05} the difference is at most by a logarithmic factor. Another important difference is that the treatment of the regular orbital integrals is quite a bit more involved. We proceed by reducing the problem to bounding local integrals which we do by splitting into many subcases.   In the present paper, we evaluate the average of the product $\mcP(\vphi,\vphi'_1)\ov{\mcP(\vphi,\vphi'_2)}$ for $\vphi'_1$ and $\vphi'_2$ belonging to the {\em same} representation $\pi'$; it turns out that, most of the time, the identity contribution is non-zero and in fact dominates the unipotent contribution. As in \cite{RR05}, if $\vphi'_1$ and $\vphi'_2$ belong to distinct representations, one can check easily that the identity contribution  vanishes (because $\vphi'_1$ and $\vphi'_2$ are orthogonal). As for the unipotent contribution, we expect  it to become the dominant term (proportional to $L(1,{\pi}'_{1,E}\times {\pi}'_{2,E})$); this will lead to simultaneous non-vanishing results analogous to those of \cite{RR05,jana2023moments,yang2023relative}, namely the existence of $\pi$ for which  
$$L(1/2,\pi_E\times{\pi}'_{1,E})L(1/2,\pi_E\times{\pi}'_{2,E})\not=0.$$
We will come back to this problem in a forthcoming work.

Let us now provide a few more details. We will allow ourselves, in this introduction, to be a bit imprecise and write things which are only ``morally'' true. So the sketch should not be taken as a precise reflection of the details of our argument.
\medskip

Let $\varphi'$ be a primitive holomorphic cusp form on $G'(\mathbb{A})$ of weight $k$, level $N'$ and trivial central character.  Let $\pi'$ be the corresponding cuspidal representation. We consider Jacquet's relative trace formula which takes the shape
\begin{equation}\label{rtf}
\text{Spectral Side}=\int_{[G']}\int_{[G']}\K(x,y)\varphi'(x)\overline{\varphi'(y)}dx\,dy=\text{Geometric Side},
\end{equation}
where $\K(x,y)=\K^{f}(x,y)$ is the kernel function of the Hecke operator $R(f)$ associated to a test function $f$ and $[G]$ denotes the quotient $G(\Qq)\bash G(\Aa)$ for any algebraic group over $\Qq$; see Section \ref{sec2} for details. Note that \eqref{rtf} can be thought of as a `section' of the Jacquet-Rallis trace formula \cite{JR11}. In Sec. \ref{secfnchoice} we construct an explicit test function $f^{{\mathfrak{n}}}$ and apply it to \eqref{rtf} to compute/estimate both sides. The archimedean component of this test function is the simple matrix coefficient of $\pi_\infty$ and we exploit a specific feature of it, namely that its restriction to $U(1,1)$ coincides with the matrix coefficient of $\pi'_\infty$. This allows us to vary the weight $k$ and the levels $N, N'$ in a somewhat hybrid fashion.
\medskip

The spectral side of \eqref{rtf} is handled in Sec. \ref{S}. We show that the operator $R(f^{{\mathfrak{n}}})$ eliminates the non-cuspidal spectrum so that the spectral side \eqref{m1} becomes a (finite) second moment of Bessel periods relative to specific holomorphic cusp forms on $G$ and $G'$. 

For these automorphic forms, we use the recent work \cite{BPLZZ19} to obtain an explicit Gan-Gross-Prasad formula of Ichino-Ikeda type for $G\times G'$ relating the central $L$-values $L(1/2,\pi_E\times\pi_E')$ to local and global period integrals. For this, we need to compute explicitly several integrals of local matrix coefficients; this is done in  \S \ref{4.3} in the Appendix, and the main result in this section is Proposition \ref{prop33}. With this, one can write the spectral side of \eqref{rtf} as a weighted sum of central $L$-values $L(1/2,\pi_E\times\pi_E')$.
\medskip

Next we evaluate the geometric side which is a sum of orbital integrals indexed by the double quotient $G'(\Qq)\bash G(\Qq)/G'(\Qq)$. In Sec. \ref{Geo} we decompose these  orbital integrals into three subsets according to the properties of the classes in the quotient: the identity element, the unipotent type and regular type; a priori, there also could be a term associated with an element of the shape $s.u$ with $s,u$ non-trivial and respectively semisimple and unipotent but luckily, with our choice of global double coset representatives (Proposition \ref{repres}) such a term does not occur. So \eqref{rtf} becomes
\begin{equation}\label{RTF}
\text{Geometric Side}= \text{Identity Orb}+\text{Unipotent Orb}+\text{Regular Orb},
\end{equation}
where `Orb' refers to orbital integrals. The first  term is made of a single orbital integral, the second is a finite sum of unipotent orbital integrals while the third term is an infinite sum of regular orbital integrals. They will be handled by different approaches in the subsequent sections.
\medskip

The identity orbital integral is calculated in Section \ref{SecIdentity}. Its contribution provides the main term in Theorem \ref{thmB}. In Section \ref{Secunipotent}, we estimate the unipotent orbital integral by local computations. The contribution from this orbital integral gives a second main term on the right-hand side of \eqref{m1} which decays exponentially fast as $k$ grows. Lastly, the more involved regular orbital integrals are investigated in Sections \ref{sec7} and \ref{sec8.3}. The main result in this part is Theorem \ref{regularglobalbound}, which provides an upper bound for the infinite sum of the regular orbital integrals. A particular feature of this bound is that in the stable range \eqref{stablerange} the contribution of the regular orbital integrals again decays exponentially fast with $k$.
\medskip

Gathering these estimates in Section \ref{SecMainThmPf}, we then prove Theorem \ref{thmB} in its more precise form, Theorem \ref{firstmomentwithl}. 

In \S \ref{STsec} we also interpret Theorem \ref{firstmomentwithl} as a vertical Sato--Tate type equidistribution result for the Hecke eigenvalues of the $\pi$ at a finite fixed set of inert primes weighted by the periods $|\mcP(\vphi,\vphi')|^2$. This is inspired by the work of Royer \cite{Royer} 
who obtained vertical Sato--Tate type equidistribution results for Hecke eigenvalues of holomorphic modular forms of weight $2$ and large level weighted by the Hecke $L$-values $L(1/2,f)$. 

Notice that in \cite{Royer}, Royer combined his results with a technique of Serre \cite{Serre} to exhibit irreducible factors of $\mathrm{Jac}(X_0(N))$ of dimension $\gg \log\log N$ and  rank $0$ (or of rank equal to the dimension). We expect that  the ongoing work of X. Zhu and his collaborators will make it possible to obtain  results of similar flavor.

Combining Theorem \ref{firstmomentwithl} with Proposition \ref{prop33},  one can deduce  Theorem \ref{thmA}; however  we need to be able to average only over {\em newforms} (if $N>1$). In \S \ref{secoldnew} we show  that the oldforms contribution is indeed smaller.
 
In \S \ref{non-vanish} we prove Theorem \ref{upperboundperiodthm} using Theorem \ref{firstmomentwithl} and the amplification method. Using again Proposition \ref{prop33}, we eventually prove Theorem \ref{thmnonvanishpower}.


\section{\bf Notations}\label{2.1}
\setcounter{subsection}{-1}
\subsection{Adelic quotient} For any algebraic group $G$ defined over $\Qq$ we write
\begin{equation}
	\label{bGdef}
	[G]:=G(\Qq)\bash G(\Aa).
\end{equation}

\subsection{Vinogradov Symbols}\label{vino} Given $A,B:\mcX\ra\Cc$ two complex valued functions on a set $\mcX$, we use the notation
$$A\ll B$$ to mean that there exists positive constant $C>0$ such that
$$\forall x\in\mcX,\ |A(x)|\leq C|B(x)|.$$
Likewise
$$A\asymp B$$
 means that there exist positive constants $0<c<C$ such that
$$\forall x\in\mcX,\ c|A(x)|\leq |B(x)|\leq C|A(x)|.$$

In particular, $A$ and $B$ have the same support. 

If  $\mcX$ and $A,B$ belong to families of sets $(\mcX_E)$ and functions $(A_E,B_E:\mcX_E\ra \Rr)$ indexed by some parameter $E$, we write
	$$A\ll_E B,\hbox{ or } A\asymp_E B$$
	to mean that there exist functions $E\ra c_E,C_E\in\Rr_{>0}$ such that
	$\forall E,\forall x\in\mcX_E$,
	  $$|A_E(x)|\leq C_E|B_E(x)|\hbox{ or }c_E|A_E(x)|\leq |B_E(x)|\leq C_E|A_E(x)|.$$

\subsection{The quadratic field $E$}

Let $E=\mathbb{Q}(\sqrt{-D})\hookrightarrow \Cc$ be an imaginary quadratic field; we denote by $\eta$ the associated Legendre symbol which we view indifferently as a quadratic Dirichlet character, a character on the group of id\`eles or  a character on the Galois group of $\Qq$. We denote the Galois involution by $\sigma\in\Gal(E/\Qq)$; it will also be useful to write
$$\sigma(z)=\ov z.$$
The trace  and the norm are denoted by
$$z\mapsto \tr_{E/\Qq}(z)=z+\ov z, z\mapsto \Nr_{E/\Qq}(z)=z.\ov z$$
respectively. 

We denote by $E^\times$ the multiplicative group of invertible elements and by $$E^1=\{z\in E^\times,\ z.\overline z=1\}\subset E^\times$$ the subgroup of norm 1 elements; whenever useful we will denote in the same way the corresponding $\Qq$-algebraic groups.

\subsubsection{Integers}
Let $\mcO_E$ be the ring of integers of $E$ and $$\mathcal{O}_E^1=E^1\cap \mathcal{O}_E\subset \mcO_E^\times$$ the group of units; set $w_E:=\# \mathcal{O}_E^1$. 

Let $D_E<0$ be the discriminant of $\mcO_E$; we set 
\begin{equation}
	\label{Deltadef}
	\Delta=i|D_E|^{1/2}.
\end{equation}
The fractional ideal $\mcD_E^{-1}=\Delta^{-1}\mcO_E$ is the different: the $\mathbb{Z}$-dual of $\mathcal{O}_E$ with respect to the trace bilinear form $$(z,z')\mapsto \tr_{E/\mathbb{Q}}(zz').$$

\subsection{The Hermitian space and its unitary group}\label{SecHerm}
Let $V$ be a  3-dimensional vector space over $E,$ with basis $\{e_{1}, e_{0}, e_{-1}\}.$ Let $\peter{\cdot,\cdot}_J$ be a Hermitian form on $V$ whose matrix  with respect to $\{e_{1}, e_{0}, e_{-1}\}$ is 
\begin{equation}\label{Jmdef}
J=\left(
\begin{array}{ccc}
&&1\\
&1&\\
1&&
\end{array}
\right).
\end{equation}
We denote by $$G=U(V)$$ the unitary group preserving the form $\peter{\cdot,\cdot}_J$. This is an algebraic group defined over $\Qq$ and for any $\mathbb{Q}$-algebra $R$, the group of its $R$-points is  
$$G(R)=\bigl\{g\in \GL(V\otimes_\Qq R),\ \transp{\overline g}Jg=J \bigr\}.$$
The center of $G$ is denoted $Z_G$ and made of the diagonal hermitian matrices
$$Z_G(R)=\big\{\begin{pmatrix}
	z&&\\&z&\\&&z
\end{pmatrix},\ z\in E^1(R)\big\}$$
so that
$$Z_G\simeq E^1= U(1)$$
The special hermitian subgroup is noted $\SU(V)$ and its $R$-points are given by
$$\SU(V)(R)=\{g\in U(V)(R),\ \det g=1\}.$$

Also, for $n=p+q$ with $p,q\geq 0$, we denote by $U(p,q)$ the unitary group for the space $E^n$ equipped with the hermitian form $\peter{\cdot,\cdot}_{p,q}$ with $n\times n$ matrix
\begin{align}
J_{p,q}:=\left(
\begin{array}{ccc}
\Id_p&\\
&-\Id_q
\end{array}
\right).\label{Jpqdef}
\end{align}
In other terms
$$U(p,q)(R)=\bigl\{g\in \GL(V\otimes_\Qq R),\ \transp{\overline g}J_{p,q}g=J_{p,q} \bigr\}.$$
We set $\SU(p,q)$ its special subgroup of elements of determinant $1$. As usual we write $U(n)$ and $SU(n)$ for $U(n,0)$ and $SU(n,0)$. In particular we have
$$U(1)=E^1.$$

In fact, in this paper, excepted for $(p,q)=(1,0)$, we will only need the $\Rr$-points of the groups $U(p,q)$ so to shorten the notation, we will often write
$$U(p,q)\hbox{ for }U(p,q)(\Rr).$$

\subsubsection{The subgroup $G'$}\label{secG'def}
Let $G'\leq G$ be the stabilizer of the anisotropic line $\peter{e_0}.$ Then $G'$ also preserves $$W=\langle e_0\rangle^{\bot}=\langle e_{1},e_{-1}\rangle,$$ the orthocomplement of $\peter{e_0}.$ Note that $W$ is a 2-dimensional Hermitian space, whose Hermitian form matrix is 
\begin{equation}\label{J'def}
	J':=\left(
\begin{array}{cc}
&1\\
1&
\end{array}
\right)
\end{equation} with respect to the basis $\{ e_{1},e_{-1}\}$. Hence we have an isomorphism of $\Qq$-algebraic subgroups $U(W)\simeq G'$ via the embedding
\begin{equation}\label{emd}
i:\quad \left(
\begin{array}{cc}
a&b\\
c&d
\end{array}
\right)\mapsto\left(
\begin{array}{ccc}
a&&b\\
&1&\\
c&&d
\end{array}
\right),\quad \left(
\begin{array}{cc}
a&b\\
c&d
\end{array}
\right)\in U(W\otimes_\Qq R),
\end{equation}
for any $\mathbb{Q}$-algebra $R.$ \medskip
We will identify $U(W)$ with $G'$ henceforth. In particular we will sometimes represent an element of $G'$ as a $2\times 2$ matrix (its matrix in the above basis).

 We also set 
\begin{equation}\label{Hdef}
H:=G'\times G'\subset G\times G.	
\end{equation}

\subsubsection{Convention regarding the split places}\label{secchoice}

Let $p$ be a finite prime. For any $\Qq$-algebra $R$ we denote by $R_p$ its completion with respect to the $p$-adic valuation. If $p$ is split, we have a decomposition $p\mcO_E=\mfp.\ov\mfp$ into a product of distinct prime ideals of $\mcO_E$. Let us choose such a prime, say $\mfp$.
The injection $\Qq\hookrightarrow E$ of fields induces isomorphisms $$E_\mfp\simeq \Qp,\ V_\mfp=E_\mfp.e_{1}\oplus E_\mfp.e_0\oplus E_\mfp.e_{-1}\simeq \Qp.e_{1}\oplus\Qp.e_0\oplus \Qp.e_{-1}$$ 
and an isomorphism of linear groups
\begin{equation}\label{isomlinear}
	G(\Qp)=U(V)(\Qp)\simeq \GL(3,\Qp).
\end{equation}
For every split prime, we make a such choice, once and for all, and represent the elements of $G(\Qp)$ as $3\times 3$ matrices with coefficients in $\Qp$. Likewise we represent the elements of $G'(\Qp)$ either as $2\times 2$ or $3\times 3$ matrices with coefficients in $\Qp$ (with a $1$ as central coefficient for the latter).

\section{\bf A Relative Trace Formula on $U(2,1)$}\label{sec2}

\subsection{Recollection of the general principles of the trace formula on $U(V)$}
Denote by $\mathbb{A}$ the adele ring of $\mathbb{Q}.$ Let $\mathcal{A}_0(G)$ be the space of cuspidal representations on $G(\mathbb{A}).$ In this section, we will introduce the framework of a relative trace formula on $U(V)$ for general test functions. We give the coarse geometric side of the trace formula, regardless of  convergence issues. In Sec. \ref{secfnchoice} we will specify our test function. Further careful analysis and computation towards the trace formula will be provided in following sections.

\subsubsection{Automorphic Kernel}Let $K_\infty \subset G(\Rr)$ be a maximal compact subgroup ($K_\infty\simeq U(2)(\Rr)\times U(1)(\Rr)$). We consider a smooth function $h\in C_c^{\infty}(G(\mathbb{A}))$ which is left and right $K_\infty$-finite, transforms by a unitary character $\omega$ of $Z_G\left(\mathbb{A}\right).$ Denote by $\mathcal{H}(G(\mathbb{A}))$ the space of such functions. Then $h\in \mathcal{H}(G(\mathbb{A}))$ defines a convolution operator
\begin{equation}\label{62}
R(h)\varphi=h*\vphi: x\mapsto \int_{G(\mathbb{A})}h(y)\varphi(xy)dy,
\end{equation}
on the space $L^2\left(G(F)\backslash G(\mathbb{A}),\omega^{-1}\right)$ of functions on $G(F)\backslash G(\mathbb{A})$ which transform under $Z_{G}(\mathbb{A})$ by $\omega^{-1}$ and are square integrable on $G(F)\backslash G(\mathbb{A}).$ This operator is represented by the kernel function
\begin{equation}\label{kernel}
\K^{h}(x,y)=\sum_{\gamma\in G(\mathbb{Q})}h(x^{-1}\gamma y).
\end{equation}
When the test function $h$ is clear or fixed, we simply write $\K(x,y)$ for $\K^{h}(x,y).$
\medskip

It is well known that $L^2\left(G(\mathbb{Q})\backslash G(\mathbb{A}),\omega^{-1}\right)$ decomposes into the direct sums of the space $L_0^2\left(G(\mathbb{Q})\backslash G(\mathbb{A}),\omega^{-1}\right)$ of cusp forms and spaces $L_{\Eis}^2\left(G(\mathbb{Q})\backslash G(\mathbb{A}),\omega^{-1}\right)$ and $L_{\Res}^2\left(G(\mathbb{Q})\backslash G(\mathbb{A}),\omega^{-1}\right)$ defined using Eisenstein series and residues of Eisenstein series respectively and the operator $\K$ splits up as: 
$$\K=\K_0+\K_{\Eis}+\K_{\Res}.$$

Let the notation be as above. Given $\vphi'$ a cuspidal automorphic form on $G'(\Aa)$ we consider the distribution 
\begin{equation}\label{Jdef}
h\mapsto J(h):=\int_{H(\mathbb{Q})\backslash H(\mathbb{A})}\K^h(x_1,x_2)\varphi'(x_1)\ov{\varphi}'(x_2)dx_1dx_2,
\end{equation}
where $d$ refers to the Tamagawa measure;
more generally for  $*\in\{0,\Eis, \Res\}$, we set \begin{equation}\label{223}
h\mapsto J_{*}(h):=\int_{H(\mathbb{Q})\backslash H(\mathbb{A})}\K^h_*(x_1,x_2)\varphi'(x_1)\ov{\varphi}'(x_2)dx_1dx_2.
\end{equation}
 Typically, because of convergence issues, one needs to introduce some regularization in \eqref{223} to make these expressions  well defined and then we have
 $$J(h)=J_0(h)+J_{\Eis}(h)+J_{\Res}(h).$$
In our situation we have (\S \ref{secspectralexp})
\begin{equation}
	\label{Eisvanishing}
	J_{\Eis}(h)=J_{\Res}(h)=0
\end{equation}
      so that $$J(h)=J_0(h).$$
  Since $\K^h$ is the kernel of $R(h)$, $J(h)$ admits a spectral expansion, i.e. is a weighted sum, over an orthogonal family $\vphi\in \mcB^{\tilde\mfn}_k(N)$ of cuspforms of level $N$ and weight $k$, of the periods squared $|\mcP(\vphi,\vphi')|^2$ (see Lemma \ref{lem34}). We refer to this expression as the spectral side of the relative trace formula.

\subsubsection{Geometric Reduction}

Assume now that $\omega=1.$ 

Let $\Phi$ be a set of representatives of the double quotient $G'(\mathbb{Q})\backslash G(\mathbb{Q})/ G'(\mathbb{Q}).$ For each $\gamma\in \Phi,$ we denote by $${H}_{\gamma}=\{(u,v)\in H=G'\times G',\ u^{-1}\gamma v=\gamma\}$$ its stabilizer in $G'\times G'$.  Then one has (assuming that everything converges absolutely) 
\begin{align}\nonumber
J(h)=&\int_{H(\mathbb{Q})\backslash H(\mathbb{A})}\sum_{\gamma\in G(\mathbb{Q})}h(x_1^{-1}\gamma x_2)\vphi'(x_1)\ov{\vphi}'(x_2)dx_1dx_2\\
=&\int_{{H}(\mathbb{Q})\backslash {H}(\mathbb{A})}\sum_{\gamma\in \Phi}\sum_{\delta\in [\gamma]}h(x_1^{-1}\delta x_2)\vphi'(x_1)\ov{\vphi}'(x_2)dx_1dx_2.\label{221}
\end{align}

Therefore (assuming that everything converges absolutely) one can write $J(h)$ as
$$
J(h)=\int_{{H}(\mathbb{Q})\backslash {H}(\mathbb{A})}\sum_{\gamma\in \Phi}\sum_{(\delta_1,\delta_2)\in {H}_{\gamma}(\mathbb{Q})\backslash {H}(\mathbb{Q})}h(x_1^{-1}\delta_1^{-1}\gamma\delta_2 x_2)\vphi'(x_1)\ov{\vphi}'(x_2)dx_1dx_2.
$$
Then switching the sums and noticing the automorphy of $\varphi',$ one then obtains
\begin{equation}\label{220}
J(h)=\sum_{\gamma\in \Phi}\int_{{H}_{\gamma}(\mathbb{Q})\backslash {H}(\mathbb{A})}h(x_1^{-1}\gamma x_2)\vphi'(x_1)\ov{\vphi}'(x_2)dx_1dx_2.
\end{equation}
In Sec. \ref{secglobalf} (see \eqref{globalfnchoice}) we will choose $h=f^{{\mathfrak{n}}}$ precisely  to make \eqref{221} converge absolutely so that \eqref{220} holds rigorously.
\medskip


\section{\bf Choice of Local and Global Data}\label{secfnchoice}

In this section, we describe our choices of the test function $h$ and the automorphic form $\vphi'$ so that the relative trace formula captures the family of automorphic forms indicated above.

The test function $h\in \mcC^\infty_c(G(\Aa))$ will be a linear combination of factorable test functions of the shape
\begin{equation}\label{hchoice}
f^{{\mathfrak{n}}}=f_{\infty}\otimes\otimes_{p}'f^{{\mathfrak{n}}}_p.	
\end{equation}
 
 The non-archimedean components $f^{{\mathfrak{n}}}_p$ are discussed in \S \ref{seclevel}. As we will see these components also depend (in addition to $N$ and $N'$) on an integer $\ell\geq 1$ coprime to $NN'D$.

The archimedean component $f_{\infty}$ is discussed in the next subsections. It is obtained from matrix coefficients of holomorphic discrete series of $U(2,1)$. 

\subsection{Holomorphic Discrete Series Representation of $U(2,1)$}\label{sec3.1}
Let us recall that there are three types of discrete series of $U(2,1)$ which embed in the non-unitary principal series, namely the holomorphic, the antiholomorphic, and the nonholomorphic discrete series. A full description of these three discrete series and models for their respective representation spaces can be found in \cite{Wal76}. In this paper, we will focus on holomorphic discrete series.

We recall that $U(2,1)$ is the unitary group of the hermitian space $\Cc^3$ with Hermitian form given by the matrix \begin{align*}
J_{2,1}:=\left(
\begin{array}{ccc}
1&&\\
&1&\\
&&-1
\end{array}
\right)
\end{align*}
Its maximal compact subgroup
 is $$U(2,1)\cap U(3)=\left(
\begin{array}{cc}
U(2)&\\
&U(1)
\end{array}
\right)\simeq U(2)\times U(1)$$

This relates to the group $G(\Rr)$ as follows: let
 \begin{equation}\label{b}
\rmB=\left(
\begin{array}{ccc}
1/\sqrt{2}&&1/\sqrt{2}\\
&1&\\
1/\sqrt{2}&&-1/\sqrt{2}
\end{array}
\right).
\end{equation} 
Then we have ${\rmB}={\rmB}^{-1}$,
$
J_{2,1}={\rmB}J{\rmB}^{-1}
$
and we have an isomorphism
\begin{equation}\label{7}
\iota_{{\rmB}}:\ G(\mathbb{R})\overset{\sim}{\longrightarrow} G_{J_{2,1}}(\mathbb{R}),\ \quad g\mapsto{\rmB}g{\rmB}^{-1}.
\end{equation}
Consequently the maximal compact subgroup of $G(\Rr)$ equals 
\begin{equation}\label{K}
K_{\infty}={\rmB}\left(
\begin{array}{cc}
U(2)(\mathbb{R})&\\
&U(1)(\mathbb{R})
\end{array}
\right){\rmB}^{-1}.
\end{equation}

\subsubsection{Holomorphic Discrete Series on $SU(2,1)$}
Let $SU(2,1)\subset U(2,1)$ be its special subgroup. It's maximal compact subgroup is noted $${K_{\infty,2,1}}=SU(2,1)\cap U(3)=\Bigg\{\left(
\begin{array}{cc}
u&0\\
0&(\det u)^{-1}
\end{array}
\right):\ u\in U(2)\Bigg\}\simeq U(2).$$
Since $\rank SU(2,1)=\rank {K_{\infty,2,1}}=1,$ $SU(2,1)$ has discrete series representations. In this section we recall that explicit description of holomorphic discrete series given by Wallach \cite{Wal76}.

 Let $S^3$ be the unit sphere 
$$S^3=\big\{z=\transp{(z_1,z_2)}\in\mathbb{C}^2:\ |z_1|^2+|z_2|^2=1\big\}.$$
We have an homeomorphism $S^3\simeq SU(2)(\Rr)$ 
\begin{equation}
	\label{isomSU(2)}
u:	\transp{(z_1,z_2)}\in S^3\mapsto u(z_1,z_2)=\left(
\begin{array}{cc}
z_1&-\overline{z}_2\\
z_2&\ov z_1
\end{array}
\right)\in SU(2).
\end{equation}

The group $SU(2,1)(\Rr)$ acts on $S^3$ via
\begin{equation}\label{7.1}
g.z=\frac{Az+b}{\langle z,\transp{\overline{c}}\rangle+d},\quad g=\left(
\begin{array}{cc}
A&b\\
c&d
\end{array}
\right)\in SU(2,1)(\mathbb{R}),
\end{equation}
(here the $\langle\cdot,\cdot \rangle$ is the usual hermitian product on $\mathbb{C}^2$).

Let $$\alpha_1=(1,-1,0)\hbox{ and }\alpha_2=(0,1,-1);$$ these form a basis of simple roots in the root system of $\mathfrak{sl}(3,\mathbb{C})$ relative to the diagonal $\mathfrak{h}.$ The basic highest weights for corresponding order are 
\begin{equation}\label{4.7}
\Lambda_1=(2/3,-1/3,-1/3)\hbox{ and }\Lambda_2=(1/3,1/3,-2/3).
\end{equation}

For $(k_1,k_2)\in\mathbb{Z}^2$, let $$\Lambda=k_1\Lambda_1+k_2\Lambda_2\in \mathfrak{h}^*.$$ For ${h}\in C^{\infty}(S^3)$ we define
\begin{align*}
(\pi_{\Lambda}(g)({h}))(z):=a(g,z)^{k_1}\overline{a(g,z)}^{k_2}{h}(g^{-1}.z)
\end{align*}
where
$$a(g,z)=\overline{d}-\langle z,b\rangle$$
for $z$ and $g$ as in \eqref{7.1}. Then $\pi_{\Lambda}$ extends to a bounded operator on $L^2(S^3)$ and $(\pi_{\Lambda},L^2(S^3))$ defines a continuous representation of $SU(2,1)(\Rr).$

This representation restricted to the compact subgroup ${K_0}$ decomposes  into irreducible as follows: let $p, q$ be nonnegative integers and $\mathcal{H}^{p,q}$ be the space of polynomials $h\in \mathbb{C}[z_1, z_2, \overline{z}_1,\overline{z}_2]$ which are homogeneous of degree $p$ in $z_1,$ $z_2,$ degree $q$ in $\overline{z}_1,\overline{z}_2;$ and harmonic, namely,
\begin{align*}
\Delta h=\left(\frac{\partial^2}{\partial z_1\partial\overline{z}_1}+\frac{\partial^2}{\partial z_2\partial\overline{z}_2}\right)h\equiv 0.
\end{align*}
Denote by $\mathscr{H}^{p,q}=\mathcal{H}^{p,q}\mid_{S^3}.$ Then $(\pi_{\Lambda}\mid_{{K_0}},\mathscr{H}^{p,q})$ is irreducible and $$L^2(S^3)=\oplus_{p\geq 0}\oplus_{q\geq 0}\mathscr{H}^{p,q}.$$

To describe the holomorphic discrete series, we also assume  that $k_1<0$ and $k_2\geq 0$. Let $\rho$ be the half sum of positive roots, i.e., $$\rho=(1,0,-1)=\Lambda_1+\Lambda_2.$$ Given integers $p\geq 0$ and $0\leq q\leq k_2$ we set
\begin{align*}
c_{p,q}(\Lambda)=&\prod^{p}_{k=1}\frac{\langle\Lambda+(k+1)\rho,\alpha_2\rangle}{\langle-\Lambda+(k-1)\rho,\alpha_1\rangle}\cdot \prod^{q}_{j=1}\frac{\langle \Lambda+(j+1)\rho,\alpha_1\rangle}{\langle-\Lambda+(j-1)\rho,\alpha_2\rangle}
\end{align*}
with each of the above products equal to $1$ if $p$ or $q=0$.

A straightforward calculation shows that $$c_{p,q}(\Lambda)=\prod_{k=1}^p\frac{k+k_2+1}{k-k_1-1}\cdot \prod_{j=1}^q\frac{j+k_1+1}{j-k_2-1}.$$
In particular $c_{p,q}(\Lambda)$ is well defined and  if $k_2+k_1+1<0$, it is nonvanishing. 

From these, we define an inner product $\langle\cdot,\cdot\rangle_{\Lambda}$  with respect to $\Lambda=k_1\Lambda_1+k_2\Lambda_2$ via the one on $L^2(S^3)$ as follows:
 given $h_1,h_2\in C^{\infty}(S^3),$ by spectral decomposition we can write
\begin{align*}
h_1=\sum h_{1,p,q},\quad h_2=\sum h_{2,p,q},\quad h_{1,p,q},\ h_{2,p,q}\in\mathscr{H}^{p,q};
\end{align*}
 set
\begin{equation}\label{inn}
\langle h_1,h_2\rangle_{\Lambda}:=\sum_{p}\sum_{q}c_{p,q}(\Lambda)\langle h_{1,p,q},h_{2,p,q}\rangle.
\end{equation}
Let $$V_{k_2}^+:=\{h\in C^{\infty}(S^3):\ h_{p,q}=0\ \text{if $q>k_2$}\},$$ 
(note that the inner product $\langle\cdot,\cdot\rangle_{\Lambda}$ with  $\Lambda=k_1\Lambda_1+k_2\Lambda_2$ is well defined on $V_{k_2}^+$) and
let $V_{+}^{\Lambda}$ be the Hilbert space completion of $V_{k_2}^+$ relative to the inner product \eqref{inn}.

The following parametrization of holomorphic discrete series of $SU(2,1)(\Rr)$ is due to Wallach \cite[p. 183]{Wal76}:

\begin{defn} Let 
\begin{equation}\label{4.9}
w_1:(x,y,z)\mapsto (y,x,z),\quad w_2:(x,y,z)\mapsto (x,z,y)
\end{equation}
be the simple Weyl reflections.
We say that the weight   $$\Lambda=k_1\Lambda_1+k_2\Lambda_2\in \mathfrak{h}^*$$
 is \textit{holomorphic} if for $1\leq i\leq 2$, we have $\langle\Lambda+\rho,w_1w_2\alpha_i\rangle>0$.
\end{defn}

\begin{prop}\label{prop2}
Let $\Lambda=k_1\Lambda_1+k_2\Lambda_2$ be an holomorphic weight in $\mathfrak{h}^*$, then $D_{\Lambda}^+:=\pi_{\Lambda}\mid_{V_+^{\Lambda}}$ is a unitary holomorphic discrete series representation of $SU(2,1)(\Rr).$ Moreover, the holomorphic discrete series representations of $SU(2,1)(\Rr)$ are of form $D_{\Lambda}^+.$
\end{prop}

Later we will need to compute explicitly some inner products $\langle h_{1,p,q},h_{2,p,q}\rangle$: after decomposing $h_{1,p,q}$ (resp. $h_{2,p,q}$) into a finite linear combination of monomials of the form $z_1^{a}z_2^{b}\overline{z}_1^{c}\overline{z_2}^{d}$ one can use the following
\begin{lemma}
	\label{cal1}
Let the notation be as above. Let $a, b, c, d$ be nonnegative integers. Then
\begin{equation}\label{cal2}
\langle z_1^{a}z_2^{b},{z}_1^{c}{z_2}^{d}\rangle= \frac{\delta_{a,c}\delta_{b,d}a!b!}{(a+b+1)!}.
\end{equation}
\end{lemma}
\begin{proof} The inner product
$$\langle z_1^{a}z_2^{b},{z}_1^{c}{z_2}^{d}\rangle=\int_{S^3}z_1^{a}z_2^{b}\overline{z}_1^{c}\overline{z_2}^{d}d\mu(z_1,z_2)$$
where $\mu(z_1,z_2)$ denotes the $SU(2,1)$-invariant probability measure on the sphere $S^3$ (which is also the Haar measure on $SU(2)(\Rr)$ under \eqref{isomSU(2)}).  Write 
\begin{equation}
	\label{z1z2param}
	z_1=e^{i(\alpha+\beta)}\cos \theta,\ z_2=e^{i(\alpha-\beta)}\sin\theta,
\end{equation} where $\theta\in[0,\pi/2],$ $\alpha\in [-\pi,\pi],$ $\beta\in[0,\pi].$ In this polar coordinate system we have 
\begin{equation}\label{4.10}
d\mu(z,\overline{z})=\frac{\sin\theta\cos\theta}{\pi^2}d\theta\,d\alpha\,d\beta.
\end{equation}
Then the left hand side of \eqref{cal2} is equal to
\begin{equation}\label{cal3}
\frac{1}{\pi^2}\int_{0}^{\frac{\pi}{2}}\int_{0}^{\pi}\int_{-\pi}^{\pi}e^{(a-c)(\alpha+\beta)i}\cos^{a+c+1}\theta e^{(b-d)(\alpha-\beta)i}\sin^{b+d+1}\theta\, d\alpha\, d\beta\, d\theta.
\end{equation}
By orthogonality we then see that \eqref{cal3} is equal to (GR 3.621.3)
\begin{align*}
{2\delta_{a,c}\delta_{b,d}}\int_{0}^{\frac{\pi}{2}}\cos^{2a+1}\theta \sin^{2b+1}\theta\,  d\theta=2\delta_{a,c}\delta_{b,d}\frac{\Gamma(a+1)\Gamma(b+1)}{2\Gamma(a+b+2)}=\frac{\delta_{a,c}\delta_{b,d}a!b!}{(a+b+1)!}
\end{align*}
which yields \eqref{cal2}.
\end{proof}

%

\medskip

\subsubsection{$K$-types}
Recall that the maximal compact subgroup ${K_0}$ of  $SU_{J_{2,1}}(\Rr)$ consisting of $SU_{J_{2,1}}\cap U(3)$ can be identified with $U(2)$.

Let $K_{0c}\simeq U(1)$ be the central part of ${K_0}$, and $K_{0s}\simeq SU(2)$ be the semisimple part. An irreducible unitary representation of ${K_0}$ is completely determined by its restriction to ${K_{0c}}$ and ${K_{0s}}.$ Therefore, such representations are parameterized by two integers $m, n,$ such that $n\geq 0$ and $m-n$ even: $m$ determines the character of ${K_{0c}}$ and $n+1$ is the dimension of the irreducible representation of ${K_{0s}}$.

Write $z=\transp{(z_1,z_2)}.$ For each integer $N,$ the group ${K_0}$ acts on $\mathcal{H}^{p,q}$ via
\begin{align*}
\tau_{p,q}^N\left(
\begin{array}{cc}
u&\\
&(\det u)^{-1}
\end{array}
\right)h(z,\overline{z})=(\det u)^{-N}h(uz,u^{-1}\overline{z}),\ u\in U(2).
\end{align*}

Let $T$ denote the Cartan subgroup of $SU_{J_{2,1}}(\Rr)$:
\begin{align*}
T=\Big\{\diag(z_1,z_2,z_3):\ |z_1|=|z_2|=|z_3|=1,\ z_1z_2z_3=1\Big\}.
\end{align*}

When restricted to ${K_0}_s$ we can take $\phi_{p,q}(z,\overline{z})=z_1^p\overline{z}_2^q$ as a highest weight vector in $\mathcal{H}^{p,q}$. Observing that
\begin{align*}
\tau_{p,q}^N\left(
\begin{array}{ccc}
e^{i\alpha}&&\\
&e^{i\beta}&\\
&&e^{-i(\alpha+\beta)}
\end{array}
\right)\phi_{p,q}=e^{pi\alpha}e^{-qi\beta}e^{-Ni(\alpha+\beta)}\phi_{p,q},
\end{align*}
 the parametrization of irreducible unitary representations of ${K_0}$ becomes $(m,n)=(p-q-2N,p+q).$ A straightforward computation shows that the highest weight of the representation $\tau_{p,q}^N$ is $(p+q)\Lambda_1-(q+N)\Lambda_2.$

\subsubsection{Highest Weight Vector in Minimal $K$-types and Matrix Coefficients}

Let $$\Lambda=k_1\Lambda_1+k_2\Lambda_2$$ be an holomorphic weight. In this section, we will find a minimal $K$-type of the discrete series $D_{\Lambda}^+.$ By definition and Theorem 9.20 in \cite{Kna01} we see that the minimal $K$-type we are seeking is the Blattner parameter $\widetilde{\Lambda}$ of $D_{\Lambda}^+.$

\begin{lemma}\label{lem4}
Let the notation be as before. Assume $\Lambda$ is holomorphic, then we have $k_2\geq 0$, $k_1+k_2+2<0$ and
\begin{equation}\label{8}
\widetilde{\Lambda}=k_2\Lambda_1+k_1\Lambda_2.
\end{equation}
\end{lemma}
\begin{proof}
Let $l=k_2-k_1\in\mathbb{Z}_{>0}.$ By Lemma 7.9 in \cite{Wal76}, we have $$(\pi_{\Lambda}\mid_K,\mathcal{H}^{p,q})\equiv \tau_{p+q}^{2l+3(p-q)},$$ which is $K$-equivariant.  Hence, for $p\geq 0$ and $0\leq q\leq k_2,$ $$(D_{\Lambda}^+\mid_K,\mathcal{H}^{p,q})\equiv \tau_{p+q}^{2l+3(p-q)}.$$ We can describe the corresponding highest weight in terms of coordinates in $\mathbb{C}^3.$ Choose the basis $\alpha_1=(1,-1,0),$ $\alpha_2=(0,1,-1)$ of simple roots in the root system of $\mathfrak{su}(2,1)^{\mathbb{C}}=\mathfrak{sl}(3,\mathbb{C})$ relative to the diagonal $\mathfrak{h}.$ Under this choice, the basic weights are $\Lambda_1=(2/3,-1/3,-1/3)$ and $\Lambda_2=(1/3,1/3,-2/3).$ So the highest weight is
\begin{align*}
H(p,q)=\left(\frac{3q-l}{3},\frac{-3p-l}{3},\frac{2l+3p-3q}{3}\right),\ p\geq 0,\ 0\leq q\leq k_2.
\end{align*}

Note that $\alpha_1$ is the only positive compact root. Let $$G(p,q):=\|H(p,q)+\alpha_1\|^2.$$ We then need to find pairs $(p,q)$ such that $G(p,q)$ is minimal. Since $\frac{\partial G}{\partial p}(p,q)>0$ for all $p, q\geq 0,$ we have $G(p,q)\geq G(0,q).$ Note that
\begin{align*}
G(0,q)=\frac{1}{9}\cdot \bigg[(3q-l+3)^2+(l+3)^2+(2l-3q)^2\bigg]=2\left(q-\frac{l-1}{2}\right)^2+\frac{(l+3)^2}{6}.
\end{align*}

On the other hand, we have $\langle\Lambda+\rho,w_1w_2\alpha_i\rangle>0,$ where $1\leq i\leq 2,$ and $w_i$ is the Weyl element.  By definition, for a weight $\nu,$
\begin{align*}
w_i\nu=\nu-\frac{2\langle\nu,\alpha_i\rangle}{\langle\alpha_i,\alpha_i\rangle}\cdot\alpha_i=\nu-\langle\nu,\alpha_i\rangle\alpha_{i},\quad 1\leq i\leq 2.
\end{align*}
Hence $\langle\Lambda+\rho,w_1w_2\alpha_i\rangle>0$ is equivalent to the conditions $k_2\geq 0$ and $k_1+k_2+2<0.$ Therefore, $k_2<(l-1)/2,$ implying that $$G(p,q)\geq G(0,q)\geq G(0,k_2)$$ for all $p\geq 0$ and $0\leq q\leq k_2.$ Then \eqref{8} follows.
\end{proof}

From Lemma \ref{lem4} we have a highest weight vector 
\begin{equation}\label{highestweight}
\phi(z,\overline{z})=\overline{z}_2^{k_2}	
\end{equation}
 for the minimal $K$-type of $D_{\Lambda}^+.$ We then compute the corresponding matrix coefficient in Proposition \ref{pro4}. To prepare for the proof, we need the following auxiliary computation:
\begin{lemma}\label{lem5}
	Given $A, B\in\mathbb{C}$ be such that $|A|\neq |B|$ and $m,n\in\mathbb{Z}$ such that $|m|+n\geq 1$, let
	$$I_{m,n}(A,B):=\frac{1}{2\pi}\int_{0}^{2\pi}\frac{e^{-mi\alpha}}{(A+Be^{i\alpha})^{n}}d\alpha.$$
	We have the equality
	\begin{align*}
	I_{m,n}(A,B)=
	\begin{cases}
	\frac{(-B)^m}{A^{m+n}}\binom{m+n-1}{m}&\hbox{ if $|A|>|B|;$}\\
	0&\hbox{ if $|A|<|B|.$}
	\end{cases}
	\end{align*}
\end{lemma}
\begin{remark}In particular when $m<0$ and $|m|+n\geq 1$ we have
$I_{m,n}(A,B)=0$.
	
\end{remark}
\begin{proof}
	Suppose $|A|>|B|\geq 0.$ Let $C=B/A.$ Then 	
	\begin{align*}
	I_{m,n}(A,B)=\frac{1}{2\pi A^n}\int_{0}^{2\pi}\frac{e^{-mi\alpha}}{(1+Ce^{i\alpha})^{n}}d\alpha=\frac{1}{2\pi A^n}\sum_{k\geq 0}C_{n,k}\int_{0}^{2\pi}e^{(k-m)i\alpha}d\alpha,
	\end{align*}
	which is vanishing if $m<0$. 
	Suppose $m\geq 0$. We have
	\begin{align*}
	I_{m,n}(A,B)=\frac{C_{n,m}}{A^n}=\frac{C^m}{A^n}\binom{-n}{m}=\frac{(-B)^m}{A^{n+m}}\binom{n+m-1}{m}.
	\end{align*}
	
	Now we suppose $|A|<|B|.$ Let $D=A/B.$ Then
	\begin{align*}
	I_{m,n}(A,B)=\frac{1}{2\pi B^n}\int_{0}^{2\pi}\frac{e^{-(m+n)i\alpha}}{(1+De^{-i\alpha})^n}d\alpha=\frac{1}{2\pi B^n}\sum_{k\geq 0}D_{n,k}\int_{0}^{2\pi}e^{-(k+m+n)i\alpha}d\alpha,
	\end{align*}
	which is vanishing since $m+n\geq 1.$
\end{proof}

\begin{prop}\label{pro4}
Let $\Lambda=k_1\Lambda_1+k_2\Lambda_2$ be holomorphic and let $g=(g_{i,j})_{1\leq i,j\leq 3}\in SU_{J_{2,1}}(\Rr)$ and, for $\phi$ defined in  \eqref{highestweight}, let
\begin{equation}\label{phi0def}
	\phi_{\circ}=\phi/\peter{\phi,\phi}_\Lambda^{1/2}. 
\end{equation}
 We have
\begin{equation}\label{prop4}
\langle D_+^{\Lambda}(g)\phi_{\circ},\phi_{\circ}\rangle_{\Lambda}=g_{22}^{k_2}\overline{g}_{33}^{k_1}.
\end{equation}
\end{prop}
\begin{proof}
By definition (see Proposition \ref{prop2}) $D_{\Lambda}^+=\pi_{\Lambda}\mid_{V_+^{\Lambda}}$ and $\phi\in V_+^{\Lambda}.$ Therefore,
\begin{align*}
\langle D_+^{\Lambda}(g)\phi,\phi\rangle_{\Lambda}=&\langle \pi_{\Lambda}(g)\phi,\phi\rangle_{\Lambda}=\sum_{p}\sum_{q}c_{p,q}(\Lambda)\langle (\pi_{\Lambda}(g)\phi)_{p,q},\phi_{p,q}\rangle\\
=&c_{0,k_2}(\Lambda)\langle (\pi_{\Lambda}(g)\phi)_{0,k_2},\phi\rangle.
\end{align*}

 Since $g\in SU_{J_{2,1}}(\Rr)$, we have $\transp{\overline{g}}J_{2,1}g=J_{2,1}$ and therefore
\begin{equation}\label{9}
g^{-1}=J_{2,1}^{-1}\transp{\overline{g}}J_{2,1}=\left(
\begin{array}{ccc}
\overline{g}_{11}&\overline{g}_{21}&-\overline{g}_{31}\\
\overline{g}_{12}&\overline{g}_{22}&-\overline{g}_{32}\\
-\overline{g}_{13}&-\overline{g}_{23}&\overline{g}_{33}
\end{array}
\right).
\end{equation}

According to the group action \eqref{7.1} we obtain
\begin{align*}
g^{-1}.z=\transp{\left(\frac{\overline{g}_{11}z_1+\overline{g}_{21}z_2-\overline{g}_{31}}{-\overline{g}_{13}{z}_1-\overline{g}_{23}{z}_2+\overline{g}_{33}},\frac{\overline{g}_{12}z_1+\overline{g}_{22}z_2-\overline{g}_{32}}{-\overline{g}_{13}{z}_1-\overline{g}_{23}{z}_2+\overline{g}_{33}}\right)}\in S^3.
\end{align*}

Thus $\pi_{\Lambda}(g)\phi(z,\overline{z})$ is equal to
\begin{align*}
&(\overline{g}_{33}-\overline{g}_{13}{z}_1-\overline{g}_{23}z_2)^{k_1}\overline{(\overline{g}_{33}-\overline{g}_{13}{z}_1-\overline{g}_{23}z_2)}^{k_2}\cdot \overline{\left(\frac{\overline{g}_{12}z_1+\overline{g}_{22}z_2-\overline{g}_{32}}{-\overline{g}_{13}{z}_1-\overline{g}_{23}{z}_2+\overline{g}_{33}}\right)}^{k_2},
\end{align*}
namely, $$\pi_{\Lambda}(g)\phi(z,\overline{z})=(\overline{g}_{33}-\overline{g}_{13}{z}_1-\overline{g}_{23}z_2)^{k_1}\cdot {\left({g}_{12}\overline{z}_1+{g}_{22}\overline{z}_2-{g}_{32}\right)}^{k_2}.$$
Since $$\phi_{\circ}=\frac{1}{c_{0,k_2}(\Lambda)^{1/2}}\frac{\phi}{\langle\phi,\phi\rangle^{1/2}},$$
we have by Lemma \ref{cal1} (on page \pageref{cal1})
\begin{align*}
\langle D_+^{\Lambda}(g)\phi_{\circ},\phi_{\circ}\rangle_{\Lambda}=&\frac{c_{0,k_2}(\Lambda)\cdot\int_{S^3}\frac{{\left({g}_{12}\overline{z}_1+{g}_{22}\overline{z}_2-{g}_{32}\right)}^{k_2}}{(\overline{g}_{33}-\overline{g}_{13}{z}_1-\overline{g}_{23}z_2)^{-k_1}}\cdot z_2^{k_2}d\mu(z,\overline{z})}{c_{0,k_2}(\Lambda)\cdot\int_{S^3} \overline{z}_2^{k_2}\cdot z_2^{k_2}d\mu(z,\overline{z})}\\
=&(k_2+1)\cdot\int_{S^3}\frac{{\left({g}_{12}\overline{z}_1+{g}_{22}\overline{z}_2-{g}_{32}\right)}^{k_2}}{(\overline{g}_{33}-\overline{g}_{13}{z}_1-\overline{g}_{23}z_2)^{|k_1|}}\cdot z_2^{k_2}d\mu(z,\overline{z})
\end{align*}

From the parametrization \eqref{z1z2param} and the identity \eqref{4.10}, we have
\begin{gather}\label{S3integral}
\pi^2\int_{S^3}\frac{{\left({g}_{12}\overline{z}_1+{g}_{22}\overline{z}_2-{g}_{32}\right)}^{k_2}}{(\overline{g}_{33}-\overline{g}_{13}{z}_1-\overline{g}_{23}z_2)^{|k_1|}}\cdot z_2^{k_2}d\mu(z,\overline{z})\\
=\int_{0}^{\frac{\pi}{2}}\int_{0}^{\pi}\int_{-\pi}^{\pi}\frac{{\left({g}_{12}e^{-2i\beta}\cos \theta+{g}_{22}\sin\theta-{g}_{32}e^{i(\alpha-\beta)}\right)}^{k_2}}{(\overline{g}_{33}-\overline{g}_{13}e^{i(\alpha+\beta)}\cos \theta-\overline{g}_{23}e^{i(\alpha-\beta)}\sin\theta)^{|k_1|}}\cdot \sin^{k_2+1}\theta\cos\theta d\alpha\, d\beta\, d\theta.\nonumber
\end{gather}

Applying the expression \eqref{9} for $g^{-1}$ in $g^{-1}g=\Id$ one has $$|g_{33}|^2=|g_{13}|^2+|g_{23}|^2+1.$$ Then by Cauchy inequality, we get
\begin{align*}
|\overline{g}_{13}e^{i\beta}\cos \theta-\overline{g}_{23}e^{-i\beta}\sin\theta|\leq  \sqrt{|g_{13}|^2+|g_{23}|^2}=\sqrt{|g_{33}|^2-1}<|g_{33}|.
\end{align*}
We write
$$A=\ov{g_{33}},\ B=-\overline{g}_{13}e^{i\beta}\cos \theta-\overline{g}_{23}e^{-i\beta}\sin \theta,$$
$$C={g}_{12}e^{-2i\beta}\cos \theta+{g}_{22}\sin\theta,\ D=-{g}_{32}e^{-i\beta}.$$
We have
$$\frac{(C+De^{i\alpha})^{k_2}}{(A+Be^{i\alpha})^{|k_1|}}=\sum_{k=0}^{k_2}\binom{k_2}{k}C^kD^{k_2-k}\frac{e^{i(k_2-k) \alpha}}{(A+Be^{i\alpha})^{|k_1|}}.$$
We compute the $\alpha$-integral in \eqref{S3integral} using Lemma \ref{lem5} (which is applicable since  $\Lambda$ is holomorphic); the only contributing term is for  $k=k_2$ and equals
${C^{k_2}}/{A^{|k_1|}}={C^{k_2}}{A^{k_1}}$.
We conclude that
\begin{align*}
\frac{\langle D_+^{\Lambda}(g)\phi_{\circ},\phi_{\circ}\rangle_{\Lambda}}{k_2+1}=&\frac{2\overline{g}_{33}^{k_1}}{\pi }\int_{0}^{\frac{\pi}{2}}\int_{0}^{\pi}{\left({g}_{12}e^{-2i\beta}\cos \theta+{g}_{22}\sin\theta\right)}^{k_2}\cdot \sin^{k_2+1}\theta\cos\theta\, d\beta\, d\theta\\
=&{2g_{22}^{k_2}\overline{g}_{33}^{k_1}}\int_{0}^{\frac{\pi}{2}}\sin^{2k_2+1}\theta\cos\theta d\theta=\frac{g_{22}^{k_2}\overline{g}_{33}^{k_1}}{k_2+1}.
\end{align*}
\end{proof}

\subsubsection{Discrete Series on $G(\mathbb{R})$}
Recall that $\rmB=\rmB^{-1}$ (see \eqref{b})
and $$G(\mathbb{R})= \rmB U(2,1)(\mathbb{R})\rmB^{-1}.$$
Therefore, setting 
$$G^1=\ker(\det:G\mapsto \Gm)=SU(W),$$ we have
$$G^1(\mathbb{R})= \rmB SU(2,1)\rmB^{-1}= \rmB SU(2,1)\rmB.$$
Letting for $g\in G^1(\Rr)$
$$D_{+,{\rmB}}^{\Lambda}(g):=D_{+}^{\Lambda}({\rmB}g{\rmB}),$$
we denote by $(D_{+,{\rmB}}^{\Lambda},V_+^{\Lambda})$ the discrete series representation on $G^1(\Rr)$.

From the split exact sequence 
\begin{align*}
1\longrightarrow G^1(\mathbb{R})\longrightarrow G(\mathbb{R})\longrightarrow Z_{G}(\mathbb{R})\longrightarrow 1,
\end{align*}
(with $Z_G(\Rr)\simeq U(1)(\Rr)=\Cc^1$, we have the decomposition
\begin{equation}\label{11}
G(\mathbb{R})=Z_{G}(\mathbb{R})^{+}G^1(\mathbb{R}).
\end{equation}
where $$Z_{G}^+(\mathbb{R})=\{\diag(e^{i\theta}, e^{i\theta}, e^{i\theta}):\ -\pi/3<\theta\leq \pi/3\}.$$ Using \eqref{11} we extend the $G^1(\Rr)$-action $(D_{+,{\rmB}}^{\Lambda},V_+^{\Lambda})$ to $G(\mathbb{R})$ by requiring $Z_{G}(\mathbb{R})^{+}$ to act trivially. Let $z=\diag(e^{i\theta}, e^{i\theta}, e^{i\theta})\in Z_{G}(\mathbb{R}),$ $-\pi<\theta\leq \pi.$ Let $\theta_{\circ}\in (-\pi/3,\pi/3]$ be such that
\begin{equation}\label{equ}
\frac{3\theta}{2\pi}\equiv \frac{3\theta_{\circ}}{2\pi}\Mod{1}.
\end{equation}
Such a $\theta_{\circ}$ is uniquely determined by $\theta.$ Set $$z_{\circ}=\diag(e^{i\theta_{\circ}}, e^{i\theta_{\circ}}, e^{i\theta_{\circ}})\in Z_{G}(\mathbb{R})^{+}.$$ Then by \eqref{equ} there exists a unique $k_z\in\{-1,0,1\}$ such that $z=z_{\circ}e^{2\pi k_zi/3}.$ We have $$z\cdot z_{\circ}^{-1}=\diag(e^{2\pi k_zi/3}, e^{2\pi k_zi/3}, e^{2\pi k_zi/3})\in Z_{G}(\mathbb{R})\cap G^1(\mathbb{R}).$$  Let $\rho(z):=e^{2\pi k_zi/3}.$ Therefore, $z$ acts by the scalar
\begin{align*}
\omega_{\Lambda}(z)=D_{+}^{\Lambda}(z\cdot z_{\circ}^{-1})=\overline{\rho(z)}^{k_1}\cdot \rho(z)^{k_2}=\rho(z)^{k_2-k_1}.
\end{align*}
However, $\omega_{\Lambda}$ is typically not a homomorphism unless we assume  that $k_2\equiv k_1\Mod{3}$ so that $\omega_{\Lambda}(z)\equiv 1$. 

Under the assumption that $k_2\equiv k_1\Mod{3}$ we obtain a representation of $G(\mathbb{R})$ on $V_+^{\Lambda}$ with trivial central character which we denote by $D^{\Lambda}$. We have
\begin{prop}
Let $\Lambda=k_1\Lambda_1+k_2\Lambda_2$ be holomorphic with $k_1\equiv k_2\Mod{3}.$ Let $(D^{\Lambda},V_+^{\Lambda})$ be the representation of $G(\mathbb{R})$ defined as above. Then $D^{\Lambda}$ is irreducible and square-integrable. Furthermore, every square-integrable holomorphic representation of $G(\mathbb{R})$ is of the form $D^{\Lambda}\otimes\chi$ for some holomorphic $\Lambda$ and a unitary character $\chi.$
\end{prop}

Now we compute the matrix coefficient of $(D^{\Lambda},V_+^{\Lambda}).$ For $g\in G(\mathbb{R})$ we set
\begin{equation}
	\label{thetagdef}
	\det g=e^{i\theta_g}\in U(1)
\end{equation} for $-\pi<\theta_g\leq \pi$ uniquely defined. 
\begin{lemma}\label{26}
Let $g=(g_{ij})_{1\leq i,j\leq 3}\in G(\mathbb{R})$ and $\phi^\circ$ as in \eqref{phi0def} we have
\begin{equation}\label{10}
\langle D^{\Lambda}(g)\phi_{\circ},\phi_{\circ}\rangle_{\Lambda}=\frac{(\overline{g}_{11}-\overline{g}_{13}-\overline{g}_{31}+\overline{g}_{33})^{k_1}g_{22}^{k_2}}{2^{k_1}\cdot (\det g)^{\frac{k_2-k_1}3}}.
\end{equation}
\end{lemma}
\begin{proof} Denote by $z_g=\diag(e^{i\theta_g/3},e^{i\theta_g/3},e^{i\theta_g/3})\in Z_{G}(\mathbb{R})^+.$ Then $z_g^{-1}\cdot g\in G^1(\mathbb{R}).$
We have $$\langle D^{\Lambda}(g)\phi_{\circ},\phi_{\circ}\rangle_{\Lambda}=\langle D_+^{\Lambda}(z_g^{-1}\cdot{\rmB} g{\rmB})\phi_{\circ},\phi_{\circ}\rangle_{\Lambda}.$$ Note that
\begin{align*}
{\rmB}\left(
\begin{array}{ccc}
g_{11}&{g}_{12}&{g}_{13}\\
g_{21}&{g}_{22}&{g}_{23}\\
g_{31}&{g}_{32}&{g}_{33}
\end{array}
\right){\rmB}
=\left(
\begin{array}{ccc}
\frac{g_{11}+g_{13}+g_{31}+g_{33}}{2}&\frac{g_{12}+g_{32}}{\sqrt{2}}&\frac{g_{11}-g_{13}+g_{31}-g_{33}}{2}\\
\frac{g_{21}+g_{23}}{{\sqrt{2}}}&g_{22}&\frac{g_{21}-g_{23}}{\sqrt{2}}\\
\frac{g_{11}+g_{13}-g_{31}-g_{33}}{{2}}&\frac{g_{12}-g_{32}}{\sqrt{2}}&\frac{g_{11}-g_{13}-g_{31}+g_{33}}{{2}}
\end{array}
\right).
\end{align*}

Hence it follows from Proposition \ref{pro4} that
\begin{align*}
\langle D_+^{\Lambda}(z_g^{-1}\cdot{\rmB} g{\rmB})\phi_{\circ},\phi_{\circ}\rangle_{\Lambda}=\frac{e^{-il\theta_g/3}(\overline{g}_{11}-\overline{g}_{13}-\overline{g}_{31}+\overline{g}_{33})^{k_1}g_{22}^{k_2}}{2^{k_1}}.
\end{align*}
Then \eqref{10} follows.
\end{proof}

\subsubsection{Choice of the archimedean component}\label{3.1.5}
  Recall that $E=\mathbb{Q}(\sqrt{-D})$ is an imaginary quadratic extension. Let $D_E$ be its fundamental discriminant. Let $$g_E\:=\diag(|D_E|^{1/4},1,|D_E|^{-1/4})\in G(\Rr);$$ 
we set 
\begin{equation}\label{finfchoice}
f_{\infty}(g):=\langle D^{\Lambda}(g_E^{-1}gg_E)\phi_{\circ},\phi_{\circ}\rangle_{\Lambda}=\langle D^{\Lambda}(g) D^{\Lambda}(g_E)\phi_{\circ},D^{\Lambda}(g_E)\phi_{\circ}\rangle_{\Lambda},	
\end{equation}
 or more explicitly
 \begin{equation}\label{55}
f_{\infty}(g)=\frac{e^{-il\theta_g/3}(\overline{g}_{11}-\overline{g}_{13}|D_E|^{-1/2}-\overline{g}_{31}|D_E|^{1/2}+\overline{g}_{33})^{k_1}g_{22}^{k_2}}{2^{k_1}},
\end{equation}
where the notation is the same as that of Lemma \ref{26} and $l=k_2-k_1$.

\subsubsection{Restriction of matrix coefficients}
As we explain  below we have an isomorphism 
$\SL(2,\mathbb{R})\xrightarrow{\sim}SU(W)(\mathbb{R})$ given by
\begin{align}\label{iEdef}
\iota_{\infty}:\ \ \begin{pmatrix}
a&b\\
c&d
\end{pmatrix}\in \SL(2,\mathbb{R})\mapsto g_E^{-1}\begin{pmatrix}
a&&-b\Delta\\
&1&\\
-c\Delta^{-1}&&d
\end{pmatrix}g_E\in SU(W)(\mathbb{R}).
\end{align}
where $\Delta$ is as in \eqref{Deltadef}. Under $\iota_E,$ the maximal compact subgroup $\SO_2(\Rr)$ is mapped to the maximal compact subgroup $K_\infty'$ whose matrices are given by
\begin{align*}
 \kappa_{\theta}=\begin{pmatrix}
\cos\theta&-\sin\theta\\
\sin\theta&\cos\theta
\end{pmatrix}\xmapsto{\iota_E}\begin{pmatrix}
\cos\theta&&-i\sin\theta\\
&1&\\
i\sin\theta&&\cos\theta
\end{pmatrix}={\rmB}\begin{pmatrix}
e^{i\theta}&&\\
&1&\\&&e^{-i\theta}
\end{pmatrix}{\rmB},
\end{align*}
where ${\rmB}$ is defined in  \eqref{b}, and $\theta\in[-\pi,\pi).$ 
\begin{lemma}\label{restrictionmatrix}
The restriction  to $SU(W)(\mathbb{R})$ of the matrix coefficients $f_{\infty}(g)$ is, via the isomorphism \eqref{iEdef},  a matrix coefficient of the holomorphic discrete series on $\SL(2,\mathbb{R})$ of weight $-k_1$
(recall that $-k_1\geq k_2+2$ and $k_2\geq 0$) 	
\end{lemma}

 \proof

Let $\pi_{k}^{+}$ be the holomorphic discrete series of weight $k>1$ on $\SL(2,\mathbb{R}).$ Let $F_{k}$ be the  matrix coefficient of its normalized lowest weight vector $v_k^{\circ}$. It is given explicitly for $g=\begin{pmatrix}
a&b\\
c&d
\end{pmatrix}\in\SL(2,\mathbb{R})$ by (see \cite[Prop. 14.1]{KL06})
\begin{equation}\label{61}
F_k(g):=\langle\pi_k^+(g)v_k^{\circ},v_k^{\circ}\rangle=\int_{\mathbb{H}}\pi_k^+(g)v_k^{\circ}(z)\overline{v_k^{\circ}(z)}y^k\frac{dx\,dy}{y^2}=\frac{2^k}{(a+d-i(b-c))^k},
\end{equation}
where  $z=x+iy\in\mathbb{H}$ the  upper half-plane. 

By Lemma \ref{26} we have for $g\in \SL(2,\Rr)$ 
\begin{equation}\label{56}
f_{\infty}(\iota_{\infty}(g))=F_{-k_1}(g)=\frac{2^{|k_1|}}{(a+d-ib+ic)^{|k_1|}}.
\end{equation}
\qed

\subsection{Non-archimedean components}\label{3.2}\label{seclevel}

Let $V$ be the $3$-dimensional Hermitian space with respect to $J$ as in \S \ref{SecHerm} with basis $\{e_{1},e_0,e_{-1}\}$ and $U(V)=G(\mathbb{Q}).$ Let 
\begin{equation}
	\label{defL}
	L=\mathcal{O}_{E}e_{1}\oplus\mathcal{O}_{E}e_{0}\oplus\mathcal{D}_E^{-1}e_{-1}\subseteq V.
\end{equation}
 Then $L$ is a lattice of $V$. Its $\mathbb{Z}$-dual lattice is 
 $$L^*=\big\{v\in V:\ \tr_{E/\mathbb{Q}}\langle v,L\rangle_J\subset\mathbb{Z}\big\}=\mathcal{O}_{E}e_{1}\oplus\mathcal{D}_{E}^{-1}e_{0}\oplus\mathcal{D}_E^{-1}e_{-1}.$$
  One can verify that $L$ is an $\mathcal{O}_E$-module of full rank equipped with $\langle\cdot,\cdot\rangle_{J}.$ Since for $v,v'\in L,$ $\langle v,v'\rangle_{J}\in \mathcal{D}_{E}^{-1},$ and $\langle v,v\rangle_{J}\in \mathbb{Z},$ the lattice $L$ is integral and even. 
  
  Let $$G(\mathbb{Z}):=U(L)$$ be the group of isometries preserving $L$ Then $U(L)$ is an arithmetic subgroup of $G(\mathbb{Q}).$ Explicitly, we have
\begin{equation}\label{78}
G(\mathbb{Z})=\{g\in G(\mathbb{Q}):\ g.L=L\}=G(\mathbb{Q})\cap \Bigg\{\begin{pmatrix}
\mathcal{O}_{E}&\mathcal{O}_{E}&\mathcal{D}_{E}\\
\mathcal{O}_{E}&\mathcal{O}_{E}&\mathcal{D}_{E}\\
\mathcal{D}_{E}^{-1}&\mathcal{D}_{E}^{-1}&\mathcal{O}_{E}
\end{pmatrix}\Bigg\}.
\end{equation}
Let $$G(\widehat{\mathbb{Z}}):=\big\{g\in G(\mathbb{A}_f):\ g.L\otimes_{\mathbb{Z}}\widehat{\mathbb{Z}}=L\otimes_{\mathbb{Z}}\widehat{\mathbb{Z}}\big\}=\prod_p G(\Zp).$$
Moreover, recall (cf \S \ref{secchoice}) that for any split prime $p$, we have fixed  a place $\mathfrak{p}$ above $p$ so that 
$$G(\Qp)\simeq \GL(3,\mathbb{Q}_p).$$
Since $\mcO_{E,\mfp}\simeq \Zp$ we have
$$L_\mfp\simeq \Zp e_{1}\oplus\Zp e_{0}\oplus\Zp e_{-1}$$  and under the above isomorphism we have
 $$G(\Zp)\simeq \GL_3(\Zp).$$

\subsubsection{The open compact $K_f(N)$}
Let $N$ be either $1$ or a positive odd  prime that remains inert in $E$ Let $$K_f(N):=\prod_{p<\infty}K_p(N)\subset G(\Af),$$ 
be the open compact subgroup whose local components $K_p(N)$ are defined as follows:

\subsubsection*{The case $p=2$}\label{tits} Let $v$ be the place above $2$ (when $2$ is unramified, $v=2$). Choose  $\lambda\in E_v$ such that 
$$\tr_{E_v/\mathbb{Q}_2}(\lambda)=1\hbox{ and }
\|\lambda\|=\min_{\substack{x\in E\\ \tr_{K_v/\mathbb{Q}_2}(x)=1}}\|x\|
$$
with
$\|x\|=2^{-\nu_v(x)}.$
Let $K_2$ be the stabilizer in $G(\mathbb{Q}_2)$ of the lattice 
$$\big\{ae_1+be_0+ce_{-1},\ a,b,\lambda c\in\mathcal{O}_{E_v}\big\}.$$ Then $K_2$ is a special maximal compact subgroup according to \cite{Tit79} and we set $$K_2(N)=K_2.$$

\subsubsection*{The generic case} If $p\nmid 2N$ (that is $p$ is either ramified, split or inert but does not divide $N$) we set $$K_p(N)=K_p:=G(\Zp).$$
\subsubsection*{The case $p=N$} The prime $p$ is then inert and we set
\begin{equation}\label{ke}
 	K_p(N):=I_p=\big\{g=(g_{i,j})\in G(\mathbb{Z}_p),\ g\equiv \begin{pmatrix}
 		*&*&*\\0&*&*\\0&0&*
 	\end{pmatrix}\mods{p}\big\}
 	\end{equation}
 the Iwahori subgroup of $K_p=G(\Zp)$: the inverse image of the Borel subgroup 
$B(\Ff_p)\subset G(\Ff_p)$ under the reduction modulo $p$ map.

%

\subsection{Choice of the local test functions $f_p^\mfn$}
We can now define the non-archimedean components
$f^{{\mathfrak{n}}}_p$ of $f^{{\mathfrak{n}}}$ in \eqref{hchoice}.
Let
\begin{equation}
	\label{elldefinition}
	\ell=\prod_{p}p^{r_p}\geq 1
\end{equation}
be an integer whose prime divisors (if any) $p$ are inert primes coprime with $N$.

\subsubsection{The generic case $p\nmid\ell N$}
We set
 \begin{equation}\label{fpchoice}
 f_p:=\frac{1}{\mu(K_p)}\cdot\textbf{1}_{K_p}; 
 \end{equation}
 here $\mu(K_p)$ is the volume of $K_p$ with respect to the Tamagawa measure. 
 
\subsubsection{The case $p=N$} Recall that $K_p(N)=I_p$ is the Iwahori subgroup and we choose the normalized characteristic function
$$f_p=\frac{1}{\mu(I_p)}\cdot\textbf{1}_{I_p}$$

\subsubsection{If $p\mid \ell$}
If $p|\ell$ (in particular inert and coprime with $N$) we take 
$f_p$ to be the compactly supported bi-$G(\Zp)$ invariant characteristic function
\begin{equation}
	\label{fpHeckechoice}
	f_p=\textbf{1}_{G(\Zp)A_{p^{r_p}}G(\Zp)}
\end{equation}
where $r_p\geq 0$ depends on $p$; for any $r\geq 0$ we have set
\begin{equation}
	\label{Aprdef}
	A_{p^r}:=\begin{pmatrix}
	p^r&&\\&1&\\ &&p^{-r}
\end{pmatrix}.
\end{equation}

\begin{remark} Occasionally, when $p$ is understood we may change the notation and write $A_r $ in place of $A_{p^r}$ (see \S \ref{S}). We hope it won't create confusion.
\end{remark}

\subsubsection{An extra twist for $p=N'$}\label{secextratwist}
For $p=N'$ a prime split in $E$, we define $\mathfrak{n}_p\in G(\Qp)$ to be the element corresponding to the  matrix
  $$\mathfrak{n}_p\simeq \begin{pmatrix}
1&&p^{-1}\\
&1&\\
&&1
\end{pmatrix}\in \GL(3,\mathbb{Q}_p)$$
under the isomorphism \eqref{isomlinear}. 
Let $w'$ be the Weyl element $w'=\begin{pmatrix}
1&&\\
&&1\\
&1&
\end{pmatrix}$ fixing $e_1$ and switching $e_0$ and $e_{-1}$. We then define 
$$\widetilde{\mathfrak{n}}_p=w'\mathfrak{n}_p w'\simeq \begin{pmatrix}
1&p^{-1}&\\
&1&\\
&&1
\end{pmatrix}$$
and for $f_p$ as in \eqref{fpchoice} let
\begin{equation}\label{fnpchoice}
f^{{\mathfrak{n}}_p}_p:\ x\in G(\Qp)\mapsto f_p(\widetilde{\mathfrak{n}}_p^{-1}x\widetilde{\mathfrak{n}}_p).	
 \end{equation}
\begin{remark}The introduction of this extra twist is absolutely crucial: without it, the spectral size of the relative trace formula would select the spherical vector in $\pi_p$, and the local period integral at $p$ would vanish. See Remark \ref{REmntilde}.
\end{remark}

 \subsection{Choice of the global test function $f^{\mfn}$}\label{secglobalf}
  We then set 
 $$f=f_\infty.\prod_p f_p,$$
\begin{equation}
	\label{deffn}
	f^{\mfn}=f_\infty.\prod_{p|N'}f^{{\mathfrak{n}}_p}_p.\prod_{p\nmid N'}f_p.
\end{equation}
Alternatively, if we set, for any place $v$
\begin{equation}\label{ntildepdef}
\widetilde{\mathfrak{n}}_v=\begin{cases}
	w'\mathfrak{n}_p w'\simeq \begin{pmatrix}
1&p^{-1}&\\
&1&\\
&&1
\end{pmatrix}&\hbox{ if $v=p=N'$ and }\\
	w'.w'=\mathrm{Id}_3&\hbox{ if $v=\infty$ or $v=p\nmid N'$}
\end{cases}	
\end{equation}
and
$$\widetilde\mfn=(\widetilde\mfn_v)_v\in G(\Aa),$$
  we have 
 \begin{equation}\label{globalfnchoice}
f^{\mfn}: x\mapsto f(\widetilde{\mathfrak{n}}^{-1}x\widetilde{\mathfrak{n}}). \end{equation}

\subsection{Cusp forms on $U(W)$}\label{sec7.1}\label{U(W)choice}
In this section, we discuss the cuspidal representation $$\pi'\simeq\pi'_{\infty}\otimes{\bigotimes}'_p\pi'_{p}$$ of $G'(\Aa)=U(W)(\Aa)$ and a cuspform $\vphi'\in\pi'$.

We recall that we want $\pi'$ to have trivial central character, its archimedean component $\pi'_{\infty}$ isomorphic to the holomorphic discrete series of weight $k$, for every $p|N'$ its $p$-component, $\pi'_{p}$ is isomorphic to the Steinberg representation $\St_p$ and for any  prime $p$ not dividing $N'$, $\pi_p'$  is unramified. 

Using the trace formula, one can show that such $\pi'$ exists (see below); alternatively one can construct $\pi'$  by functoriality: let $\pi_1$ be a cuspidal automorphic representation of $\GL(2,\Aa)$ with trivial central character such that \begin{itemize}
	\item[-] $\pi_{1\infty}\simeq \pi^+_{k}$ is the holomorphic discrete series of weight $k$,
	\item[-] for $p=N'$, $\pi_{1p}$ is the Steinberg representation,
	 \item[-]  for every $p\not=N'$, $\pi_{1p}$ is unramified.
\end{itemize}
 Let $\pi_{1,E}$ be the base change of $\pi_1$ to $E$; it is conjugate dual. It follows from 
the works of Flicker and Rogawski (\cite{Fli82,Rog}) that $\pi_{1,E}$ descends to an automorphic cuspidal representation $\pi'$ of $G'(\Aa)$ with trivial central character which has the required local properties. Note that by construction, the base change $\pi_E$ of $\pi$ to $\GL(2)/E$ is isomorphic to $\pi_{1,E}$. 

Indeed, because $\pi_1$ is selfdual and $\pi_{1,E}$ is by construction invariant under the complex conjugation $\sigma_E$, the representation $\pi_{1,E}$ is conjugate dual:
$$\pi_{1,E}^\vee\simeq \pi_{1,E}\simeq \pi_{1,E}\circ \sigma_E.$$
Since $W$ is even dimensional, it is sufficient to see that $\pi_{1,E}$ is conjugate symplectic \cite{GGPW12} or in other terms that the automorphic induction 
$\Ind^{\Qq}_{E}(\pi_{1,E})$ is symplectic. We can verify this by considering the global $2$-dimensional Galois representation $\rho'$ attached to $\pi'$ and showing that the Galois representation
$\Ind^{\Qq}_{E}(\mathrm{res}_{E}(\rho'))$ is symplectic: its exterior square contains the trivial representation. We have 
$$\Ind^{\Qq}_{E}(\mathrm{res}_{E}(\rho'))\simeq \rho'\oplus \rho'.\eta$$
($\eta=\eta_E$ is the quadratic character corresponding to $E/\Qq$)
so that $$\Lambda^2(\rho'\oplus \rho'.\eta)\simeq \Lambda^2(\rho')\oplus \Lambda^2(\rho'.\eta)\oplus \Lambda^2(\rho')\oplus (\rho'\otimes \rho'.\eta)=1\oplus 1\oplus \sym_2(\rho').\eta\oplus \eta$$
since the determinant of $\rho'$ is trivial.

Conversely one can show using \cite{Rog} that any representation $\pi'$ can be obtained in that way.

%
%
%
%

\subsubsection{The automorphic form $\vphi'$}\label{phipdef}

We choose $\vphi'$ to be the "newform" on $U(W)(\Aa)$ corresponding to a pure tensor $$\vphi'\simeq \otimes_v\xi'_v\in{\bigotimes}'_v\pi'_v$$ where 
\begin{itemize}
	\item[-]$\xi'_\infty\in\pi'_\infty$ is a vector of lowest weight $k$ (i.e. is multiplied by $e^{-ik\theta}$ under the action  of the matrix $\kappa(\theta)=\diag(e^{i\theta},1,e^{-i\theta})\in SU(W)(\mathbb{R})$),
	\item[-]  for every prime $p$, $\xi'_p$ is invariant under the open-compact subgroup $K'_p(N')$ defined below.
\end{itemize}

Regarding the last point we set
 $$G'(\whZ)=\prod_p G'(\Zp)=G'(\Af)\cap G(\whZ)$$ 
and define
\begin{equation}
	\label{eqK'def}
	K'_f(N')=\prod_p K'_p(N')\subset G'(\Zp)
\end{equation}
where
\begin{itemize}
	\item[-] If $p$ split in $E$,
\begin{equation}
	\label{K'defsplit}K_p'(N'):=
\Big\{g=\begin{pmatrix}
a&0&b\\0&1&0\\
c&0&d
\end{pmatrix}\in G'(\Zp):\ c\in N'\mathbb{Z}_p\Big\}
\end{equation}
\item[-] If $p$ is inert in $E$,
 \begin{equation}
	\label{K'definert}K_p'(N'):=\Big\{g=\begin{pmatrix}
a&0&b\\0&1&0\\
c&0&d
\end{pmatrix}\in G'(\Zp):\ c\in N'\mathcal{O}_{E_p}\Big\}
\end{equation}
\item[-] If $p$ is ramified in $E$, $K_p'(N')=G'(\Zp).$
\end{itemize}

\begin{remark} The subgroup $G'(\whZ)$ is the stabilizer in $G'(\Af)$ of the hermitian $\mathcal{O}_E$-submodule generated by $e_1$ and $\Delta^{-1}e_{-1}$,
\begin{equation}
	\label{defL'}
	L'=\mathcal{O}_{E}e_{1}\oplus\mathcal{D}_E^{-1}e_{-1}\subseteq W;
\end{equation}
In other terms $G'(\whZ)$ is the closure in $G'(\Af)$ of $U(L')$ the stabilizer of $L'$ in $U(W)(\Qq)$.

\end{remark}

\subsubsection{The exceptional isomorphism}\label{exceptionalisom}
To be more concrete we recall that we have an exceptional isomorphism of $\Qq$-algebraic groups
$$\SU(W)\simeq \SL(2)$$
 given by
\begin{equation}\label{isomSL2}
\iota_E=\iota:\ \ \begin{pmatrix}
a&b\\
c&d
\end{pmatrix}\in \SL(2)\mapsto \begin{pmatrix}
a&0&-b\Delta\\
0&1&0\\
-c\Delta^{-1}&0&d
\end{pmatrix}\in SU(W).
\end{equation}
(observe that the isomorphism $\iota_{\infty}$ from \eqref{iEdef} is conjugate to $\iota_{E}$ when tensored with $\Rr$).

Under this isomorphism, the image of the full congruence subgroup $\SL(2,\Zz)$ is given by
 $$\iota(\SL(2,\mathbb{Z}))= SU(L'):=G'(\whZ)\cap SU(W)(\Qq);$$
more generally, if $N'$ is a prime coprime with $D_E$, the image of the usual congruence subgroup of level $N'$
$$\Gamma_0(N')=\bigg\{\begin{pmatrix}a&b\\c&d	
\end{pmatrix}\in SL(2,\BZ):\ N'|c\bigg\}$$
is
\begin{align*}
	\iota(\Gamma_0(N'))&=\Bigg\{\begin{pmatrix}
a&0&-b\Delta\\
0&1&0\\
-c\Delta^{-1}&0&d
\end{pmatrix},\ \begin{pmatrix}a&b\\c&d	
\end{pmatrix}\in SL(2,\BZ):\ N'|c\Bigg\}\\
&= K'_f(N')\cap SU(W)(\Qq).
\end{align*}

Given our automorphic form $\vphi':G'(\Aa)\ra\Cc$ as above let $\phi':\Hh\mapsto \Cc$ be the function on the upper half-plane defined by
\begin{equation}\label{47}
	\phi'(z)=j(g_\infty,i)^{k}\vphi'(\iota(g_\infty))
\end{equation}
for $g_\infty\in\SL(2,\Rr)$ such that $g_\infty.i=z$ and
 $$j(g,z)=(\det g)^{-1/2}(cz+d)$$ is the usual automorphy factor on $\GL(2,\Rr)^+\times\Hh$. The invariance  of $\vphi'$ along with the strong approximation property for $\SL(2)$ imply that the function $\phi'$ is a well defined holomorphic cuspform of weight $k$ with trivial nebentypus and a newform of level $N'$. Now using again the strong approximation property for $\SL(2)$ one can construct out of $\phi'$ an automorphic form which (abusing the notation) we denote 
 $$\vphi':\GL(2,\Qq)\bash \GL(2,\Aa)\mapsto \Cc$$
 of weight $k$, i.e. which multiplies by the character $e^{-ik\theta}$ under the left action of
$$\kappa_\theta=\begin{pmatrix}
\cos\theta&-\sin\theta\\
\sin\theta&\cos\theta
\end{pmatrix}\in \SO_2(\BR)$$
and invariant under the center $Z_{\GL(2)}(\Aa)$ and under the open compact congruence subgroup of level $N'$
 $$K_{0,f}(N')=\{g=\begin{pmatrix}
 	a&b\\c&d
 \end{pmatrix}\in \GL(2,\whZ),\ c\in N'\whZ\}.$$ 

The automorphic form $\vphi'$ generates the representation denoted (abusing the notation)  $\pi'$ and is a new, lowest weight vector of it.

We assume that $\vphi'$ is $L^2$-normalized, that is
\begin{equation}
	\label{vphi'norm}
	\peter{\vphi',\vphi'}=\int_{[G']}|\vphi'(g')|^2dg'=1
\end{equation} where
 $dg'$ denotes the Tamagawa measure and $[G']=G'(\Qq)\bash G'(\Aa)$. With this normalization one has for the classical form $\phi'$
\begin{equation}
	\label{L2norm}
	\peter{\phi',\phi'}:=\frac{1}{\vol(\Gamma_0(N')\bash \BH)}\int_{\Gamma_0(N')\bash \BH}y^k|\phi'(z)|^2\frac{dx\,dy}{y^2}=c_E
\end{equation}
where  $$\vol(\Gamma_0(N')\bash \BH)=\int_{\Gamma_0(N')\bash \BH}\frac{dx\,dy}{y^2}=\frac{\pi}3N'\prod_{p|N'}(1+\frac1p)$$
and $c_E>0$ depends only on $E$.

Consider its Fourier expansion 
$$\phi'(z)=\sum_{n\geq 1}a_ne(nz)$$ 
where
 \begin{equation}\label{andef}
a_n=a_n(\phi'):=e^{-2\pi in\tau}\int_{0}^1\phi'(\tau+x)e^{-2\pi i n x}dx.
 \end{equation}
Since $\phi'$ is a newform, we also have
\begin{equation}
	\label{coefficienthecke}
	a_n=a_1.n^{\frac{k-1}2}\lambda_{\pi_1}(n);
	\end{equation}
here for any integer $n$, $\lambda_{\pi_1}(n)$ denotes  the $n$-th coefficient of the Hecke $L$-function $L(s,\pi_1)$ (normalized analytically); it satisfies Deligne's bound
 \begin{equation}
 	\label{delignesbound}
 	|\lambda_{\pi_1}(n)|\leq d(n)=n^{o(1)}
 \end{equation}
  ( $d(n)$ denotes the divisor function) and its first coefficient $a_1$ satisfies (see \cite[(7)]{Nelson})
  \begin{equation}
	\label{a1coefficient}
	{|a_1|^2}=c_E\frac{2\pi^3}{3}{\prod_{p|N'}(1+\frac1p)}\frac{(4\pi)^{k-1}}{\Gamma(k)L(1,\Ad,\pi')}
\end{equation}
where $L(s,\Ad,\pi')$ is the adjoint $L$-function. We have therefore
\begin{equation}
	\label{anformula}
	|a_n|^2=c_E\frac{2\pi^3}{3}\frac{\prod_{p|N'}(1+1/p)}{L(1,\Ad,\pi')}\frac{(4\pi)^{k-1}n^{k-1}}{\Gamma(k)}|\lambda_{\pi'}(n)|^2
\end{equation}
Also, since $N'$ is squarefree $L(s,\Ad,\pi')$ does not have a Landau-Siegel zero (see \cite{HL94}) and one has
\begin{equation} 
	\label{Siegel}
	L(1,\Ad,\pi')=(kN')^{o(1)}
\end{equation}
so that
\begin{equation}
	\label{anbound}
	|a_n|^2\leq (kN')^{o(1)}\frac{(4\pi)^{k-1}n^{k-1}}{\Gamma(k)}n^{o(1)}
\end{equation}

\section{\bf The Spectral Side}\label{S}

Let $E/\Qq$ be a quadratic extension and let $W\subset V$ be Hermitian spaces of dimensions $n$ and $n+1$ over $E$ respectively. Set $G=U(V)$ and $G'=U(W)$ be the corresponding unitary groups where $G'$ is embedded in $G$ in the obvious way as the stabilizer of $W^\perp\subset V$.

Let $(\pi,V_{\pi})$ (resp. $(\pi',V_{\pi'})$) be cuspidal representations on $G(\mathbb{A})$ (resp $G'(\mathbb{A})$). Define the global Petersson pairing $\langle\cdot,\cdot\rangle$ on $G(\BA)$ by
\begin{align*}
\peter{\vphi_1,\vphi_2}=\int_{G(\BQ)\backslash G(\mathbb{A})}\vphi_1(g)\overline{\vphi_2(g)}dg,
\end{align*}
where $dg$ denotes the Tamagawa measure on $G(\Qq)\bash G(\mathbb{A}).$ Similarly one defines $\langle\vphi_1',\vphi_2'\rangle$ on $G'(\Aa).$ 

For any place $v$ set $G_{v}=G(\BQ_v)$, $G'_{v}=G'(\BQ_v)$ and $dg_v,\ dg'_v$ are local Haar measures on $G_{v},\ G'_{v}$ such that $\prod dg_v=dg,\ \prod dg'_v=dg'.$ 

Under the decomposition $\pi=\otimes_v'\pi_{v}$ and $\pi'=\otimes_v'\pi'_{v}$ we fix a decomposition of either of the global inner products $\langle\cdot,\cdot\rangle$ into local ones as
\begin{align*}
\peter{\cdot,\cdot}=\prod_{v}\peter{\cdot,\cdot}_{v}.
\end{align*}

\subsection{Spectral Expansion}\label{sec4.5}\label{secspectralexp}

Given integers $k_1,k_2$ such that $k_2\geq 2$, $k_1+k_2+2<0$ and $k_1\equiv k_2\mods 3$, let $$\Lambda=k_1\Lambda_1+k_2\Lambda_2$$ 
as in \S \ref{sec3.1} and let $D^{\Lambda}$ be the corresponding holomorphic discrete series of $G(\mathbb{R})$. 

Let
$$\mcA_\Lambda(N)=\{\pi=\pi_{\infty}\otimes\pi_{f}\in\mathcal{A}(G),\  \omega_{\pi}=\bfone,\ \pi_\infty\simeq D^{\Lambda},\ \pi_{f}^{K_f(N)}\not=\{0\}\}$$
the set of automorphic representations of $G(\Aa)$ having trivial central character, whose archimedean component $\pi_\infty$ is isomorphic to $D_\Lambda$ and whose non-archimedean component $\pi_f$ admits non-trivial $K_f(N)$-invariant vectors. The set $\mathcal{A}_\Lambda(N)$ is finite and contains only cuspidal representations.

For any $\pi\in \mathcal{A}_\Lambda(N)$, the subspace of lowest weight vectors of $\pi_\infty\simeq D_\Lambda$ is one dimensional (generated by $\phi_\circ$ say) and let $v_{\pi_\infty}\not=0$ be such a non-zero vector. We define the finite dimensional vector spaces of automorphic forms 
$$\mcV_{\pi,\Lambda}(N):=\Im(\mathbb{C}v_{\pi_\infty} \otimes \pi_{f}^{K_f(N)}\hookrightarrow L^2([G]))$$ 
(recall that $[G]=G(\Qq)\bash G(\Aa)$)
\begin{align*}
\mcV_{\Lambda}(N):=\bigoplus_{\pi\in \mathcal{A}_\Lambda(N)}\mcV_{\pi,\Lambda}(N).
\end{align*}
Let $\tfn$ be the matrix
$$\widetilde{\mfn}=\prod_{p|N'}\begin{pmatrix}
1&p^{-1}&\\
&1&\\
&&1
\end{pmatrix},$$
  which we view as an element in $G(\mathbb{A})$ (cf. \S\ \ref{seclevel}); we denote by
$$\mcV_{\pi,\Lambda}^{\widetilde{\mfn}}(N):=\pi(\widetilde{\mfn}).\mathcal{V}_{\pi,\Lambda}(N)=\{\tilde \vphi:=\pi(\widetilde{\mfn})\vphi,\ \vphi\in \mathcal{V}_{\pi,\Lambda}(N) \},$$
the image of $\mathcal{V}_{\pi,\Lambda}(N)$ under the action of $\pi(\tfn)$ 
and
\begin{align*}
\mcV_{\Lambda}^{\widetilde{\mfn}}(N):=\bigoplus_{\pi\in \mathcal{A}_\Lambda(N)}\mcV^{\widetilde{\mfn}}_{\pi,\Lambda}(N).
\end{align*}
We fix orthonormal bases of the $\mcV_{\pi,\Lambda}(N)$ and $ \mcV^{\widetilde{\mfn}}_{\pi,\Lambda}(N)$ by
$$\mcB_{\pi,\Lambda}(N)\hbox{ and }\mcB^{\widetilde{\mfn}}_{\pi,\Lambda}(N)=\pi(\widetilde{\mfn})\mcB_{\pi,\Lambda}(N)$$
and finally let 
$$\mcB_{\Lambda}(N)\hbox{ and }\mcB^{\widetilde{\mfn}}_{\Lambda}(N)$$
be their unions over the $\pi\in \mcA_\Lambda(N)$.

\begin{lemma}
With the notation as above, let $\ell\geq 1$ be an integer whose prime divisors are all inert and coprime with $N$ and let $f^{\mfn}$ be the function defined in \S \ref{secglobalf}, \eqref{deffn}.  The image of $R(f^{{\mfn}})$ is contained in $\mathcal{V}_{\Lambda}^{\widetilde{\mfn}}(N)$. Moreover, for $\pi\in\mcA_\Lambda(N)$ and $\phi$ contained in $ \mcV_{\pi,\Lambda}(N),$ we have
\begin{equation}\label{134}
R(f^{\mfn})\pi(\widetilde{\mfn})\phi=\frac{1}{d_{\Lambda}}\lambda_{\pi}(f_\ell)\pi(\widetilde{\mfn})\phi;
\end{equation}
here $d_{\Lambda}$ denotes the formal degree of $D^\Lambda$ and
$$\lambda_{\pi}(f_\ell)=\prod_{p|\ell}\lambda_{\pi_p}(f_p)\in\Cc$$ is a scalar depending on $\pi_\ell=\otimes_{p|\ell}\pi_p$ and on the test functions $(f_p)_{p|\ell}$.
\end{lemma}
\begin{proof}
We denote by $R^{\widetilde{\mfn}}(f)=\pi(\widetilde{\mfn})^{-1}R(f^{\mfn})\pi(\widetilde{\mfn}).$ Then
\begin{align*}
R^{\widetilde{\mfn}}(f)\phi(x)=&\int_{\overline{G}(\mathbb{A})}f(\widetilde{\mfn}^{-1}y\widetilde{\mfn})\phi(x\widetilde{\mfn}^{-1}y\widetilde{\mfn})dy=\int_{\overline{G}(\mathbb{A})}f(y)\phi(xy)dy=\pi(f)\phi(x).
\end{align*}	

From the definition of $f$ : $$f=f_{\infty}\otimes\bigotimes_{p}f_p\in C^{\infty}(G(\mathbb{A})),$$
where $f_\infty$ is a matrix coefficient of $D^\Lambda$ and each $f_p$ is right $K_p(N)$-invariant, we see that the image of $R^{\widetilde{\mfn}}(f)$ is contained in $\mathcal{V}_{\Lambda}(N)$ (in particular $R^{\widetilde{\mfn}}(f)$ vanishes on the Eisenstein and residual spectra: this comes from the choice of $f_\infty$ to be the matrix coefficient of the holomorphic discrete series $D_\Lambda$). For the same reason we have
$$\pi_\infty(f_\infty).v_{\pi_\infty}=\frac{1}{d_\Lambda}v_{\pi_\infty}.$$
Moreover, for $p\mid \ell$, $\pi_p$ is unramified  and $\pi_p^{K_p(N)}$ is one dimensional; since $f_p$ is bi-$K_p(N)$ invariant, for any $v_p\in \pi_p^{K_p(N)}$,  one has
$$\pi_p(f_p).v_p=\lambda_{\pi_p}(f_p).v_p.$$
It follows that  $\mcV_{\pi,\Lambda}(N)$ is one dimensional made of factorable vectors and that for any $\phi\in \mathcal{V}_{\pi,\Lambda}(N)$ one has 
$$R(f)\phi=\frac{1}{d_{\Lambda}}\prod_{p|\ell }\lambda_{\pi_p}(f_p)\phi=\frac{1}{d_{\Lambda}}\lambda_{\pi}(f_\ell)\phi$$
 and that
$$
R(f^{\mfn})\pi(\widetilde{\mfn})\phi=\frac{1}{d_{\Lambda}}\lambda_{\pi}(f_\ell)\pi(\widetilde{\mfn})\phi,
$$	
proving \eqref{134}.
\end{proof}

Recall that the automorphic kernel is given by
\begin{align}
\label{autodecomp}
	J(f^{\mfn})=\sum_{\gamma\in {G}(\mathbb{Q})}f^{\mfn}(x^{-1}\gamma y).
       \end{align}
Applying the expansion \eqref{autodecomp} in $J(f^{\mfn})$ (see \eqref{Jdef}) we then conclude that
\begin{equation}
	\label{s}
J(f^{\mfn})=\frac{1}{d_{{\Lambda}}}\sum_{\pi\in \mcA_{\Lambda}(N)}\lambda_\pi(f_\ell)\sum_{\tilde\vphi\in \mcB^{\tfn}_{\pi,\Lambda}(N)}\frac{\mcP(\tilde\vphi,\vphi')\overline{\mcP(\tilde\vphi,\vphi')}}{\langle\tilde\vphi,\tilde\vphi\rangle}\end{equation}
where $\mcP(\vphi,\vphi')$ denotes
the automorphic period integral
\begin{equation}\label{1.1}
\mcP(\tilde\vphi,\vphi')=\int_{G'(\BQ)\backslash G'(\mathbb{A})}\tilde\vphi(g)\vphi'(g)dg.
\end{equation}
Since $\tilde\vphi$ and $\vphi'$ are cusp forms, $\mcP(\tilde\vphi,\vphi')$ converges absolutely. Let $\mathcal{R}_{\Lambda}^{\widetilde{\mfn}}(N)$ denote the image of $\mcV_{\Lambda}^{\widetilde{\mfn}}(N)$ under the map $R(f^{\tfn})$; it is finite dimensional hence the right-hand side of \eqref{s} is  absolutely convergent.

Set for $\tilde\vphi\in\mcB_{\pi,\Lambda}^{\tfn}(N)$
$$\lambda_{\tilde\vphi}(f_\ell):=\lambda_\pi(f_\ell).$$
 The identity \eqref{s} becomes
 
\begin{lemma}\label{lem34}
With the notation as above, we have
\begin{equation}\label{222}
J(f^{\mfn})=\frac{1}{d_{{\Lambda}}}\sum_{\tilde\vphi\in \mcB_\Lambda^{\tfn}(N)}\lambda_{\tilde\vphi}(f_\ell)\frac{\big|\mcP(\tilde\vphi,\vphi')\big|^2}{\langle\tilde\vphi,\tilde\vphi\rangle}.
\end{equation}
\end{lemma}

\subsection{The Ichino-Ikeda Conjecture for Unitary Groups}

In this section we review the global Ichino-Ikeda conjecture for automorphic forms of $G\times G'= U(V)\times U(W)$ (e.g., see \cite{Har14}, Conjecture 1.3), which is a refinement of the Gan-Gross-Prasad conjecture \cite{GGPW12}  by giving an explicit formula relating the periods and central values of Rankin-Selberg $L$-functions. We now recall the definition of the local analogs of the global period integral \eqref{1.1} (see the beginning of \S \ref{S} for the notation we use).

We recall that $\pi\simeq \otimes_v\pi_v$ and $\pi'\simeq \otimes_v\pi'_v$ denote suitable automorphic representations of $U(V)$ and $U(W)$ which we assume are everywhere tempered. Their respective base changes to $\GL_{3,E}$ and $\GL_{2,E}$ are noted $\pi_E\simeq \otimes_v\pi_{E_v} $ and $\pi'_E\simeq \otimes_v\pi'_{E_v}$. We denote by
$$L(s,\pi_E\times\pi'_E)=\prod_{p}L_p(s,\pi_E\times\pi'_E)=\prod_{\mfp|p}\prod_{\mfp}L_\mfp(s,\pi_E\times\pi'_E)$$
the finite part of their Rankin-Selberg $L$-function and
$$\Lambda(s,\pi_E\times\pi'_E)=L_\infty(s,\pi_E\times\pi'_E)L(s,\pi_E\times\pi'_E)$$
its completed version (see Prop. \ref{propperiodarch} for the exact expression of $L_\infty(s,\pi_E\times\pi'_E)$). As recalled in the introduction, it admits analytic continuation to $\Cc$ and a functional equation\begin{equation}
	\label{fcteqn2}
	\Lambda(s,\pi_E\times\pi'_E)=\eps(\pi_E\times\pi'_E)C_f(\pi_E\times\pi'_E)^{1/2-s}\Lambda(1-s,\pi_E\times\pi'_E)
\end{equation}
In Proposition \ref{conductorthm} below we prove that
$$C_f(\pi_E\times\pi'_E)=N^4{N'}^6$$
and that the root number equals
$$\eps(\pi_E\times\pi'_E)=+1.$$

Let $$\Delta_G=\Lambda(M_G^{\vee}(1),0)$$ be the special (complete) $L$-value where $M_G^{\vee}(1)$ is the twisted dual of the motive $M_G$ associated to $G$ by Gross \cite{Gro97}. Locally, for $v$ any place, we set $$\Delta_{G,v}=L_v(M_{G}^{\vee}(1),0).$$
Explicitly, let $\eta=\prod_v\eta_v$ denote the quadratic character of $\mathbb{Q}^{\times}\backslash\mathbb{A}^{\times}$ associated to $E/\mathbb{Q}$ by class field theory. We have
\begin{align*}
\Delta_{G,v}=\prod_{j=1}^{3}L_v(j,\eta^j)=L_v(1,\eta)L_v(2,\bfone)L_v(3,\eta^3),
\end{align*}
and $$\Delta_{G}=\Lambda(1,\eta)\Lambda(2,\bfone)\Lambda(3,\eta^3).$$
We set
\begin{equation}
	\label{Lpipidef}
	\Lambda(\pi,\pi'):=\Delta_{G}\frac{\Lambda(1/2,\pi_E\times\pi'_E)}{\Lambda(1,\Ad,\pi_E)\Lambda(1,\Ad,\pi_E')}
\end{equation}
and for any place $v$ we set
\begin{equation}
	\label{Lvpipidef}
	L_v(\pi_v,\pi'_v):=\Delta_{G,v}\frac{L_v(1/2,\pi_{E_v}\times\pi'_{E_v})}{L_v(1,\Ad,\pi_{E_v})L_v(1,\Ad,\pi')}.
\end{equation}
Note that, by temperedness, we have for any prime $p$
$$L_p(\pi_p,\pi'_p)=1+O(p^{-1/2})$$
where the implicit constant is absolute; moreover
there exists an absolute constant $C\geq 1$ such that  we have for any prime $p$
\begin{equation}
	\label{upperlowerL}
	C^{-1}\leq L_p(\pi_p,\pi'_p)\leq C.
\end{equation}
We also denote by
\begin{equation}
	\label{Lfpipidef}
	L(\pi,\pi'):=\frac{\Lambda(\pi,\pi')}{L_\infty(\pi_\infty,\pi'_\infty)}
\end{equation}
the "finite part" of the complete Euler product $\Lambda(\pi,\pi')$.

Given any place $v$ of $\BQ$ and any tuple of local vectors $$(\xi_{1,v},\xi_{2,v},\xi'_{1,v},\xi'_{2,v})\in\pi_v\times\pi_v\times\pi'_v\times\pi'_v,$$ the local period is defined formally by
\begin{align*}
\mcP_v(\xi_{1,v},\xi_{2,v};\xi_{1,v}',\xi_{2,v}'):=\int_{G'_{v}}\langle\pi_{v}(g_v)\xi_{1,v},\xi_{2,v}\rangle_v\cdot\ov{\langle\pi'_{v}(g_v)\xi'_{1,v},\xi'_{2,v}\rangle_v}dg_v;
\end{align*}
by a result of Harris \cite{Har14} the integral $\mcP_v(\xi_{1,v},\xi_{2,v};\xi_{1,v}',\xi_{2,v}')$ converges absolutely when both $\pi_{v}$ and $\pi_v'$ are tempered and 
\begin{equation}
	\label{eqlocalpositivity}
	\mcP_v(\xi_{1,v},\xi_{1,v};\xi_{1,v}',\xi_{1,v}')\geq 0.
\end{equation}
\medskip

One then defines the unitarily and arithmetically normalized local period as
 \begin{align}\label{unitarydef}
\mcP_v^{*}(\xi_{1,v},\xi_{2,v};\xi_{1,v}',\xi_{2,v}'):=& \frac{\mcP_v(\xi_{1,v},\xi_{2,v};\xi_{1,v}',\xi_{2,v}')}{\langle\xi_{1,v},\xi_{2,v}\rangle_v \langle\xi'_{1,v},\xi_{2,v}'\rangle_v},\\  \mcP_v^{\natural}(\xi_{1,v},\xi_{2,v};\xi_{1,v}',\xi_{2,v}'):=&\frac{\mcP_v^{*}(\xi_{1,v},\xi_{2,v};\xi_{1,v}',\xi_{2,v}')}{L_v(\pi_v,\pi'_v)}.\label{1.3}
\end{align}

 According to Theorem 2.12 in \cite{Har14}, we have $$\mcP_v^{\natural}=1$$ for almost all places and
 \begin{align*}
\prod_{v}\mcP_v^{\natural}:\ (V_{\pi}\boxtimes{V}_{\pi})\otimes (V_{\pi'}\boxtimes{V}_{\pi'})\longrightarrow \mathbb{C}.
\end{align*}
 is a well defined $G(\mathbb{A})\times G'(\mathbb{A})$-invariant functional.

The global Ichino-Ikeda conjecture for the unitary groups $G\times G'$ then provides an explicit constant of proportionality between $|\mcP|^2$ and $\prod\mcP^{\natural}_v.$ It is now a theorem due to the recent work \cite{BPLZZ19}:

\begin{thm}[\cite{BPLZZ19}, Theorem 1.9]\label{thmBPLZZ}
Let the notation be as before. Assume $\pi$ and $\pi'$ are tempered. Let $\vphi_1, \vphi_2\in V_{\pi},$ $\vphi_1', \vphi_2'\in V_{\pi'}$ be factorable vectors. We have
\begin{equation}\label{eqBPLZZ}
\frac{\mcP(\vphi_1,\vphi_1')\overline{\mcP(\vphi_2,\vphi_2')}}{\langle\vphi_1,\vphi_2\rangle\langle\vphi_1',\vphi_2'\rangle}=\frac{1}{2}\cdot \Lambda(\pi,\pi')\cdot \prod_{v}\mcP_v^{\natural},
\end{equation}
where $\mcP_v^{\natural}$ is defined by \eqref{1.3}.
\end{thm}

\subsubsection{Explication of the Ichino-Ikeda formula}\label{SubSecIIk}
In the next subsection, we explicate the right-hand side of the formula \eqref{eqBPLZZ} for the pairs $$(\vphi_1,\vphi_1')=(\vphi_2,\vphi_2')=(\tilde\vphi,\vphi')\in \pi\times\pi'$$ for  $\tilde\vphi=\pi(\tilde\mfn)\vphi$ appearing in \eqref{222} and $\vphi'$ discussed in \S \ref{secfnchoice}.

Let us recall that 
\begin{itemize}
	\item[--] the representation $\pi=\otimes_{v\leq \infty}'\pi_v$ is a cuspidal representation of $G(\mathbb{A})$ with trivial central character, of level $N$ equal to $1$ or to a prime inert in $E$ such that $\pi_{\infty}\simeq D^{\Lambda},$ with 
	$$\Lambda=k_1\Lambda_1+k_2\Lambda_2,\ k_1+k_2+2<0,\ k_2\geq 0$$ and $k_1\equiv k_2\mods 3$ namely, $D^{\Lambda}$ is a holomorphic discrete series. 
	\item[--] the representation $\pi'=\otimes_{v\leq \infty}'\pi'_v$ is a cuspidal representation of $G'(\Aa)$ with trivial central character, which at finite places is everywhere unramified excepted for at most one place $N'$, split in $E$, where it is Steinberg and at the infinite place has restriction, under the isomorphism $SU(W)\simeq \SL_2$, such that $$\pi'_{\infty}\simeq \pi'_k$$  the holomorphic discrete series representation of even weight $k\geq 2$.
\end{itemize}

Let us now recall (this is a slight generalization of the discussion in $\S \ref{sectempered}$) why the representations $\pi$ and $\pi'$ are tempered. For $\pi'$ this is classical and due to Deligne (\cite{DeligneBBKI}). For $\pi$, if $N$ is a prime (in which case $\pi_N$ is the Steinberg representation whose parameters are $3$-dimensional and indecomposable),  $\pi$ cannot be globally endoscopic. If $\pi$ has level $1$ it could a priori be endoscopic but then come from a representation $\tau=\tau_1\otimes\tau_2$ on $U(1,1)\times U(1)$ as the parameter of $\tau_\infty$ but match the parameter of $\pi_\infty$, $\tau_{1,\infty}$ must lie in the discrete series and hence $\tau_1$ must be tempered by Deligne and therefore $\tau$ is tempered.

By the works of Kottwitz, Milne, Rogawski et al, as assembled in \cite{LR} (see Theorems A p. 291 and B p. 293) the $L$-function $L(s-1,\pi_E)$ is a factor of $L^{(2)}(s,\widetilde S^K,V_\ell)$, the $L$-function defined by the $\Gal(\ov E/E)$-action on $\mathrm{IH}^2_{et}(\ov S^K_{\ov\Qq},V_\ell)$,   the intersection cohomology in degree $2$ of the Baily-Borel-Satake compactification $\ov S^K_{\ov\Qq}$ with coefficient in a local system $V$ depending on the weight of $\pi_\infty$, of the associated Picard modular surface $S^K$ for $K=K_0(N)$. By the work of  Gabber, one knows that the intersection etale cohomology is pure and by the Beilinson-Bernstein-Deligne decomposition theorem, $\mathrm{IH}^*_{et}(\ov S^K_{\ov\Qq},V_\ell)$ is a direct summand of ${H}^*_{et}(\widetilde S^K_{\ov\Qq},V_\ell)$ for any smooth toroidal compactification $\widetilde S^K$ of $S^K$ relative to $\widetilde S^K\mapsto \ov S^K$. Now by Deligne's proof of the Weil conjectures \cite{WeilI}, the eigenvalues of $\Frob_\mfp$ at any prime $\mfp|p\nmid ND$ acting on ${H}^*_{et}(\widetilde S^K_{\ov\Qq},\Qq_\ell)$ have absolute value $\Nr_{E/\Qq}(\mfp)^{j/2}$ (for $j$ depending on the degree and $V_\ell$ ) from which it follows that $\pi_{E,p}$ is tempered.
\begin{remark}
If $\pi$ weren't stable, it could be that only a factor of $L(s-1,\pi_E)$  divides $L^{(2)}(s,\widetilde S^K,V_\ell)$.	
\end{remark}

Let us recall that the automorphic forms $\vphi,\tilde\vphi=\pi(\tilde\mfn)\vphi$ and $\vphi'$ are factorable vectors and correspond to  pure tensors which we denote by
\begin{equation}\label{phiphi'local}
	\vphi\simeq {\otimes}'_{v}\xi_v,\ \tilde\vphi\simeq {\otimes}'_{v}\tilde\xi_v,\ \vphi'\simeq {\otimes}'_{v}\xi'_v.
\end{equation} 
and that the local vectors $$\xi_v,\ \tilde\xi_v=\pi_v(\tilde\mfn_v)\xi_v,\ \xi'_v$$ have the following properties and are uniquely defined up to scalars:
\begin{itemize}
\item[--] $\tilde\xi_v=\xi_v$ unless $v=p=N'$.
\item[--] If $v=\infty$,	$\tilde\xi_\infty=\xi_\infty$ is an highest weight vector of the minimal $K$-type of $D^\Lambda$ and  $\xi'_\infty$ is of minimal weight $k$ (see \S \ref{phipdef}).
\item[--] For every $p$, $\xi_p$ is $K_p(N)$-invariant and $\xi'_p$ is $K'_p(N')$-invariant.
\item[--] In particular if $p$ does not divide $D_ENN'$, then $\xi_p=\tilde\xi_p$ and $\xi'_p$ are invariant under the maximal compact subgroups $G(\Zp)$ and $G'(\Zp)$ respectively and $\pi_p$, $\pi'_p$ are unramified principal series representations.
\end{itemize}

	Since $\pi'$ and $\pi$ are everywhere tempered,  formula \eqref{eqBPLZZ} holds and in the next subsections we will  evaluate the local period integrals and will  provide for 
	$$\Lambda=k_1\Lambda_1+k_2\Lambda_2,\ \ k_1=-k,\ \ k_2=-k/2$$
	an explicit approximation of the central value
$L(1/2,\pi_{E}\times\pi'_E)$ in terms of the square of the period $|\mcP(\vphi,\vphi')|^2$. Our main objective in this section is the following
\begin{prop}\label{prop33}
	With the same notation and assumptions as above. We have
\begin{equation}\label{1.4}
\frac{\big|\mcP(\tilde\vphi,\vphi')\big|^2}{\langle\tilde\vphi,\tilde\vphi\rangle\langle\vphi',\vphi'\rangle}\asymp \frac{1}{d_kN{N'}^2}\frac{L(1/2,\pi_{E}\times\pi'_E)}{L(1,\pi_E,\Ad)L(1,\pi_E',\Ad)}
\end{equation}
where $L(\cdot)$ refers to the finite part of the $L$-functions, and the implicit (positive) constants in $\asymp$ depends at most on the absolute discriminant of $E$. 

\end{prop}

\begin{remark} In particular, this implies \begin{equation}\label{PositivityLfct}
	L(1/2,\pi_{E}\times\pi'_E)\geq 0
	\end{equation} whenever the local periods $\mcP_v^{\natural}$ are non-zero.
	This follows immediately from  \eqref{eqBPLZZ} and \eqref{eqlocalpositivity} and was likely known to experts. However, what we achieve here is an effective dependency between the sizes of the global period and the central $L$-value for our explicit test vectors. This will be crucial for our forthcoming argument (see the proof of Theorem \ref{thmA} in \S \ref{secproofthmA}).
\end{remark}

The proof is a consequence of Theorem \ref{thmBPLZZ} and of the following proposition which evaluates the local period \eqref{1.3} for each place $v$:

\begin{thm}
	\label{proplocalperiod}
 Let $\pi$ and $\pi'$ be the automorphic representations of $G(\Aa)$ and $G'(\Aa)$ described in \S \ref{SubSecIIk} with $\Lambda=-k\Lambda_1+\frac{k}{2}\Lambda_2$,  
 and let $\vphi,\vphi'$ be the factorable automorphic forms described in \eqref{phiphi'local} and below. Let $v$ be a place of $\Qq$. We have

-- Archimedean case: If $v=\infty$ we have
 	\begin{equation}\label{135}
L_\infty(\pi_\infty,\pi'_\infty)\mcP_{\infty}^{\natural}(\xi_{\infty},\xi_{\infty};\xi_{\infty}',\xi'_{\infty})=\frac{1}{d_k}
\end{equation}
where
\begin{equation}
	\label{dkdef}
	d_{k}={k-1}
\end{equation} is the formal degree of $\pi^+_k$.

-- Unramified case: If $v=p$ does not divide $2D_ENN'$, one has \begin{equation}
	\label{unramlocalperiod}
	\mcP_p^\natural(\xi_{p},\xi_{p};\xi_{p}',\xi'_{p})=1.
\end{equation}

-- The case $v=p|D_E$: we have
\begin{align}\label{ramifiedlocalperiod}
	 \Big|L_p(\pi_p,\pi'_p)\mcP_p^{\natural}(\xi_{p},\xi_{p};\xi_{p}',\xi'_{p})-1\Big|\leq \frac{7}{p^3}.
	\end{align}
	In particular, we have
	\begin{align}\label{pEram}
	C^{-1}\leq \mcP_p^{\natural}(\xi_{p},\xi_{p};\xi_{p}',\xi'_{p})\leq C,
	\end{align}
for some absolute constant $C$.

-- The case $v=p=N$: we have
	\begin{equation}\label{inertlocalperiod}
L_p(\pi_p,\pi'_p)\mcP_p^{\natural}(\xi_{p},\xi_{p};\xi_{p}',\xi'_{p})=\frac{1}{p}(1-\frac1p)+\frac{O}{p^2}
	\end{equation}
	with $|O|\leq 4$.	In particular (since $p\geq 3$) we have 
	\begin{equation}
		\label{lowerboundforPinert}\frac{C^{-1}}{p}\leq \mcP_p^{\natural}(\xi_{p},\xi_{p};\xi_{p}',\xi'_{p})\leq \frac{C}{p}
	\end{equation}
for  some absolute constant $C\geq 1$.	
 
-- The case $v=p=N'$: we have
	\begin{equation}\label{180}
	L_p(\pi_p,\pi'_p)\mcP_p^{\natural}(\tilde\xi_p,\tilde{\xi}_p;\xi'_{p},\xi'_{p})=\frac{1}{p^2-1}(1+\frac{O}p)
	\end{equation}
	with
	$|O|\leq \Npbound$. In particular for $p> \Npbound$ we have
	\begin{equation}
		\label{lowerboundforPsplit}\frac{C^{-1}}{p^2}\leq \mcP_p^{\natural}({\tilde\xi}_p,{\tilde\xi}_p;\xi'_{p},\xi'_{p})\leq \frac{C}{p^2}	\end{equation}
	for some absolute constant $C\geq 1$.
\end{thm}

\subsection{The archimedean local period}

In this subsection we discuss \eqref{135}. We start with the following proposition which justifies our choice $(k_1,k_2)$:

\begin{prop}\label{59}
Let the notation and assumptions be the same as above. The global period \eqref{1.1} vanishes unless $k_1=-k$ and $k_2=k/2$.
\end{prop}

\begin{proof}
For $\theta\in[-\pi,\pi),$ let $$z'(\theta)=\diag(e^{i\theta},1,e^{i\theta})\in Z_{G'}(\mathbb{R}),\ \kappa(\theta)=\diag(e^{i\theta},1,e^{-i\theta})\in G'(\Rr),$$
$$\widetilde{z}(\theta)=\diag(e^{i\theta/3},e^{-2i\theta/3},e^{i\theta/3})\in SU(V)(\mathbb{R}).$$ Then an explicit computation (as in the proof of Proposition \ref{pro4}) shows that for $$\forall g\in G'(\mathbb{A}),\ \vphi(g\widetilde{z}(\theta))=e^{-(k_1+2k_2)i\theta}\vphi(g).$$  
One has for any $\theta$
\begin{equation}\label{phi'vary}
	\vphi'(gz'(\theta))=\vphi'(g),\ \vphi'(g\kappa(\theta))=e^{-ik\theta}\vphi'(g).
\end{equation} 
The first equality implies that $$\vphi(gz'(\theta))=\vphi(g)$$
(or the period integral would be $0$). Since $$z'(\theta)\tilde z(\theta)^{-1}=\diag(e^{2i\theta/3},e^{2i\theta/3},e^{2i\theta/3})\in Z_G(\Rr)$$ and $\vphi$ is invariant under the center one has
$$\vphi(g)=\vphi(gz'(\theta))=\vphi(g\tilde z(\theta))=e^{-(k_1+2k_2)i\theta}\vphi(g)$$
and
$$k_1+2k_2=0.$$
%

 
Since $\tilde\vphi=\pi(\tilde\mfn)\vphi\in\mcB_{\pi,\Lambda}^{\tfn}(N)$, then 
$$\vphi(g\kappa(\theta))=e^{ik_1\theta}\vphi(g)$$
and by the second equality in \eqref{phi'vary} we have $k+k_1=0$ or otherwise $\mcP(\vphi,\vphi')=0$.

\end{proof}

\begin{remark}
One could also obtain Proposition \ref{59} through the relative trace formula by computing the geometric side, i.e., orbital integrals: as a consequence of Lemma \ref{26} we have $f_{\infty}(gz(\theta))=e^{-(k_1+2k_2)i\theta}f_{\infty}(g),$ for all $g\in G'(\mathbb{R}).$ Hence, if $k_1+2k_2\neq 0,$ the geometric side vanishes. Then the sum of $|\mcP(\vphi,\vphi')|^2$ is zero. Since each term is nonnegative, then each period vanishes.
\end{remark}

\begin{remark}
Due to Proposition \ref{59}, we will, from now on, take  $(k_1,k_2)=(-k,k/2)$, i.e. 
\begin{equation}
	\label{Lambdadef}
	\Lambda=-k\Lambda_1+\frac{k}{2}\Lambda_2.
\end{equation}
 To ensure absolute convergence of various integrals later we will moreover assume that $$k\geq \kmin$$ an even integer. In the sequel, to simplify notations and since $\Lambda$ is defined in terms of $k$, we will sometimes replace the index $\Lambda$ by $k$ and write $\mcV_{k}(N)$ for $\mcV_{\Lambda}(N)$, $\mcB_{\pi,k}(N)$ for $\mcB_{\pi,\Lambda}(N)$, etc...
	
\end{remark}



The next lemma provides the value of $d_{\Lambda}$:
\begin{lemma}\label{lemformaldegreeDk} Let $$\Lambda=k_1\Lambda_1+k_2\Lambda_2=-k\Lambda_1+\frac{k}2\Lambda_2$$ be as in \eqref{Lambdadef} and $D^\Lambda$ be the corresponding holomorphic discrete series. When $dg$ is the Euler-Poincar\'e measure, its formal degree equals
\begin{equation}\label{dLambdadef}
	d_{\Lambda}=\frac{(k_1+1)(k_2+1)(k_1+k_2+2)}{6}=\dLambda\asymp \frac23k^3.
\end{equation}
\end{lemma}

\begin{proof}
We recall that the simple positive root in this case are
$e_1-e_2$ and $e_2-e_3$ where $e_i$ are the standard basis vectors in $\Rr^3$ and the root space is the hyperplane $\{(x,y,z)\in\Rr^3, x+y+z=0\}$. The $\Lambda_j,\ j=1,2$ are defined by \eqref{4.7}. 

Let $\rho$ be half the sum of the positive roots, it is given by
$$\rho=\frac{1}2(e_1-e_2+e_2-e_3+e_1-e_3)=e_1-e_3=(1,0,-1).$$
Let $w_1$ and $w_2$ be the Weyl reflections defined in \eqref{4.9}, and define $\Lambda'$ by
\[
k_1\Lambda_1+k_2\Lambda_2
=
w_1\circ w_2(\Lambda'+\rho)-\rho.
\]

\begin{remark}\label{Wallacherr} In \cite[Lemma 9.4]{Wal76} $\peter{\lambda+\rho,\alpha}/\peter{\alpha,\alpha}$ should be $\peter{\lambda+\rho,\alpha}/\peter{\rho,\alpha}$.
\end{remark} 
Let us compute the Langlands parameter $\Lambda'=(a,b,c)$: $$w_1\circ w_2(\Lambda'+\rho)-\rho=(c-2,a+1,b+1)=\frac{k_1}3(2,-1,-1)+\frac{k_2}3(1,1,-2)$$
so that
\begin{equation}\label{langlands}
\Lambda'=(\frac{-k_1+k_2}{3}-1,-\frac{k_1+2k_2}{3}-1,\frac{2k_1+k_2}3+2).
\end{equation}
We have the Blattner parameter 
$$\Lambda'+\rho=(\frac{-k_1+k_2}{3},-\frac{k_1+2k_2}{3}-1,\frac{2k_1+k_2}3+1)$$
and
$$\peter{\Lambda'+\rho,\alpha}=\begin{cases}
	k_2+1&\ \alpha=e_1-e_2\\
	-k_1-k_2-2&\ \alpha=e_2-e_3\\
	-k_1-1&\ \alpha=e_1-e_3.
\end{cases}$$

Then \eqref{dLambdadef} follows from Harish-Chandra's formula \cite{HC} that
\begin{align*}
d(D^\Lambda)=\frac13\prod_{\alpha>0}\frac{\peter{\Lambda'+\rho,\alpha}}{\peter{\rho,\alpha}},
\end{align*}
with the product over all the positive roots $\alpha\in \{e_1-e_2, e_2-e_3, e_1-e_3\}.$
\end{proof} 


We can now evaluate explicitly the archimedean local period integral:
\begin{prop}\label{propperiodarch}
 Let $\pi$ and $\pi'$ be the automorphic representations of $G(\Aa)$ and $G'(\Aa)$ described in \S \ref{SubSecIIk}. 
 We have (with $k_1=-k,\ k_2=k/2$)
 
\begin{multline}
	L_\infty(s,\pi_E\times\pi'_E)=\Gamma_\Cc\left(s+\frac{1}{2}\right)\Gamma_\Cc\left(s+\frac{3}{2}\right)\Gamma_\Cc\left(s+\frac{k}2-\frac{3}{2}\right)\\\Gamma_\Cc\left(s+\frac{k}2+\frac{1}{2}\right)	\Gamma_\Cc\left(s+k-\frac{5}{2}\right)\Gamma_\Cc\left(s+{k}-\frac{3}{2}\right)	.	\label{Linftyshape}
\end{multline}
where 
	$$\Gamma_\Cc(s)=2(2\pi)^{-s}\Gamma(s)=\Gamma_\Rr(s+1)\Gamma_\Rr(s),\ \Gamma_\Rr(s)=\pi^{-s/2}\Gamma(s/2).$$
\end{prop}

\proof 
 Recall that   the Langlands parameter of the base changed representation of the holomorphic discrete series of weight $k$ is $(\frac{k-1}{2},-\frac{k-1}{2}),$ i.e., $$z\mapsto \diag((\ov z/z)^{\frac{k-1}{2}},(\ov z/z)^{-\frac{k-1}{2}}).$$ Therefore, the archimedean parameter of the base changed representation $\pi'_E$ is given by 
\begin{align*}
(\frac{k-1}{2},\frac{1-k}{2}):\ z\mapsto \diag((\bar{z}/z)^{\frac{k-1}{2}},(\bar{z}/z)^{-\frac{k-1}{2}}),
\end{align*}

Recall that (cf. \eqref{langlands}) the Langlands parameter of $\pi_{\infty}$ is  
\begin{align*}
\Lambda'=(\frac{-k_1+k_2}{3}-1,-\frac{k_1+2k_2}{3}-1,\frac{2k_1+k_2}3+2)=(\frac{k}2-1,-1,-\frac{k}2+2).
\end{align*}
Let $\pi_{\infty,\mathbb{C}}$ be the base change of $\pi_{\infty}$ to $\mathrm{GL}_3(\mathbb{C}).$ Then 
the archimedean parameter of $\pi_{\infty,\mathbb{C}}\otimes\pi_{\infty,\mathbb{C}}'$ is given by
$$
(\frac{k}2-1,-1,-\frac{k}2+2)\otimes (\frac{k-1}{2},\frac{1-k}{2})=(k-\frac{3}{2},\frac{k}2-\frac{3}2,\frac{3}{2},
-\frac{1}{2},-\frac{k}{2}-\frac{1}{2},-k+\frac{5}{2}).
$$
and the result follows since for $r\in\frac12\Zz$
$$L_\infty(z\mapsto (\ov z/z)^{r},s)=L_\infty((z\ov z)^s(\ov z/z)^{r})=\Gamma_\Cc(s+|r|)$$
\qed

\subsubsection{Proof of \eqref{135}}
By definition we have
$$
\mcP_{\infty}(\xi_{\infty},\xi_{\infty};\xi_{\infty}',\xi'_{\infty})
=\int_{G'(\Rr)}\peter{g_{\infty}.\xi_{\infty},\xi_{\infty}}_{\infty}\ov{\peter{\pi'_{\infty}(g_{\infty})\xi'_{\infty},\xi'_{\infty}}}_{\infty}dg_{\infty}.
$$

Note that by Lemma \ref{26}, 
$$\frac{\peter{\pi_{\infty}(g_{\infty})\xi_{\infty},\xi_{\infty}}_{\infty}}{\peter{\xi_\infty,\xi_\infty}_\infty}=f_{\infty}(g_{\infty}).$$
 Then by Lemma \ref{restrictionmatrix} and the Schur orthogonality relations we obtain
$$
\frac{\mcP_{\infty}(\xi_{\infty},\xi_{\infty};\xi_{\infty}',\xi'_{\infty})}{
\peter{\xi_\infty,\xi_\infty}_\infty\peter{\xi'_\infty,\xi'_\infty}_\infty}=\int_{{G}'(\mathbb{R})}f_{\infty}(g_{\infty})\frac{\peter{\pi'_{\infty}(g_{\infty})\xi'_{\infty},\xi'_{\infty}}_{\infty}}{\peter{\xi'_\infty,\xi'_\infty}_\infty}dg_{\infty}=\frac1{d_{\pi'_\infty}}.
$$
where $d_{\pi'_\infty}=d_{k}$ is the formal degree of $ \pi^+_k$
Then \eqref{135} follows.
\qed

\subsection{The root number and the conductor of $L(s,\pi_E\times\pi'_E)$}\label{seccond}
In this section we compute the root number and the arithmetic conductor of $L(s,\pi_E\times\pi'_E)$. First we observe that $\pi_E$ and $\pi'_E$ are conjugate self-dual, i.e. if $c\in\Aut_\Qq(E)$ denotes the non-trivial automorphism, we have
$$\pi_E\circ c\simeq \pi^\vee_E\simeq \ov\pi_E,\ {\pi'}_E\circ c\simeq {\pi'}^\vee_E\simeq \ov{\pi'}_E$$
(since the representations are unitary). In particular, the functional equation relates $\Lambda(s,\pi_E\times\pi'_E)$ to $\Lambda(1-s,\pi_E\times\pi'_E)$ and $\eps(\pi_E\times\pi'_E)\in\{\pm 1\}$.

\begin{prop}
	\label{conductorthm} Let $\pi$ and $\pi'$ be the automorphic representations of $G(\Aa)$ and $G'(\Aa)$ described in \S \ref{SubSecIIk} ; let $\pi_E$, $\pi'_E$ be the corresponding base change to $\GL(3,\Aa_E)$ and $\GL(2,\Aa_E)$ and $L(s,\pi_E\times\pi'_E)$ be (the finite part of) its associated Rankin-Selberg $L$-function. Its arithmetic conductor equals
	$$C_{f}(\pi_E\times\pi'_E)=N^4{N'}^6$$ and its root number equals
	$$\eps(\pi_E\times\pi'_E)=+1.$$
	Consequently, its analytic conductor $C(s,\pi_E\times\pi'_E)$ satisfies (for $\Re s=1/2$)
	\begin{equation}
		\label{eqcondbound}C(s,\pi_E\times\pi'_E)\asymp N^4{N'}^6|s|^4(|s|+k)^{8}.
	\end{equation}
	
	In particular for $s=1/2$ one has the convexity bound 
	$$L(1/2,\pi_E\times\pi'_E)\ll (kNN')^{o(1)}C(1/2,\pi_E\times\pi'_E)^{1/4}\ll (kNN')^{o(1)}k^2N{N'}^{3/2}.$$
\end{prop}

\begin{proof} Since $(N,N')=1,$ the arithmetic conductor of $L(s,\pi_E\times\pi'_E)$ is simply $$C_{f}(\pi_E\times\pi'_E)=(N^2)^2({N'}^2)^3=N^4{N'}^6;$$
indeed for $N$ a prime inert in $E$, the conductor of the Steinberg representation for $\GL_3(E_N)$ is $N^2$ (the norm of the ideal $N\mcO_{E_N}$)  and for $N'$ a prime split in $E$, $\GL_2(E_{N'})\simeq \GL_2(\Qq_{N'})^2$ and $\pi_{N'}\simeq {\St}_{N'}\otimes {\St}_{N'}$ has conductor ${N'}^2$.

By \eqref{Linftyshape} the archimedean conductor is (for $\Re s=1/2$),
\begin{align*}
	C_{\infty}(s,\pi_E\times\pi'_E)&\asymp |s+1/2|^2|s+3/2|^2|s+k/2-3/2|^2\\&\ \ \times|s+k/2+1/2|^2|s+k-5/2|^2|s+k-3/2|^2\\
&\asymp |s|^4(|s|+k)^8.
\end{align*}
and the analytic conductor $$C(s,\pi_E\times\pi'_E)=C_{\infty}(s,\pi_E\times\pi'_E)C_f(\pi_E\times\pi'_E)$$
satisfies \eqref{eqcondbound}.

The convexity bound follows from applying the approximate functional equation for $L(1/2,\pi_E\times\pi_E')$
and from the following bounds for the coefficients of $L(s,\pi_E\times\pi_E')$ 
$$\lambda_{\pi_E\times\pi_E'}(n)\ll_\eps n^{\eps}$$
for any $\eps>0$ (the later is a consequence of the temperedness of $\pi$ and $\pi'$).

Let us turn to the computation of the root number: it decomposes as a product of local root numbers (along the places of $\Qq$ and, say, relative to the usual unramified additive character)
$$\eps(\pi_E\times\pi'_E)=\eps_\infty(\pi_E\times\pi'_E)\prod_p\eps_p(\pi_E\times\pi'_E)$$
and for any place $v$ of $\Qq$, we have the further factorization 
$$\eps_v(\pi_E\times\pi'_E)=\prod_{w|v}\eps_w(\pi_E\times\pi'_E)$$
where $w$ runs over places of $E$ above $v$.
If $v=\!\! p\not\!| NN'$ then $\pi_E$ and $\pi'_E$ are unramified at any place $w|p$ and
$$\eps_w(\pi_E\times\pi'_E)=1.$$

If $p=N$ which is inert in $E$, then $\pi'_E$ is unramified at $p$ and
$$\eps_p(\pi_E\times\pi'_E)=\eps_p(\pi_E)^2=1$$
($\pi_{E,p}=\St_p$ and $\eps_p(\pi_E)=\pm 1$).

If $p|N'$ then
$$\eps_p(\pi_E\times\pi'_E)=\eps_\mfp(\pi_E\times\pi'_E)\eps_{\ov\mfp}(\pi_E\times\pi'_E)=(\eps_\mfp(\pi'_E)\eps_{\ov\mfp}(\pi'_E))^3$$ since $\pi_E$ is unramified at these places. Also since $\pi'_{E_\mfp}$ is the base change of the Steinberg representation, $\pi'_{E_\mfp}\simeq \pi'_{E_{\ov\mfp}}$ and
$\eps_\mfp(\pi'_E)\eps_{\ov\mfp}(\pi'_E)=1$.

Finally for $v=\infty$ the unique archimedean place we have seen above that $\pi_E\times\pi'_E$ has parameters
$$(k-\frac{3}{2},\frac{k}2-\frac{3}2,\frac{3}{2},
-\frac{1}{2},-\frac{k}{2}-\frac{1}{2},-k+\frac{5}{2}).$$
By \cite[\S 3.2]{Tate} for $r\in\frac12\Zz$
$$\eps_\infty(z\ra (\ov z/z)^r)=i^{2r}$$ and in this case we have
$$k-\frac{3}{2}+\frac{k}2-\frac{3}2+\frac{3}{2}-\frac12-\frac{k}{2}-\frac{1}{2}-k+\frac{5}{2}=0
$$
hence
$$\eps_\infty(\pi_E\times\pi'_E)=1.$$
\end{proof}
\begin{remark} Notice that by \cite{GGPW12}, $\pi_E$ is conjugate orthogonal and $\pi'_E$  conjugate symplectic so with this information,  the sign $\eps(\pi_E\times\pi'_E)$ could only be $+1$ or $-1$.
	
\end{remark}

\subsection{The local periods at unramified primes}\label{unramifiedsection}

Let now $v=p$ be a prime that does not divide $NN'D_E$ then (see the discussion in \S \ref{SubSecIIk}) $\xi_p,\xi'_p$ are unramified vectors and the 7 conditions on page 308 of \cite{Har14} (see also \cite{II10}, (U1)-(U6), p.5) are satisfied for $\pi$ and $\pi'.$ Consequently, by Theorem 2.12 of \cite{Har14}, $$\mcP_p^\natural(\xi_{p},\xi_{p};\xi_{p}',\xi'_{p})=1$$
which is \eqref{unramlocalperiod}.

Until the end of Section \ref{S} the following notation will be used for diagonal matrices: given a finite prime $\mfp$, $\varpi$ a uniformizer at $\mfp$ and $n\in \Zz$, we write\footnote{Notice the difference with the earlier notation \eqref{Aprdef}; the present one is for convenience.}
 $$A_{n}:=\diag(\varpi^n,1,\overline{\varpi}^{-n}),\ A'_n=\diag(\varpi^n,\overline{\varpi}^{-n}).$$

\subsection{The local periods at primes ramified in $E$}

In this section we establish \eqref{ramifiedlocalperiod}. Let $p=\mathfrak{p}^2$ be a ramified prime a Since $\pi_p$ and $\pi'_p$ are tempered principal series, we may assume $\pi_p=\Ind\chi_p$ and $\pi_p'=\Ind\chi_p',$ for some unramified unitary  characters $\chi_p$ and $\chi_p'$ of the respective diagonal tori. We set
  $$\gamma_p=\chi_p(A_1), \gamma_p'=\chi_p'(A_1').$$

 By Macdonald's formula (cf. \cite{Mac73} or \cite{Cas80}) we have for $n\geq 0,$
\begin{equation}\label{mcdoG'}
	\frac{\langle\pi_p'(A_n)\xi'_p,\xi'_p\rangle_p}{\langle\xi'_p,\xi'_p\rangle_p}= \frac{(1-p^{-1}\gamma_p'^{-1})\gamma_p'^n-(1-p^{-1}\gamma_p') \gamma_p'^{-n-1}}{p^{n}\big[(1-p^{-1}\gamma_p'^{-1})-(1-p^{-1}\gamma'_p)\gamma_p'^{-1}\big]},
\end{equation}
\begin{multline}
\label{mcdoG}
	\frac{\langle\pi_p(A_n)\xi_p,\xi_p\rangle_p}{\peter{\xi_{p},\xi_{p}}_p}=\\
	\frac{1}{p^{2n}}\Big[\frac{(\gamma_p-p^{-2})(1+p^{-1}\gamma_p^{-1})\gamma_p^n-(\gamma_p^{-1}-p^{-2})(1+p^{-1}\gamma_p)\gamma_p^{-n}}{(\gamma_p-p^{-2})(1+p^{-1}\gamma_p^{-1})\gamma_p-(\gamma_p^{-1}-p^{-2})(1+p^{-1}\gamma_p)\gamma_p^{-1}}\Big]
\end{multline}
and these inner product vanish for $n<0$. 

Therefore, if the Haar measure on $G'(\Qp)$ is such that $G'(\mathbb{Z}_p)$ has measure $1$, by the Cartan decomposition we have
\begin{equation}\label{219}
\int_{}\frac{\langle\pi_{p}(g_p)\xi_{p},\xi_{p}\rangle_p\langle\pi'_{p}(g_p)\xi'_{p},\xi'_{p}\rangle_p}{\peter{\xi_{p},\xi_{p}}_p\peter{\xi'_{p},\xi'_{p}}_p}dg_p=\sum_{n\geq 0}\frac{\langle\pi_{p}(A_n)\xi_{p},\xi_{p}\rangle_p\langle\pi'_{p}(A_n)\xi'_{p},\xi'_{p}\rangle_p}{\peter{\xi_{p},\xi_{p}}_p\peter{\xi'_{p},\xi'_{p}}_p},
\end{equation}
where the integral on the left hand side is taken over ${G}'(\mathbb{Q}_p).$

A straightforward calculation shows that
\begin{align*}
\Bigg|\frac{\langle\pi_p'(A_n)\xi'_p,\xi'_p\rangle_p}{\langle\xi'_p,\xi'_p\rangle_p}\Bigg|=\Bigg|\frac{1-\gamma_p'^{2n+1}-p^{-1}\gamma_p'(1-\gamma_p'^{2n-1})}{(1+p^{-1})(1-\gamma_p')p^n}\Bigg|.
\end{align*}
Then expanding the fractions into a geometric sum and appealing to the triangle inequality, we obtain, when $n\geq 1,$ that
\begin{equation}\label{212}
L_n':=\Bigg|\frac{\langle\pi_p'(A_n)\xi'_p,\xi'_p\rangle_p}{\langle\xi'_p,\xi'_p\rangle_p}\Bigg|\leq \frac{2n+1+p^{-1}(2n-1)}{(1+p^{-1})p^n}.
\end{equation}

Similarly, a straightforward calculation gives
\begin{align*}
L_n:=
\left|
\frac{\langle \pi_p(A_n)\xi_p,\xi_p\rangle_p}
{\peter{\xi_p,\xi_p}_p}
\right|=
\left|
\frac{
\widetilde{\gamma}_p^{n+1}
+p^{-1}(1-p^{-1})\widetilde{\gamma}_p^n
-p^{-3}\widetilde{\gamma}_p^{n-1}
}{
p^{2n}(1+p^{-3})(\gamma_p-\gamma_p^{-1})
}
\right|,
\end{align*}
where, for $m\in\Zz$,
\[
\widetilde{\gamma}_p^m
:=
\gamma_p^m-\gamma_p^{-m}.
\]
For $m\geq 0$, we have
\[
\frac{|\widetilde{\gamma}_p^m|}
{|\widetilde{\gamma}_p^1|}
\leq m,
\qquad
\frac{
|\widetilde{\gamma}_p^{m+1}-\widetilde{\gamma}_p^{m-1}|
}{
|\widetilde{\gamma}_p^1|
}
\leq 2.
\]
Applying these estimates to the preceding expression, after writing
$\widetilde{\gamma}_p^{n+1}-p^{-3}\widetilde{\gamma}_p^{n-1}
=
(\widetilde{\gamma}_p^{n+1}-\widetilde{\gamma}_p^{n-1})
+(1-p^{-3})\widetilde{\gamma}_p^{n-1}$,
we immediately obtain, for $n\geq 1$,
\begin{equation}\label{213}
L_n\leq\frac{2+(1-p^{-3})(n-1)+p^{-1}(1-p^{-1})n}{(1+p^{-3})p^{2n}}.
\end{equation}

It follows from \eqref{212} and \eqref{213} that
\begin{equation}\label{5.38}
\sum_{n\geq 1}
\left|
\frac{
\langle \pi_p(A_n)\xi_p,\xi_p\rangle_p
\langle \pi'_p(A'_n)\xi'_p,\xi'_p\rangle_p
}{
\peter{\xi_p,\xi_p}_p
\peter{\xi'_p,\xi'_p}_p
}
\right|=\sum_{n\geq 1}L'_nL_n\leq
p^{-3}S(p),
\end{equation}
where
\begin{align*}
S(p)
:=
p^3\sum_{n\geq 1}
\frac{2n+1+p^{-1}(2n-1)}
{(1+p^{-1})p^n}
\cdot
\frac{
2+(1-p^{-3})(n-1)+p^{-1}(1-p^{-1})n
}{
(1+p^{-3})p^{2n}
}.
\end{align*}
Using the identity
\begin{align*}
\sum_{n\geq 1}\frac{An^2+Bn+C}{p^{3n}}
&=
\frac{Ap^3(p^3+1)}{(p^3-1)^3}
+\frac{Bp^3}{(p^3-1)^2}
+\frac{C}{p^3-1},
\qquad A,B,C\in\mathbb{R},
\end{align*}
we obtain
\begin{align*}
S(p)
=
\frac{
p^3(6p^8+5p^7+4p^6+p^5+8p^4+3p^3+p^2-p+1)
}{
(p^3+1)(p-1)^2(p^2+p+1)^3
}.
\end{align*}
A direct calculation shows that
\begin{align*}
S(2)=\frac{2328}{343}<7,
\qquad
S(p)<6\quad\text{for all }p\geq 3.
\end{align*}
Consequently, $S(p)\leq 7$ for every prime $p$. It therefore follows from
\eqref{5.38} that
\begin{equation}\label{217}
\sum_{n\geq 1}
\left|
\frac{
\langle \pi_p(A_n)\xi_p,\xi_p\rangle_p
\langle \pi'_p(A'_n)\xi'_p,\xi'_p\rangle_p
}{
\peter{\xi_p,\xi_p}_p
\peter{\xi'_p,\xi'_p}_p
}
\right|
\leq
7p^{-3}.
\end{equation}

Finally, combining \eqref{219} and \eqref{217}, we obtain
\begin{align*}
\left|
\int_{G'(\mathbb{Q}_p)}
\frac{
\langle \pi_p(g_p)\xi_p,\xi_p\rangle_p
\langle \pi'_p(g_p)\xi'_p,\xi'_p\rangle_p
}{
\peter{\xi_p,\xi_p}_p
\peter{\xi'_p,\xi'_p}_p
}
\,dg_p
-1
\right|
\leq
\frac{7}{p^3}.
\end{align*}
This proves \eqref{ramifiedlocalperiod}.
\qed

\subsection{The matrix coefficient of the Steinberg representation for $U(V)$}\label{secmatrixStG}
We continue to assume that $p$ is inert in $E$. Let $\mu=\mu_p$ be a Haar  measure on $G(\Qp)$.
Denote by $W_0=\big\{\bfone_p, J\big\}.$ For $n\geq 1,$ set $$W_n=\big\{A_n, JA_n, A_nJ, JA_nJ\big\}.$$
 Let $$W:=\bigsqcup_{n\geq 0}W_n.$$
 
 We recall that $I_p$ denotes the Iwahori subgroup as defined in \eqref{ke}.

 By Lemma \ref{162} (see the Appendix) and the Cartan decomposition we then obtain
\begin{equation}\label{163}
G(\mathbb{Q}_p)=\bigsqcup_{n\geq 0}G(\mathbb{Z}_p)A_nG(\mathbb{Z}_p)=\bigsqcup_{w\in W}I_pwI_p.
\end{equation}

By Lemma \ref{lemIwahoriCartanU(V)} one has $\mu(I_pJI_p)=p^3\mu(I_p).$ Moreover, by Lemma \ref{160}, for $n\geq 1,$  
\begin{gather*}
\mu(I_pA_nI_p)=p^{4n}\mu(I_p),\ \mu(I_pJA_nI_p)=p^{4n-3}\mu(I_p),\\
\ \mu(I_pA_nJI_p)=p^{4n+3}\mu(I_p),\ \mu(I_pJA_nJI_p)=p^{4n}\mu(I_p).
\end{gather*}
Then for $w\in W,$ there exists a unique integer $\lambda(w)\in\mathbb{Z}_{\geq 0}$ such that $$\mu(I_pwI_p)=p^{\lambda(w)}\mu(I_p).$$ In particular, $\lambda(\bfone_p)=0,$ $\lambda(J)=3;$ and for $n\geq 1,$ $$\lambda(A_n)=4n,\ \lambda(JA_n)=4n-3,\ \lambda(A_nJ)=4n+3\hbox{ and }\lambda(JA_nJ)=4n.$$
 Thus, the Poincar\'{e} series of the Weyl group
\begin{equation}\label{168}
\sum_{w\in W}p^{-2\lambda(w)}=1+p^{-6}+(2+p^6+p^{-6})\cdot \sum_{n\geq 1}p^{-8n}
\end{equation}
converges absolutely; this will be used below to evaluate the matrix coefficients of the Steinberg representation.

Let $\Theta_p$ be the $I_p$-bi-invariant function on $G(\mathbb{Q}_p)$ defined by 
\begin{equation}
	\label{Xipdef}
	\Theta_p(w):=(-p)^{-\lambda(w)}\mu(I_p),\ w\in W.
\end{equation}

\begin{lemma}\label{171}
Let $f$ be a function on $G(\mathbb{Q}_p)$ defined by $$f(bk)=\delta_{P}(b)\Theta_p(k),$$ for $b\in P(\mathbb{Q}_p)$ and $k\in G(\mathbb{Z}_p)$ and $\delta_{P}$ the module character of the Borel subgroup $P.$ Then
\begin{equation}\label{164}
\vphi_f(g):=\int_{I_p}f(\kappa g)d\kappa=\Theta_p(g)\Theta_p(\bfone_p), \quad \forall g\in G(\mathbb{Q}_p).
\end{equation}
\end{lemma}
\begin{proof}
Write $g=bk$ in its Iwasawa decomposition, then by the Iwahori decomposition $G(\mathbb{Z}_p)=I_p\bigsqcup I_pJI_p$ we have
\begin{multline*}
\int_{I_p\bigsqcup I_pJI_p}f(g\kappa )d\kappa=\int_{G(\mathbb{Z}_p)}f(b\kappa )d\kappa\\=\delta_{P}(b)\int_{G(\mathbb{Z}_p)}\Theta_p(\kappa)d\kappa=\delta_{P}(b)(1+p^3\Theta_p(J))\mu(I_p),
\end{multline*}
where the last equality comes from Lemma \ref{lemIwahoriCartanU(V)}. By \eqref{Xipdef} we have $1+p^3\Theta_p(J)=0$, so that
\begin{equation}\label{165}
\int_{I_p\bigsqcup I_pJI_p}f(g\kappa )d\kappa=0, \quad \forall g\in G(\mathbb{Q}_p).
\end{equation}

Let $g=bk$ with $k\in I_p.$ Then we have by definition
\begin{align*}
\int_{I_p\bigsqcup I_pJA_1I_p}f(g\kappa )d\kappa=&\delta_{P}(b)\cdot\bigg[\int_{I_p}\Theta_p(\kappa)d\kappa+\int_{I_pJA_1I_p}\Theta_p(\kappa)d\kappa\bigg]\\
=&\delta_{P}(b)\cdot \left(\mu(I_p)+\Theta_p(JA_1)\mu(I_pJA_1I_p)\right)=0.
\end{align*}

When $g=bk$ and $k\in I_pJI_p,$ we have by \eqref{120} 
$$
k=n(\delta,\tau)J\kappa_1,\quad n(\delta,\tau):=\begin{pmatrix}
	1&\delta&\tau\\&1&-\ov \delta\\&&1
\end{pmatrix}
$$  
for some $\delta,\tau\in\Ff_{p^2}$ and $\kappa_1\in I_p$. So
\begin{align*}
\int_{I_p\bigsqcup I_pJA_1I_p}f(g\kappa )d\kappa=&\int_{I_p}f(bn(\delta,\tau)J\kappa )d\kappa+\int_{I_pJA_1I_p}f(bn(\delta,\tau)J\kappa )d\kappa\\
=&\delta_{P}(b)\cdot\bigg[\int_{I_p}\Theta_p(J\kappa)d\kappa+\int_{I_pJA_1I_p}\Theta_p(J\kappa)d\kappa\bigg].
\end{align*}

By Lemma \ref{160} we have
\begin{align*}
\int_{I_p}\Theta_p(J\kappa)d\kappa+\int_{I_pJA_1I_p}\Theta_p(J\kappa)d\kappa=&\Big[\Theta_p(J)+\sum_{\substack{
		\tau\in p\mathcal{O}_{p}/p^{2}\mathcal{O}_{p}\\ \tau+\overline{\tau}=0}}\Theta_p(n(0,\tau)A_1)\Big]\mu(I_p)\\
=&\Theta_p(J)\mu(I_p)+p\Theta_p(A_1)\mu(I_p).
\end{align*}
Note that $\Theta_p(J)+p\Theta_p(A_1)=-p^{-3}+p\cdot (-p)^{-4}=0.$ Hence, we have
\begin{equation}\label{166}
\int_{I_p\bigsqcup I_pJA_1I_p}f(g\kappa )d\kappa=0, \quad \forall g\in G(\mathbb{Q}_p).
\end{equation}

Note that the function $\vphi_f$ is $I_p$-bi-invariant. So we only need to verify \eqref{164} for all $g\in W.$ Clearly, \eqref{164} holds for $g=\bfone_p\in W.$ Moreover, by \eqref{165} and \eqref{166}, \eqref{164} holds for $g\in \{J, JA_1\}\subset W.$ Hence, we have
\begin{equation}\label{167}
\int_{I_p\bigsqcup I_pwI_p}\vphi_f(g\kappa )d\kappa=0,\ \forall w\in \big\{J, JA_1\big\},\ g\in G(\mathbb{Q}_p).
\end{equation}

Hence, expanding \eqref{167} we then see $\vphi_f(w)=\Theta_p(w)$ holds for all $w\in W$ with $\lambda(w)\leq 3.$ 

Let $n\geq 3.$ Suppose $\vphi_f(w)=\Theta_p(w)$ holds for all $w\in W$ such that $\lambda(w)\leq n.$ Let $w'\in W$ be such that $\lambda(w')=n+1.$ Then there exists $w_1\in \big\{J, JA_1\big\},$ and $w_2\in W-\big\{\bfone_p\big\}$ with $w_2w_1=w'$ and $\lambda(w_2)+\lambda(w_1)=n+1.$ Explicitly, suppose $w'\in W_m,$ $m\geq 1.$ When $w'=A_mJ,$ then $w_1=J$ and $w_2=A_m;$ when $w'=JA_mJ,$ then $w_1=J$ and $w_2=JA_m;$ when $w'=A_m,$ then $w_1=JA_1$ and $w_2=A_{m-1}J;$ when $w'=JA_m,$ then $w_1=JA_1$ and $w_2=JA_{m-1}J.$

We have $1\leq \lambda(w_1)\leq 3$ and $\lambda(w_2)\leq n.$ Also, by Lemma \ref{160} one has $$I_pw_2I_pw_1I_p= I_pwI_p.$$ Hence, by our assumption, and taking $g=w_1,$ one then has
\begin{equation}\label{169}
\int_{I_p}\vphi_f(w_2\kappa )d\kappa+\int_{I_pw_1I_p}\vphi_f(w_2\kappa )d\kappa=\int_{I_p\bigsqcup I_pw_1I_p}\vphi_f(w_2\kappa )d\kappa=0.
\end{equation}

Since $\vphi_f$ is bi-invariant under $I_p$ and $I_pw_1I_pw_2I_p=I_pwI_p,$ \eqref{169} then becomes
\begin{equation}\label{170}
\vphi_f(w_2)+\frac{\mu(I_pw_1I_p)}{\mu(I_p)}\vphi_f(w)=0,\ i.e.,\ \vphi_f(w_2)+q^{\lambda(w_1)}\vphi_f(w')=0.
\end{equation}

By assumption, $\vphi_f(w_2)=\Theta_p(w_2).$ Hence by \eqref{170}, $$\vphi_f(w')=-p^{-\lambda(w_1)}\Theta_p(w_2)=(-p)^{-\lambda(w_1)-\lambda(w_2)}\mu(I_p)=\Theta_p(w').$$ Thus \eqref{164} follows by induction.
\end{proof}

\begin{prop}\label{matrixsteinberg} Let $p$ be a prime inert in $E$ and $\mathrm{St}_p$ be the Steinberg representation of $G(\Qp)$ ; the function $\Theta_p$ is a matrix coefficient of $\mathrm{St}_p.$ Precisely, let $\xi_p\neq 0$ be a local new vector in $\mathrm{St}_p$ (a generator of the one dimensional space of $I_p$-invariant vectors), then
\begin{equation}\label{176}
\frac{\langle\mathrm{St}_p(g)\xi_p,\xi_p\rangle_p}{\peter{\xi_{p},\xi_{p}}_p}=\frac{\Theta_p(g)}{\Theta_p(\bfone_p)}, \quad \forall\ g\in G(\mathbb{Q}_p).
\end{equation}
\end{prop}
\begin{proof}
For $g_1\in G(\mathbb{Q}_p)$, let $\vphi=\vphi_{g_1}$ be the function on $G(\mathbb{Q}_p)$ given by
\begin{equation}\label{174}
\vphi(g):=\int_{I_p}\Theta_p(g_1\kappa g)d\kappa.
\end{equation}
That function is $I_p$-bi-invariant. 

By a  computation similar to \eqref{165} and \eqref{166}, we have
\begin{equation}\label{172}
\int_{I_p\bigsqcup I_pwI_p}\Theta_p(g'\kappa )d\kappa=0,\ \forall w\in \big\{J, JA_1\big\},\ g'\in G(\mathbb{Q}_p).
\end{equation}

Write $g'$  in the form of $g'=g_1\kappa'g$ and integrate over $\kappa'\in I_p$ in \eqref{172}  to obtain for $w\in \big\{J, JA_1\big\}$
\begin{align*}
\int_{I_p\bigsqcup I_pwI_p}\vphi(g\kappa )d\kappa=\int_{I_p\bigsqcup I_pwI_p}\int_{I_p}\Theta_p(g_1\kappa'g\kappa )d\kappa'd\kappa=0.
\end{align*}

By an  induction argument similar to the proof of Lemma \ref{171}, we deduce that
 \begin{equation}\label{phiproduct}
 \vphi(g)=\Theta_p(g_1)\Theta_p(g).	
 \end{equation}
Taking $g=g_0$ and letting $g_1$ vary we obtain that
\begin{equation}\label{173}
\Theta_p(g_0)\Theta_p(g_1)=\vphi(g_0)=\int_{I_p}\Theta_p(g_1\kappa g_0)d\kappa=\int_{I_p}(\pi_p(\kappa g_0))\Theta_p(g_1)d\kappa.
\end{equation}

Let $\mcV$ be the vector space spanned by the right translates of the function $\Theta_p.$ Then $\mcV$ is a smooth representation of $G(\mathbb{Q}_p)$ which we denote by $(\pi_p,\mcV)$.  Suppose that  $\pi_p$ is reducible: there exists some nonzero $G(\mathbb{Q}_p)$-invariant subspace $\mcV'\subsetneq \mcV$ and some $g_0\in G(\mathbb{Q}_p)$ such that $\pi_p(g_0)\Theta_p\in \mcV'$ and $$\Theta_p(g_0)=\Theta_p(\bfone_p\cdot g_0)=\pi_p(g_0)\Theta_p(\bfone_p)\neq 0.$$ 
Note by our assumption, $\Theta_p(g_0)\neq 0.$ Hence, from \eqref{173} we get
\begin{align*}
\Theta_p(\bullet)=\frac{1}{\Theta_p(g_0)}\int_{I_p}\pi_p(\kappa g_0)\Theta_p(\bullet)d\kappa\in \mcV'.
\end{align*}
Therefore, $\mcV\subseteq \mcV',$ a contradiction! So $(\pi_p,\mcV)$ is irreducible.
Let $$\Pi_p=\Ind_{P(\mathbb{Q}_p)}^{G(\mathbb{Q}_p)}(|\cdot|_p^1,1,|\cdot|_p^{-1})$$ be the induced representation. Then the functions $f$ on $G(\mathbb{Q}_p)$ which belong to the space of $\Pi_p$ are precisely the functions $G(\Zp)$-finite on the right which satisfy $$f(pg)=\delta_P(p)f(g).$$ Hence the function $\vphi_f$ defined in Lemma \ref{171} belongs to the space of $\Pi_p$ and therefore $\pi_p$ is an irreducible component of $\Pi_p$ . 
The classification of irreducible admissible representations on $G(\mathbb{Q}_p)$ implies that  $\pi_p$ is the Steinberg representation $\St_p$. 

Since $\Theta_p$ is $I_p$-invariant and the space of $I_p$-invariant vectors (the space of local new-vectors) in the Steinberg representation has dimension $1$, there exists a nonzero constant $\lambda$ such that $\xi_p=\lambda\cdot \Theta_p.$ It remains to
 compute the matrix coefficient of $\Theta_p$. Let $g\in G(\mathbb{Q}_p).$ Then
\begin{multline*}
\langle \pi_p(g)\Theta_p,\Theta_p\rangle_p=\int_{\overline{G}(\mathbb{Q}_p)}\Theta_p(g_1g)\overline{\Xi}_p(g_1)dg_1
=\\
\frac{1}{\Theta_p(\bfone_p)}\int_{\overline{G}(\mathbb{Q}_p)}\overline{\Xi}_p(g_1)dg_1\int_{I_p}\Theta_p(g_1\kappa g)d\kappa.
\end{multline*}

Then by \eqref{174} and \eqref{phiproduct} (noting that $\Theta_p$ is $L^2$-integrable by \eqref{168}), we deduce that
\begin{equation}\label{177}
\langle\pi_p(g)\Theta_p,\Theta_p\rangle_p=\frac{\Theta_p(g)}{\Theta_p(\bfone_p)}\int_{\overline{G}(\mathbb{Q}_p)}\Theta_p(g_1)\overline{\Xi}_p(g_1)dg_1=\frac{\peter{\xi_{p},\xi_{p}}_p}{\Theta_p(\bfone_p)}\cdot \Theta_p(g),
\end{equation}
which proves \eqref{176}.
\end{proof}

\subsection{The local period at $p=N$}\label{4.2}

In this section, we use the results of the previous section to establish \eqref{inertlocalperiod}. 

We recall that $p=N$ is inert in $E$, that $\pi_p$ is the Steinberg representation and that $\pi_p'$ is a (tempered) unramified principal series representation.

Let $\xi_p$ and $\xi'_p$ be local new vectors, of $\pi_p$ and $\pi'_p$ respectively. We will show that
	\begin{equation}\label{183'}
	\Bigg|L_p(\pi_p,\pi'_p)\cdot \mcP_p^{\natural}(\xi_{p},\xi_{p};\xi_{p}',\xi'_{p})-\frac{p-1}{p^2}\Bigg|\leq \frac{3(1-p^{-2})}{p}\cdot \frac{p-1}{p^2}.
	\end{equation}
This implies \eqref{inertlocalperiod} as well as \eqref{lowerboundforPinert} since the local factors  do not vanish at $s=1/2$ and $s=1$ (see \eqref{upperlowerL}) and are of the shape $1+o(1)$ as $p$ becomes large.

Write $\pi_p'=\Ind\chi_p'$ and let $$\gamma_p'=\chi_p'(A_1')$$ where for $n\in\Zz$, we have set $$A'_n=\diag(p^n,p^{-n}).$$

Let $W_0'=\big\{\bfone_p', J'\big\}$ where $\bfone_p'$ is the identity in $G'(\mathbb{Q}_p)$, and for $n\geq 1$ let $$W_n'=\big\{A'_n, J'A'_n, A'_nJ', J'A'_nJ'\big\}$$
and
$$W':=\bigsqcup_{n\geq 0}W_n'.$$
Let $\mu'$ be a Haar measure on $G'(\Qp)$ and let $$I'_p\subset G'(\Zp)$$ be the Iwahori subgroup corresponding to matrices which are upper-triangular modulo $p$.  For any $w'\in W',$ there exists a unique integer $\lambda'(w')\in\mathbb{Z}_{\geq 0}$ such that $$\mu'(I_p'w'I_p')=p^{\lambda'(w')}\mu(I_p').$$ In particular, $\lambda'(\bfone_p')=0,$ $\lambda'(J')=1;$ and for $n\geq 1,$ $$\lambda'(A'_n)=2n,\ \lambda'(J'A'_n)=2n-1,\ \lambda'(A'_nJ')=2n+1\hbox{ and }\lambda'(J'A'_nJ')=2n.$$

  By Macdonald's formula (cf. \cite{Mac73} or \cite{Cas80}) and Lemma \ref{K'cosetlemma}, we have 
\begin{equation}\label{184}
\frac{\langle\pi_p'(A'_n)\xi'_p,\xi'_p\rangle_p}{\langle\xi'_p,\xi'_p\rangle_p}= \frac{(1-p^{-1}\gamma_p'^{-1})\gamma_p'^n-(1-p^{-1}\gamma_p') \gamma_p'^{-n-1}}{p^{n}\big[(1-p^{-1}\gamma_p'^{-1})-(1-p^{-1}\gamma'_p)\gamma_p'^{-1}\big]}.
\end{equation}

Recall that $i(\bullet)$ is defined in \eqref{emd}. By Lemma \ref{160}, Lemma \ref{162}, Proposition \ref{matrixsteinberg}, Lemma \ref{178} and Lemma \ref{162'}, we have
\begin{align*}
\frac{\mcP_p(\xi_{p},\xi_{p};\xi_{p}',\xi'_{p})}{\peter{\xi_{p},\xi_{p}}_p\cdot\peter{\xi'_{p},\xi'_{p}}_p}=&\sum_{n\geq 0}\sum_{w'\in W_n'}(-p)^{-\lambda(i(w'))}p^{\lambda'(w')}\cdot \frac{\langle\pi_p'(A'_n)\xi'_p,\xi'_p\rangle_p}{\langle\xi'_p,\xi'_p\rangle_p}\cdot \mu(I_p')\\
=&\frac{1-p^{-2}}{p+1}+\sum_{n\geq 1}\frac{\langle\pi_p'(A'_n)\xi'_p,\xi'_p\rangle_p}{\langle\xi'_p,\xi'_p\rangle_p}\cdot \frac{2p^{-2n}-p^{-2n+2}-p^{-2n-2}}{p+1},
\end{align*}
where the last equality follows from the fact that $\xi'_p$ is spherical. Therefore we have
\begin{equation}\label{185}
\frac{\mcP_p(\xi_{p},\xi_{p};\xi_{p}',\xi'_{p})}{\peter{\xi_{p},\xi_{p}}_p\cdot\peter{\xi'_{p},\xi'_{p}}_p}=\frac{p-1}{p^2}-\sum_{n\geq 1}\frac{\langle\pi_p'(A'_n)\xi'_p,\xi'_p\rangle_p}{\langle\xi'_p,\xi'_p\rangle_p}\cdot \frac{(p-1)^2(p+1)}{p^{2n+2}}.
\end{equation}

Since $\pi'_p$ is tempered, $|\gamma_p'|=1.$ Hence we obtain
\begin{equation}\label{186}
\sum_{n\geq 1}\frac{(1-p^{-1}\overline{\gamma_p'})\gamma_p'^n-(1-p^{-1}\gamma_p') \overline{\gamma_p'}^{n+1}}{p^{3n}\big[(1-p^{-1}\overline{\gamma_p})-(1-p^{-1}\gamma'_p)\gamma_p'^{-1}\big]}=\frac{\frac{(1-p^{-1}\gamma_p'^{-1})\gamma_p'}{p^{3}-\gamma_p'}-\frac{(1-p^{-1}\gamma_p') \overline{\gamma_p'}^{2}}{p^{3}-\overline{\gamma_p'}}}{(1+p^{-1})(1-\overline{\gamma_p'})}.
\end{equation}

Substituting \eqref{184} and \eqref{186} in \eqref{185} one then obtains
\begin{align*}
\frac{\mcP_p(\xi_{p},\xi_{p};\xi_{p}',\xi'_{p})}{\peter{\xi_{p},\xi_{p}}_p\cdot\peter{\xi'_{p},\xi'_{p}}_p}=\frac{p-1}{p^2}- \frac{(p-1)^2(p+1)}{p^{2}}\cdot \frac{\frac{(1-p^{-1}\gamma_p'^{-1})\gamma_p'}{p^{3}-\gamma_p'}-\frac{(1-p^{-1}\gamma_p') \gamma_p'^{-2}}{p^{3}-\gamma_p'^{-1}}}{(1+p^{-1})(1-\gamma_p'^{-1})}.
\end{align*}

A straightforward simplification shows that
\begin{align*}
\frac{\frac{(1-p^{-1}\gamma_p'^{-1})\gamma_p'}{p^{3}-\gamma_p'}-\frac{(1-p^{-1}\gamma_p') \gamma_p'^{-2}}{p^{3}-\gamma_p'^{-1}}}{(1+p^{-1})(1-\gamma_p'^{-1})}
=&\frac{p^3\gamma_p'(1+\gamma_p'^{-1}+\gamma_p'^{-2})-p^2-1-p^{-1}}{(p^{3}-\gamma_p')(p^{3}-\gamma_p'^{-1})(1+p^{-1})}.
\end{align*}

In conjunction with $|\gamma_p'|=1$ we then conclude, when $p\geq 3,$ that
\begin{equation}\label{225}
\Bigg|\frac{\frac{(1-p^{-1}\gamma_p'^{-1})\gamma_p'}{p^{3}-\gamma_p'}-\frac{(1-p^{-1}\gamma_p') \gamma_p'^{-2}}{p^{3}-\gamma_p'^{-1}}}{(1+p^{-1})(1-\gamma_p'^{-1})}\Bigg|\leq \frac{3p^3+p^2+1+p^{-1}}{(p^{3}-1)^2(1+p^{-1})}\leq \frac{3}{p^3}.
\end{equation}

 Therefore, we have by \eqref{225} that
\begin{align*}
\Bigg|\frac{\mcP_p(\xi_{p},\xi_{p};\xi_{p}',\xi'_{p})}{\peter{\xi_{p},\xi_{p}}_p\cdot\peter{\xi'_{p},\xi'_{p}}_p}-\frac{p-1}{p^2}\Bigg|\leq \frac{3(p-1)^2(p+1)}{p^5}= \frac{3(1-p^{-2})}{p}\cdot \frac{p-1}{p^2}
\end{align*}
and \eqref{183'} follows.
\qed

\subsection{The matrix coefficient of the Steinberg representation for $U(W)$} 
In this section, we assume that $p=N'$ is {\em split}.
 By Lemma \ref{162'} and the Cartan decomposition we have
\begin{align*}
G'(\mathbb{Q}_p)=\bigsqcup_{n\geq 0}G'(\mathbb{Z}_p)A'_nG'(\mathbb{Z}_p)=\bigsqcup_{w'\in W'}I_p'w'I_p'.
\end{align*}

By Lemma \ref{K'cosetlemma} one has $$\mu'(I_p'J'I_p')=p\mu'(I_p')$$
where $J'$ is defined in \eqref{J'def}. Moreover, by Lemma \ref{178}, for $n\geq 1,$ one has 
\begin{gather*}
\mu'(I_p'A'_nI_p')=p^{2n}\mu'(I_p'),\ \mu'(I_p'J'A'_nI_p')=p^{2n-1}\mu'(I_p'),\\ \mu'(I_p'A'_nJ'I_p')=p^{2n+1}\mu'(I_p'),\hbox{ and }\mu'(I_p'J'A'_nJ'I_p')=p^{2n}\mu'(I_p').	
\end{gather*}
  Let $\Xi'_p$ be the $I_p'$-bi-invariant function on $G'(\mathbb{Q}_p)$ defined by $$\Xi'_p(w')=(-p)^{-\lambda'(w')}\mu'(I_p'),\ w'\in W'.$$ Then by similar analysis as that in \S \ref{secmatrixStG} we have a counterpart of Proposition \ref{matrixsteinberg}:

\begin{prop}\label{175'}\label{propmatStG'}
Let the notation be as before. Let $\St_p'$ be the Steinberg representation of $G'(\mathbb{Q}_p)$. The function $\Xi'_p$ is a matrix coefficient of $\St_p'$. Precisely, let $\xi'_p\neq 0$ be a local new vector in $\St_p'$ then 
	\begin{equation}\label{176'}
	\frac{\langle\pi_p'(g)\xi'_p,\xi'_p\rangle_p}{\langle\xi'_p,\xi'_p\rangle_p}=\frac{\Xi'_p(g)}{\Xi'_p(\bfone_p')}, \quad \forall\ g\in G'(\mathbb{Q}_p).
	\end{equation}
\end{prop}

\subsection{The local period at $p=N'$}\label{secpdivN'notN}
In this section we deal with the case where $v=p=N'$ is a split prime and establish \eqref{180} and \eqref{lowerboundforPsplit}.

We recall that in this case, we have the identifications $$G(\Qp)\simeq \GL_3(\Qp),\ K_p\simeq \GL_3(\Zp)$$ and
$$G'(\Qp)\simeq \GL_2(\Qp),\ K'_p\simeq \GL_2(\Zp).$$
Recall (see \S \ref{secG'def}) that the subgroup $G'\subset G$ is identified with the subgroup of $\GL_3$ leaving invariant the second element of the canonical basis 
\begin{equation}
	\label{cornerembedding}
	\begin{pmatrix}
	a&b\\c&d
\end{pmatrix}\mapsto
\begin{pmatrix}
	a&&b\\&1&\\c&&d
\end{pmatrix}.
\end{equation}
 
We recall that $\pi'_p\simeq \St'_p$ is the Steinberg representation and we denote its new vector by $\xi'_p$; the representation $$\pi_p\simeq \Ind^{G(\Qp)}_{P(\Qp)}\chi_p$$ is a tempered unramified principal series induced from  a unitary character $\chi_p=\chi$ of the diagonal torus. We denote by $\xi_p$ a nonzero spherical vector and by $\Theta_p$ its associated  matrix coefficient:
\begin{equation}
	\label{thetapdef}
	\Theta_p(g_p):={\peter{\pi_p(g_p)\xi_p,\xi_p}}.
\end{equation}
Finally we set
$$\tilde\xi_p:=\pi_p(\tilde \mfn_p)\xi_p$$
where we recall that
$$\tilde\mfn_p\simeq \begin{pmatrix}
	1&p^{-1}&\\&1&\\&&1
\end{pmatrix}=w'\mfn_pw',\ \mfn_p=\begin{pmatrix}
	1&&p^{-1}\\&1&\\&&1
\end{pmatrix},\ w'=\begin{pmatrix}
	1&&\\&&1\\&1&
\end{pmatrix}.$$ 
Our aim is to compute  the normalized period
\begin{equation}
	\mcP^*(\widetilde{\xi}_p,\widetilde{\xi}_p;\xi'_p,\xi'_p):=\int_{{G}'(\mathbb{Q}_p)}\frac{\peter{\pi_{p}(g_p)\widetilde{\xi}_{p},\widetilde{\xi}_{p}}_p\langle\pi'_{p}(g_p)\xi'_{p},\xi'_{p}\rangle_p}{\langle\widetilde{\xi}_{p},\widetilde{\xi}_{p}\rangle_p\peter{\xi'_{p},\xi'_{p}}_p}dg_p
\label{P*def}
\end{equation}
We have
$$\frac{\peter{\pi_{p}(g_p)\widetilde{\xi}_{p},\widetilde{\xi}_{p}}}{\peter{\widetilde{\xi}_{p},\widetilde{\xi}_{p}}}=\frac{\peter{\pi_{p}(\tilde\mfn^{-1}_pg_p\tilde\mfn_p)\xi_p,\xi_p }}{\peter{\xi_p,\xi_p}}=
\frac{\Theta_p(\tilde\mfn^{-1}_pg_p\tilde\mfn_p)}{\Theta_p(\bfone_p)}$$
so that
\begin{equation}\label{newP*def}
	\mcP^*(\widetilde{\xi}_p,\widetilde{\xi}_p;\xi'_p,\xi'_p)=\int_{G'(\Qp)}\frac{\Theta_p(\tilde\mfn^{-1}_pg_p\tilde\mfn_p)\Xi'_p(g_p)}{\Theta_p(\bfone_p)\Xi'_p(\bfone_p')}dg_p
\end{equation}
by Proposition \ref{propmatStG'}. 

\begin{remark}\label{remswitchembedding}
	Observe that since $w'\in K_p$ we have
$$\Theta_p(\tilde\mfn^{-1}_pg_p\tilde\mfn_p)=
\Theta_p(w'.\mfn^{-1}_p.w'.g_p.w'.\mfn_p.w')=
\Theta_p(\mfn^{-1}_p.w'.g_p.w'.\mfn_p)$$
and for $g_p\in G'(\Qp)$
$$w'.g_p.w'=g'_p=\begin{pmatrix}
	a&b&\\c&d&\\&&1
\end{pmatrix}$$
with the  usual embedding of $\GL_2\hookrightarrow\GL_3$. 
For the rest of this section and to simplify the notation, we will use this later embedding in place of  \eqref{cornerembedding}. This will allow us to replace $\tilde\mfn_p$ by $\mfn_p$ in all our forthcoming computation.
\end{remark}

Let 
$$w=\begin{pmatrix}
&1\\
1&
\end{pmatrix}\in K'_p\hbox{ and }I_p'\simeq \big\{\begin{pmatrix}
a&b\\c&d	
\end{pmatrix}
\in \GL_2(\Zp),\ c\in p\Zp\big\}$$ be the Iwahori subgroup of $G'(\Qp)$. 

For $(m,n)\in\mathbb{Z}^2,$ we set $$A'_{m,n}=\begin{pmatrix}
p^m&\\
&p^n
\end{pmatrix}\in G'(\Qp).$$

By the Iwahori-Cartan decomposition one has $$G'(\Qp)=G_1'\bigsqcup G_2',$$ where 
\begin{equation}\label{EqG'1}G_1':=\bigsqcup_{n\in\mathbb{Z}}\left(I_p'A'_{n,n}\bigsqcup I_p'A'_{n,n}wI_p'\right)	
\end{equation}
 and
\begin{equation}\label{EqG'2}
G_2':=\bigsqcup_{m\geq n+1}\left(I_p'A'_{m,n}I_p'\bigsqcup I_p'wA'_{m,n}I_p'\bigsqcup I_p'A'_{m,n}wI_p'\bigsqcup I_p'wA'_{m,n}I_p'\right).
\end{equation}
From \eqref{P*def}, we have
\begin{equation}
	\mcP^*(\widetilde{\xi}_p,\widetilde{\xi}_p;\xi'_p,\xi'_p)=\sum_{j=1}^2\int_{G_j'}\frac{\Theta_p(\mfn_p^{-1}g_p\mfn_p)\Xi'_p(g_p)}{\Theta_p(\bfone_p)\Xi'_p(\bfone_p')}dg_p.
\label{197}
\end{equation}

Let 

\begin{equation}\label{defI'p1}
 	I_p'(1)\simeq \bigg\{\begin{pmatrix}
a&b\\
c&d
\end{pmatrix}\in \GL_2(\Zp):\ c,a-1\in p\mathbb{Z}_p\bigg\}\subset I'_p.
 \end{equation}
 We have 
 $$\mfn_p^{-1}\begin{pmatrix}
	a&b&\\c&d&\\&&1
\end{pmatrix}\mfn_p=\begin{pmatrix}
	a&b&(a-1)/p\\c&d&c/p\\&&1
\end{pmatrix}$$
 so that 
 $$\mfn_p^{-1}I_p'(1)\mfn_p\subseteq K_p.$$
 It follows that
\begin{align*}
\int_{I_p'A'_{n,n}}\frac{\Theta_p(\mfn_p^{-1}g_p\mfn_p)\Xi'_p(g_p)}{\Theta_p(\bfone_p)\Xi'_p(\bfone_p')}dg_p=\frac{\mu(I_p'(1))}{\Theta_p(\bfone_p)}\sum_{\delta}\Theta_p\left(\mfn_p^{-1}A'_{n,n}\begin{pmatrix}
\delta&&\\
&1&\\
&&1
\end{pmatrix}\mfn_p\right),
\end{align*}
where $\delta$ runs over $(\mathbb{Z}_p/p\mathbb{Z}_p)^{\times}$ and $A'_{m,n}$ stands for $\begin{pmatrix}
	A'_{m,n}&\\&1
\end{pmatrix}$

 We will compute the above integral depending on $n$: for this we notice that
\begin{enumerate}
	\item[1.] For $n\geq 1$, we have
	\begin{align*}
	\mfn_p^{-1}A'_{n,n}\begin{pmatrix}
	\delta&&\\
	&1&\\
	&&1
	\end{pmatrix}\mfn_p&=\begin{pmatrix}
	\delta.p^n&&(\delta p^n-1)/p\\
	&p^n&\\
	&&1
	\end{pmatrix}\\
	&\in K_p\begin{pmatrix}
	p^{n+1}&&\\
	&p^n&\\
	&&p^{-1}
	\end{pmatrix}K_p.
	\end{align*}
	\item[2.] For $n\leq -1$ we have similarly
	\begin{align*}
	\mfn_p^{-1}A'_{n,n}\begin{pmatrix}
	\delta&&\\
	&1&\\
	&&1
	\end{pmatrix}\mfn_p\in K_p\begin{pmatrix}
	p&&\\
	&p^n&\\
	&&p^{n-1}
	\end{pmatrix}K_p.
	\end{align*}
	\item[3.] For $n=0$ and $\delta\neq 1\in(\mathbb{Z}_p/p\mathbb{Z}_p)^{\times}$ we have
	\begin{align*}
	\mfn_p^{-1}A'_{n,n}\begin{pmatrix}
	\delta&&\\
	&1&\\
	&&1
	\end{pmatrix}\mfn_p\in K_p\begin{pmatrix}
	p&&\\
	&1&\\
	&&p^{-1}
	\end{pmatrix}K_p.
	\end{align*}
\end{enumerate}

From the above discussion we obtain that
\begin{equation}\label{189}
\int_{I_p'A'_{n,n}}\frac{\Theta_p(\mfn_p^{-1}g_p\mfn_p)\Xi'_p(g_p)}{\Theta_p(\bfone_p)\Xi'_p(\bfone_p')}dg_p=\Sigma_{01}+\Sigma_{02}+\Sigma_{03},
\end{equation}
where
\begin{align*}
\Sigma_{01}:=&\frac{\mu'(I_p'(1))}{\Theta_p(\bfone_p)}\Bigg[\Theta_p(\bfone_p)+(p-2)\Theta_p\begin{pmatrix}
p&&\\
&1&\\
&&p^{-1}
\end{pmatrix}\Bigg],\\
\Sigma_{02}:=&\frac{(p-1)\mu'(I_p'(1))}{\Theta_p(\bfone_p)}\sum_{n\geq 1}\Theta_p\begin{pmatrix}
p^{n+1}&&\\
&p^n&\\
&&p^{-1}
\end{pmatrix},\\
\Sigma_{03}:=&\frac{(p-1)\mu'(I_p'(1))}{\Theta_p(\bfone_p)}\sum_{n\leq -1}\Theta_p\begin{pmatrix}
p^{n+1}&&\\
&p^n&\\
&&p^{-1}
\end{pmatrix}.
\end{align*}

Regarding the integral along $I_p'A'_{n,n}wI_p'$ we have
\begin{align*}
\int_{I_p'A'_{n,n}w'I_p'}\frac{\Theta_p(\mfn_p^{-1}g_p\mfn_p)\Xi'_p(g_p)}{\Theta_p(\bfone_p)\Xi'_p(\bfone_p')}dg_p
=-\frac{\mu(I_p'(1))}{\Theta_p(\bfone_p)}\sum_{\delta}\Theta_p\left(\mfn_p^{-1}w\begin{pmatrix}
\delta&&\\
&1&\\
&&1
\end{pmatrix}\mfn_p\right),
\end{align*}
where $\delta$ runs over $(\mathbb{Z}_p/p\mathbb{Z}_p)^{\times}.$ 

We will compute the above integral depending on the value of $n$: for this we observe that
\begin{enumerate}
	\item[1.] For $n\geq 1$, we have
	\begin{align*}
	\mfn_p^{-1}A'_{n,n}w\begin{pmatrix}
	\delta&&\\
	&1&\\
	&&1
	\end{pmatrix}\mfn_p&=\begin{pmatrix}
	\delta p^n&&\delta p^{n-1}\\
	&p^n&-p^{-1}\\
	&&1
	\end{pmatrix}\\
	&\in K_p\begin{pmatrix}
	p^{n+1}&&\\
	&p^n&\\
	&&p^{-1}
	\end{pmatrix}K_p.
	\end{align*}
	\item[2.] For $n\leq -1$ we have
	\begin{align*}
		\mfn_p^{-1}A'_{n,n}w\begin{pmatrix}
	\delta&&\\
	&1&\\
	&&1
	\end{pmatrix}\mfn_p\in K_p\begin{pmatrix}
	p&&\\
	&p^n&\\
	&&p^{n-1}
	\end{pmatrix}K_p.
	\end{align*}
	\item[3.] For $n=0$ we have
	\begin{align*}
	\mfn_p^{-1}w\begin{pmatrix}
	\delta&&\\
	&1&\\
	&&1
	\end{pmatrix}\mfn_p\in K_p\begin{pmatrix}
	p&&\\
	&1&\\
	&&p^{-1}
	\end{pmatrix}K_p.
	\end{align*}
\end{enumerate}

Thus from the above discussion we then obtain
\begin{equation}\label{190}
\int_{I_p'A'_{n,n}w'I_p'}\frac{\Theta_p(\mfn_p^{-1}g_p\mfn_p)\Xi'_p(g_p)}{\Theta_p(\bfone_p)\Xi'_p(\bfone_p')}dg_p
=\Sigma_{01}'+\Sigma_{02}'+\Sigma_{03}',
\end{equation}
where
\begin{align*}
\Sigma_{01}':=&-\frac{(p-1)\mu(I_p'(1))}{\Theta_p(\bfone_p)}\Theta_p\begin{pmatrix}
p&&\\
&1&\\
&&p^{-1}
\end{pmatrix},\\
\Sigma_{02}':=&-\frac{(p-1)\mu(I_p'(1))}{\Theta_p(\bfone_p)}\sum_{n\geq 1}\Theta_p\begin{pmatrix}
p^{n+1}&&\\
&p^n&\\
&&p^{-1}
\end{pmatrix},\\
\Sigma_{03}':=&-\frac{(p-1)\mu(I_p'(1))}{\Theta_p(\bfone_p)}\sum_{n\leq -1}\Theta_p\begin{pmatrix}
p^{n+1}&&\\
&p^n&\\
&&p^{-1}
\end{pmatrix}.
\end{align*}

Then we have by \eqref{189} and \eqref{190} that
\begin{equation}\label{191}
\int_{G_1'}\frac{\Theta_p(\mfn_p^{-1}g_p\mfn_p)\Xi'_p(g_p)}{\Theta_p(\bfone_p)\Xi'_p(\bfone_p')}dg_p=\mu(I_p'(1))-\frac{\mu(I_p'(1))}{\Theta_p(\bfone_p)}\Theta_p\begin{pmatrix}
p&&\\
&1&\\
&&p^{-1}
\end{pmatrix}.
\end{equation}
\medskip

Let $n\in \mathbb{Z}.$ Let $m\geq n+1.$ We have, similar to Lemma \ref{178}, that
\begin{align*}
I_p'A'_{m,n}I_p'=\bigsqcup_{\tau\in \mathbb{Z}_p/p^{m-n}\mathbb{Z}_p}\begin{pmatrix}
1&\tau\\
&1
\end{pmatrix}A'_{m,n}I_p'.
\end{align*}

A straightforward calculation shows that
\begin{align*}
\mfn_p^{-1}\begin{pmatrix}
1&\tau&\\
&1&\\
&&1
\end{pmatrix}A'_{m,n}\begin{pmatrix}
\delta&&\\
&1&\\
&&1
\end{pmatrix}\mfn_p\in K_p\begin{pmatrix}
p^{m}&&(\delta p^m-1)p^{-1}\\
&p^n&\\
&&1
\end{pmatrix}K_p.
\end{align*}
\begin{enumerate}
	\item[1.] Suppose $n\geq 0.$ Then $m\geq n+1\geq 1.$ Then
	\begin{align*}
	\mfn_p^{-1}\begin{pmatrix}
	1&\tau&\\
	&1&\\
	&&1
	\end{pmatrix}A'_{m,n}\begin{pmatrix}
	\delta&&\\
	&1&\\
	&&1
	\end{pmatrix}\mfn_p\in K_p\begin{pmatrix}
	p^{m+1}&&\\
	&p^n&\\
	&&p^{-1}
	\end{pmatrix}K_p.
	\end{align*}
   \item[2.] Suppose $n\leq -1$ and $m\geq 1.$ Then
   \begin{align*}
   \mfn_p^{-1}\begin{pmatrix}
   1&\tau&\\
   &1&\\
   &&1
   \end{pmatrix}A'_{m,n}\begin{pmatrix}
   \delta&&\\
   &1&\\
   &&1
   \end{pmatrix}\mfn_p\in K_p\begin{pmatrix}
   p^{m+1}&&\\
   &p^{-1}&\\
   &&p^{n}
   \end{pmatrix}K_p.
   \end{align*}
   \item[3.] Suppose $n\leq -1$ and $m=0.$ If $\delta\neq 1\in \mathbb{F}_p,$ then
   \begin{align*}
   \mfn_p^{-1}\begin{pmatrix}
   1&\tau&\\
   &1&\\
   &&1
   \end{pmatrix}A'_{m,n}\begin{pmatrix}
   \delta&&\\
   &1&\\
   &&1
   \end{pmatrix}\mfn_p\in K_p\begin{pmatrix}
   p^{m+1}&&\\
   &p^{-1}&\\
   &&p^{n}
   \end{pmatrix}K_p.
   \end{align*}
   \item[4.] Suppose $n\leq -1$ and $m\leq -1.$ Note that $m-1\geq n.$ Then
   \begin{align*}
   \mfn_p^{-1}\begin{pmatrix}
   1&\tau&\\
   &1&\\
   &&1
   \end{pmatrix}A'_{m,n}\begin{pmatrix}
   \delta&&\\
   &1&\\
   &&1
   \end{pmatrix}\mfn_p\in K_p\begin{pmatrix}
   p&&\\
   &p^{m-1}&\\
   &&p^{n}
   \end{pmatrix}K_p.
   \end{align*}
\end{enumerate}

Therefore, similar to \eqref{189} and \eqref{190} we have that
\begin{align}\nonumber
\Sigma_1&:=\sum_{n\in \mathbb{Z}}\sum_{m\geq n+1}\int_{I_p'A'_{m,n}I_p'}\frac{\Theta_p(\mfn_p^{-1}g_p\mfn_p)\Xi'_p(g_p)}{\Theta_p(\bfone_p)\Xi'_p(\bfone_p')}dg_py\\
&\ \ =\Sigma_{11}+\Sigma_{12}+\Sigma_{13}+\Sigma_{14},\label{192}
\end{align}
where
\begin{align*}
\Sigma_{11}:=&\frac{(p-1)\mu(I_p'(1))}{\Theta_p(\bfone_p)}\sum_{n\geq 0}\sum_{m\geq n+1}\Theta_p\begin{pmatrix}
p^{m+1}&&\\
&p^n&\\
&&p^{-1}
\end{pmatrix},\\
\Sigma_{12}:=&\frac{(p-1)\mu(I_p'(1))}{\Theta_p(\bfone_p)}\sum_{n\leq -1}\sum_{m\geq 1}\Theta_p\begin{pmatrix}
p^{m+1}&&\\
&p^{-1}&\\
&&p^{n}
\end{pmatrix},\\
\Sigma_{13}:=&\frac{\mu(I_p'(1))}{\Theta_p(\bfone_p)}\sum_{n\leq -1}\Bigg[\Theta_p\begin{pmatrix}
1&&\\
&1&\\
&&p^{n}
\end{pmatrix}+(p-2)\Theta_p\begin{pmatrix}
p&&\\
&p^{-1}&\\
&&p^{n}
\end{pmatrix}\Bigg],\\
\Sigma_{14}:=&\frac{(p-1)\mu(I_p'(1))}{\Theta_p(\bfone_p)}\sum_{n\leq -2}\sum_{n+1\leq m\leq -1}\Theta_p\begin{pmatrix}
p&&\\
&p^{m-1}&\\
&&p^{n}
\end{pmatrix}.
\end{align*}

\bigskip

Let $m\geq n+1.$ We have, similar to Lemma \ref{178}, that
\begin{align*}
I_p'w'A'_{m,n}I_p'=\bigsqcup_{\tau\in p\mathbb{Z}_p/p^{m-n}\mathbb{Z}_p}\begin{pmatrix}
1&\\
\tau&1
\end{pmatrix}w'A'_{m,n}I_p'.
\end{align*}

\begin{align*}
\mfn_p^{-1}\begin{pmatrix}
1&&\\
\tau&1&\\
&&1
\end{pmatrix}w'A'_{m,n}\begin{pmatrix}
\delta&&\\
&1&\\
&&1
\end{pmatrix}\mfn_p\in K_p\begin{pmatrix}
p^{m}&&p^{m-1}\\
&p^n&-p^{-1}\\
&&1
\end{pmatrix}K_p.
\end{align*}
\begin{enumerate}
	\item[1.] Suppose $n\geq 0.$ Then $m\geq n+1\geq 1.$ Then
	\begin{align*}
	\mfn_p^{-1}\begin{pmatrix}
	1&&\\
	\tau&1&\\
	&&1
	\end{pmatrix}w'A'_{m,n}\begin{pmatrix}
	\delta&&\\
	&1&\\
	&&1
	\end{pmatrix}\mfn_p\in K_p\begin{pmatrix}
	p^{m}&&\\
	&p^{n+1}&\\
	&&p^{-1}
	\end{pmatrix}K_p.
	\end{align*}
	\item[2.] Suppose $n\leq -1$ and $m\geq 1.$ Then
	\begin{align*}
	\mfn_p^{-1}\begin{pmatrix}
	1&&\\
	\tau&1&\\
	&&1
	\end{pmatrix}w'A'_{m,n}\begin{pmatrix}
	\delta&&\\
	&1&\\
	&&1
	\end{pmatrix}\mfn_p\in K_p\begin{pmatrix}
	p^{m}&&\\
	&1&\\
	&&p^{n}
	\end{pmatrix}K_p.
	\end{align*}
	\item[3.] Suppose $n\leq -1$ and $m=0.$ Then
	\begin{align*}
	\mfn_p^{-1}\begin{pmatrix}
	1&&\\
	\tau&1&\\
	&&1
	\end{pmatrix}w'A'_{m,n}\begin{pmatrix}
	\delta&&\\
	&1&\\
	&&1
	\end{pmatrix}\mfn_p\in K_p\begin{pmatrix}
	p&&\\
	&p^{-1}&\\
	&&p^{n}
	\end{pmatrix}K_p.
	\end{align*}
	\item[4.] Suppose $n\leq -1$ and $m\leq -1.$ Note that $m-1\geq n.$ Then
	\begin{align*}
	\mfn_p^{-1}\begin{pmatrix}
	1&&\\
	\tau&1&\\
	&&1
	\end{pmatrix}w'A'_{m,n}\begin{pmatrix}
	\delta&&\\
	&1&\\
	&&1
	\end{pmatrix}\mfn_p\in K_p\begin{pmatrix}
	p&&\\
	&p^{m-1}&\\
	&&p^{n}
	\end{pmatrix}K_p.
	\end{align*}
\end{enumerate}

Therefore,similar to \eqref{189} and \eqref{192} we have that
\begin{align}\nonumber
\Sigma_2&:=\sum_{n\in \mathbb{Z}}\sum_{m\geq n+1}\int_{I_p'w'A'_{m,n}I_p'}\frac{\Theta_p(\mfn_p^{-1}g_p\mfn_p)\Xi'_p(g_p)}{\Theta_p(\bfone_p)\Xi'_p(\bfone_p')}dg_p\\
&\ \ =\Sigma_{21}+\Sigma_{22}+\Sigma_{23}+\Sigma_{24},
\label{193}
\end{align}

\begin{align*}
\Sigma_{21}:=&-\frac{(p-1)\mu(I_p'(1))}{\Theta_p(\bfone_p)}\sum_{n\geq 0}\sum_{m\geq n+1}\Theta_p\begin{pmatrix}
p^{m}&&\\
&p^{n+1}&\\
&&p^{-1}
\end{pmatrix},\\
\Sigma_{22}:=&-\frac{(p-1)\mu(I_p'(1))}{\Theta_p(\bfone_p)}\sum_{n\leq -1}\sum_{m\geq 1}\Theta_p\begin{pmatrix}
p^{m}&&\\
&1&\\
&&p^{n}
\end{pmatrix},\\
\Sigma_{23}:=&-\frac{(p-1)\mu(I_p'(1))}{\Theta_p(\bfone_p)}\sum_{n\leq -1}\Theta_p\begin{pmatrix}
p&&\\
&p^{-1}&\\
&&p^{n}
\end{pmatrix},\\
\Sigma_{24}:=&-\frac{(p-1)\mu(I_p'(1))}{\Theta_p(\bfone_p)}\sum_{n\leq -2}\sum_{n+1\leq m\leq -1}\Theta_p\begin{pmatrix}
p&&\\
&p^{m-1}&\\
&&p^{n}
\end{pmatrix}.
\end{align*}

Let $m\geq n+1.$ We have, similar to Lemma \ref{178}, that
\begin{align*}
I_p'A'_{m,n}w'I_p'=\bigsqcup_{\tau\in \mathbb{Z}_p/p^{m-n+1}\mathbb{Z}_p}\begin{pmatrix}
1&\tau\\
&1
\end{pmatrix}A'_{m,n}w'I_p'.
\end{align*}

\begin{align*}
\mfn_p^{-1}\begin{pmatrix}
1&\tau&\\
&1&\\
&&1
\end{pmatrix}A'_{m,n}w'\begin{pmatrix}
\delta&&\\
&1&\\
&&1
\end{pmatrix}\mfn_p\in K_p\begin{pmatrix}
p^{m}&&-p^{-1}\\
&p^n&p^{n-1}\\
&&1
\end{pmatrix}K_p.
\end{align*}
\begin{enumerate}
	\item[1.] Suppose $n\geq 0.$ Then $m\geq n+1\geq 1.$ Then
	\begin{align*}
	\mfn_p^{-1}\begin{pmatrix}
	1&\tau&\\
	&1&\\
	&&1
	\end{pmatrix}A'_{m,n}w'\begin{pmatrix}
	\delta&&\\
	&1&\\
	&&1
	\end{pmatrix}\mfn_p\in K_p\begin{pmatrix}
	p^{m+1}&&\\
	&p^{n}&\\
	&&p^{-1}
	\end{pmatrix}K_p.
	\end{align*}
	\item[2.] Suppose $n\leq -1$ and $m\geq 1.$ Then
	\begin{align*}
	\mfn_p^{-1}\begin{pmatrix}
	1&\tau&\\
	&1&\\
	&&1
	\end{pmatrix}A'_{m,n}w'\begin{pmatrix}
	\delta&&\\
	&1&\\
	&&1
	\end{pmatrix}\mfn_p\in K_p\begin{pmatrix}
	p^{m+1}&&\\
	&1&\\
	&&p^{n-1}
	\end{pmatrix}K_p.
	\end{align*}
	\item[3.] Suppose $n\leq -1$ and $m=0.$ Then
	\begin{align*}
	\mfn_p^{-1}\begin{pmatrix}
	1&\tau&\\
	&1&\\
	&&1
	\end{pmatrix}A'_{m,n}w'\begin{pmatrix}
	\delta&&\\
	&1&\\
	&&1
	\end{pmatrix}\mfn_p\in K_p\begin{pmatrix}
	p&&\\
	&1&\\
	&&p^{n-1}
	\end{pmatrix}K_p.
	\end{align*}
	\item[4.] Suppose $n\leq -1$ and $m\leq -1.$ Note that $m-1\geq n.$ Then
	\begin{align*}
	\mfn_p^{-1}\begin{pmatrix}
	1&\tau&\\
	&1&\\
	&&1
	\end{pmatrix}A'_{m,n}w'\begin{pmatrix}
	\delta&&\\
	&1&\\
	&&1
	\end{pmatrix}\mfn_p\in K_p\begin{pmatrix}
	p&&\\
	&p^{m}&\\
	&&p^{n-1}
	\end{pmatrix}K_p.
	\end{align*}
\end{enumerate}

Therefore, similar to \eqref{189} and \eqref{193} we have that
\begin{align}\nonumber
\Sigma_3&:=\sum_{n\in \mathbb{Z}}\sum_{m\geq n+1}\int_{I_p'w'A'_{m,n}I_p'}\frac{\Theta_p(\mfn_p^{-1}g_p\mfn_p)\Xi'_p(g_p)}{\Theta_p(\bfone_p)\Xi'_p(\bfone_p')}dg_p\\
&\ \ =\Sigma_{31}+\Sigma_{32}+\Sigma_{33}+\Sigma_{34},\label{194}
\end{align}
where
\begin{align*}
\Sigma_{31}:=&-\frac{(p-1)\mu(I_p'(1))}{\Theta_p(\bfone_p)}\sum_{n\geq 0}\sum_{m\geq n+1}\Theta_p\begin{pmatrix}
p^{m+1}&&\\
&p^{n}&\\
&&p^{-1}
\end{pmatrix},\\
\Sigma_{32}:=&-\frac{(p-1)\mu(I_p'(1))}{\Theta_p(\bfone_p)}\sum_{n\leq -1}\sum_{m\geq 1}\Theta_p\begin{pmatrix}
p^{m+1}&&\\
&1&\\
&&p^{n-1}
\end{pmatrix},\\
\Sigma_{33}:=&-\frac{(p-1)\mu(I_p'(1))}{\Theta_p(\bfone_p)}\sum_{n\leq -1}\Theta_p\begin{pmatrix}
p&&\\
&1&\\
&&p^{n-1}
\end{pmatrix},\\
\Sigma_{34}:=&-\frac{(p-1)\mu(I_p'(1))}{\Theta_p(\bfone_p)}\sum_{n\leq -2}\sum_{n+1\leq m\leq -1}\Theta_p\begin{pmatrix}
p&&\\
&p^{m}&\\
&&p^{n-1}
\end{pmatrix}.
\end{align*}

\bigskip

Let $m\geq n+1.$ We have, similar to Lemma \ref{178}, that
\begin{align*}
I_p'w'A'_{m,n}w'I_p'=\bigsqcup_{\tau\in p\mathbb{Z}_p/p^{m-n+1}\mathbb{Z}_p}\begin{pmatrix}
1&\\
\tau&1
\end{pmatrix}w'A'_{m,n}I_p'.
\end{align*}

\begin{align*}
\mfn_p^{-1}\begin{pmatrix}
1&&\\
\tau&1&\\
&&1
\end{pmatrix}w'A'_{m,n}w'\begin{pmatrix}
\delta&&\\
&1&\\
&&1
\end{pmatrix}\mfn_p\in K_p\begin{pmatrix}
p^{m}&&\\
&p^n&(\delta p^n-1)p^{-1}\\
&&1
\end{pmatrix}K_p.
\end{align*}
\begin{enumerate}
	\item[1.] Suppose $n\geq 1.$ Then $m\geq n+1\geq 2.$ Then
	\begin{align*}
	\mfn_p^{-1}\begin{pmatrix}
	1&&\\
	\tau&1&\\
	&&1
	\end{pmatrix}w'A'_{m,n}w'\begin{pmatrix}
	\delta&&\\
	&1&\\
	&&1
	\end{pmatrix}\mfn_p\in K_p\begin{pmatrix}
	p^{m}&&\\
	&p^{n+1}&\\
	&&p^{-1}
	\end{pmatrix}K_p.
	\end{align*}
	\item[2.] Suppose $n=0.$ Let $\delta\neq 1\in \mathbb{F}_p.$ Then
	\begin{align*}
	\mfn_p^{-1}\begin{pmatrix}
	1&&\\
	\tau&1&\\
	&&1
	\end{pmatrix}w'A'_{m,n}w'\begin{pmatrix}
	\delta&&\\
	&1&\\
	&&1
	\end{pmatrix}\mfn_p\in K_p\begin{pmatrix}
	p^{m}&&\\
	&p&\\
	&&p^{-1}
	\end{pmatrix}K_p.
	\end{align*}
	\item[3.] Suppose $n\leq -1$ and $m\geq 1.$ Then
	\begin{align*}
	\mfn_p^{-1}\begin{pmatrix}
	1&&\\
	\tau&1&\\
	&&1
	\end{pmatrix}w'A'_{m,n}w'\begin{pmatrix}
	\delta&&\\
	&1&\\
	&&1
	\end{pmatrix}\mfn_p\in K_p\begin{pmatrix}
	p^m&&\\
	&p&\\
	&&p^{n-1}
	\end{pmatrix}K_p.
	\end{align*}
	\item[4.] Suppose $n\leq -1$ and $m\leq 0.$  Then
	\begin{align*}
	\mfn_p^{-1}\begin{pmatrix}
	1&&\\
	\tau&1&\\
	&&1
	\end{pmatrix}w'A'_{m,n}w'\begin{pmatrix}
	\delta&&\\
	&1&\\
	&&1
	\end{pmatrix}\mfn_p\in K_p\begin{pmatrix}
	p&&\\
	&p^{m}&\\
	&&p^{n-1}
	\end{pmatrix}K_p.
	\end{align*}
\end{enumerate}

Therefore, similar to \eqref{189} and \eqref{194} we have that
\begin{align}\label{195}
\Sigma_4&:=\sum_{n\in \mathbb{Z}}\sum_{m\geq n+1}\int_{X_{m,n}}\frac{\Theta_p(\mfn_p^{-1}g_p\mfn_p)\Xi'_p(g_p)}{\Theta_p(\bfone_p)\Xi'_p(\bfone_p')}dg_p\\
&=\Sigma_{41}+\Sigma_{42}+\Sigma_{43}+\Sigma_{44},\nonumber
\end{align}
where $$X_{m,n}=I_p'w'A'_{m,n}w'I_p'$$ and
\begin{align*}
\Sigma_{41}:=&\frac{(p-1)\mu(I_p'(1))}{\Theta_p(\bfone_p)}\sum_{n\geq 1}\sum_{m\geq n+1}\Theta_p\begin{pmatrix}
p^{m}&&\\
&p^{n+1}&\\
&&p^{-1}
\end{pmatrix},\\
\Sigma_{42}:=&\frac{\mu(I_p'(1))}{\Theta_p(\bfone_p)}\sum_{m\geq 1}\Bigg[\Theta_p\begin{pmatrix}
p^{m}&&\\
&1&\\
&&1
\end{pmatrix}+(p-2)\Theta_p\begin{pmatrix}
p^{m}&&\\
&p&\\
&&p^{-1}
\end{pmatrix}\Bigg],\\
\Sigma_{43}:=&\frac{(p-1)\mu(I_p'(1))}{\Theta_p(\bfone_p)}\sum_{n\leq -1}\sum_{m\geq 1}\Theta_p\begin{pmatrix}
p^m&&\\
&p&\\
&&p^{n-1}
\end{pmatrix},\\
\Sigma_{44}:=&\frac{(p-1)\mu(I_p'(1))}{\Theta_p(\bfone_p)}\sum_{n\leq -1}\sum_{n+1\leq m\leq 0}\Theta_p\begin{pmatrix}
p&&\\
&p^{m}&\\
&&p^{n-1}
\end{pmatrix}.
\end{align*}

Recall that $\mcP^*(\widetilde{\xi}_p,\widetilde{\xi}_p;\xi'_p,\xi'_p)$ was defined in \eqref{P*def}. Then combining  \eqref{191}, \eqref{192}, \eqref{193}, \eqref{194} with \eqref{195} we obtain
\begin{equation}\label{198}
\mcP^*(\widetilde{\xi}_p,\widetilde{\xi}_p;\xi'_p,\xi'_p)-\mu(I_p'(1))=-\frac{\mu(I_p'(1))}{\Theta_p(\bfone_p)}\Theta_p\begin{pmatrix}
p&&\\
&1&\\
&&p^{-1}
\end{pmatrix}+\sum_{i=1}^4\sum_{j=1}^4\Sigma_{ij}.
\end{equation}
Denote by $\RHS$ the right hand side of \eqref{198}. Then substituting definitions of $\Sigma_{ij}$'s one finds that $\Theta_p(\bfone_p)\cdot \mu(I_p'(1))^{-1}\cdot \RHS$ is equal to

\begin{align}\nonumber
&-\Theta_p\begin{pmatrix}
p&&\\
&1&\\
&&p^{-1}
\end{pmatrix}+(p-1)\sum_{n\leq -1}\sum_{m\geq 1}\Theta_p\begin{pmatrix}
p^{m+1}&&\\
&p^{-1}&\\
&&p^{n}
\end{pmatrix}\\
&+\sum_{n\leq -1}\Bigg[\Theta_p\begin{pmatrix}
1&&\\
&1&\\
&&p^{n}
\end{pmatrix}-\Theta_p\begin{pmatrix}
p&&\\
&p^{-1}&\\
&&p^{n}
\end{pmatrix}\Bigg] \nonumber\\
&-\sum_{m\geq 1}\Theta_p\begin{pmatrix}
p^{m}&&\\
&p&\\ 
&&p^{-1}
\end{pmatrix}-(p-1)\sum_{n\leq -1}\sum_{m\geq 1}\Theta_p\begin{pmatrix}
p^{m}&&\\
&1&\\
&&p^{n}
\end{pmatrix}\label{bigsum}\\
&-(p-1)\sum_{n\leq -1}\sum_{m\geq 1}\Theta_p\begin{pmatrix}
p^{m+1}&&\\
&1&\\
&&p^{n-1}
\end{pmatrix}+\sum_{m\geq 1}\Theta_p\begin{pmatrix}
p^{m}&&\\
&1&\\
&&1
\end{pmatrix} \nonumber\\
&+(p-1)\sum_{n\leq -1}\sum_{m\geq 1}\Theta_p\begin{pmatrix}
p^m&&\\
&p&\\
&&p^{n-1}
\end{pmatrix}.\nonumber
\end{align}
To bound \eqref{bigsum}, we need to recall Macdonald's formula  for $\GL(3,\mathbb{Q}_p)$ (cf. \cite{Mac73}). Let $\boldsymbol{\lambda}=(\lambda_1,\lambda_2,\lambda_3)$ a dominant coweight (i.e. $\lambda_1\geq\lambda_2\geq \lambda_3$) and
$$p^{\boldsymbol{\lambda}}:=\diag(p^{\lambda_1},p^{\lambda_2},p^{\lambda_3}).$$
Let $\mathcal{C}_{\rho}$ be the set of weights of the irreducible representation of highest weight $\rho.$ By \cite[Theorem 5.5, p. 31]{Cas17},
\begin{multline*}
	\frac{\Theta_p(p^{\boldsymbol{\lambda}})}{\Theta_p(\textbf{1}_p)}=\frac{\delta_{B}(p^{\boldsymbol{\lambda}})^{\frac{1}{2}}}{|K_pp^{\boldsymbol{\lambda}}K_p/K_p|\sum_w p^{-l(w)}}\\
	\sum_{w\in W_{{\boldsymbol{\lambda}}}}\sgn(w)\sum_{\boldsymbol{\mu}\in [W_{\boldsymbol{\lambda}}\backslash \mathcal{C}_{\rho}]}\sum_{\substack{S\subseteq\Sigma^+\sum_{\gamma\in S}\gamma=\rho-w\boldsymbol{\mu}}}(-1)^{|S|}p^{-|S|}\tau_{\boldsymbol{\lambda}+\boldsymbol{\mu}-\rho}(\chi_p),
\end{multline*}
where $W_{{\boldsymbol{\lambda}}}$ is the group generated by the simple roots $\alpha$ such that $\langle {\boldsymbol{\lambda}},\alpha^{\vee}\rangle=0,$ $\Sigma^+$ is the set of positive roots, and $\tau_{\boldsymbol{\lambda}+\boldsymbol{\mu}-\rho}$ is the character of the irreducible representation $\sigma_{\boldsymbol{\lambda}+\boldsymbol{\mu}-\rho}$ of $\mathrm{GL}_3(\mathbb{C})$ with highest weight $\boldsymbol{\lambda}+\boldsymbol{\mu}-\rho.$

Note that $$|C_{\rho}|=\sum_{S\subseteq \Sigma^+}1=2^3=8.$$ So
\begin{align*}
\Bigg|\frac{\Theta_p(p^{\boldsymbol{\lambda}})}{\Theta_p(\textbf{1}_p)}\Bigg|\leq & \frac{\delta_{B}(p^{\boldsymbol{\lambda}})^{\frac{1}{2}}}{|K_pp^{\boldsymbol{\lambda}}K_p/K_p|\sum_w p^{-l(w)}}\sum_w\sum_{\boldsymbol{\mu}}\sum_{\substack{S\subseteq\Sigma^+\\ \sum_{\gamma\in S}\gamma=\rho-w\boldsymbol{\mu}}}p^{-|S|}\cdot \dim\sigma_{\boldsymbol{\lambda}+\boldsymbol{\mu}-\rho}\\
\leq & \frac{\delta_{B}(p^{\boldsymbol{\lambda}})^{\frac{1}{2}}|C_{\rho}|}{|K_pp^{\boldsymbol{\lambda}}K_p/K_p|}\cdot \max_{\boldsymbol{\mu}}\dim\sigma_{\boldsymbol{\lambda}+\boldsymbol{\mu}-\rho}
\end{align*}

By definition as $\boldsymbol{\mu}$ ranges through $C_{\rho},$ $\boldsymbol{\mu}-\rho$ ranges over negative roots. Let $\Lambda_1$ and $\Lambda_2$ be basic weights. Then the possible values of $\boldsymbol{\mu}-\rho$ are
\begin{align*}
0,\ -2\Lambda_1+\Lambda_2,\ \Lambda_1-2\Lambda_2,\ -\Lambda_1-\Lambda_2,\ -3\Lambda_1,\ -3\Lambda_2,\ -2\Lambda_1-2\Lambda_2.
\end{align*}

Note that $\boldsymbol{\lambda}\equiv (\lambda_1-\lambda_2)\Lambda_1+(\lambda_2-\lambda_3)\Lambda_2$ modulo the center. Hence, by \cite[Example 10.23, p. 288]{Hal15}, we have
\begin{align*}
\max_{\boldsymbol{\mu}}\dim\sigma_{\boldsymbol{\lambda}+\boldsymbol{\mu}-\rho}\leq \frac{(\lambda_1-\lambda_2+4)(\lambda_2-\lambda_3+4)(\lambda_1-\lambda_3+6)}{2}.
\end{align*}

Therefore,
\begin{equation}\label{generictermbound}
\Bigg|\frac{\Theta_p(p^{\boldsymbol{\lambda}})}{\Theta_p(\textbf{1}_p)}\Bigg|\leq  \frac{4\delta_{B}(p^{\boldsymbol{\lambda}})^{\frac{1}{2}}}{|K_pp^{\boldsymbol{\lambda}}K_p/K_p|}\cdot (\lambda_1-\lambda_2+4)(\lambda_2-\lambda_3+4)(\lambda_1-\lambda_3+6).
\end{equation}
With this general estimate we can bound the terms in the decomposition \eqref{bigsum}. For instance the fifth term 
$$-(p-1)\sum_{n\leq -1}\sum_{m\geq 1}\Theta_p\begin{pmatrix}
p^{m}&&\\
&1&\\
&&p^{n}
\end{pmatrix}$$
 is bounded, using \eqref{generictermbound}, in absolute value by
$$
	\leq |\Theta_p(\bfone_p)|(p-1)\sum_{n\leq -1} \sum_{m\geq 1}\frac{4(m+4)(|n|+4)(m+|n|+6)}{p^{m+|n|}}
	\leq |\Theta_p(\bfone_p)|\frac{10^3}{p}.
$$
Here we have used the crude bound $|K_pp^{\boldsymbol{\lambda}}K_p/K_p|\geq 1$.

Bounding the other terms similarly we deduce the following bound for the absolute value of the right hand side of \eqref{198}:
$$
 |\mcP^*(\widetilde{\xi}_p,\widetilde{\xi}_p;\xi'_p,\xi'_p)-\mu(I_p'(1))|\leq\frac{10^6\mu(I_p'(1))}{p}.
$$
In other words,
\begin{equation}\label{199}
\Bigg|\int_{G'(\mathbb{Q}_p)}\frac{\Theta_p(\mfn_p^{-1}g_p\mfn_p)\Xi'_p(g_p)}{\Theta_p(\bfone_p)\Xi'_p(\bfone_p')}dg_p-\mu(I_p'(1))\Bigg|\leq \frac{10^6\mu(I_p'(1))}{p}.
\end{equation}
Now  \eqref{180}  follows from  \eqref{199}, and the identity 
\begin{equation}
	\label{muI'p1}
	\mu(I_p'(1))=\frac{1}{(p-1)(p+1)}=\frac{1}{p^2-1}.
\end{equation}


\begin{remark}\label{REmntilde}
When one works on a spherical vector $\xi_p$ without the translation by $\mfn_p,$ the periods $\mathcal{P}_p^{\natural}(\xi_{p},\xi_{p};\xi_{p}',\xi'_{p})$ vanishes. In fact, applying Cartan-Iwahori decomposition, we obtain by Lemma \ref{178}, Lemma \ref{162'} and Proposition \ref{175'}, that
\begin{equation}\label{182}
\frac{\mathcal{P}_p(\xi_{p},\xi_{p};\xi_{p}',\xi'_{p})}{\peter{\xi_{p},\xi_{p}}_p\cdot\peter{\xi'_{p},\xi'_{p}}_p}=\sum_{n\geq 0}\sum_{w'\in W_n'}\frac{\langle\pi_p(w')\xi_p,\xi_p\rangle_p}{\peter{\xi_{p},\xi_{p}}_p}\cdot\frac{(-1)^{\lambda'(w')}}{p+1},
\end{equation}
which converges absolutely. Therefore, one can switch the sums on the right hand of \eqref{182}, obtaining
\begin{align*}
\frac{\mathcal{P}_p(\xi_{p},\xi_{p};\xi_{p}',\xi'_{p})}{\peter{\xi_{p},\xi_{p}}\peter{\xi'_{p},\xi'_{p}}_p}=\sum_{n\geq 0}\frac{\langle\pi_p(A'_n)\xi_p,\xi_p\rangle_p}{\peter{\xi_{p},\xi_{p}}_p}\cdot\mu(I_p')\sum_{w'\in W_n'}\frac{(-1)^{\lambda'(w')}}{p+1}=0.
\end{align*}
\end{remark}


\section{\bf The Geometric Side}\label{Geo}

In this section and the next three sections, we compute the terms on the right hand side of \eqref{221} with the choices of $f^\mfn$ and $\vphi'$ that have been described in the previous sections. In particular this will show the absolute convergence of the sums/integral appearing in \eqref{221} and \eqref{220} so that \eqref{220} is completely justified.

\subsection{A Basic Decomposition}
In this subsection we briefly recall a basic decomposition related to $G=U(V).$ These results will be used to calculate representatives of double cosets and estimate regular orbital integrals.

Let $P$ be the parabolic subgroup stabilizing the isotropic line through $e_{-1}.$ Explicitly, $P=MN,$ with
\begin{align*}
M=\Bigg\{m(\alpha,\beta):=\left(
\begin{array}{ccc}
\alpha&&\\
&\beta&\\
&&\overline{\alpha}^{-1}
\end{array}
\right):\ \alpha\in E^{\times},\ \beta\in E^{1}\Bigg\};
\end{align*}
and the unipotent radical
\begin{align*}
N=\Bigg\{n(b,z):=\left(
\begin{array}{ccc}
1&b&z\\
&1&-\overline{b}\\
&&1
\end{array}
\right):\ z, b\in E,\ z+\overline{z}=-b\overline{b}\Bigg\}.
\end{align*}

As algebraic groups defined over $\Qq$, $M$ is a 3-dimensional torus with split rank 1, and $N$ is a 3-dimensional unipotent group. The center  of $G$ is $$Z_G=\{m(\beta,\beta)=\diag(\beta,\beta,\beta):\ \beta\in E^1\}\subseteq M$$
so that $Z_G\simeq E^1$; given $\beta\in E^{1}$, we set 
\begin{equation}\label{gammabetadef}
 \gamma_{\beta}:=\diag (\beta,\beta,\beta)\in Z_G(\Qq).	
 \end{equation}

\begin{lemma}[Bruhat decomposition]\label{BruhatGlemma}
Let the notation be as before and $J$ given in  \eqref{Jmdef}. Then
\begin{equation}\label{BruhatG}
G=P\sqcup PJP,
\end{equation}
and $PJP=NJP=PJN.$ The expression of an element from the cell $PJP$ as $nJp$ or $pJn$ with $n\in N$ and $p\in P$ is unique.
\end{lemma}
\begin{proof}
Suppose $g\notin P.$ So $ge_{-1}\notin\langle e_{-1}\rangle.$ Suppose $ge_{-1}=c_{-1}e_{-1}+c_{0}e_{0}.$ Then $(ge_{-1}, ge_{-1})_V=c_0\overline{c_0}.$ On the other hand, $(ge_{-1}, ge_{-1})_V=(e_{-1}, e_{-1})_V=0.$ So we get a contradiction. Thus $ge_{-1}$ must involve the line $\langle e_1\rangle.$

One can write $ge_{-1}=c_{-1}e_{-1}+c_{0}e_{0}+c_1e_1$ with $c_1\neq 0.$ Then a straightforward computation using linear algebra shows that one can find some $p\in P$ satisfying $e_1=pge_{-1}.$ Hence from $(pge_0,pge_{-1})_V=0$ we deduce that $(pge_0,e_1)_V=0,$ namely, $pge_0\in \langle e_1\rangle^{\bot}=\langle e_{0}, e_1\rangle.$ Therefore,
\begin{align*}
pg=\left(
\begin{array}{ccc}
&&*\\
&*&*\\
*&*&*
\end{array}
\right)\in JP.
\end{align*}
Thus $g\in PJP,$ proving \eqref{BruhatG}. The remaining part of this lemma is similar.
\end{proof}
\begin{remark}
Note that \eqref{BruhatG} is not a decomposition as algebraic groups, since there are 6 Bruhat cells required to cover $G(E)=\GL(3,E).$
\end{remark}

\subsection{Representatives of $G'(\mathbb{Q})\backslash G(\mathbb{Q})/G'(\mathbb{Q})$}
To deal with the geometric side of the relative trace formula, we need to describe the double coset $G'(\mathbb{Q})\backslash G(\mathbb{Q})/G'(\mathbb{Q}).$ Our main tool is the Bruhat decomposition \eqref{BruhatG}.

\begin{lemma}
Let $c\in E$ be such that $c+\overline{c}=-1.$ Then
\begin{equation}\label{36}
J\left(
\begin{array}{ccc}
1&1&c\\
&1&-1\\
&&1
\end{array}
\right)J=\left(
\begin{array}{ccc}
1&-\frac{1}{1+c}&-\frac{1}{1+\overline{c}}\\
&1&\frac{1}{1+\overline{c}}\\
&&1
\end{array}
\right)J\left(
\begin{array}{ccc}
c&1&1\\
&\frac{\overline{c}}{1+\overline{c}}&-\frac{1}{1+\overline{c}}\\
&&-\frac{1}{1+c}
\end{array}
\right).
\end{equation}
\end{lemma}

\begin{proof}
By making the proof of Lemma \ref{BruhatGlemma} explicit we find
\begin{equation}\label{36'}
\left(
\begin{array}{ccc}
1&\frac{1}{1+c}&-\frac{1}{1+c}\\
&1&-\frac{1}{1+\overline{c}}\\
&&1
\end{array}
\right)\left(
\begin{array}{ccc}
1&&\\
-1&1&\\
c&1&1
\end{array}
\right)=\left(
\begin{array}{ccc}
&&-\frac{1}{1+c}\\
&\frac{\overline{c}}{1+\overline{c}}&-\frac{1}{1+\overline{c}}\\
c&1&1
\end{array}
\right).
\end{equation}

Note also that the second matrix in \eqref{36'} equals the left hand side of \eqref{36}. Hence
\begin{align*}
J\left(
\begin{array}{ccc}
1&1&c\\
&1&-1\\
&&1
\end{array}
\right)J=&\left(
\begin{array}{ccc}
1&\frac{1}{1+c}&-\frac{1}{1+c}\\
&1&-\frac{1}{1+\overline{c}}\\
&&1
\end{array}
\right)^{-1}\left(
\begin{array}{ccc}
&&-\frac{1}{1+c}\\
&\frac{\overline{c}}{1+\overline{c}}&-\frac{1}{1+\overline{c}}\\
c&1&1
\end{array}
\right)\\
=&\left(
\begin{array}{ccc}
1&\frac{1}{1+c}&-\frac{1}{1+c}\\
&1&-\frac{1}{1+\overline{c}}\\
&&1
\end{array}
\right)^{-1}J\left(
\begin{array}{ccc}
c&1&1\\
&\frac{\overline{c}}{1+\overline{c}}&-\frac{1}{1+\overline{c}}\\
&&-\frac{1}{1+c}
\end{array}
\right).
\end{align*}
Then \eqref{36} follows from computing the inverse of the first matrix in the last line.
\end{proof}
 For $x\in E$ we set 
\begin{equation}\label{44}
\gamma(x)=\left(
\begin{array}{ccc}
\frac{x\overline{x}+3\overline{x}-x+1}{4}&\frac{1+x}{2}&-\frac{1}{2}\\
\frac{(x+1)(\overline{x}-1)}{2}&x&-1\\
-\frac{(1-x)(1-\overline{x})}{2}&1-x&1
\end{array}
\right).
\end{equation}

\begin{lemma}\label{14'}
	Let $x_1, x_2\in E$, $\alpha\in E^1-\{1\}$ be such that $x_1=\alpha x_2,$. Then there exists $g_1, g_2\in G'(\mathbb{Q})$ such that
	\begin{equation}\label{37}
		\gamma(x_1)=\alpha g_1\gamma(x_2) g_2.
	\end{equation}
\end{lemma}
\begin{proof}
	A straightforward calculation shows that
	\begin{equation}\label{39}
		\gamma(x_j)J=\left(
		\begin{array}{ccc}
			1&1&-\frac{1}{2}\\
			&1&-1\\
			&&1
		\end{array}
		\right)J\left(
		\begin{array}{ccc}
			1&1-x_j&-\frac{(1-x_j)(1-\overline{x}_j)}{2}\\
			&1&\overline{x}_j-1\\
			&&1
		\end{array}
		\right),\ \ j=1, 2.
	\end{equation}
	
	By Hilbert 90, there exists $\beta=a'+b'\sqrt{-D}\in E^{\times}$ such that $\alpha=\beta\overline{\beta}^{-1},$ with $a', b'\in\mathbb{Q}.$ Since $\alpha\neq 1,$ then $b'\neq 0.$ Let $v'=a'/(2b')\in\mathbb{Q}.$ Then one has
	\begin{align*}
		\alpha=\frac{a'+b'\sqrt{-D}}{a'-b'\sqrt{-D}}=-\frac{-1/2+v'\sqrt{-D}}{-1/2-v'\sqrt{-D}}.
	\end{align*}
	
	Let $v=v'\sqrt{-D}\in E.$ Then $v+\overline{v}=0.$ Let $c=-1/2+v.$ By \eqref{36} we have
	\begin{equation}\label{38}
		J\left(
		\begin{array}{ccc}
			1&1&c\\
			&1&-1\\
			&&1
		\end{array}
		\right)J
		=\left(
		\begin{array}{ccc}
			\overline{c}^{-1}&-c^{-1}&1\\
			&-{\overline{c}}{c^{-1}}&-1\\
			&&c
		\end{array}
		\right)J\left(
		\begin{array}{ccc}
			1&c^{-1}&c^{-1}\\
			&1&-\overline{c}^{-1}\\
			&&1
		\end{array}
		\right),
	\end{equation}
	since $c+\overline{c}+1=0.$ Let $u, d\in E$ be such that $u+\overline{u}=d+\overline{d}=0.$ Let $a, b\in E^{\times}.$ Appealing to the identities \eqref{39} and \eqref{38} we then obtain
	\begin{align*}
		&\left(
		\begin{array}{ccc}
			a&&au\\
			&\alpha&\\
			&&\overline{a}^{-1}
		\end{array}
		\right)
		J
		\left(
		\begin{array}{ccc}
			1&&v\\
			&1&\\
			&&1
		\end{array}
		\right)
		\gamma(x_2)J
		\left(\begin{array}{ccc}
			b&&\\
			&1&\\
			&&\overline{b}^{-1}
		\end{array}
		\right)
		\left(
		\begin{array}{ccc}
			1&&d\\
			&1&\\
			&&1
		\end{array}
		\right)\\
		=&
		\left(
		\begin{array}{ccc}
			\frac{ab}{\overline{c}}&-\frac{a}{c}&\frac{a(1+uc)}{\overline{b}}\\
			&-\frac{\alpha\overline{c}}{c}&-\frac{\alpha}{\overline{b}}\\
			&&\frac{c}{\overline{a}\overline{b}}
		\end{array}
		\right)J\left(
		\begin{array}{ccc}
			1&\frac{1-x_2+c^{-1}}{b}&\frac{-\frac{(1-x_2)(1-\overline{x}_2)}{2}+\overline{x}_2c^{-1}}{b\overline{b}}+d\\
			&1&\frac{\overline{x}_2-1-\overline{c}^{-1}}{\overline{b}}\\
			&&1
		\end{array}
		\right).
	\end{align*}
	To compare the right hand side of this equality with $\gamma(x_1)J,$ we consider
	\begin{equation}\label{40}
		\begin{cases}
			ab=\overline{c},\ a=-c,\ \alpha=-c\overline{c}^{-1}\\
			a\overline{b}^{-1}+auc\overline{b}^{-1}=-1/2,\ \frac{1-x_2+c^{-1}}{b}=1-x_1\\
			\frac{-\frac{(1-x_2)(1-\overline{x}_2)}{2}+\overline{x}_2c^{-1}}{b\overline{b}}+d=-\frac{(1-x_1)(1-\overline{x}_1)}{2}\\
			d+\overline{d}=u+\overline{u}=0.
		\end{cases}
	\end{equation}
	Solve the system of equations \eqref{40} we have
	\begin{equation}\label{41}
		\begin{cases}
			a=-c,\ b=-\overline{c}c^{-1},\ \alpha=-c\overline{c}^{-1},\ u=-\frac{1/2+\overline{c}}{c\overline{c}},\ x_1=\alpha_1x_2\\
			d=-\frac{(1-x_1)(1-\overline{x}_1)}{2}+\frac{(1-x_2)(1-\overline{x}_2)}{2}-\overline{x}_2c^{-1}=\frac{(\alpha-1)x_2+(\overline{\alpha}-1)\overline{x}_2}{2}-\overline{x}_2c^{-1}.
		\end{cases}
	\end{equation}
	
	Let $a, b, c ,d, u, v, \alpha\in E$ be as in \eqref{41}. Then we have
	\begin{equation}\label{42}
		\gamma(x_1)J=\left(
		\begin{array}{ccc}
			a&&au\\
			&\alpha&\\
			&&\overline{a}^{-1}
		\end{array}
		\right)
		J
		\left(
		\begin{array}{ccc}
			1&&v\\
			&1&\\
			&&1
		\end{array}
		\right)
		\gamma(x_2)J
		\left(\begin{array}{ccc}
			b&&bd\\
			&1&\\
			&&\overline{b}^{-1}
		\end{array}
		\right).
	\end{equation}
	Then \eqref{37} follows from \eqref{42} by setting
	\begin{align*}
		g_1=\left(
		\begin{array}{ccc}
			\alpha^{-1}a&&\alpha^{-1}au\\
			&1&\\
			&&\alpha^{-1}\overline{a}^{-1}
		\end{array}
		\right)
		J
		\left(
		\begin{array}{ccc}
			1&&v\\
			&1&\\
			&&1
		\end{array}
		\right),\quad g_2=J
		\left(\begin{array}{ccc}
			b&&bd\\
			&1&\\
			&&\overline{b}^{-1}
		\end{array}
		\right)J.
	\end{align*}
	One verifies that $g_1, g_2\in G'(\mathbb{Q}).$ Hence Lemma \ref{14'} follows.
\end{proof}

\begin{prop}\label{repres} The
 set $$\Phi=\big\{\gamma_{\beta}: \ \beta\in E^1\}\cup \{\gamma(x):\ x\in E\big\}$$ 
 form a complete set of representatives for the double quotient $G'(\mathbb{Q})\backslash G(\mathbb{Q})/ G'(\mathbb{Q})$; here $\gamma_\beta$ and $\gamma(x)$ are defined in \eqref{gammabetadef} and \eqref{44}.
\end{prop}

\begin{proof}
We start with the case $g\in P(\mathbb{Q})$. We  decompose $g$ as $g=\beta m(\alpha,1)n(b,z)$ for some $\alpha\in E^{\times}$, $\beta\in E^1$ and $b, z\in E$. We then have
\begin{align*}
G'(\mathbb{Q})gG'(\mathbb{Q})=G'(\mathbb{Q})\left(
\begin{array}{ccc}
\alpha&&\\
&\beta&\\
&&\overline{\alpha}^{-1}
\end{array}
\right)\left(
\begin{array}{ccc}
1&b&z\\
&1&-\overline{b}\\
&&1
\end{array}
\right)G'(\mathbb{Q})
\end{align*}

Note that for any $z$ satisfying $z+\overline{z}=0,$ one has $n(0,z)\in G'(\mathbb{Q}).$

\begin{enumerate}

\item[(i)] If $b=0,$ then $G'(\mathbb{Q})gG'(\mathbb{Q})=G'(\mathbb{Q})\gamma_{\beta}G'(\mathbb{Q}),$ because the semisimple part $\diag(\alpha,1,\overline{\alpha}^{-1})\in G'(\mathbb{Q}).$

\item[(ii)] Suppose $b\neq 0.$  Hence
\begin{align*}
G'(\mathbb{Q})gG'(\mathbb{Q})=G'(\mathbb{Q})\gamma_\beta\left(
\begin{array}{ccc}
1&1&\frac{z}{b\overline{b}}\\
&1&-1\\
&&1
\end{array}
\right)G'(\mathbb{Q})=G'(\mathbb{Q})\gamma(\beta)G'(\mathbb{Q}),
\end{align*}
\end{enumerate}
where we have used Lemma \ref{14'} to replace $\gamma_\beta\gamma(1)$ by $\gamma(\beta)$.

We now deal with the case $g\in  P(\mathbb{Q})JP(\mathbb{Q})$. We can write $g=m(\alpha,\beta)n(b,z)Jn(b',z')$
for some $\alpha\in E^{\times},$ $\beta\in E^1$ and $b, b', z, z'\in E$. We then have
\begin{align*}
G'(\mathbb{Q})gG'(\mathbb{Q})=G'(\mathbb{Q})\left(
\begin{array}{ccc}
\alpha&&\\
&\beta&\\
&&\overline{\alpha}^{-1}
\end{array}
\right)\left(
\begin{array}{ccc}
1&b&z\\
&1&-\overline{b}\\
&&1
\end{array}
\right)J \left(
\begin{array}{ccc}
1&b'&z'\\
&1&-\overline{b}'\\
&&1
\end{array}
\right)G'(\mathbb{Q})
\end{align*}

\begin{enumerate}
	\item[(iii)] Suppose $b=0.$ Then $z+\overline{z}=-b\overline{b}=0,$ implying that $m(\alpha,1)n(b,z)\in G'(\mathbb{Q}).$ Note that $J\in G'(\mathbb{Q}).$ Hence this situation will boil down to case (i) or (ii), namely, we obtain, under the hypothesis $b=0,$ that
	\begin{align*}
	G'(\mathbb{Q})gG'(\mathbb{Q})\subseteq G'(\mathbb{Q})\gamma_{\beta}\cup G'(\mathbb{Q})\gamma_{\beta}\gamma(1)G'(\mathbb{Q}).
	\end{align*}
	\item[(iv)] Suppose $b'=0.$ Then similarly we obtain that
	\begin{align*}
	G'(\mathbb{Q})gG'(\mathbb{Q})\subseteq  G'(\mathbb{Q}) \gamma_{\beta} \cup G'(\mathbb{Q})\gamma_{\beta}\gamma(1)G'(\mathbb{Q}).
	\end{align*}
	\item[(v)] Suppose $b\neq 0$ and $b'\neq 0.$ Let $x=1-\overline{b}b'\neq 1$ and $\widetilde{z}=b\overline{b}z'.$ Then
	\begin{align*}
	G'(\mathbb{Q})gG'(\mathbb{Q})=G'(\mathbb{Q})\gamma_{\beta}\left(
	\begin{array}{ccc}
	1&1&\frac{z}{b\overline{b}}\\
	&1&-1\\
	&&1
	\end{array}
	\right)J\left(
	\begin{array}{ccc}
	1&1-x&\widetilde{z}\\
	&1&\overline{x}-1\\
	&&1
	\end{array}
	\right)G'(\mathbb{Q}).
	\end{align*}
	Using the fact that $n(0,z)\in G'(\mathbb{Q})$ when $z+\overline{z}=0,$ we can further deduce
	\begin{align*}
	G'(\mathbb{Q})gG'(\mathbb{Q})=G'(\mathbb{Q})\gamma_{\beta}\left(
	\begin{array}{ccc}
	1&1&-1/2\\
	&1&-1\\
	&&1
	\end{array}
	\right)J\left(
	\begin{array}{ccc}
	1&1-x&-\frac{(1-x)(1-\overline{x})}{2}\\
	&1&\overline{x}-1\\
	&&1
	\end{array}
	\right)G'(\mathbb{Q}).
	\end{align*}
\end{enumerate}

Then it follows from Lemma \ref{BruhatGlemma} and the identity
\begin{align*}
\left(
\begin{array}{ccc}
1&1&-\frac{1}{2}\\
&1&-1\\
&&1
\end{array}
\right)J\left(
\begin{array}{ccc}
1&1-x&-\frac{(1-x)(1-\overline{x})}{2}\\
&1&\overline{x}-1\\
&&1
\end{array}
\right)=\left(
\begin{array}{ccc}
-\frac{1}{2}&\frac{1+x}{2}&\frac{x\overline{x}+3\overline{x}-x+1}{4}\\
-1&x&\frac{(x+1)(\overline{x}-1)}{2}\\
1&1-x&-\frac{(1-x)(1-\overline{x})}{2}
\end{array}
\right)
\end{align*}

that 
\begin{equation}\label{43''}
G'(\mathbb{Q})\backslash G(\mathbb{Q})/ G'(\mathbb{Q})=\bigcup_{\gamma\in\Phi}Z_G(\mathbb{Q})G'(\mathbb{Q})\gamma G'(\mathbb{Q}),
\end{equation}
where $Z_G$ is the center of $G$. By Lemma \ref{14'} we can write \eqref{43''} as
\begin{equation}\label{43}
 G(\mathbb{Q})=\bigcup_{\gamma\in\Phi}G'(\mathbb{Q})\gamma G'(\mathbb{Q}).
\end{equation}
Now we show the union in \eqref{43} is actually disjoint.

Let $\gamma(x_1), \gamma(x_2)\in\Phi.$ Suppose $\gamma(x_i)$ ($1\leq i\leq 2$) are such that $$G'(\mathbb{Q})\gamma(x_1)G'(\mathbb{Q})=G'(\mathbb{Q})\gamma(x_2)G'(\mathbb{Q}).$$ Combining the definition of $\gamma(x_1), \gamma(x_2)$ in \eqref{44} and the identity
\begin{align*}
\left(
\begin{array}{ccc}
*&&*\\
&1&\\
*&&*
\end{array}
\right)\left(
\begin{array}{ccc}
*&*&*\\
*&x&*\\
*&*&*
\end{array}
\right)\left(
\begin{array}{ccc}
*&&*\\
&1&\\
*&&*
\end{array}
\right)=\left(
\begin{array}{ccc}
*&*&*\\
*&x&*\\
*&*&*
\end{array}
\right),
\end{align*}
we deduce (by comparing the $(2,2)$-th entry) that $x_1= x_2.$ 

By the Bruhat decomposition, the orbit $G'(\mathbb{Q})\gamma_{\beta}G'(\mathbb{Q})$ does not intersect  the orbit $G'(\mathbb{Q})\gamma(x)G'(\mathbb{Q})$ for any $\beta\in E^1$ and $x\in E.$ 
In conclusion, \eqref{43} is a disjoint union.
\end{proof}

Given $\gamma\in \Phi,$ we denote by ${H}_{\gamma}\subset G'\times G'$ the stabilizer of $\gamma,$ namely, $${H}_{\gamma}=\big\{(u,v)\in {G}'\times {G}':\ u^{-1}\gamma v=\gamma\big\}.$$

Let $f^{\mathfrak{n}}$ be defined by \eqref{deffn}, and let $\varphi'$ be an automorphic form on $G'(\Qq)\backslash G'(\Aa)$. We formally define the associated orbital integral by
\begin{equation}\label{orbitalintdef}
\mathcal{O}_{\gamma}(f^{\mathfrak{n}},\varphi')=\int_{{H}_{\gamma}(\Qq)\backslash (G'\times G')(\mathbb{A})}f^{\mathfrak{n}}(u^{-1}\gamma v)\varphi'(u)\overline{\varphi}'(v)dudv.	
\end{equation}
By Proposition \ref{repres} we can rewrite $J(f)$ (at least formally) as
\begin{equation}\label{3}
J(f^{\mathfrak{n}},\varphi')=\sum_{\beta\in E^1}\mathcal{O}_{\gamma_{\beta}}(f^{\mathfrak{n}},\varphi')+\sum_{x\in E}\mathcal{O}_{\gamma(x)}(f^{\mathfrak{n}},\varphi').
\end{equation}
The analysis of these integrals of course depends heavily on the structure of the stabilizer $H_\gamma$ and we give here a roadmap of what is to come.

\begin{itemize}
	\item[--] For $\beta\in E^1$, $\gamma_\beta\in Z_G(\Qq)$ and the stabilizer $H_{\gamma_\beta}$ is the diagonal subgroup
$$H_{\gamma_\beta}=\Delta G'\subset G'\times G'.$$

\item[--] We will see in Sec. \ref{Secunipotent} that the stabilizer $H_{\gamma(1)}$ (and more generally $H_{\gamma(x)}$ for $x\in E^1$) is isomorphic to the unipotent radical of the Borel subgroup of $G'$.
\item[--] Finally for $x\in E-E^1$ we will see that $H_{\gamma(x)}$ is a torus isomorphic to the unitary group $U(1).$
\end{itemize}

Moreover we can use the support and invariance properties of $f^{\mathfrak{n}}$ to infer further restrictions on the $\gamma_\beta$ and $\gamma(x)$ whose orbital integral is non-zero.

Regarding the former, notice that $u,v\in G'(\mathbb{A}),$ $u^{-1}\gamma_{\beta}v$ is of the form $\begin{pmatrix}
	*&&*\\
	&\beta&\\
	*&&*
\end{pmatrix}.$
Therefore, by definition of $f^{\mathfrak{n}}$,  one has $$f^{\mathfrak{n}}(u^{-1}\gamma_{\beta}v)=0$$ unless $$\beta\in \mathcal{O}_E^1=E^1\cap \mathcal{O}_E.$$ 
In this case $f^{\mathfrak{n}}$ is invariant under $Z_G(\mathcal{O}_E^1),$ by Lemma \ref{14'} we have 
\begin{align*}
\mathcal{O}_{\gamma_{\beta}}(f^{\mathfrak{n}},\varphi')=\mathcal{O}_{\gamma_{1}}(f^{\mathfrak{n}},\varphi'),\ \ \beta\in  \mathcal{O}_E^1,
\end{align*}
therefore
\begin{equation}\label{Ogamma(x)}
	\sum_{\beta\in E^1}\mathcal{O}_{\gamma_{\beta}}(f^{\mathfrak{n}},\varphi')=w_E\mathcal{O}_{\gamma_{1}}(f^{\mathfrak{n}},\varphi')
\end{equation}
where $w_E=\# \mathcal{O}_E^1$ is finite (as $E$ is an imaginary quadratic field).


Regarding the orbital integrals $\mathcal{O}_{\gamma(x)}(f^{\mathfrak{n}},\varphi'),\ x\in E^1$, we will see in \S \ref{Secunipotent} that these converge absolutely. Moreover, using by Lemma \ref{14'} we will show, in Proposition \ref{lem44} in Sec. \ref{sec44}, that for $x\in E^1,$ one has 
\begin{equation}\label{284}
	\mathcal{O}_{\gamma(x)}(f^{\mathfrak{n}},\varphi')=\begin{cases}
		\mathcal{O}_{\gamma(1)}(f^{\mathfrak{n}},\varphi'),\ \text{if $x\in\mathcal{O}_E^1;$}\\
		0,\ \text{otherwise}.
	\end{cases}
\end{equation}
This implies that
\begin{equation}\label{6.18}
\sum_{x\in E^1}\mathcal{O}_{\gamma(x)}(f^{\mathfrak{n}},\varphi')=w_E\mathcal{O}_{\gamma(1)}(f^{\mathfrak{n}},\varphi')z
\end{equation}

Finally, for the (more complicated) orbital integrals $\mathcal{O}_{\gamma(x)}(f^{\mathfrak{n}},\varphi'),\ x\in E-E^1$, we will see in \S  \ref{sec8.3}, that these orbital integrals  as well as  their sum  converge absolutely.
From \eqref{Ogamma(x)} and \eqref{6.18} we then obtain 
\begin{align}\label{Jsimple}
	J(f^{\mathfrak{n}},\varphi')=w_E\mathcal{O}_{\gamma_1}(f^{\mathfrak{n}},\varphi')+w_E\mathcal{O}_{\gamma(1)}(f^{\mathfrak{n}},\varphi')+\sum_{x\in E-E^1}\mathcal{O}_{\gamma(x)}(f^{\mathfrak{n}},\varphi'),
\end{align}

In what follows, we will call
\begin{enumerate}
\item[--]  $\mathcal{O}_{\gamma_1}(f^{\mathfrak{n}},\varphi')$ the \textit{identity orbital integral};	
\item[--]   $\mathcal{O}_{\gamma(1)}(f^{\mathfrak{n}},\varphi')$ and more generally the integrals $\mathcal{O}_{\gamma(x)}(f^{\mathfrak{n}},\varphi'),\ x\in E^1$   the \textit{unipotent orbital integrals},

\item[--]	 $\mathcal{O}_{\gamma(x)}(f^{\mathfrak{n}},\varphi'),\ x\in E- E^1$  \textit{the regular orbital integrals} (as these $\gamma(x)$'s are regular). 
\end{enumerate}



\section{\bf The Identity Orbital Integral}\label{SecIdentity}

In this section, we deal with the identity orbital integral (associated to the identity element $\gamma_1=\Id_3$):
\begin{align*}
	\mathcal{O}_{\gamma_1}(f^{\mathfrak{n}},\vphi')=\int_{{H}_{\gamma_1}(\mathbb{Q})\backslash {H}(\mathbb{A})}f^{\mathfrak{n}}(x^{-1} y)\ov\vphi'(x){\vphi}'(y)dx\,dy.
\end{align*}

Let $\nu(f)$ be the set of primes $p$ such that 
\begin{equation}\label{fpdef2}
	f_{p}=1_{G(\Zp)A_{p^{r_p}}G(\Zp)}
\end{equation} for some integer $r=r_p\geq 1.$
 By construction, as $p\in \nu(f),$ $p$ is inert  $(p,NN')=1$; in particular $\pi_p'$ is unramified. 

The matrix  $A_{p^{r}}$ (see \eqref{Aprdef}) also belongs to $G'(\mathbb{Q}_p)$ and convolution by the function 
$$
1_{G'(\Zp)A_{p^r}G'(\Zp)}
$$ 
is a Hecke operator; the space of $G'(\Zp)$-invariant functions ${\pi_p'}^{G'(\Zp)}$ is an eigenspace with eigenvalue
  $$p^{r}\lambda_{\pi_p'}(p^{r}).$$
  Moreover since $\pi'$ is tempered one has $|\lambda_{\pi_p'}(p^{r})|\leq 2$ 
  (and for $r=0$, $\lambda_{\pi_p'}(1)=1$).

\begin{prop}\label{propIdentity}
	Let the notation be as before. Then $\mathcal{O}_{\gamma_1}(f^{\mathfrak{n}},\vphi')$ is equal to
	\begin{equation}\label{EqIdentityOrb}\frac{\peter{\vphi',\vphi'}}{d_k}
\frac{N^2}{{N'}^2}\Psi(N)\mfS(N')\prod_{p\in \nu(f)}p^{r_p}\lambda_{\pi'}(p^{r_p}).
	\end{equation}
	where
	$$d_k={k-1},\ \Psi(N)=\prod_{p\mid N}\left(1-\frac{1}{p}+\frac{1}{p^2}\right),\ \mfS(N')=\prod_{p\mid N'}\frac{1}{1-p^{-2}}.$$
\end{prop}
\begin{proof}
	Note that ${H}_{\gamma_1}=\Delta G',$ the diagonal embedding of $G'$ in $G'\times G'.$ Hence we can change variables to obtain
	\begin{align*}
		\mathcal{O}_{\gamma_1}(f^{\mathfrak{n}})=&\int_{\Delta G'(\mathbb{A})\backslash G'(\mathbb{A})\times G'(\mathbb{A})}f(\widetilde{\mathfrak{n}}^{-1}x^{-1}y\widetilde{\mathfrak{n}})\int_{G'(\mathbb{Q})\backslash G'(\mathbb{A})}\ov\vphi'(hx){\vphi}'(hy)dhdx\,dy.
	\end{align*}
	We can write it as the Petersson inner product of cusp forms: 
	\begin{equation}\label{inner}
		\mathcal{O}_{\gamma_1}(f^{\mathfrak{n}})=\int_{G'(\mathbb{Q})\backslash G'(\mathbb{A})}\ov\vphi'(x){\pi'({f^{\mathfrak{n}}})\vphi'(x)}dx=\langle \pi'({f^{\mathfrak{n}}})\vphi',\vphi'\rangle.
	\end{equation}
	Since $\vphi'$ is  rapidly decreasing, each integral above is absolutely convergent.
	
	Write $$(\pi',V)=(\pi'_{\infty},V_{\infty})\otimes (\pi'_{\fin},V_{\fin});$$ 
	we have $\vphi'\simeq \xi'_{\infty}\otimes \xi'_{\fin}\in V_{\infty}\otimes V_{\fin}$ for some local new vector $\xi'_{\fin}\in V_{\fin}.$ Let us also recall that through the isomorphism $\iota$ (see \eqref{isomSL2}) we have $$\pi'\simeq \pi_k^+,$$ the cuspidal representation of $\SL(2,\mathbb{A})$ of level $1$ and whose archimedean component is the holomorphic discrete series of weight $k$ and set $\xi'_{\infty}=v_{k}^{\circ}\circ\iota.$ 
	
	Writing $f=f_{\infty}\times f_{\fin}$  we have
	\begin{align*}
		\pi'({f^{\mathfrak{n}}})(\xi'_{\infty}\otimes \xi'_{\fin})
		=&\int_{{G}'(\mathbb{R})}\int_{{G}'(\mathbb{A}_{\fin})}f_{\infty}(y_{\infty})\pi_{\infty}'(y_{\infty})\xi'_{\infty}\otimes f^{\mathfrak{n}}_{\fin}(y_{\fin})\pi_{f}'(y_{\fin})\xi'_{\fin}dy,
	\end{align*}
	where $dy=dy_{\infty}dy_{\fin}$ with $y_{\fin}=\otimes_{p<\infty}y_p\in G'(\mathbb{A}_{\fin})$ and we have
	$$\pi'({f^{\mathfrak{n}}})\vphi'=\pi'({f^{\mathfrak{n}}})(\xi'_{\infty}\otimes \xi'_{\fin})=\pi_{\infty}'({f}_{\infty})\xi'_{\infty} \bigotimes_{p<\infty}\pi'_{p}({f^{\mathfrak{n}}_{p}})\xi'_{p}.$$ We recall that $\xi'_{\fin}$ is unique up to scalar. Specifically, at $p\nmid N',$ $\xi'_p$ is spherical; and at $p\mid N',$ $\xi'_p$ is a nonzero  Iwahori fixed vector.

	Note that by \eqref{56} we can regard $f_{\infty}$ as the matrix coefficient $F_{-k_1}.$ Since $SU(1,1;\mathbb{R})$ is unimodular and the central characters in the above representations are trivial, we can apply Schur orthogonality relations to conclude that $$\pi_{\infty}'({f}_{\infty})\xi'_{\infty}=0$$ if $k\neq -k_1;$ and when $k=-k_1,$ we have
	\begin{equation}\label{60}
		\pi_{\infty}'({f}_{\infty})\xi'_{\infty}=\frac{1}{k-1}\xi'_{\infty}.
	\end{equation}

	\begin{enumerate}

		\item Suppose $p\mid N.$ Then $\xi'_p$ is $I_p'$-fixed, $f_p^{\mathfrak{n_p}}=f_p$ and $G'(\mathbb{Q}_p)\cap \supp f_p=I_p'$. Therefore, we have 
		\begin{align}\nonumber
			\pi'_p({f_p^{\mathfrak{n_p}}})\xi'_p&=\pi'_p({f}_p)\xi'_p=\int_{{G}'(\mathbb{Q}_{p})} f_{p}(y_{p})\pi_{p}'(y_{p})\xi'_{p}dy_{p}\\
			&=
			\frac{\mu({I_p'})}{\mu({K_p}(N))}\cdot \xi'_p=({p^2-p+1})\xi'_p\label{60N}
		\end{align}
		by Lemma \ref{lemIwahoriCartanU(V)} and Lemma \ref{K'cosetlemma} (see the Appendix).
		
		\item Suppose $p\mid N'.$ In that case $f_p^{\mathfrak{n_p}}\neq f_p$ and we then write $y_p=\begin{pmatrix}
			a&&b\\
			&1&\\
			c&&d
		\end{pmatrix}$.
		 By definition the non-vanishing $${f}_p(\widetilde{\mathfrak{n}}_p^{-1}y_p\widetilde{\mathfrak{n}}_p)\neq 0$$ amounts to
		\begin{align*}
\mathfrak{n}_p^{-1}\begin{pmatrix}
				a&b&\\
				c&d&\\
				&&1
			\end{pmatrix}\mathfrak{n}_p=\begin{pmatrix}
				a&b&(a-1)p^{-1}\\
				c&d&cp^{-1}\\
				&&1
			\end{pmatrix}\in Z_G(\Qp)K_p,
		\end{align*}
	which  is equivalent to
		\begin{align*}
			y_p=\begin{pmatrix}
				a&&b\\&1&\\
				c&&d
			\end{pmatrix}\in I_p'(1)=\Big\{\begin{pmatrix}
				g_{11}&&g_{12}\\&1&\\
				g_{21}&&g_{22}
			\end{pmatrix}\in I_p':\ g_{11}\in 1+p\mathbb{Z}_p\Big\}.
		\end{align*}
		Therefore, we have by  \eqref{muI'p1} that 
		\begin{align}\label{60N'}
			\pi'_p({f_p^{\mathfrak{n_p}}})\xi'_p=\frac{1}{\mu(K_p)}\int_{I_p'(1)}\pi'(y_p)\xi'_pdy_p=\frac{\mu(I_p'(1))}{\mu(K_p)}\xi'_p=\frac{1}{p^2-1}\xi'_p.
		\end{align}
	
	\item Suppose $p$ is inert and $p\nmid NN',$ and   $$f_p^{\mathfrak{n_p}}=f_{p}=1_{G(\Zp)A_{p^r}G(\Zp)}$$
	for some $r\geq 0.$ Since $\xi'_p$ is spherical and
	$$
	G(\Zp)A_{p^{r_p}}G(\Zp)\cap G'(\mathbb{Q}_p)=G'(\Zp)A_{p^{r_p}}G'(\Zp).
	$$
	we have 
	$$
	\pi'_p({f_p^{\mathfrak{n_p}}})\xi'_p=p^{r_p}\lambda_{\pi'}(p^{r_p})\xi'_p.
	$$

Combining \eqref{60} with \eqref{60N} and \eqref{60N'} we  obtain
		\begin{equation}\label{64}
			\pi'({f})\vphi'=\frac{1}{k-1}\frac{N^2}{{N'}^2}\Psi(N)\mfS(N')\prod_{p\in \nu(f)}p^{r_p}\lambda_{\pi'}(p^{r_p})\cdot \vphi'.
		\end{equation}
		
	\end{enumerate}
	
	Substituting  \eqref{64} in \eqref{inner},  \eqref{EqIdentityOrb} follows.
\end{proof}



\section{\bf The Unipotent Orbital Integrals}\label{Secunipotent}

In this section, we will deal with the orbital integral with respect to $\gamma(x)$ when $x\in E^1.$ We will start with the special case that $x=1;$ as we will see from Proposition \ref{lem44}, the general case reduces to this special case.

Recall that (see \eqref{44})
\begin{align*}
	\gamma(1)=\left(
	\begin{array}{ccc}
		1&1&-1/2\\
		&1&-1\\
		&&1
	\end{array}
	\right).
\end{align*}

Let ${H}_{\gamma(1)}$ be the stabilizer of $\gamma(1)$ and define (at least formally) the following integral
\begin{align*}
	\mathcal{O}_{\gamma(1)}(f,\vphi')=\int_{{H}_{\gamma(1)}(\Qq)\backslash {H}(\mathbb{A})}f(x^{-1}\gamma(1) y)\overline{\vphi}'(x){\vphi}'(y)dx\,dy.
\end{align*}
where $f$ is the function noted $f^{\mathfrak n}$ in \eqref{globalfnchoice}.
We will show below that $\mathcal{O}_{\gamma(1)}(f,\vphi')$ converges absolutely so that this integral is well defined.

\subsection{Factorization of the Unipotent Orbital Integral}
The orbital integral $\mathcal{O}_{\gamma(1)}(f,\vphi')$ does not  factor as a product of local components over $p\leq \infty$ but we will apply the Fourier expansion to it which will express it as an infinite sum of factorable integrals over $\mathbb{A}$ from which a sharp upper bound will be deduced. 

We start with the following explicit expression:
\begin{lemma}\label{23.}
	Let the notation be as above and let\footnote{Hopefully this will not create a confusion with the conductor of $\pi'$}  $N'$ be the unipotent subgroup of the standard parabolic subgroup of $G'=U(W),$ i.e., for any  $\Qq$-algebra $R$,
	\begin{align*}
		N'(R)=\Bigg\{n(0,b):=\left(
		\begin{array}{ccc}
			1&&b\\
			&1&\\
			&&1
		\end{array}
		\right)\in \GL(3,E\otimes_{\Qq}R):\ b+\overline{b}=0\Bigg\}.
	\end{align*}
Then we have
	$$H_{\gamma(1)}=\Delta N'\subset G'\times G'$$
	and
	\begin{equation}\label{49}
		\mathcal{O}_{\gamma(1)}(f,\vphi')=\int_{N'(\mathbb{A})\backslash G'(\mathbb{A})}\int_{G'(\mathbb{A})}f\left(x^{-1}\gamma(1)y\right)\int_{[N']}\overline{\vphi}'(vx){\vphi}'(vy)dvdx\,dy
	\end{equation}
	where $[N']=N'(\Qq)\bash N'(\Aa)$.
\end{lemma}
\begin{proof}
	To compute the stabilizer $H_{\gamma(1)},$ we consider the equation
	\begin{align*}
		\left(
		\begin{array}{ccc}
			a&&b\\
			&1&\\
			c&&d
		\end{array}
		\right)\left(
		\begin{array}{ccc}
			1&1&-1/2\\
			&1&-1\\
			&&1
		\end{array}
		\right)=\left(
		\begin{array}{ccc}
			1&1&-1/2\\
			&1&-1\\
			&&1
		\end{array}
		\right)\left(
		\begin{array}{ccc}
			a'&&b'\\
			&1&\\
			c'&&d'
		\end{array}
		\right).
	\end{align*}
	
	The solution is $a=d=a'=d'=1,$ $c=c'=0$ and $b=b'.$ 
	
	In other terms, $H_{\gamma(1)}=\Delta N',$  the image of the diagonal embedding $$\Delta:N'\hookrightarrow N'\times N'.$$ We have
	\begin{equation}\label{117}
		\mathcal{O}_{\gamma(1)}(f)=\int_{\Delta N'(\Qq)\backslash {H}(\mathbb{A})}f(x^{-1}\gamma(1) y)\overline{\vphi}'(x){\vphi}'(y)dx\,dy.
	\end{equation}
	
	Let $$H_1=N'(\mathbb{A})^2\backslash {G}'(\mathbb{A})^2,\ H_2=\Delta N'(\Qq)\backslash N'(\mathbb{A})^2,\ H_3=N'(\mathbb{A})\bash G'(\mathbb{A}).$$ Then the right hand side is equal to
	\begin{align*}
		&\int_{H_1}\int_{H_2}f(x^{-1}n(0,\overline{b})\gamma(1) n(0,b')y)\overline{\vphi}'(n(0,b)x){\vphi}'(n(0,b')y)dbdb'dx\,dy\\
		=&\int_{H_1}\int_{H_2}f\left(x^{-1}\left(
		\begin{array}{ccc}
			1&1&-\frac{1}{2}+b'-b\\
			&1&-1\\
			&&1
		\end{array}
		\right)y\right)\overline{\vphi}'(n(0,b)x){\vphi}'(n(0,b')y)dbdb'dx\,dy\\
		=&\int_{H_3}\!\int_{H_3}\!\int_{H_2} f\left(x^{-1}\left(
		\begin{array}{ccc}
			1&1&-\frac{1}{2}+b'\\
			&1&-1\\
			&&1
		\end{array}
		\right)y\right)\overline{\vphi}'(n(0,b)x){\vphi}'(n(0,b+b')y)dbdb'dx\,dy\\
		=&\int_{H_3}\int_{[N']}\int_{{G}'(\mathbb{A})}f\left(x^{-1}\left(
		\begin{array}{ccc}
			1&1&-\frac{1}{2}\\
			&1&-1\\
			&&1
		\end{array}
		\right)y\right)\overline{\vphi}'(n(0,b)x){\vphi}'(n(0,b)y)dbdx\,dy.
	\end{align*}
	Therefore, we obtain
	\begin{align*}
		\mathcal{O}_{\gamma(1)}(f,\vphi')=\int_{N'(\mathbb{A})\backslash {G}'(\mathbb{A})}\int_{{G}'(\mathbb{A})}f\left(x^{-1}\gamma(1)y\right)\int_{[N']}\overline{\vphi}'(vx){\vphi}'(vy)dvdx\,dy,
	\end{align*}
	which proves Lemma \ref{23.}.
\end{proof}

Let us recall (see \S\ \ref{U(W)choice}) that the automorphic forms $\vphi'$ on $G'(\Aa)$ correspond to a $\GL_2(\Aa)$-cusp form $\vphi_1$. The latter admits a Fourier expansion which translates to a corresponding expansion for $\vphi'$. We spell this out below.

Let $$\theta=\theta_\infty.\theta_{f}=\theta_\infty\prod_{p}\theta_p$$
be the usual unramified additive character of $\Aa/\Qq$: i.e. 
$\theta_{\infty}(x):=e^{-2\pi i x},$ and for $p$ a prime, $\theta_p(x)=e^{2\pi i r_p(x)}$, where $r_p(x)$ is the principal part of $x\in\Qp.$ For $n\in\mathbb{A}$ and $x\in\mathbb{A}$, we define  $$\theta_{n}(x):=\theta(n x).$$ 

 The additive character $\theta_n$ defines a character on $N'(\Qq)\backslash N'(\Aa)$  by setting 
 $$\psi_{n}(u):=\theta_{n}(x)$$
 for
 \begin{equation}\label{isomSL22}
	u=u(x)=\begin{pmatrix}
		1&0&x\Delta\\
		0&1&0\\
		0&0&1
	\end{pmatrix}\in N'(\Aa),
	\end{equation}
with $\Delta$ as in  \eqref{Deltadef}.

We also set for $R$ a commutative $\Qq$-algebra and $v,w\in R^\times$
$$a(v,w):=\begin{pmatrix}v&\\&1&\\&&w\end{pmatrix},\ a(v):=a(v,1)=\begin{pmatrix}v&\\&1&\\&&1\end{pmatrix}.$$
We denote by
\begin{align*}
	W_{n}(g;\vphi'):=\int_{N'(\Qq)\backslash N'(\mathbb{A})}{\vphi}'(ug)\overline{\psi_{n}(u)}du
\end{align*}
the $n$-Whittaker function of ${\vphi}'$. By $G'(\Qq)$-invariance of ${\vphi}'$ and the identity
$$\ \iota^{-1}(u(x))=\begin{pmatrix}
		1&-x\\
		0&1
	\end{pmatrix},$$
	for $\iota:\SL_2\simeq SU(W)$ the exceptional isomorphism discussed in \S \ref{exceptionalisom},
	we have
$$W_n(g;\vphi')=W_1(a(n)g;\vphi_1),$$ the Whittaker function associated to $\varphi_1$ relative to the character $\theta$ and the Fourier expansion
\begin{equation}\label{48}
	{\vphi}'\left(ug\right)=\sum_{n\in \Qq^{\times}}W_{n}(g;\vphi')\psi_n(u).
\end{equation}

Substituting \eqref{48} in the expression \eqref{49} of $\mathcal{O}_{\gamma(1)}(f)$ we then get
\begin{align*}
	\mathcal{O}_{\gamma(1)}(f,\vphi')=&\int_{N'(\mathbb{A})\backslash {G}'(\mathbb{A})}\int_{{G}'(\mathbb{A})}f\left(x^{-1}\gamma(1)y\right)\sum_{n\in\Qq^\times}\overline{W_{n}}(x;\vphi'){W_{n}(y;\vphi')}dx\,dy\\
	=&\sum_{n\in\Qq^\times}\int_{N'(\mathbb{A})\backslash {G}'(\mathbb{A})}\int_{{G}'(\mathbb{A})}f\left(x^{-1}\gamma(1)y\right)\overline{W_{n}}(x;\vphi'){W_{n}(y;\vphi')}dx\,dy.
\end{align*}

Since $\phi'$ has been chosen to be a primitive cusp form, $\vphi'$ is decomposable. Let $$\vphi'\simeq \otimes_v\xi'_v\in\otimes_v'V_{\pi_v'}.$$ Decomposing $\psi_n$ in local characters
$$\psi_n=\prod_v\psi_{n,v}$$ yields a decomposition of the corresponding Whittaker function ($g=(g_v)_v\in G'(\Aa)$)
\begin{equation}\label{whittakerproduct}
	W_{n}(g;\vphi')=\prod_vW_{n,v}(g_v;\xi'_v)=\prod_vW_{1,v}(a(n)g_v;\xi'_v),	
\end{equation}
where the
$W_{n,v}(g_v;\xi'_v)=W_{1,v}(a(n)g_v;\xi'_v)$ are the local Whittaker functions which satisfy 
$$W_{n,v}(u_vg_v;\xi'_v)={\psi}_{n,v}(u_v)W_{n,v}(g_v;\xi'_v)$$ for all $u_v\in N'(\Qq_v),$ $g_v\in G'(\Qq_v).$
Also to simplify the notation, we write in the rest of this section 
\begin{equation}\label{Wnvdef}
	W_{n,v}(g_v):=W_{1,v}(a(n)g_v;\xi'_v).
\end{equation}

For $p$ a prime, let $\bfone_p$ be the identity element in $G'(\Qq_p)$ and write $\bfone_f$ for the identity element of $G'(\Af)$.  We choose for all  primes $p$, $\xi'_p$'s to be the local new vector normalized such that $$W_{1,p}(\textbf{1}_p;\xi'_p)=1.$$ This normalization will be our choice for $\xi'_p$'s henceforth. Using this decomposition of Whittaker functions we can express $\mathcal{O}_{\gamma(1)}(f)$ as  a sum of products of local orbital integrals.
\begin{lemma}\label{uniporbfactolem}
	Let the notation be as before. Then
	\begin{equation}\label{unipint1facto}
		\mathcal{O}_{\gamma(1)}(f,\vphi')=\sum_{n\in\Qq^{\times}}\mathcal{O}(f;n)
	\end{equation}
	where $$\mathcal{O}(f;n)=\prod_{v}\mathcal{O}_{v}(f;n)$$
	with $v$ running through all the places of $\Qq$ and
	\begin{equation}\label{75}
		\mathcal{O}_{v}(f;n):=\int_{N'(\Qq_v)\backslash {G}'(\Qq_v)}\int_{{G}'(\Qq_v)}\overline{W_{n,v}}(x_v;\xi'_v){W_{n,v}(y_v;\xi'_v)}f_v\left(x_v^{-1}\gamma(1) y_v\right)dx_vdy_v,
	\end{equation}
	
\end{lemma}
\begin{proof}
	We have
	\begin{align*}
		\mathcal{O}_{\gamma(1)}(f,\vphi')=\sum_{n\in\Qq^{\times}}\int_{N'(\mathbb{A})\backslash {G}'(\mathbb{A})}\int_{{G}'(\mathbb{A})}f\left(x^{-1}\gamma(1)y\right)\overline{W_{n}}(x;\vphi'){W_{n}(y;\vphi')}dx\,dy.
	\end{align*}
	Then \eqref{unipint1facto} follows from the factorization of Whittaker functions.
\end{proof}
\begin{remark} As we will see below $n$ is in fact a non-zero integer.
	
\end{remark}

\subsection{Computation of $\mathcal{O}_{\infty}(f;n)$}\label{8.2}

In this section, we  compute the local orbital integral $\mathcal{O}_{\infty}(f;n)$. For this we compute explicitly the archimedean Whittaker functions $W_{n,\infty}(g_{\infty})$ and  $f_{\infty}(x_{\infty}^{-1}\gamma(1) y_{\infty})$. Identifying $g_\infty\in G'(\mathbb{R})$ with $g_\infty.\textbf{1}_f\in G'(\Aa)$ we have
\begin{equation}\label{76}
	W_{n}(g_\infty;\vphi')=W_{n,\infty}(g_\infty)\prod_pW_{n,p}(\textbf{1}_p).	
\end{equation} 
\subsubsection{Whittaker Functions}
Let $g_{\infty}\in G'(\mathbb{R}).$ Express $g_{\infty}$ in its Iwasawa form:
\begin{align*}
	g_{\infty}=n_\infty a_\infty k_\infty=&\left(
	\begin{array}{ccc}
		1&&-\Delta t\\
		&1&\\
		&&1
	\end{array}
	\right)\left(
	\begin{array}{ccc}
		a&&\\
		&1&\\
		&&\overline{a}^{-1}
	\end{array}
	\right)\left(
	\begin{array}{ccc}
		\cos\alpha&&\Delta\sin\alpha\\
		&1&\\
		-\Delta^{-1}\sin\alpha&&\cos\alpha
	\end{array}
	\right),
\end{align*}
where $\alpha\in[-\pi,\pi);$ $a\in\mathbb{C}^{\times}$ and $t\in\mathbb{R}$ and $$\Delta=i\sqrt{|D_E|}\in i\Rr_{>0}.$$ 

We also write
$$a_\infty=z_\infty.a_\infty^1=\left(
\begin{array}{ccc}
	(a/\overline{a})^{1/2}&&\\
	&1&\\
	&&(a/\overline{a})^{1/2}
\end{array}
\right)\left(
\begin{array}{ccc}
	(a\overline{a})^{1/2}&&\\
	&1&\\
	&&(a\overline{a})^{-1/2}
\end{array}
\right)$$
where $z_\infty$ is in the center $Z_{G'}(\Rr)$ and $a_\infty^1$ has determinant $1$.

\begin{lemma}\label{archwhittaker}
	Let the notation be as before. Then
	\begin{equation}\label{71}
		W_{n,\infty}(g_{\infty})\cdot \prod_{p<\infty}W_{n,p}(\textbf{1}_p)=e^{-ik\alpha}e^{-2\pi na\overline{a}+2\pi n i t}\cdot\left({a}{\overline{a}}\right)^{k/2} a_n.
	\end{equation}
	where $ a_n$ denotes the $n$-th Fourier coefficient of the classical form $\phi'$ (cf. \eqref{andef}). 
\end{lemma}
\begin{proof}
	By \eqref{47}, we have (since $\vphi'$ is invariant by  $Z_{G'}(\Rr)$)
	\begin{align*}
		W_{n}(g_{\infty};\vphi')=&\psi_{n,\infty}(n_{\infty})\int_{N'(\mathbb{Z})\backslash N'(\mathbb{R})}\vphi'(u_{\infty}a^1_{\infty}k_{\infty})\overline{\psi_{n,\infty}(u_{\infty})}du_{\infty}\\
		=&\psi_{n,\infty}(n_{\infty})\int_{N'(\mathbb{Z})\backslash N'(\mathbb{R})}\frac{\phi'(\iota^{-1}(u_{\infty}a^1_{\infty}k_{\infty}).i)}{j(\iota^{-1}(u_{\infty}a^1_{\infty}k_{\infty}),i)^{k}}\overline{\psi_{n,\infty}(u_{\infty})}du_{\infty}\\	
		=&\frac{\psi_{n,\infty}(n_{\infty})}{j(\iota^{-1}(a^1_{\infty}),i)^{k}j(\iota^{-1}(k_{\infty}),i)^{k}}\int_{N'(\mathbb{Z})\backslash N'(\mathbb{R})}\phi'(\iota^{-1}(u_{\infty}a^1_{\infty}).i)\overline{\psi_{n,\infty}(u_{\infty})}du_{\infty}\\
		=&\frac{\psi_{n,\infty}(n_{\infty})}{j(\iota^{-1}(a^1_{\infty}),i)^{k}j(\iota^{-1}(k_{\infty}),i)^{k}}\int_{0}^1\phi'(a\overline{a}i+t)e^{-2\pi ni t}dt\\
		=&\frac{\psi_{n,\infty}(n_{\infty})}{(a\overline{a})^{-k/2}e^{ik\alpha}}\int_{0}^1\phi'(a\overline{a}i+t)e^{-2\pi ni t}dt.
	\end{align*}
	Then \eqref{71} follows from \eqref{76}, \eqref{andef} and the equality $\psi_{n,\infty}(n_{\infty})=e^{2\pi n i t}.$
\end{proof}

\subsubsection{The archimedean test function}
Let $x_{\infty},$ $y_{\infty}\in G'(\mathbb{R})$ have the Iwasawa forms:
\begin{align*}
	x_{\infty}=&\left(
	\begin{array}{ccc}
		e^{i\alpha_1}&&\\
		&1&\\
		&&e^{i\alpha_1}
	\end{array}
	\right)\left(
	\begin{array}{ccc}
		a_1&&\\
		&1&\\
		&&{a}_1^{-1}
	\end{array}
	\right)\left(
	\begin{array}{ccc}
		\cos\alpha&&\Delta\sin\alpha\\
		&1&\\
		-\Delta^{-1}\sin\alpha&&\cos\alpha
	\end{array}
	\right);\\
	y_{\infty}=&\left(
	\begin{array}{ccc}
		e^{i\alpha_2}&&\\
		&1&\\
		&&e^{i\alpha_2}
	\end{array}
	\right)\left(
	\begin{array}{ccc}
		a_2&&-{a}_2^{-1}\Delta t\\
		&1&\\
		&&{a}_2^{-1}
	\end{array}
	\right)\left(
	\begin{array}{ccc}
		\cos\beta&&\Delta\sin\beta\\
		&1&\\
		-\Delta^{-1}\sin\beta&&\cos\beta
	\end{array}
	\right),
\end{align*}
where $\alpha_1, \alpha_2, \alpha, \beta\in[-\pi,\pi);$ $a_1, a_2\in\mathbb{R}^{\times}$ and $t\in\mathbb{R}.$

Let  $f_\infty$ be as in \eqref{55} (with $\theta_g$ as in  \eqref{thetagdef})	with $k_1=-k,\ k_2=k/2$.

\begin{lemma}\label{archf}
With the notation as above, we have
$$
f_{\infty}(x_{\infty}^{-1}\gamma(1)y_{\infty})=2^{k}{\big[(a_1^{-1}a_2+a_1a_2^{-1})-a_1^{-1}a_2^{-1}i(t-1/(2\Delta))\big]^{-k}}\cdot e^{ik(\alpha-\beta)}.
$$
\end{lemma}
\begin{proof} It follows from the definition that $f_{\infty}(x_{\infty}^{-1}\gamma(1)y_{\infty})$ does not depend on $\alpha_i,\ i=1,2$ so we may assume that $\alpha_1=\alpha_2=0$.
	
	 Computing the matrices one obtains
	\begin{align*}
		x_{\infty}=&\left(
		\begin{array}{ccc}
			a_1\cos\alpha&&a_1\Delta\sin\alpha\\
			&1&\\
			-{a}_1^{-1}\Delta^{-1}\sin\alpha&&{a}_1^{-1}\cos\alpha
		\end{array}
		\right),\\
		y_{\infty}=&\left(
		\begin{array}{ccc}
			a_2\cos\beta+{a}_2^{-1}t\sin\beta&&a_2\Delta\sin\beta-{a}_2^{-1}\Delta t\cos\beta\\
			&1&\\
			-{a}_2^{-1}\Delta^{-1}\sin\beta&&{a}_2^{-1}\cos\beta
		\end{array}
		\right).
	\end{align*}
	
	We set $$a=a_1\cos\alpha,\ b=a_1\Delta\sin\alpha,\ c=-{a}_1^{-1}\Delta^{-1} \sin\alpha,\ d=a^{-1}_1\cos\alpha;$$ and 
	$$a'=a_2\cos\beta+{a}_2^{-1}t\sin\beta,\ b'=a_2\Delta\sin\beta-{a}_2^{-1}\Delta t\cos\beta,$$ $$c'=-{a}_2^{-1}\Delta^{-1}\sin\beta,\ d'={a}_2^{-1}\cos\beta.$$
	With this notation, we have that $x_{\infty}^{-1}\gamma(1)y_{\infty}$ is equal to
	\begin{align*}
		\left(
		\begin{array}{ccc}
			\overline{d}&&\overline{b}\\
			&1&\\
			\overline{c}&&\overline{a}
		\end{array}
		\right)\gamma(1)\left(
		\begin{array}{ccc}
			a'&&b'\\
			&1&\\
			c'&&d'
		\end{array}
		\right)
		=&\left(
		\begin{array}{ccc}
			{d}&&-{b}\\
			&1&\\
			-{c}&&{a}
		\end{array}
		\right)\gamma(1)\left(
		\begin{array}{ccc}
			a'&&b'\\
			&1&\\
			c'&&d'
		\end{array}
		\right)
		\\=&\left(
		\begin{array}{ccc}
			{d}(a'-\frac{c'}{2})-{b}c'&{d}&{d}(b'-\frac{d'}{2})-{b}d'\\
			-c'&1&-d'\\
			-{c}(a'-\frac{c'}{2})+{a}c'&-{c}&-{c}(b'-\frac{d'}{2})+{a}d'
		\end{array}
		\right).
	\end{align*}
	Then by definition \eqref{55} we have  
	\begin{equation}\label{67}
		f_{\infty}(x_{\infty}^{-1}\gamma(1)y_{\infty})=2^{k}\cdot (\overline{A}-\overline{B})^{-k},
	\end{equation}
	where $$A={d}(a'-\frac{c'}{2})-{b}c'-{c}(b'-\frac{d'}{2})+{a}d'$$
	and
	$$B=\big[{d}(b'-\frac{d'}{2})-{b}d'\big]\cdot |D_E|^{-1/2}+\big[-{c}(a'-\frac{c'}{2})+{a}c'\big]\cdot|D_E|^{1/2}.$$
	
	Substituting expressions of $a, b,c ,d$ and $a', b', c', d'$ we then have
	\begin{align*}
		A=(a_1^{-1}a_2+a_1a_2^{-1})\cos(\alpha-\beta)-a_1^{-1}a_2^{-1}t'\sin(\alpha-\beta)
	\end{align*}
	with  $t'=t-1/(2\Delta)$.
	We have 	
	\begin{align*}
		B=-i\Bigl[(a_1^{-1}a_2+a_1a_2^{-1})\sin(\alpha-\beta)+a_1^{-1}a_2^{-1}t'\cos(\alpha-\beta)\Bigr].
	\end{align*}
Therefore,
	$$
	A-B=\big[(a_1^{-1}a_2+a_1a_2^{-1})+a_1^{-1}a_2^{-1}it'\big]\cdot e^{i(\alpha-\beta)}.	
	$$
Since $\overline{\Delta}=-\Delta$, we have
\[
\overline{t'}
=
t-\frac{1}{2\Delta}.
\]
It follows that
\[
\overline{A}-\overline{B}
=
\left[
(a_1^{-1}a_2+a_1a_2^{-1})
-a_1^{-1}a_2^{-1}i
\left(t-\frac{1}{2\Delta}\right)
\right]e^{-i(\alpha-\beta)}.
\]
The lemma now follows from \eqref{67}.
\end{proof}

\subsubsection{Orbital Integrals}\label{subsecOI}
In this subsection, we will combine Lemma \ref{archwhittaker} and Lemma \ref{archf} to compute the archimedean unipotent orbital integral.
We set
\begin{equation}
	\label{Wfnsq}|W_{n,f}(\textbf{1})|^2:=\prod_{p<\infty}|W_{n,p}(\textbf{1}_p)|^2
\end{equation}
\begin{prop}\label{73}
	Let the notation be as before. 
	If $|W_{n,f}(\textbf{1})|^2=0$ we have
	$$\mathcal{O}_{\infty}(f;n)=0.$$
	
	Otherwise we have
	\begin{align}\nonumber
		\frac{\mathcal{O}_{\infty}(f;n)}{|W_{n,f}(\textbf{1})|^2}&=
		\frac{2^3\pi^2}{2^{4(k-1)}}\frac{\Gamma(k-1)^2}{\Gamma(k)}\frac{|a_n|^2}{(4\pi n)^{k-1}}e^{-\pi n |D_E|^{-1/2}}\\
		&=\frac{2^4\pi^5}{3\cdot 2^{4(k-1)}}\frac{1}{(k-1)^2}
		{\prod_{p|N'}(1+\frac1p)}\frac{|\lambda_{\pi'}(n)|^2}{L(1,\pi',\Ad)}e^{-\pi n |D_E|^{-1/2}}.\label{unipinfty}
	\end{align}
	In particular we have
	\begin{equation}
		\label{Oinftybound}
		{\mathcal{O}_{\infty}(f;n)}\leq \frac{(kN'n)^{o(1)}}{2^{4k}k^2}e^{-\pi n |D_E|^{-1/2}}\prod_{p<\infty}|W_{n,p}(\textbf{1}_p)|^2.
	\end{equation}
\end{prop}

\begin{proof} 
	
	By \eqref{75} with $v=\infty$ we have
	$$
	{\mathcal{O}_{\infty}(f;n)}=\int_{N'(\Rr)\backslash {G}'(\Rr)}\int_{{G}'(\Rr)}{\overline{W_{n,\infty}}(x){W_{n,\infty}(y)}}f_\infty\left(x^{-1}\gamma(1) y\right)dx\,dy.
	$$
	
	By Lemma \ref{archwhittaker} and Lemma \ref{archf} we have
	\begin{align*}
		{\mathcal{O}_{\infty}(f;n)}=&2^k|a_n|^2 |W_{n,f}(\textbf{1})|^2\int_{\mathbb{R}}\int_{\mathbb{R}_+}\int_{\mathbb{R}_+} \int_{0}^{2\pi}\int_{0}^{2\pi}e^{-2\pi n(a_1^2+a_2^2)-2\pi n i t}\\
		&\frac{a_1^ka_2^k}{\big[(a_1^{-1}a_2+a_1a_2^{-1})-a_1^{-1}a_2^{-1}i(t-1/(2\Delta))\big]^{k}} \frac{d\alpha d\beta}{4\pi^2} \frac{da_1}{a_1^3}\frac{da_2}{a_2^3}dt.
	\end{align*}
		
	Since $k=|k_1|\geq 8$, the integral $\mathcal{O}_{\infty}(f;n)$ converges absolutely. Moreover if $W_{n,f}(\textbf{1})=0$, the orbital integral vanishes so we may assume $W_{n,f}(\textbf{1})\not=0$.

	After passing to polar coordinates, $a_1^2+a_2^2=r^2(\cos^2\theta+\sin^2\theta)$, we obtain
	$$\frac{\mathcal{O}_{\infty}(f;n)}{|W_{n,f}(\textbf{1})|^2}=2^k|a_n|^2\cdot \int_{\mathbb{R}_+}\int_{\mathbb{R}_+}\int_{\mathbb{R}} \frac{e^{-2\pi n(a_1^2+a_2^2)-2\pi n i t}\cdot (a_1a_2)^{2k}}{\big[a_1^{2}+a_2^2-i(t-1/(2\Delta))\big]^{k}}dt\frac{da_1}{a_1^3}\frac{da_2}{a_2^3}.$$
	We view the $t$-integral as a complex integral along the line $\Im t=0$ and shift the contour to $\Im t=-|D_E|^{-1/2}/2$ (recall that $k\geq 8$ is large). Noting there is no pole between the two lines, we obtain
\begin{align*}
\frac{\mathcal{O}_{\infty}(f;n)}{|W_{n,f}(\textbf{1})|^2}=&2^k|a_n|^2e^{-\pi in \Delta^{-1}}\iint_{\mathbb{R}^2_+}\int_{\mathbb{R}} \frac{e^{-2\pi n(a_1^2+a_2^2)-2\pi n i t}\cdot (a_1a_2)^{2k}}{(a_1^{2}+a_2^2-it)^{k}}dt\,\frac{da_1}{a_1^3}\,\frac{da_2}{a_2^3}\\
=&2^{k-2k+3}|a_n|^2e^{-\pi in \Delta^{-1}}\int_{0}^{\infty}\int_{0}^{\frac{\pi}{2}}\int_{\mathbb{R}} \frac{e^{-2\pi nr^2-2\pi n i t}\cdot (r^2 2\cos\theta\sin\theta)^{2k-3}}{(r^2-it)^{k}}\,dt\,rdr\,d\theta\\
=&\frac{2}{2^{k-1}}|a_n|^2{e^{-\pi in \Delta^{-1}}}\int_{0}^{{\pi}}(\sin\theta)^{2k-3} d\theta\cdot \int_{0}^{\infty}\int_{\mathbb{R}} \frac{e^{-2\pi nr-2\pi n i t}\cdot r^{2k-3}}{(r-it)^{k}}\,dt\,dr.
\end{align*}
after making the changes of variable $2\theta\leftrightarrow \theta$, $r^2\leftrightarrow r$.
	
	We have
	$$\int_{0}^{{\pi}}\sin^{2k-{3}}\theta d\theta=\pi^{1/2}\frac{\Gamma(k-{1})}{\Gamma(k-1/2)}.$$
	Appealing to the Cauchy integral formula we then obtain
\begin{align*}
e^{\pi in \Delta^{-1}}\frac{\mathcal{O}_{\infty}(f;n)}{|W_{n,f}(\textbf{1})|^2}
=&\frac{\pi^{1/2}}{2^{k-2}}|a_n|^2\frac{\Gamma(k-{1})}{\Gamma(k-1/2)} \frac{(-1)^k}{i}\int_{0}^{\infty}\int_{i\mathbb{R}} \frac{e^{-2\pi nr-2\pi n z}\cdot r^{2k-{3}}}{(z-r)^{k}}dzdr\\
=&\frac{2\pi^{1/2}}{2^{k-1}}|a_n|^2\frac{\Gamma(k-{1})}{\Gamma(k-1/2)}\cdot\frac{2\pi (-1)^{k-1}}{(k-1)!} \int_{0}^{\infty}\frac{r^{2k-{3}}}{e^{2\pi nr}} \Bigg[\frac{d^{k-1}e^{-2\pi n z}}{dz^{k-1}}\Bigg]_{z=r}dr\\
=&\frac{2\pi^{1/2}}{2^{k-1}}|a_n|^2\frac{\Gamma(k-{1})}{\Gamma(k-1/2)} \frac{2\pi (-1)^{k-1}}{\Gamma(k)} (-2\pi n)^{k-1}\cdot \int_{0}^{\infty}e^{-4\pi nr}\cdot r^{2k-{2}}\frac{dr}{r}\\
=&\frac{2\pi^{1/2}}{2^{k-1}}|a_n|^2\frac{\Gamma(k-{1})}{\Gamma(k-1/2)} \frac{2\pi}{\Gamma(k)} \frac{(2\pi n)^{k-1}}{(4\pi n)^{2(k-1)}}\cdot \Gamma(2(k-1))\\
=&\frac{2\pi^{1/2}}{2^{k-1}}|a_n|^2\frac{\Gamma(k-{1})}{\Gamma(k-1/2)} \frac{2\pi}{\Gamma(k)} \frac{(2\pi n)^{k-1}}{(4\pi n)^{2(k-1)}}\pi^{1/2}2^{1-2(k-1)}\Gamma(k-1)\Gamma(k-1/2)\\
=&\frac{2^3\pi^2}{2^{4(k-1)}}\frac{\Gamma(k-1)^2}{\Gamma(k)}\frac{|a_n|^2}{(4\pi n)^{k-1}}
\end{align*}
on using the duplication formula
$$
\Gamma(2(k-1))=\pi^{1/2}2^{1-2(k-1)}\Gamma(k-1)\Gamma(k-1/2).
$$
Then formula \eqref{unipinfty} follows from \eqref{coefficienthecke} and the bound \eqref{Oinftybound} results from \eqref{delignesbound} and \eqref{Siegel}.
\end{proof}

\begin{remark} The reason for this normalization by the factor 
$$|W_{n,f}(\textbf{1})|^2:=\prod_{p<\infty}|W_{n,p}(\textbf{1}_p)|^2$$
is that as we will see in the forthcoming lemma, for any $n$, the local orbital integrals $\mcO_p(f;n)$ appearing in Lemma \ref{uniporbfactolem} are equal to $|W_{n,p}(\textbf{1}_p)|^2$ for almost every $p$.	
\end{remark}

\subsection{Computation of $\mathcal{O}_{p}(f;n)$ when $\pi'_p$ is unramified}\label{8.3}
In this section and the next ones, we compute the local orbital integrals over nonarchimedean places. 

Let $p$ be a rational prime we denote by $\nu$ the usual $p$-adic valuation.

In this subsection, we will assume $p\nmid N'.$ Thus $\xi_p'$ is a spherical vector. 

The next four lemmas treat the following cases:
\begin{itemize}
	\item $p$ splits in $E$;
	\item $p$ is inert in $E$ and $\pi_p$ is unramified (i.e., $p\nmid N$);
	\item $p$ is inert in $E$ and $\pi_p$ is ramified (i.e., $p\mid N$);
	\item $p$ is ramified in $E$.
\end{itemize}

For the sequel we denote by $\nu$ the usual valuation on $\Qp$.

\begin{lemma}\label{lemunipunramsplit}
	Let the  notation be as before. Let $p$ be a  prime split in $E$ and coprime to $N$. Under the isomorphism 
	$$G'(\Qq_p)\simeq \GL(2,\Qq_p),$$
	$\pi_p'$ is isomorphic to a principal series representation  $$\pi_p'\simeq \Ind(\chi\otimes\overline{\chi})$$ for some unramified unitary character $\chi$. We have
	\begin{multline}\label{unipunramsplit}{\mathcal{O}_{p}(f;n)}\\={|W_{n,p}(\textbf{1}_p)|^2}\frac{e^{2\pi ni r_p(\frac{1}{2\Delta})}}{{p^{\nu(n)}}}\cdot \sum_{l=0}^{\nu(n)} (l+1)\cdot \Bigg|\frac{\chi(p)^{\nu(n)-l+1}-\overline{\chi}(p)^{\nu(n)-l+1}}{\chi(p)^{\nu(n)+1}-\overline{\chi}(p)^{\nu(n)+1}}\Bigg|^2.
	\end{multline}
	In particular, ${\mathcal{O}_{p}(f;n)}$ equals zero if $\nu(n)<0$ and equals $1$ if $\nu(n)=0$ and for $\nu(n)\geq 1$ we have
	${\mathcal{O}_{p}(f;n)}$ is bounded in absolute value by
	$$\ll {|W_{n,p}(\textbf{1}_p)|^2}\frac{1}{{p^{\nu(n)}}}\nu(n)^4$$
	where the implicit constant is absolute.
\end{lemma}
\begin{proof}
	Given $g_p\in G'(\Qp)$ with the Iwasawa decomposition  $g_p=u_pa_p\kappa_p$.
	Then
	\begin{equation}\label{51}
		W_{n,p}(g_p)=|n|_p^{-1/2}{\psi}_{n,p}(u_p)\overline{\chi}(n)W_{1,p}\left(\begin{pmatrix}
			n&&\\
			&1&\\
			&&1
		\end{pmatrix}a_p\right).
	\end{equation}
	Note that the product of $|n|_v$ over all places $v$  is equal to $1$. 
	
Write $a_p=\diag(p^{j_1},1,p^{j_2})$ with $(j_1,j_2)\in\mathbb{Z}^2$. Since $W_{1,p}$ is spherical, its value at 
	$\begin{pmatrix}
			n&&\\
			&1&\\
			&&1
		\end{pmatrix}a_p$ is zero unless $j_1\geq j_2-\nu(n)$. It follows from \eqref{51} that $W_{n,p}(g_p)=0$ if we are not in this range.
	
	Let $x_p\in N'(\Qp)\backslash G'(\Qp)$ and $y_p\in G'(\Qp).$ We can express them in their Iwasawa coordinates:
	\begin{align*}
		x_p=\begin{pmatrix}
			p^{i_1}&&\\&1&\\
			& &p^{i_2}
		\end{pmatrix}\kappa_1,\quad y_p=\begin{pmatrix}
			1&&\Delta t\\&1&\\
			&&1
		\end{pmatrix}\begin{pmatrix}
			p^{j_1}&&\\&1&\\
			&&p^{j_2}
		\end{pmatrix}\kappa_2,
	\end{align*}
	where $(i_1,i_2)\in\mathbb{Z}^2,$ $(j_1,j_2)\in\mathbb{Z}^2;$ $t\in\Qq_p$ and $\kappa_1, \kappa_2\in G'(\mathbb{Z}_p).$ 
	
	Since $p\nmid N,$ $\iota(\kappa_1), \iota(\kappa_2)\in K_p=\GL(3,\mathbb{Z}_p)$ via the natural embedding \eqref{emd} and in this case we have $$f_p=\textbf{1}_{\GL(3,\mathbb{Z}_p)}.$$ Hence $f_p(x_p^{-1}\gamma(1)y_p)\neq 0$ if and only if
	\begin{align*}
		&\begin{pmatrix}
			p^{-i_1}&&\\
			&1&\\
			&&p^{-i_2}
		\end{pmatrix}\left(
		\begin{array}{ccc}
			1&1&-1/2\\
			&1&-1\\
			&&1
		\end{array}
		\right)\begin{pmatrix}
			p^{j_1}&&p^{j_2}\Delta t\\
			&1&\\
			&&p^{j_2}
		\end{pmatrix}\in \GL(3,\mathbb{Z}_p).
	\end{align*}
	which, by a direct computation, can be further reduced to
	\begin{align*}
		\left(
		\begin{array}{ccc}
			p^{j_1-i_1}&p^{-i_1}&p^{j_2-i_1}(\Delta t-\frac{1}{2})\\
			&1&-p^{j_2}\\
			&&p^{j_2-i_2}
		\end{array}
		\right)\in \GL(3,\mathbb{Z}_p).
	\end{align*}

	Hence $j_1=i_1\leq 0$ and $j_2=i_2\geq 0.$ By the support of Whittaker functions we have $$\overline{W_{n,p}}(x_p){W_{n,p}(y_p)}=0$$ unless $i_1\geq i_2-\nu(n)$ and $j_1\geq j_2-\nu(n).$ Hence $\mathcal{O}_{p}(f;n)=0$ unless $$0\geq i_1\geq i_2-\nu(n)\geq -\nu(n).$$
	In particular we have $\nu(n)\geq 0$.
	
	Since $\vol(G'(\mathbb{Z}_p))=1$, we have
	\begin{align*}
		{\mathcal{O}_{p}(f;n)}=&\int_{N'(\Qq_p)\backslash G'(\Qq_p)}\int_{G'(\Qq_p)}{f_p(x_p^{-1}\gamma(1)y_p)\overline{W_{n,p}}(x_p){W_{n,p}(y_p)}}dx_pdy_p\\
		=&|n|_p^{-1}\sum_{i_1=-\nu(n)}^{0}\sum_{i_2=0}^{i_1+\nu(n)}p^{2i_1-2i_2}I_p(i_1,i_2)\Big|W_{1,p}\begin{pmatrix}
			np^{i_1}&&\\
			&1&\\
			&&p^{i_2}
		\end{pmatrix}\Big|^2,
	\end{align*}
	where $$I_p(i_1,i_2):=\int_{\Qq_p}\textbf{1}_{\mathbb{Z}_p}(p^{i_2-i_1}(\Delta t-1/{2}))\theta_{n,p}(t)dt.$$
	Since $\Delta\in\Zz_p^{\times}$ and $i_2-i_1\leq \nu(n),$ we have $$I_p(i_1,i_2)=e^{2\pi in r_p(\frac{1}{2\Delta})}p^{i_2-i_1}.$$ Hence we have
	\begin{align*}
		{\mathcal{O}_{p}(f;n)}=&\frac{e^{2\pi ni r_p(1/(2\Delta))}}{|n|_p}\cdot \sum_{l=0}^{\nu(n)}(\nu(n)-l+1)p^{l-\nu(n)}\Bigg|W_{1,p}\left(\iota\begin{pmatrix}
			p^{l}&\\
			&1
		\end{pmatrix}\right)\Bigg|^2.
	\end{align*}
	Substituting the Casselman-Shalika formula (cf. \cite{CS80}) in the  computation above, we then obtain
	\begin{align*}
		\frac{\mathcal{O}_{p}(f;n)}{e^{2\pi ni r_p(\frac{1}{2\Delta})}}=&\frac{|W_{1,p}(\textbf{1}_p)|^2}{|n|_p}\cdot \sum_{l=0}^{\nu(n)}(\nu(n)-l+1)p^{l-\nu(n)}\cdot \Bigg|\frac{p^{-\frac{l}{2}}(\chi(p)^{l+1}-\overline{\chi}(p)^{l+1})}{\chi(p)-\overline{\chi}(p)}\Bigg|^2\\
		=&|W_{n,p}(\textbf{1}_p)|^2\cdot \sum_{l=0}^{\nu(n)} (l+1)p^{-\nu(n)}\cdot \Bigg|\frac{\chi(p)^{\nu(n)-l+1}-\overline{\chi}(p)^{\nu(n)-l+1}}{\chi(p)^{\nu(n)+1}-\overline{\chi}(p)^{\nu(n)+1}}\Bigg|^2.
	\end{align*}
	
	Then \eqref{unipunramsplit} follows.
\end{proof}

\begin{lemma}\label{lemunipunraminert}
	Let the notation be as before. Let $p$ be a rational prime   inert in $E$ and such that $p\nmid N'$. 
	For $$f_p=\mathrm{1}_{G(\Zp)},$$ we have
	\begin{equation}\label{unipunraminert}
		{\mathcal{O}_{p}(f;n)}= {|W_{n,p}(\textbf{1}_p)|^2}\sum_{-\nu(n)/2\leq i\leq 0}\Bigg|\frac{\chi(p)^{\nu(n)+2i+1}-\overline{\chi}(p)^{\nu(n)+2i+1}}{\chi(p)^{\nu(n)+1}-\overline{\chi}(p)^{\nu(n)+1}}\Bigg|^2.
	\end{equation}
	In particular, ${\mathcal{O}_{p}(f;n)}$ is equal to $0$ if $\nu(n)<0$ and is equal to $1$ if $\nu(n)=0.$
	
	For $$f_p=\mathrm{1}_{G(\Zp)A_{p^r}G(\Zp)},$$ for $r\geq 1$ (cf. \eqref{Andef} in the Appendix), we have 
	\begin{equation}\label{unipHeckeinert}
		{\mathcal{O}_{p}(f;n)}\ll {|W_{n,p}(\textbf{1}_p)|^2}(r+\nu(n))^4p^{r}
	\end{equation}
	where the implied constant is absolute.
\end{lemma}
\begin{proof}
	Since $p$ is split and does not divide $N'$, we have (see \eqref{K'defsplit})
	 $$K'_p(N')=G'(\mathbb{Z}_p)=U(L'\otimes\Zp),$$
	 where $L'$ is the lattice defined in \eqref{defL'}.
	 
	 Let $g_p\in G'(\Qq_p)$. Let $g_p=u_pa_pk_p$ be the Iwasawa decomposition of $g_p$. Then one has
	\begin{equation}\label{53}
		W_{n,p}(g_p)=W_{n,p}(u_pa_pk_p)={\psi}_{n,p}(u_p)W_{n,p}(a_p).
	\end{equation}
	For $a_p=\diag(p^{i},1,p^{-i}),$ where $i\in\mathbb{Z}$, we have, by Casselman-Shalika formula, that 
	\begin{equation}\label{138}
		W_{n,p}\begin{pmatrix}
			p^{i}&&\\
			&1&\\
			&&p^{-i}
		\end{pmatrix}=\frac{\textbf{1}_{\mathcal{O}_{E_p}}(np^{2i})}{p^i}\cdot\frac{\chi(p)^{\nu(n)+2i+1}-\overline{\chi}(p)^{\nu(n)+2i+1}}{\chi(p)^{\nu(n)+1}-\overline{\chi}(p)^{\nu(n)+1}}\cdot W_{n,p}(\textbf{1}_p).
	\end{equation}

	Let $x_p\in N'(\Qq_p)\backslash{G}'(\Qq_p)$ and $y_p\in {G}'(\Qq_p).$ We can express these in their Iwasawa coordinates:
	\begin{align*}
		x_p=\begin{pmatrix}
			p^{i}&&\\
			&1&\\
			&&p^{-i}
		\end{pmatrix}\kappa_1,\quad y_p=\begin{pmatrix}
			1&&\Delta t\\
			&1&\\
			&&1
		\end{pmatrix}\begin{pmatrix}
			p^{j}&&\\
			&1&\\
			&&p^{-j}
		\end{pmatrix}\kappa_2,
	\end{align*}
	for $i,j\in\mathbb{Z}^2$; $t\in \Qq_p$ and $\kappa_1, \kappa_2\in K'_p(N).$
	
	By definition, $f_p=\textbf{1}_{K_p}.$ Noting that $p\nmid N,$ then $\kappa_1, \kappa_2\in K_p$ via the natural embedding \eqref{emd}. Hence $f_p(x_p^{-1}\gamma(1)y_p)\neq 0$ if and only if
	\begin{align*}
		&\begin{pmatrix}
			p^{-i}&&\\
			&1&\\
			&&p^{i}
		\end{pmatrix}\left(
		\begin{array}{ccc}
			1&1&-1/2\\
			&1&-1\\
			&&1
		\end{array}
		\right)\begin{pmatrix}
			1&&\Delta t\\
			&1&\\
			&&1
		\end{pmatrix}\begin{pmatrix}
			p^{j}&&\\
			&1&\\
			&&p^{-j}
		\end{pmatrix}\in G(\mathbb{Z}_p),
	\end{align*}
	which could be further reduced to
	\begin{align*}
		\left(
		\begin{array}{ccc}
			p^{j-i}&p^{-i}&p^{-j-i}(\Delta t-\frac{1}{2})\\
			&1&-p^{-j}\\
			&&p^{-j+i}
		\end{array}
		\right)\in G(\mathbb{Z}_p).
	\end{align*}
	
	Hence $j=i\leq 0.$ By the support of Whittaker functions we have $$\overline{W_{n,p}}(x_p){W_{n,p}(y_p)}=0$$ unless $$\nu(n)+2i\geq 0,\ \nu(n)+2j\geq 0.$$
	In particular $\nu(n)\geq 0$.
	
	It follows that $\mathcal{O}_{p}(f;n)=0$ unless $$i=j\geq -\nu(n)/2.$$ Therefore, we have
	\begin{align*}
		\mathcal{O}_{p}(f;n)=&\int_{N'(\Qq_p)\backslash{G}'(\Qq_p)}\int_{{G}'(\Qq_p)}f_p(x_p^{-1}\gamma(1)y_p)\overline{W_{n,p}}(x_p){W_{n,p}(y_p)}dx_pdy_p\\
		=&\sum_{-\nu(n)/2\leq i\leq 0}{p^{4i}}\int_{\Qp}\textbf{1}_{\mathcal{O}_{E_p}}(p^{-2i}(\Delta t-1/{2}))\theta_{n,p}(t)dt\cdot \Bigg|W_{n,p}\begin{pmatrix}
			p^{i}&&\\
			&1&\\
			&&p^{-i}
		\end{pmatrix}\Bigg|^2.
	\end{align*}
	The condition $p^{-2i}(\Delta t-1/{2})\in\mcO_{E_p}$ is equivalent to $t\in p^{2i}\Zp$ and since $-2i\leq \nu(n)$, we have $$\int_{\Qp}\textbf{1}_{\mathcal{O}_{E_p}}(p^{-2i}(\Delta t-1/{2}))\theta_{n,p}(t)dt=p^{-2i}.$$
	This in conjunction with formula \eqref{138} implies that
	\begin{align*}
		\frac{\mathcal{O}_{p}(f;n)}{|W_{n,p}(\textbf{1}_p)|^2}= \sum_{-\nu(n)/2\leq i\leq 0}\Bigg|\frac{\chi(p)^{\nu(n)+2i+1}-\overline{\chi}(p)^{\nu(n)+2i+1}}{\chi(p)^{\nu(n)+1}-\overline{\chi}(p)^{\nu(n)+1}}\Bigg|^2,
	\end{align*}
	proving \eqref{unipunraminert}.
	
	Now let's take $f_p=\mathrm{1}_{G(\Zp)A_{p^r}G(\Zp)},$ for $r\geq 1.$ Then  $f_p(x_p^{-1}\gamma(1)y_p)\neq 0$ if and only if
	\begin{align*}
		&\begin{pmatrix}
			p^{-i}&&\\
			&1&\\
			&&p^{i}
		\end{pmatrix}\left(
		\begin{array}{ccc}
			1&1&-1/2\\
			&1&-1\\
			&&1
		\end{array}
		\right)\begin{pmatrix}
			1&&\Delta t\\
			&1&\\
			&&1
		\end{pmatrix}\begin{pmatrix}
			p^{j}&&\\
			&1&\\
			&&p^{-j}
		\end{pmatrix}\in G(\mathbb{Z}_p)A_{p^r}G(\mathbb{Z}_p),
	\end{align*}
	which could be further reduced to $t\in \mathcal{E}(i,j;r).$ Here $\mathcal{E}(i,j;r)$ is defined by 
	\begin{align*}
		\Bigg\{t\in \Qq_p:\ \left(
		\begin{array}{ccc}
			p^{j-i}&p^{-i}&p^{-j-i}(\Delta t-\frac{1}{2})\\
			&1&-p^{-j}\\
			&&p^{-j+i}
		\end{array}
		\right)\in G(\mathbb{Z}_p)	\left(
		\begin{array}{ccc}
			p^{r}&\\
			&1&\\
			&&p^{-r}
		\end{array}
		\right)G(\mathbb{Z}_p)\Bigg\}.
	\end{align*}
	Note that $\mathcal{E}(i,j;r)$ is empty unless $i\leq r,$ $j\leq r,$ and $\Delta t-1/2\in p^{i+j-r}\Zp.$   
	
	Considering the support of Whittaker functions as above, it follows that 
	\begin{align*}
		\mathcal{O}_{p}(f;n)=&\int_{N'(\Qq_p)\backslash{G}'(\Qq_p)}\int_{{G}'(\Qq_p)}f_p(x_p^{-1}\gamma(1)y_p)\overline{W_{n,p}}(x_p){W_{n,p}(y_p)}dx_pdy_p
	\end{align*}
	is equal to 
	\begin{align*}
		\sum_{\substack{-\frac{\nu(n)}{2}\leq i\leq r\\ -\frac{\nu(n)}{2}\leq j\leq r}}p^{2(i+j)}\int_{\Qp}\textbf{1}_{\mathcal{E}(i,j;r)}(t)\theta_{n,p}(t)dt W_{n,p}\begin{pmatrix}
			p^{i}&&\\
			&1&\\
			&&p^{-i}
		\end{pmatrix}\overline{W_{n,p}\begin{pmatrix}
				p^{j}&&\\
				&1&\\
				&&p^{-j}
		\end{pmatrix}}.
	\end{align*}
	Appealing to the triangle inequality, $|\mathcal{O}_{p}(f;n)|$ is 
	\begin{align*}
		\leq& \sum_{\substack{-\frac{\nu(n)}{2}\leq i\leq r\\ -\frac{\nu(n)}{2}\leq j\leq r}}p^{2(i+j)}\int_{\Qp}\textbf{1}_{\mathcal{E}(i,j;r)}(t)dt \Bigg|W_{n,p}\begin{pmatrix}
			p^{i}&&\\
			&1&\\
			&&p^{-i}
		\end{pmatrix}\overline{W_{n,p}\begin{pmatrix}
				p^{j}&&\\
				&1&\\
				&&p^{-j}
		\end{pmatrix}}\Bigg|\\
		\leq &\sum_{\substack{-\frac{\nu(n)}{2}\leq i\leq r\\ -\frac{\nu(n)}{2}\leq j\leq r}} p^{r} \Bigg|\frac{\chi(p)^{\nu(n)+2i+1}-\overline{\chi}(p)^{\nu(n)+2i+1}}{\chi(p)^{\nu(n)+1}-\overline{\chi}(p)^{\nu(n)+1}}\cdot \frac{\chi(p)^{\nu(n)+2j+1}-\overline{\chi}(p)^{\nu(n)+2j+1}}{\chi(p)^{\nu(n)+1}-\overline{\chi}(p)^{\nu(n)+1}}\Bigg|{|W_{n,p}(\textbf{1}_p)|^2}.
	\end{align*}
	Since $\chi$ is unitary  we have
	$$\Bigg| \frac{\chi(p)^{\nu(n)+2i+1}-\overline{\chi}(p)^{\nu(n)+2i+1}}{\chi(p)^{\nu(n)+1}-\overline{\chi}(p)^{\nu(n)+1}}\cdot \frac{\chi(p)^{\nu(n)+2j+1}-\overline{\chi}(p)^{\nu(n)+2j+1}}{\chi(p)^{\nu(n)+1}-\overline{\chi}(p)^{\nu(n)+1}}\Bigg| \leq 2|i|\cdot 2|j|$$ and \eqref{unipHeckeinert} follows.	
\end{proof}

If $p\mid N$ then  $p$ is inert in $E$ and $\pi_p\simeq \St_p$ is the Steinberg representation. From \eqref{78} we have 
$$G(\mathbb{Z}_p)=U(L\otimes_{\Zz}\Zp)\hbox{ and }G'(\Zp)=U(L'\otimes_\Zz\Zp)$$
where $L,L'$ are the lattices defined in \eqref{defL} and \eqref{defL'}. Explicitly,
\begin{align*}
	G(\mathbb{Z}_p)&=\{g\in G(\Qq_p):\ g.L\otimes_{\Qq}\Qq_p=L\otimes_{\Zz}\Zp\}=G(\Qp)\cap \GL(3,\mathcal{O}_{E_p}),\\
	G'(\Zp)&=\{g\in G'(\Qq_p):\ g.L'\otimes_{\Zz}\Zp=L'\otimes_{\Zz}\Zp\}=G'(\Qp)\cap \GL(2,\mathcal{O}_{E_p}).
\end{align*}

\begin{lemma}\label{lemunipNcoprime}
	Let the notation be as before. Suppose that $p\mid N$ (in particular $p$ is inert) and $(p,n)=1$. Then
	\begin{equation}\label{unipNcoprime}
		\mathcal{O}_{p}(f;n)=|W_{n,p}(\textbf{1}_p)|^2\frac{\mu(I_p')^2}{\mu(I_p)}=|W_{n,p}(\textbf{1}_p)|^2\frac{p^2-p+1}{p+1}.
	\end{equation}
\end{lemma}
\begin{proof}
	Let $x_p\in N'(\Qq_p)\backslash {G}'(\Qq_p)$ and $y_p\in {G}'(\Qq_p)$.  By Lemma \ref{K'cosetlemma} and the Iwasawa decomposition, we can write $$x_p=\begin{pmatrix}
		p^{i}&&\\
		&1&\\
		&&p^{-i}
	\end{pmatrix}\mu_1 \kappa_1',\ 
	y_p=\begin{pmatrix}
		1&&t\\
		&1&\\
		&&1
	\end{pmatrix}\begin{pmatrix}
		p^{j}&&\\
		&1&\\
		&&p^{-j}
	\end{pmatrix}\mu_2 \kappa_2',
	$$
	where $\kappa_1', \kappa_2'\in I_p';$ $i, j\in\mathbb{Z};$ and $\mu_1,$ $\mu_2$ are either trivial or of form $\begin{pmatrix}
		\delta&&1\\&1&\\
		1&&
	\end{pmatrix}$ as in Lemma \ref{K'cosetlemma}. By definition, $f_p(x_p^{-1}\gamma(1)y_p)\neq 0$ is equivalent to
	\begin{align*}
		x_p^{-1}\gamma(1)y_p=\mu_1^{-1}\begin{pmatrix}
			p^{-i}&&\\
			&1&\\
			&&p^{i}
		\end{pmatrix}\begin{pmatrix}
			1&&\Delta t\\
			&1&\\
			&&1
		\end{pmatrix}\begin{pmatrix}
			p^{j}&&\\
			&1&\\
			&&p^{-j}
		\end{pmatrix}\mu_2\in I_p,
	\end{align*}
	which could be further simplified to
	\begin{equation}\label{85}
		\mu_1^{-1}\left(
		\begin{array}{ccc}
			p^{j-i}&p^{-i}&p^{-j-i}(\Delta t-\frac{1}{2})\\
			&1&-p^{-j}\\
			&&p^{-j+i}
		\end{array}
		\right)\mu_2\in I_p\subseteq G(\mathbb{Z}_p).
	\end{equation}
	
	Hence $j=i\leq 0.$ However, when $i<0,$ from the properties of Whittaker functions (since $(n,p)=1$) we have $W_{n,p}(x_p)=W_{n,p}(\diag(p^{i},1,p^{-i}))=0.$ Then $\mathcal{O}_{p}(f;n)=0$ is this case. So $i=j=0$ and \eqref{85} becomes
	\begin{equation}\label{86}
		\mu_1^{-1}\left(
		\begin{array}{ccc}
			1&1&\Delta t-\frac{1}{2}\\
			&1&-1\\
			&&1
		\end{array}
		\right)\mu_2\in I_p.
	\end{equation}
	
	If $\mu_1\neq \textbf{1}_p$ and $\mu_2\neq \textbf{1}_p,$ we may write $\mu_1=\begin{pmatrix}
		\delta_1&&1\\&1&\\
		1&&
	\end{pmatrix},$ $\mu_2=\begin{pmatrix}
		\delta_2&&1\\&1&\\
		1&&
	\end{pmatrix},$ with $\delta_j+\overline{\delta}_j=0,$ $j=1,2.$ Since
	\begin{align*}
		\begin{pmatrix}
			\delta_1&&1\\
			&1&\\
			1&&
		\end{pmatrix}^{-1}\left(
		\begin{array}{ccc}
			1&1&\Delta t-\frac{1}{2}\\
			&1&-1\\
			&&1
		\end{array}
		\right)\begin{pmatrix}
			\delta_2&&1\\
			&1&\\
			1&&
		\end{pmatrix}
		=&\left(
		\begin{array}{ccc}
			1&0&1\\
			-1&1&-1\\
			\Delta t-\frac{1}{2}-\delta_1+\delta_2&1&1
		\end{array}
		\right)
	\end{align*}
	does not belong to $I_p,$ \eqref{86} cannot hold. A contradiction! Hence our assumption fails, namely, at least one $\mu_1$ and $\mu_2$ is equal to $\textbf{1}_p.$ 
	
	The following computations
	\begin{align*}
		&\left(
		\begin{array}{ccc}
			1&1&\Delta t-\frac{1}{2}\\
			&1&-1\\
			&&1
		\end{array}
		\right)\begin{pmatrix}
			\delta_2&&1\\
			&1&\\
			1&&
		\end{pmatrix}
		=\left(
		\begin{array}{ccc}
			\Delta t-\frac{1}{2}+\delta_2&1&1\\
			-1&1&\\
			1&&
		\end{array}
		\right)\not\in I_p,\\
		&\begin{pmatrix}
			\delta_1&&1\\
			&1&\\
			1&&
		\end{pmatrix}^{-1}\left(
		\begin{array}{ccc}
			1&1&\Delta t-\frac{1}{2}\\
			&1&-1\\
			&&1
		\end{array}
		\right)
		=\left(
		\begin{array}{ccc}
			&&1\\
			&1&-1\\
			1&1&\Delta t-\frac{1}{2}-\delta_1
		\end{array}
		\right)\not\in I_p,
	\end{align*}
	show that in order for \eqref{86} to hold, one must have $\mu_1=\mu_2=\bfone_p.$ 
	
	Therefore,
	\begin{align*}
		\frac{\mathcal{O}_{p}(f;n)}{|W_{n,p}(\textbf{1}_p)|^2}=&\int_{N'(\Qq_p)\backslash G'(\Qq_p)}\int_{G'(\Qq_p)}\frac{f_p(x_p^{-1}\gamma(1)y_p)\overline{W_{n,p}}(x_p){W_{n,p}(y_p)}}{|W_{n,p}(\textbf{1}_p)|^2}dx_pdy_p\\
		=&\frac{\mu(I_p')^2}{\mu(I_p)} \int_{\Qp}\textbf{1}_{\mathcal{O}_{E,p}}(\Delta t-\frac{1}{2}){\theta}_{n,p}(t)dt=\frac{\mu(I_p')^2}{\mu(I_p)}.
	\end{align*}
	Thus \eqref{unipNcoprime} follows from this computation and Lemmas \ref{K'cosetlemma} and \ref{lemIwahoriCartanU(V)}.
\end{proof}

\begin{lemma}\label{lemunipNdiv}
	Let the notation be as before. Suppose  that $p\mid (n,N)$. Then we have
	\begin{equation}\label{unipNdiv}
		{\mathcal{O}_{p}(f;n)}=|W_{n,p}(\textbf{1}_p)|^2\frac{\mu(I_p')^2}{\mu(I_p)}\cdot \Bigg\{1+2p\sum_{-\nu(n)/2\leq i\leq -1}\Bigg|\frac{\chi(p)^{\nu(n)+2i+1}-\overline{\chi}(p)^{\nu(n)+2i+1}}{\chi(p)^{\nu(n)+1}-\overline{\chi}(p)^{\nu(n)+1}}\Bigg|^2\Bigg\}.
	\end{equation}
	In particular $\mathcal{O}_{p}(f;n)$ equals $0$ if $\nu(n)\leq 0$, equals 
	$${\mathcal{O}_{p}(f;n)}=|W_{n,p}(\textbf{1}_p)|^2\frac{p^2-p+1}{p+1}$$
	and for $\nu(n)\geq 2$, is bounded by
	$${\mathcal{O}_{p}(f;n)}\ll |W_{n,p}(\textbf{1}_p)|^2\frac{p^2-p+1}{p+1}p\nu(n)^3$$
	where the implicit constant is absolute.  
\end{lemma}
\begin{proof}
	By the proof of Lemma \ref{lemunipNcoprime} (we use the same notation here), $f_p(x_p^{-1}\gamma(1)y_p)\neq 0$ is equivalent to \eqref{85}, i.e.,
$$
		\mu_1^{-1}\left(
		\begin{array}{ccc}
			p^{j-i}&p^{-i}&p^{-j-i}(\Delta t-\frac{1}{2})\\
			&1&-p^{-j}\\
			&&p^{-j+i}
		\end{array}
		\right)\mu_2\in I_p\subseteq G(\mathbb{Z}_p).
	$$
	
	If $\mu_1\neq \textbf{1}_p$ and $\mu_2\neq \textbf{1}_p,$ we may write $\mu_1=\begin{pmatrix}
		\delta_1&&1\\
		&1&\\1&&
	\end{pmatrix},$ $\mu_2=\begin{pmatrix}
		\delta_2&&1\\
		&1&\\1&&
	\end{pmatrix},$ with $\delta_j+\overline{\delta}_j=0,$ $j=1,2.$ Then the condition \eqref{85} becomes
	$$
		\left(
		\begin{array}{ccc}
			1&&\\
			p^{-i}&1&\\
			p^{-2i}(\Delta t-\frac{1}{2})+\delta_2-\delta_1&p^{-i}&1
		\end{array}
		\right)\in I_p\subseteq G(\mathbb{Z}_p),
	$$
	which is equivalent to $$i\leq -1,\ p^{-2i}(\Delta t-\frac{1}{2})\in \mathcal{O}_{E_p},\hbox{ and }\delta_1-\delta_2\equiv p^{-2i}(\Delta t-\frac{1}{2})\pmod{N_1}.$$
	
	On the other hand, we have
	\begin{align*}
		&\left(
		\begin{array}{ccc}
			1&p^{-i}&p^{-2i}(\Delta t-\frac{1}{2})\\
			&1&-p^{-i}\\
			&&1
		\end{array}
		\right)\begin{pmatrix}
			\delta_2&&1\\
			&1&\\
			1&&
		\end{pmatrix}
		=\left(
		\begin{array}{ccc}
			\frac{\Delta t-{1}/{2}}{p^{2i}}+\delta_2&p^{-i}&1\\
			-p^{-i}&1&\\
			1&&
		\end{array}
		\right)\not\in I_p,\\
		&\begin{pmatrix}
			\delta_1&&1\\
			&1&\\
			1&&
		\end{pmatrix}^{-1}\left(
		\begin{array}{ccc}
			1&p^{-i}&p^{-2i}(\Delta t-\frac{1}{2})\\
			&1&-p^{-i}\\
			&&1
		\end{array}
		\right)
		=\left(
		\begin{array}{ccc}
			&&1\\
			&1&-p^{-i}\\
			1&p^{-i}&\frac{\Delta t-{1}/{2}}{p^{2i}}-\delta_1
		\end{array}
		\right)\not\in I_p,
	\end{align*}
	therefore to make \eqref{86} hold, one must have either $\mu_1=\mu_2=\textbf{1}_p$ or $$\mu_1=\begin{pmatrix}
		\delta_1&&1\\
		&1&\\
		1&&
	\end{pmatrix},\ \mu_2=\begin{pmatrix}
		\delta_2&&1\\
		&1&\\
		1&&
	\end{pmatrix},$$ for some $\delta_j\in \mathcal{O}_{E_p}/p\mathcal{O}_{E_p},$ $\delta_j+\overline{\delta}_j=0,$ $j=1,2.$ 
	
	By \eqref{138} we obtain
	\begin{align*}
		\frac{\mathcal{O}_{p}(f;n)}{\mu(I_p')^2/\mu(I_p)}=&\iint f_p(x_p^{-1}\gamma(1)y_p)\overline{W_{n,p}}(x_p){W_{n,p}(y_p)}dx_pdy_p\\
		=&\sum_{\substack{i\geq -\nu(n)/2\\
				i\leq 0}}\int_{\Qp}{p^{4i}\textbf{1}_{\mathcal{O}_{E_p}}(p^{-2i}(\Delta t-1/{2}))\theta_{n,p}(t)}dt \Bigg|W_{n,p}\begin{pmatrix}
			p^{i}&&\\
			&1&\\
			&&p^{-i}
		\end{pmatrix}\Bigg|^2\\
		&\quad+\sum_{-\nu(n)/2\leq i\leq -1}\int_{\Qp}{p^{4i}}\cdot \Bigg|W_{n,p}\begin{pmatrix}
			p^{i}&&\\
			&1&\\
			&&p^{-i}
		\end{pmatrix}\Bigg|^2\\
		&\qquad\sum_{\substack{\delta_1, \delta_2\in \mathcal{O}_{E_p}/N\mathcal{O}_{E_p}\\ \delta_1+\overline{\delta}_1=0,\ \delta_2+\overline{\delta}_2=0\\\delta_1-\delta_2\equiv p^{-2i}(\Delta t-\frac{1}{2})\pmod{N}}}{\textbf{1}_{\mathcal{O}_{E_p}}(p^{-2i}(\Delta t-1/{2}))\theta_{n,p}(t)}dt\\
		=&|W_{n,p}(\textbf{1}_p)|^2\Bigg\{\sum_{-\nu(n)/2\leq i\leq 0}\Bigg|\frac{\chi(p)^{\nu(n)+2i+1}-\overline{\chi}(p)^{\nu(n)+2i+1}}{\chi(p)^{\nu(n)+1}-\overline{\chi}(p)^{\nu(n)+1}}\Bigg|^2\\
		&\qquad +\sum_{-\nu(n)/2\leq i\leq -1}p \Bigg|\frac{\chi(p)^{\nu(n)+2i+1}-\overline{\chi}(p)^{\nu(n)+2i+1}}{\chi(p)^{\nu(n)+1}-\overline{\chi}(p)^{\nu(n)+1}}\Bigg|^2\Bigg\}.
	\end{align*}
	
	Lemma \ref{lemunipNdiv} follows from this identity and Lemmas \ref{K'cosetlemma} and \ref{lemIwahoriCartanU(V)}.
\end{proof}

\begin{lemma}\label{lemunipramified}
	Let the notation be as before. Let $p$ be a  prime  ramified in $E.$ We have
	\begin{equation}\label{unipramified}
		{\mathcal{O}_{p}(f;n)}={|W_{n,p}(\textbf{1}_p)|^2}\sum_{-\nu(n)+1\leq i\leq -2\nu(D_E)}\Bigg|\frac{\chi(p)^{\nu(n)+2i+1}-\overline{\chi}(p)^{\nu(n)+2i+1}}{\chi(p)^{\nu(n)+1}-\overline{\chi}(p)^{\nu(n)+1}}\Bigg|^2.
	\end{equation}
	In particular, $\mathcal{O}_{p}(f;n)=0$ if $\nu(n)\leq 2\nu(D_E)$, and, for $\nu(n)\geq 2\nu(D_E)+1$, it is bounded by 
	$${\mathcal{O}_{p}(f;n)}\ll {|W_{n,p}(\textbf{1}_p)|^2}\nu(n)^3,$$
	where the implicit constant is absolute.
\end{lemma}
\begin{proof}
	Let $p=\mathfrak{p}^2$ be ramified in $E$ and $\nu_{\mathfrak p}$ be the corresponding valuation ($\nu_\mfp=2\nu$  on $\Qp$). Then $p\mid D_E.$ Let $\varpi$ be a uniformizer in $\mathfrak{p}.$  Writing $g_p=u_pa_pk_p$ for the Iwasawa decomposition of $g_p\in G'(\Qq_p)$, we have again
	\begin{equation}\label{83}
		W_{n,p}(g_p)=W_{n,p}(u_pa_pk_p)={\psi}_{n,p}(u_p)W_{n,p}(a_p).
	\end{equation}
	Let $a_p=\diag(\varpi^{j},1,\varpi^{-j}),$ where $j\in\mathbb{Z}.$ From the support properties of the Whittaker function, we have $W_{n,p}(g_p)=0$ unless $2j+\nu(n)\geq 0.$
	
	Let $x_p\in N'(\Qq_p)\backslash {G}'(\Qq_p)$ and $y_p\in {G}'(\Qq_p).$ We can express them in their Iwasawa coordinates:
	\begin{align*}
		x_p=\begin{pmatrix}
			\varpi^{i}&&\\
			&1&\\
			&&\overline{\varpi}^{-i}
		\end{pmatrix}\kappa_1,\quad y_p=\begin{pmatrix}
			1&&\Delta t\\
			&1&\\
			&&1
		\end{pmatrix}\begin{pmatrix}
			\varpi^{j}&&\\
			&1&\\
			&&\overline{\varpi}^{-j}
		\end{pmatrix}\kappa_2,
	\end{align*}
	where $i,j\in\mathbb{Z}^2$, $t\in \Qp$ and $\kappa_1, \kappa_2\in K'_p(N).$
	
	By definition, $f_p=\textbf{1}_{K_p}.$ Since $p\nmid N,$ we have $\kappa_1, \kappa_2\in K_p$ and  $f_p(x_p^{-1}\gamma(1)y_p)\neq 0$ if and only if
	\begin{align*}
		&\begin{pmatrix}
			\varpi^{-i}&&\\
			&1&\\
			&&\overline{\varpi}^{i}
		\end{pmatrix}\left(
		\begin{array}{ccc}
			1&1&-1/2\\
			&1&-1\\
			&&1
		\end{array}
		\right)\begin{pmatrix}
			1&&\Delta t\\
			&1&\\
			&&1
		\end{pmatrix}\begin{pmatrix}
			\varpi^{j}&&\\
			&1&\\
			&&\overline{\varpi}^{-j}
		\end{pmatrix}\in G(\mathbb{Z}_p),
	\end{align*}
	which could be further reduced to
	\begin{align*}
		\left(
		\begin{array}{ccc}
			\varpi^{j-i}&\varpi^{-i}&\varpi^{-j}\overline{\varpi}^{-i}(\Delta t-\frac{1}{2})\\
			&1&-\overline{\varpi}^{-j}\\
			&&\overline{\varpi}^{-j+i}
		\end{array}
		\right)\in G(\mathbb{Z}_p).
	\end{align*}
	
	Hence $j=i\leq -2\nu(D_E).$ By the support of Whittaker functions we have $\overline{W_{n,p}}(x_p){W_{n,p}(y_p)}=0$ unless $2i+\nu_{\mathfrak{p}}(n)\geq 0.$   Hence $\mathcal{O}_{p}(f;n)=0$ unless $$-\nu(n)\leq i=j\leq -2\nu(D_E).$$ Therefore, we have
	\begin{align*}
		\mathcal{O}_{p}(f;n)
		=&\int_{N'(\Qq_p)\backslash{G}'(\Qq_p)}\int_{{G}'(\Qq_p)}f_p(x_p^{-1}\gamma(1)y_p)\overline{W_{n,p}}(x_p){W_{n,p}(y_p)}dx_pdy_p\\
		=&\sum_{\substack{i\geq -\nu(n)\\ i\leq -2\nu(D_E)}}{p^{2i}}\int_{E_{\mathfrak{p}}}\textbf{1}_{\mathcal{O}_{E_p}}(p^{-i}(\Delta t-1/{2}))\theta_{n,p}(t)dt\cdot \Bigg|W_{n,p}\begin{pmatrix}
			\varpi^{i}&&\\&1&\\
			&&\overline{\varpi}^{-i}
		\end{pmatrix}\Bigg|^2\\
		=&|W_{n,p}(\textbf{1}_p)|^2\cdot \sum_{-\nu(n)+\nu(D_E)/2\leq i\leq -2\nu(D_E)}\Bigg|\frac{\chi(p)^{\nu(n)+2i+1}-\overline{\chi}(p)^{\nu(n)+2i+1}}{\chi(p)^{\nu(n)+1}-\overline{\chi}(p)^{\nu(n)+1}}\Bigg|^2.
	\end{align*}
	
	We then conclude \eqref{unipramified}.
\end{proof}

\subsection{Computation of $\mathcal{O}_{p}(f;n)$ when  $\pi'_p$ is ramified}\label{8.3'}
In this subsection we deal with the case  $p=N'$. In particular $p$ is split and $\pi_p'\simeq \St'_p$ is the Steinberg representation. We continue to denote by $\nu$ the usual valuation on $\Qp$.

\begin{lemma}\label{lemunipN'div}
	Suppose that $p=N'$. 
	
	For $\nu(n)<0$, we have $$\mathcal{O}_{p}(f^{\mathfrak{n}_p};n)=0.$$ 
	
	For $\nu(n)=0$ we have
	\begin{equation}\label{unipN'0}
		{\mathcal{O}_{p}(f^{\mathfrak{n}_p};n)}={|W_{n,p}(\textbf{1}_p)|^2}\frac{(p-2)\mu(I_p'(1))^2}{\mu(K_p)}={|W_{n,p}(\textbf{1}_p)|^2}\frac{p-2}{(p+1)^2}.
	\end{equation}
	
	For $\nu(n)\geq 1$ we have
	\begin{equation}\label{unipN'>0}
		{\mathcal{O}_{p}(f^{\mathfrak{n}_p};n)}\ll {|W_{n,p}(\textbf{1}_p)|^2}\frac{\nu(n)^2}{p},
	\end{equation}
	where the implied constant is absolute. We recall that $I'_p(1)$ is defined in \eqref{defI'p1} (when we identify it with a subgroup of $\GL_2(\Zp)$).
\end{lemma}

\begin{proof}
	By definition of the test function $f_p$ (see \eqref{fpchoice}),  the orbital integral
	\begin{align*}
		\mathcal{O}_{p}(f;n)=&\int_{N'(\Qq_p)\backslash G'(\Qq_p)}\int_{G'(\Qq_p)}f_p(\mathfrak{n}_p^{-1}x_p^{-1}\gamma(1)y_p\mathfrak{n}_p)\overline{W_{n,p}}(x_p){W_{n,p}(y_p)}dx_pdy_p,
	\end{align*}	
	is zero unless
	\begin{equation}\label{200}
		\mathfrak{n}_p^{-1}\left(
		\begin{array}{ccc}
			a&b&\\
			c&d&\\
			&&1
		\end{array}
		\right)^{-1}\left(
		\begin{array}{ccc}
			1&-\frac{1}{2}&1\\
			&1&\\
			&-1&1
		\end{array}
		\right)\left(
		\begin{array}{ccc}
			a'&b'&\\
			c'&d'&\\
			&&1
		\end{array}
		\right)\mathfrak{n}_p\in K_p.
	\end{equation}
	
	By the Iwasawa decomposition for $G(\Qp)$ and the Iwahori decomposition for $K'_p$, we write
	\begin{equation}\label{201}
		\left(\begin{array}{ccc}
			a'&b'&\\
			c'&d'&\\
			&&1
		\end{array}\right)\in \left(\begin{array}{ccc}
			p^{j'}&&\\
			&p^{k'}&\\
			&&1
		\end{array}\right)\left(\begin{array}{ccc}
			1&b'&\\
			&1&\\
			&&1
		\end{array}\right)\left(\begin{array}{ccc}
			\delta&&\\
			&1&\\
			&&1
		\end{array}\right)I_p'(1)
	\end{equation}
	or
	\begin{equation}\label{202}
		\left(\begin{array}{ccc}
			a'&b'&\\
			c'&d'&\\
			&&1
		\end{array}\right)\in \left(\begin{array}{ccc}
			p^{j'}&b'p^{k'}&\\
			&p^{k'}&\\
			&&1
		\end{array}\right)\begin{pmatrix}
			\mu_2&1&\\
			1&&\\
			&&1
		\end{pmatrix}\left(\begin{array}{ccc}
			\delta&&\\
			&1&\\
			&&1
		\end{array}\right)I_p'(1)
	\end{equation}
	with $j,j',k,k'\in\Zz$ (in this section $k$ denotes a general integral parameter and has nothing to do with the archimedean weight of $\pi'$).
	
	\begin{enumerate}
		\item[1.] Suppose $x_p$ and $y_p$ are both of the form in \eqref{201}, namely, suppose
		\begin{align*}
			x_p=&\begin{pmatrix}
				p^{j}&&\\
				&p^{k}&\\
				&&1
			\end{pmatrix}\begin{pmatrix}
				\delta&&\\
				&1&\\
				&&1
			\end{pmatrix}\gamma_1,\\
			y_p=&\begin{pmatrix}
				p^{j'}&&\\
				&p^{k'}&\\
				&&1
			\end{pmatrix}\begin{pmatrix}
				1&b'&\\
				&1&\\
				&&1
			\end{pmatrix}\begin{pmatrix}
				\delta&&\\
				&1&\\
				&&1
			\end{pmatrix}\gamma_2,
		\end{align*}
		where $\gamma_1, \gamma_2\in I_p'(1).$ Denote by $\mathcal{O}_{p}^{(1)}(f;n)$ the contribution of $x_p, y_p$ in the above forms.
		Then \eqref{200} is equivalent to
		\begin{align*}
			\mathfrak{n}_p^{-1}\left(\begin{array}{ccc}
				\tau p^j&&\\
				&p^k&\\
				&&1
			\end{array}\right)^{-1}\left(
			\begin{array}{ccc}
				1&-1/2&1\\
				&1&\\
				&-1&1
			\end{array}
			\right)\left(\begin{array}{ccc}
				p^{j'}&&\\
				&p^{k'}&\\
				&&1
			\end{array}\right)\left(\begin{array}{ccc}
				\delta&b'&\\
				&1&\\
				&&1
			\end{array}\right)\mathfrak{n}_p\in K_p
		\end{align*}
		A direct calculation shows that

		\begin{align*}
			\left(
			\begin{array}{ccc}
				p^{j'-j}&p^{j'-j}b'-2^{-1}p^{k'-j}+\tau p^{k-1}&\delta p^{j'-j-1}+p^{-j}-\tau p^{-1}\\
				&p^{k'-k}&\\
				&-p^{k'}&1
			\end{array}
			\right)\in K_p.
		\end{align*}

		Similarly, taking the inverse, we have
		\begin{align*}
			\left(\begin{array}{ccc}
				p^{-j'}&-b'p^{-k'}&-\delta p^{-1}\\
				&p^{-k'}&\\
				&&1
			\end{array}\right)\left(
			\begin{array}{ccc}
				1&-1/2&-1\\
				&1&\\
				&1&1
			\end{array}
			\right)\left(\begin{array}{ccc}
				p^{j}&&\tau p^{j-1}\\
				&p^{k}&\\
				&&1
			\end{array}\right)\in K_p.
		\end{align*}
		Expanding the matrices we then obtain

		\begin{align*}
			\left(
			\begin{array}{ccc}
				p^{j-j'}&-2^{-1}p^{k-j'}-b'p^{k-k'}-\delta p^{k-1}&\tau p^{k-j'-1}-p^{-j'}-\delta p^{-1}\\
				0&p^{k-k'}&0\\
				0&p^{k}&1
			\end{array}
			\right)\in K_p.
		\end{align*}
		So $j-j'=k-k'$ and $k=k'\geq 0$, therefore $j=j'.$ Hence the above constraints reduces to the following:
		\begin{align*}
			\left(
			\begin{array}{ccc}
				1&b'-2^{-1}p^{k-j}&(\delta-\tau)p^{-1}+p^{-j}\\
				0&1&0\\
				0&-p^{k}&1
			\end{array}
			\right)\in K_p.
		\end{align*}
		\begin{align*}
			\left(
			\begin{array}{ccc}
				1&-2^{-1}p^{k-j}-b'-\delta p^{k-1}&\tau p^{k-j-1}-p^{-j}-\delta p^{-1}\\
				0&1&0\\
				0&p^{k}&1
			\end{array}
			\right)\in K_p.
		\end{align*}
		
		Since $\tau p^{k-j-1}-p^{-j}-\delta^{-1} p^{-1}\in\mathbb{Z}_p$ and $k\geq 0,$ then $0\leq j\leq 1.$ Also, $p^{k-j}+\delta p^{k-1}\in\mathbb{Z}_p.$ So $j\neq 0,$ forcing that $j=1.$ From $$\tau p^{k-j-1}-p^{-j}-\delta^{-1} p^{-1}\in\mathbb{Z}_p$$ we then see that $k\geq 1.$ 
		
		Also, by the support properties of Whittaker functions, we have $$\nu(n)+1\geq k\geq 1.$$
		If $\nu(n)=0$ we have thus $k=1$ and
		\begin{align*}
			\mathcal{O}_{p}^{(1)}(f;n)=&\frac{\mu(I_p'(1))^2}{\mu(K_p)}\sum_{\substack{\delta, \tau\\ (\delta-\tau)p^{-1}+p^{-1}\in \mathbb{Z}_p
			}}\bigg|W_{n,p}\left(\textbf{1}_p\right)\bigg|^2\\
			=&\frac{(p-2)\mu(I_p'(1))^2}{\mu(K_p)}|W_{n,p}(\textbf{1}_p)|^2.
		\end{align*}	
		Assume now that $\nu(n)\geq 1$. If $k\geq 2,$ then there is only one choice for the pair $(\delta,\tau)$ and 
		\begin{align*}
			\mathcal{O}_{p}^{(1)}(f;n)=&\frac{\mu(I_p'(1))^2}{\mu(K_p)}\sum_{1\leq k\leq \nu(n)}p^{1-k}\bigg|W_{n,p}\begin{pmatrix}
				p&&\\&1&\\&
				&p^{k}
			\end{pmatrix}\bigg|^2\ll \frac{\mu(I_p'(1))^2\nu(n)}{\mu(K_p)}|W_{n,p}(\textbf{1}_p)|^2.
		\end{align*}

		\item[2.] Suppose $x_p$ and $y_p$ are both of the form in \eqref{202}, namely, suppose
		\begin{align*}
			x_p=&\begin{pmatrix}
				p^{j}&&\\
				&p^{k}&\\
				&&1
			\end{pmatrix}\begin{pmatrix}
				\mu_1&1&\\
				1&&\\
				&&1
			\end{pmatrix}\begin{pmatrix}
				\tau^{-1}&&\\
				&1&\\
				&&1
			\end{pmatrix}\gamma_1,\\
			y_p=&\begin{pmatrix}
				p^{j'}&&\\
				&p^{k'}&\\
				&&1
			\end{pmatrix}\begin{pmatrix}
				1&b'&\\
				&1&\\
				&&1
			\end{pmatrix}\begin{pmatrix}
				\mu_2&1&\\
				1&&\\
				&&1
			\end{pmatrix}\begin{pmatrix}
				\delta^{-1}&&\\
				&1&\\
				&&1
			\end{pmatrix}\gamma_2,
		\end{align*}
		where $\gamma_1, \gamma_2\in I_p'(1).$
		Then \eqref{200} is equivalent to
		\begin{align*}
			\begin{pmatrix}
				p^{-j}&&\\
				&p^{-k}&-\tau p^{-1}\\
				&&1
			\end{pmatrix}\begin{pmatrix}
				1&-1/2&1\\
				&1&\\
				&-1&1
			\end{pmatrix}\begin{pmatrix}
				p^{j'}&bp^{j'}&\delta^{-1}bp^{j'-1}\\
				&p^{k'}&\delta^{-1}p^{k'-1}\\
				&&1
			\end{pmatrix}\in K_p,
		\end{align*}
		where $b=b'+\mu_2-\mu_1p^{j-k+k'-j'}.$ Expanding the left hand side we  obtain
			
		\begin{align*}
			\begin{pmatrix}
				p^{j'-j}&bp^{j'-j}-2^{-1}p^{k'-j}&z\\
				0&p^{k'-k}+\tau p^{k'-1}&\delta^{-1}p^{k'-1-k}-\tau p^{-1}+\tau\delta^{-1}p^{k'-2}\\
				0&-p^{k'}&1-\delta^{-1}p^{k'-1}
			\end{pmatrix}\in K_p,
		\end{align*}
		where $z=\delta^{-1}bp^{j'-j-1}-2^{-1}\delta^{-1}p^{k'-1-j}+p^{-j}.$
		
		Taking the inverse of the above matrix we then obtain
		\begin{align*}
			\begin{pmatrix}
				p^{j-j'}&-2^{-1}p^{k-j'}-bp^{k-k'}&z'\\
				0&p^{k-k'}-\delta p^{k-1}&\tau^{-1}p^{k-k'-1}-\delta p^{-1}-\delta\tau^{-1}p^{k-2}\\
				0&p^k&1+\tau^{-1}p^{k-1}
			\end{pmatrix}\in K_p,
		\end{align*}
		where $$z'=-2^{-1}\tau^{-1}p^{k-j'-1}-p^{-j'}-b\tau^{-1}p^{k-k'-1}.$$
		In particular we have
		$$k\geq 1,\ k'\geq 1,\ j-j'=k-k'.$$ 
		
		If $k>k'$ then $k\geq 2,$ which contradicts the condition $$\tau^{-1}p^{k-k'-1}-\delta p^{-1}-\delta\tau^{-1}p^{k-2}\in\mathbb{Z}_p.$$Hence $k\leq k'.$ Likewise, $k'\leq k.$ So we must have $k=k'$ and therefore $j=j'$. Eventually, the above constraints reduces to
		\begin{align*}
			\begin{pmatrix}
				1&b-2^{-1}p^{k-j}+\mu_1p^{k-1}&z_1\\
				0&1+\tau p^{k-1}&\delta^{-1}p^{-1}-\tau p^{-1}+\tau\delta^{-1}p^{k-2}\\
				0&-p^{k}&1-\delta^{-1}p^{k-1}
			\end{pmatrix}\in K_p,
		\end{align*}
		where $z_1=\delta^{-1}bp^{-1}-2^{-1}\delta^{-1}p^{-1-j}+p^{-j}.$ From 
		$$
		bp^{j'-j}-2^{-1}p^{k'-j}\in\mathbb{Z}_p,\ \ -2^{-1}p^{k-j'}-bp^{k-k'}\in\mathbb{Z}_p
		$$
		one has $k\geq j$ and $b\in\mathbb{Z}_p.$ On the other hand, from the support of Whittaker functions we have necessarily that $\nu(n)+j\geq k\geq 1.$
		
		We have the following cases
		\begin{itemize}
			\item[(a)] Suppose $\nu(n)=0.$ Then $k=j\geq 1.$ Since $b\in\mathbb{Z}_p,$ then from $$z'=-2^{-1}\tau^{-1}p^{k-j'-1}-p^{-j'}-b\tau^{-1}p^{k-k'-1}\in\mathbb{Z}_p$$ one concludes that $-j\geq -1,$ i.e., $j\leq 1.$ Hence $j=k=1.$ Then it follows from
			\begin{equation}\label{290}
				\delta^{-1}p^{k'-1-k}-\tau p^{-1}+\tau\delta^{-1}p^{k'-2},\ 
				\tau^{-1}p^{k-k'-1}-\delta p^{-1}-\delta\tau^{-1}p^{k-2}\in\mathbb{Z}_p
			\end{equation}
			that $\tau\equiv -\delta\pmod{p}$ and $1+\tau+\tau^2\equiv 0\pmod{p}.$ From $\tau z'+\delta z\in\mathbb{Z}_p$ we have $1+2\tau\equiv 0\pmod{p},$ which in conjunction with $1+\tau+\tau^2\equiv 0\pmod{p}$ implies $p\mid 3.$ However, since we have assumed that $p\geq 5$ we encounter a contradiction if $\nu(n)=0.$ 
			
			\item[(b)] Therefore $\nu(n)\geq 1.$ Note that from \eqref{290} $\delta$ and $\tau$ should satisfy either $\delta\tau\equiv 1\pmod{p}$ or $\tau^2+\tau+1\equiv 0\pmod{p}$ and $\tau+\delta\equiv 0\pmod{p}.$ Hence, in conjunction with 
			\begin{gather*}
				W_{n,p}(x_p)=\theta_{n,p}(p^{j-k}\mu_1)W_{n,p}\begin{pmatrix}
					p^{j-k}&&\\&1&\\&
					&1
				\end{pmatrix}\\
				W_{n,p}(y_p)=\theta_{n,p}(p^{j-k}\mu_2)\theta_{n,p}(p^{j-k}b')W_{n,p}\begin{pmatrix}
					p^{j-k}&&\\&1&\\&
					&1
				\end{pmatrix},
			\end{gather*}
			we conclude that 
			\begin{align*}
				\mathcal{O}_{p}^{(2)}(f;n)\leq&p^3\mu(I_p'(1))^2\sum_{k=1}^{\nu(n)+1}\sum_{j=k-\nu(n)}^1p^{j-k}\bigg|W_{n,p}\begin{pmatrix}
					p^{j-k}&&\\&1&\\&
					&1
				\end{pmatrix}\bigg|^2\ll \frac{\nu(n)^2|W_{n,p}(\textbf{1}_p)|^2}{p\mu(K_p)}.
			\end{align*}	
		\end{itemize}
		
		\item[3.] Suppose $x_p$  is of the form in \eqref{201} and $y_p$ is of the form in \eqref{202}, namely, suppose
		\begin{align*}
			x_p=&\begin{pmatrix}
				p^{j}&&\\
				&p^{k}&\\
				&&1
			\end{pmatrix}\begin{pmatrix}
				\tau&&\\
				&1&\\
				&&1
			\end{pmatrix}\gamma_1,\\
			y_p=&\begin{pmatrix}
				p^{j'}&&\\
				&p^{k'}&\\
				&&1
			\end{pmatrix}\begin{pmatrix}
				1&b'&\\
				&1&\\
				&&1
			\end{pmatrix}\begin{pmatrix}
				\mu_2&1&\\
				1&&\\
				&&1
			\end{pmatrix}\begin{pmatrix}
				\delta&&\\
				&1&\\
				&&1
			\end{pmatrix}\gamma_2,
		\end{align*}
		where $\gamma_1, \gamma_2\in I_p'(1).$ Denote by $\mathcal{O}_{p}^{(3)}(f;n)$ the contribution of $x_p, y_p$ in the above forms.
		Then \eqref{200} is equivalent to
		\begin{align*}
			\left(\begin{array}{ccc}
				\tau^{-1} p^{-j}&&-p^{-1}\\
				&p^{-k}&\\
				&&1
			\end{array}\right)\begin{pmatrix}
				1&-1/2&1\\
				&1&\\
				&-1&1
			\end{pmatrix}\begin{pmatrix}
				p^{j'}&bp^{j'}&\delta^{-1}bp^{j'-1}\\
				&p^{k'}&\delta^{-1}p^{k'-1}\\
				&&1
			\end{pmatrix}\in K_p
		\end{align*}
		where $b=b'+\mu_2.$ Expanding the left hand side, the above constraint becomes
			
		\begin{align*}
			\begin{pmatrix}
				p^{j'-j}&bp^{j'-j}-2^{-1}p^{k'-j}+\tau p^{k'-1}&z_2\\
				0&p^{k'-k}&\delta^{-1}p^{k'-k-1}\\
				0&-p^{k'}&1-\delta^{-1}p^{k'-1}
			\end{pmatrix}\in K_p,
		\end{align*}
		where $$z_2:=\delta^{-1}bp^{j'-j-1}-2^{-1}\delta^{-1}p^{k'-j-1}+p^{-j}-\tau p^{-1}+\tau\delta^{-1}p^{k'-2}.$$
		Considering the inverse as before, we  obtain
	
		\begin{align*}
			\begin{pmatrix}
				p^{-j'}&-bp^{-k'}&\\
				&p^{-k'}&-\delta p^{-1}\\
				&&1
			\end{pmatrix}\begin{pmatrix}
				p^j&-2^{-1}p^k&\tau^{-1}p^{j-1}-1\\
				&p^k&\\
				&p^k&1
			\end{pmatrix}\in K_p,
		\end{align*}
		which amounts to the following condition:
		\begin{align*}
			\begin{pmatrix}
				p^{j-j'}&-2^{-1}p^{k-j'}-bp^{k-k'}&\tau^{-1}p^{j-j'-1}-p^{-j'}\\
				0&p^{k-k'}-\delta p^{k-1}&-\delta p^{-1}\\
				0&p^k&1
			\end{pmatrix}\in K_p.
		\end{align*}
		Then we get a contradiction. So in this case, one has $\mathcal{O}_{p}^{(3)}(f;n)=0.$

		\bigskip
		
		\item[4.] Suppose $x_p$  is of the form in \eqref{202} and $y_p$ is of the form in \eqref{201}, namely, suppose
		\begin{align*}
			x_p=&\begin{pmatrix}
				p^{j}&&\\
				&p^{k}&\\
				&&1
			\end{pmatrix}\begin{pmatrix}
				\mu_1&1&\\
				1&&\\
				&&1
			\end{pmatrix}\begin{pmatrix}
				\tau&&\\
				&1&\\
				&&1
			\end{pmatrix}\gamma_1,\\
			y_p=&\begin{pmatrix}
				p^{j'}&&\\
				&p^{k'}&\\
				&&1
			\end{pmatrix}\begin{pmatrix}
				1&b'&\\
				&1&\\
				&&1
			\end{pmatrix}\begin{pmatrix}
				\delta&&\\
				&1&\\
				&&1
			\end{pmatrix}\gamma_2,
		\end{align*}
		where $\gamma_1, \gamma_2\in I_p'(1).$ Denote by $\mathcal{O}_{p}^{(4)}(f;n)$ the contribution of $x_p, y_p$ in the above forms.
		Then \eqref{200} is equivalent to
		\begin{align*}
			\begin{pmatrix}
				p^{-j}&&-\mu_1 p^{-1}\\
				&p^{-k}&-\tau p^{-1}\\
				&&1
			\end{pmatrix}\left(
			\begin{array}{ccc}
				p^{j'}&-2^{-1}p^{k'}&1\\
				&p^{k'}&\\
				&-p^{k'}&1
			\end{array}
			\right)\left(\begin{array}{ccc}
				1&b'&\delta^{-1}p^{-1}\\
				&1&\\
				&&1
			\end{array}\right)\in K_p.
		\end{align*}
		Expanding the left hand side this becomes
			
		\begin{align*}
			\left(
			\begin{array}{ccc}
				p^{j'-j}&p^{j'-j}b'-2^{-1}p^{k'-j}+\mu_1p^{k'-1}&\delta^{-1}p^{j'-j-1}+p^{-j}-\mu_1p^{-1}\\
				0&p^{k'-k}+\tau p^{k'-1}&-\tau p^{-1}\\
				0&-p^{k'}&1
			\end{array}
			\right)\in K_p.
		\end{align*}
		Taking the inverse we then obtain
		\begin{align*}
			\begin{pmatrix}
				p^{-j'}&-bp^{-k'}&-\delta p^{-1}\\
				&p^{-k'}&\\
				&&1
			\end{pmatrix}\begin{pmatrix}
				p^j&-2^{-1}p^k&-1\\
				&p^k&\\
				&p^k&1
			\end{pmatrix}\begin{pmatrix}
				1&\mu_1&\mu_1\tau^{-1}p^{-1}\\
				&1&\tau^{-1}p^{-1}\\
				&&1
			\end{pmatrix}\in K_p
		\end{align*}
		By a calculation, this is equivalent to
			
		\begin{align*}
			\begin{pmatrix}
				p^{j-j'}&-2^{-1}p^{k-j'}-bp^{k-k'}-\delta p^{k-1}&z_2'\\
				0&p^{k-k'}&\tau^{-1}p^{k-k'-1}\\
				0&p^k&1+\tau^{-1}p^{k-1}
			\end{pmatrix}\in K_p,
		\end{align*}
		where $$z_2'=\mu_1\tau^{-1}p^{j-j'-1}-2^{-1}\tau^{-1}p^{k-j'-1}-p^{-j'}-\tau^{-1}bp^{k-k'-1}-\delta p^{-1}-\tau^{-1}\delta p^{k-2}.$$ So $j=j'$ and $k=k'\geq 1.$ Thus,
		\begin{align*}
			\left(
			\begin{array}{ccc}
				1&b'-2^{-1}p^{k-j}+\mu_1p^{k-1}&\delta^{-1}p^{-1}+p^{-j}-\mu_1p^{-1}\\
				0&1+\tau p^{k-1}&-\tau p^{-1}\\
				0&-p^{k}&1
			\end{array}
			\right)\in K_p.
		\end{align*}
		Then we get a contradiction and conclude that $\mathcal{O}_{p}^{(4)}(f;n)=0.$
	\end{enumerate}
	
	In conclusion,  since $\mu(K_p)=1$, \eqref{unipN'0} and \eqref{unipN'>0} hold, giving Lemma \ref{lemunipN'div}.
\end{proof}

\subsection{An upper bound for the global orbital integral $\mathcal{O}_{\gamma(1)}(f;\vphi')$}
In this section, we combine results from previous sections to deduce an upper bound for $\mathcal{O}_{\gamma(1)}(f,\vphi').$

Let $\nu(f)$ be the set of (inert) primes $p$ (coprime with $NN'D$) such that $$f_p=\mathrm{1}_{G(\Zp)A_{r_p}G(\Zp)}$$ for some $r_p\geq 1.$ 

\begin{prop}\label{propglobalunipsingle}
	Let the notation be as before. We have
	\begin{equation}\label{eqglobalunipsingle}
		|\mathcal{O}_{\gamma(1)}(f,\vphi')|\leq \frac{(\ell N')^{o(1)}}{2^{4k}k^2}\frac{N}{{N'}^3}\prod_{p\in \nu(f)}p^{r_p}.	
	\end{equation}
	where the implicit constants depend  on $E$ and $\pi'$ (via $L(1,\pi',\Ad)$).
\end{prop}
\begin{proof}
	It follows from \eqref{unipint1facto}, \eqref{unipinfty}, and Lemmas \ref{lemunipunramsplit}, \ref{lemunipunraminert},  \ref{lemunipNcoprime}, \ref{lemunipNdiv},  \ref{lemunipramified}, \ref{lemunipN'div} that 	
	\begin{equation}\label{eqglobaluniporbint}
		\mathcal{O}_{\gamma(1)}(f,\vphi')=\eta(\Delta)\frac{C_kC_{N,N'}}{L(1,\pi',\Ad)}  \sum_{n\geq 1}|\lambda_{\pi'}(n)|^2 e^{-\frac{\pi n}{{|D_E|^{1/2}}}}\prod_{p<\infty}I_p(n,\vphi'),
	\end{equation}
	where $\eta(\Delta)$ is a complex number of modulus $1$, 
	$$C_k=\frac{2^4\pi^5}{3\cdot 2^{4(k-1)}}\frac{1}{(k-1)^2},$$ 
	
	\begin{equation}\label{CNN'def}
		C_{N,N'}=\prod_{p\mid N}\frac{p^2-p+1}{p+1}\prod_{p\mid N'}\frac{(p-2)(p+1)}{p(p^2-1)^2} \ll\frac{N}{{N'}^3},\nonumber 
	\end{equation} 
	and	where $I_p(n,\vphi')=1$ if $(n,p\ell)=1$, and in general satisfies,
	\begin{equation}\label{Inpbound}
		\begin{cases}
			I_p(n,\vphi')\leq 
			(n,p)^{1+o(1)},&\ \text{if $p\notin \nu(f)$}\\
			I_p(n,\vphi')\ll (r_p+1)^4(n,p)^{1+o(1)}p^{r_p}&\ \text{if $p\in \nu(f)$}.
		\end{cases}
	\end{equation}

	Hence, \eqref{eqglobaluniporbint}, Deligne's bound \eqref{delignesbound} and \eqref{Siegel} yields
	\begin{equation}\label{eqglobaluniporbintbound}
		\mathcal{O}_{\gamma(1)}(f,\vphi')\ll \frac{(\ell N')^{o(1)}}{2^{4k}k^2}\frac{N}{{N'}^3}\cdot \prod_{p\in \nu(f)}p^{r_p}	.
	\end{equation}
	

	
\end{proof}


\subsection{Computation of the remaining unipotent orbital integrals}\label{sec44}
In this subsection, we prove the claim \eqref{284} and reduce the computation of $\mathcal{O}_{\gamma(x)}(f^{\mathfrak{n}},\vphi')$ $x\in E^1$ to $\mathcal{O}_{\gamma(1)}(f^{\mathfrak{n}},\vphi')$. Our result is the following:

\begin{prop}\label{lem44}
	Let the notation be as before. Let $x\in E^1.$ Then 
	\begin{align*}
		\mathcal{O}_{\gamma(x)}(f^{\mathfrak{n}},\vphi')=\begin{cases}
			\mathcal{O}_{\gamma(1)}(f^{\mathfrak{n}},\vphi'),\ \text{if $x\in\mathcal{O}_E^1;$}\\
			0,\ \text{otherwise}.
		\end{cases}
	\end{align*}
\end{prop}
\begin{proof}
	Let $x\neq 1.$ By Lemma \ref{14'} we have 
	$$\gamma(x)=g_1.x\gamma(1).g_2,\ g_1,g_2\in G'(\Qq)$$ and by $G'(\Qq)$ invariance of $\vphi'$, we have
	\begin{equation}
		\begin{split}
			\mathcal{O}_{\gamma(x)}(f^{\mathfrak{n}},\vphi')&=
			\int_{{H}_{\gamma(x)}(\Qq)\backslash  G'(\mathbb{A})^2}f^{\mathfrak{n}}(u^{-1}g_1.x\gamma(1).g_2 v)\overline{\vphi}'(u){\vphi}'(v)dudv\\&
			=\mathcal{O}_{x\gamma(1)}(f^{\mathfrak{n}},\vphi').\label{285}
		\end{split}
	\end{equation}

	Similar to the proof of Lemma \ref{23.} the orbital integral $\mathcal{O}_{x\gamma(1)}(f^{\mathfrak{n}},\vphi')$ is equal to 
	\begin{equation}
		\int_{N'(\mathbb{A})\backslash G'(\mathbb{A})}\int_{G'(\mathbb{A})}f^{\mathfrak{n}}\left(y_1^{-1}x\gamma(1)y_2\right)\int_{[N']}\overline{\vphi}'(vy_1){\vphi}'(vy_2)dvdy_1dy_2,
	\end{equation}
	which is absolutely convergent  by Proposition \ref{propglobalunipsingle}. 
	
	Note that $f^{\mathfrak{n}}\left(y_1^{-1}x\gamma(1)y_2\right)= 0$ unless $$\widetilde{\mathfrak{n}}^{-1}y_1^{-1}x\gamma(1)y_2\widetilde{\mathfrak{n}}\in 
	\prod_{\substack{p<\infty\\ p\notin\nu(f)}}K_p(N)\prod_{p\in\nu(f)}G(\mathbb{Z}_p)A_{p^{r_p}}G(\mathbb{Z}_p).$$ Also, 
	\begin{align*}
		\left(
		\begin{array}{ccc}
			a&&b\\
			&1&\\
			c&&d
		\end{array}
		\right)
		\left(
		\begin{array}{ccc}
			1&1&-1/2\\
			&1&-1\\
			&&1
		\end{array}
		\right)
		\left(
		\begin{array}{ccc}
			a'&&b'\\
			&1&\\
			c'&&d'
		\end{array}
		\right)=\left(
		\begin{array}{ccc}
			*&&*\\
			&1&\\
			*&&*
		\end{array}
		\right).
	\end{align*}
	
	Hence $f^{\mathfrak{n}}\left(y_1^{-1}x\gamma(1)y_2\right)= 0$ unless $\nu_{v}(x)\geq 0$ (i.e., $x\in \mathcal{O}_{E_v}$) for all places $v$ in $E$ such that $v$ is not above $N'$ and $v\notin\nu(f).$ 
	
	Let $p\in\nu(f).$ Then $f_p^{\mathfrak{n}_p}\left(y_1^{-1}x\gamma(1)y_2\right)= 0$ unless 
	\begin{equation}\label{280}
		\widetilde{\mathfrak{n}}_p^{-1}y_{1,p}^{-1}\gamma(1)y_{2,p}\widetilde{\mathfrak{n}}_p\in x^{-1}G(\mathbb{Z}_p)A_{p^{r_p}}G(\mathbb{Z}_p).
	\end{equation}
	Taking the determinant on both sides of \eqref{280} we then derive that $|x|_p=1,$ implying that $x\in \mathcal{O}_{E_p}^{\times}.$ 
	
	Let $p\mid N'.$ Then $f_p^{\mathfrak{n}_p}\left(y_1^{-1}x\gamma(1)y_2\right)= 0$ unless 
	\begin{equation}\label{286}
		\widetilde{\mathfrak{n}}_p^{-1}y_{1,p}^{-1}\gamma(1)y_{2,p}\widetilde{\mathfrak{n}}_p\in x^{-1}K_p.
	\end{equation}
	
	We then apply the manipulation as in the proof of Lemma \ref{lemunipN'div} to show that, if \eqref{286} holds for some $x\in E^1,$ then necessarily $x\in \mathcal{O}_{E_p}^{\times}.$ 
	
	Since $p|N'$ we may identify $G(\Qq_p)$ (resp. $G'(\Qq_p)$) with $\GL(3,\Qq_p)$ (resp. $\GL(2,\Qq_p)$); suppose for the contrary that $\nu(x)\neq 0.$  In the notation of the proof of Lemma \ref{lemunipN'div}, we have four cases to consider. 
	
	In  Case 1, \eqref{286} implies that 
	\begin{align*}
		\left(
		\begin{array}{ccc}
			p^{j-j'}&-2^{-1}p^{k-j'}-b'p^{k-k'}-\delta p^{k-1}&\tau p^{k-j'-1}-p^{-j'}-\delta p^{-1}\\
			0&p^{k-k'}&0\\
			0&p^{k}&1
		\end{array}
		\right)\in x^{-1}K_p,
	\end{align*}
	which contradicts the assumption that $\nu(x)<0.$ Similarly, taking the inverse, we see that $\nu(x)>0$ cannot happen as well. 
	
	In the Case 3, \eqref{286} implies that 
	\begin{equation}\label{287}
		x\begin{pmatrix}
			p^{j'-j}&bp^{j'-j}-2^{-1}p^{k'-j}+\tau p^{k'-1}&z_2\\
			&p^{k'-k}&\delta^{-1}p^{k'-k-1}\\
			&-p^{k'}&1-\delta^{-1}p^{k'-1}
		\end{pmatrix}\in K_p,
	\end{equation}
	where $$z_2:=\delta^{-1}bp^{j'-j-1}-2^{-1}\delta^{-1}p^{k'-j-1}+p^{-j}-\tau p^{-1}+\tau\delta^{-1}p^{k'-2}.$$
	Considering the inverse as before, we then obtain
	\begin{equation}\label{288}
		x^{-1}\begin{pmatrix}
			p^{j-j'}&-2^{-1}p^{k-j'}-bp^{k-k'}&\tau^{-1}p^{j-j'-1}-p^{-j'}\\
			0&p^{k-k'}-\delta p^{k-1}&-\delta p^{-1}\\
			0	&p^k&1
		\end{pmatrix}\in K_p.
	\end{equation}
	
	By \eqref{288} one has $\nu(x)\leq -1$ and $j-j'=\nu(x).$ Computing the determinant of \eqref{287} we then have 
	\begin{align*}
		j-j'+k-k'=3\nu(x),\ i.e.,\ k-k'=2\nu(x).
	\end{align*}
	Therefore $$j=1, j'=2, k=-1, k'=1,\hbox{ and }\nu(x)=-1.$$ Applying inversion in \eqref{288} we then get $\nu(x)-1\geq 0$ by considering its $(2,3)$-th entry and therefore $\nu(x)\geq 1$, a contradiction! In conclusion, the Case 3 does not contribute to the orbital integral, just as the situation in Lemma \ref{lemunipN'div}. Similarly  the contribution of Case 4 is also zero.
	
	Finally we consider the Case 2; where \eqref{286} implies that 
	\begin{align*}
		\begin{pmatrix}
			p^{j'-j}&bp^{j'-j}-2^{-1}p^{k'-j}+\mu_1p^{k'-1}&z\\
			&p^{k'-k}+\tau p^{k'-1}&\delta^{-1}p^{k'-1-k}-\tau p^{-1}+\tau\delta^{-1}p^{k'-2}\\
			&-p^{k'}&1-\delta^{-1}p^{k'-1}
		\end{pmatrix}\in x^{-1}K_p,
	\end{align*}
	where $$
	z=\delta^{-1}bp^{j'-j-1}-2^{-1}\delta^{-1}p^{-1-j}+p^{-j}-\mu_1 p^{-1}+\mu_1 \delta^{-1}p^{k'-2}.$$ Suppose $\nu(x)\leq -1.$ Then $k'=1=k$ and $j'-j=-\nu(x)$ by analyzing the diagonals. However, taking determinant we then obtain $$j'-j+k'-k=-3\nu(x),$$ a contradiction! Suppose $\nu(x)\geq 1.$ Then taking the inverse of the above matrix we then obtain
	\begin{align*}
		\begin{pmatrix}
			p^{j-j'}&\mu_1p^{j-j'}-2^{-1}p^{k-j'}-bp^{k-k'}&z'\\
			&p^{k-k'}-\delta p^{k-1}&\tau^{-1}p^{k-k'-1}-\delta p^{-1}-\delta\tau^{-1}p^{k-2}\\
			&p^k&1+\tau^{-1}p^{k-1}
		\end{pmatrix}\in xK_p,
	\end{align*}
	where $$z'=\mu_1\tau^{-1}p^{j-j'-1}-2^{-1}\tau^{-1}p^{k-j'-1}-p^{-j'}-b\tau^{-1}p^{k-k'-1}.$$
	 So $k=k'=1$ and $j-j'=1.$ Again, we encounter a contradiction by taking determinant. Hence, $\nu(x)=0,$ i.e., $x\in \mathcal{O}_{E_p}^{\times}.$
	
	In summary, we have shown that $\nu(x)\geq 0$ in all finite places. So $x\in \mathcal{O}_E.$ Since $x\overline{x}=1$ we have $x\in\mathcal{O}_E^1.$
	
	Since the test function $f$ is $Z_G(\mathcal{O}_E^1)$-invariant, then Proposition \ref{lem44} follows.
\end{proof}

As a consequence of Proposition \ref{propglobalunipsingle} and \ref{lem44} we have
\begin{cor}\label{corunipotentorb}
	The sum of the unipotent orbital integrals satisfies the bound
	\begin{align*}
		\sum_{x\in E^1}	\mathcal{O}_{\gamma(x)}(f,\vphi')\ll \frac{(\ell N')^{o(1)}}{2^{4k}k^2}\frac{N}{{N'}^3} \prod_{p\in \nu(f)}p^{r_p}
	\end{align*}
	where the implicit constant depends on $E$ and $\pi'$.
	
\end{cor}



\section{\bf The Regular Orbital Integrals}\label{sec7}
In this section and the next we handle the contribution to \eqref{Jsimple} of the regular orbital integrals, i.e. the $\mcO_{\gamma(x)}(f,\vphi')$ when $x\in E^\times\!-E^1$.

In the case of $\GL(2)\times\GL(1)$ this was handled in \cite[\S 2.6]{RR05}; here the geometric structure of $U(2,1)\times U(1,1)$ is much more complicated. 

In this section, we provide bounds for some local integrals which we combine together in Section \ref{sec8.3} to prove   Theorem \ref{regularglobalbound}.

\subsection{The Stabilizer $H_{\gamma(x)}$}

We start by computing the stabilizer $H_{\gamma(x)}$ associated to $\gamma(x)$ when $x\in E^\times\!-E^1.$ Recall that for any $\Qq$-algebra $R$
\begin{align*}
	H_{{\gamma(x)}}(R):=\big\{(g_1,g_2)\in G'(R)^2:\ g_1^{-1}\gamma(x)g_2=\gamma(x)\big\}.
\end{align*}
To save space we will  represent matrices in $G'$ either as  $2\times 2$ or a $3\times 3$ matrices (in the standard bases $\{e_{-1},e_1\}$ or $\{e_{-1},e_0,e_1\}$), that is
$\hbox{ either }\begin{pmatrix}
	a& b\\
	c&d
\end{pmatrix}\hbox{ or }\begin{pmatrix}
	a& &b\\&1&\\
	c&&d
\end{pmatrix}.$ Of course whenever matrices from $G$ are involved we will use the $3\times 3$ notation.

\begin{lemma}\label{23'}
	Let the notation be as above, for any $\Qq$-algebra $R$, $H_{\gamma(x)}(R)$ is equal to
	\begin{align*}
		\Bigg\{&\begin{pmatrix}
			1+\frac{y(x-1)(\overline{x}+1)}{2}& \frac{y(1+x)(1+\overline{x})}{4}\\
			y(1-x)(1-\overline{x})&1+\frac{y(x+1)(\overline{x}-1)}{2}
		\end{pmatrix}\times \begin{pmatrix}
			1+\frac{y(x-1)(\overline{x}+1)}{2}&y\\
			\frac{y(1+x)(1+\overline{x})(1-x)(1-\overline{x})}{4}&1+\frac{y(x+1)(\overline{x}-1)}{2}
		\end{pmatrix}\\
		&\qquad\in G'(R)\times G'(R):
		y\in E^{\times}(R)\ \text{and}\ y+\overline{y}+y\overline{y}(x\overline{x}-1)=0\Bigg\}.
	\end{align*}
\end{lemma}
\begin{proof}
	Let $$g=\begin{pmatrix}
		a& b\\
		c&d
	\end{pmatrix}, g'=\begin{pmatrix}
		a'& b'\\
		c'&d'
	\end{pmatrix}\in G'(R).$$
	Assume
	\begin{equation}\label{18}
		g\left(
		\begin{array}{ccc}
			\frac{x\overline{x}+3\overline{x}-x+1}{4}&\frac{1+x}{2}&-\frac{1}{2}\\
			\frac{(x+1)(\overline{x}-1)}{2}&x&-1\\
			-\frac{(1-x)(1-\overline{x})}{2}&1-x&1
		\end{array}
		\right)=\left(
		\begin{array}{ccc}
			\frac{x\overline{x}+3\overline{x}-x+1}{4}&\frac{1+x}{2}&-\frac{1}{2}\\
			\frac{(x+1)(\overline{x}-1)}{2}&x&-1\\
			-\frac{(1-x)(1-\overline{x})}{2}&1-x&1
		\end{array}
		\right)g'.
	\end{equation}

	Expanding both sides of equation \eqref{18} we then obtain an equality between the matrix $$\left(
	\begin{array}{ccc}
		a'(\overline{x}+\frac{(1-x)(1-\overline{x})}{4})-\frac{c'}{2}&\frac{1+x}{2}&b'(\overline{x}+\frac{(1-x)(1-\overline{x})}{4})-\frac{d'}{2}\\
		\frac{a'(x+1)(\overline{x}-1)}{2}-c'&x&\frac{b'(x+1)(\overline{x}-1)}{2}-d'\\
		-\frac{a'(1-x)(1-\overline{x})}{2}+c'&1-x&-\frac{b'(1-x)(1-\overline{x})}{2}+d'
	\end{array}
	\right)$$ and the matrix $$\left(
	\begin{array}{ccc}
		a(\overline{x}+\frac{(1-x)(1-\overline{x})}{4})-\frac{b(1-x)(1-\overline{x})}{2}&\frac{a(1+x)}{2})+b(1-x)&-\frac{a}{2}+b\\
		\frac{(x+1)(\overline{x}-1)}{2}&x&-1\\
		c(\overline{x}+\frac{(1-x)(1-\overline{x})}{4})-\frac{d(1-x)(1-\overline{x})}{2}&\frac{c(1+x)}{2}+d(1-x)&-\frac{c}{2}+d
	\end{array}
	\right),$$ which is equivalent to the system of equations
	\begin{equation}\label{stab}
		\begin{cases}
			-\frac{a}{2}+b=b'(\overline{x}+\frac{(1-x)(1-\overline{x})}{4})-\frac{d'}{2}\\
			\frac{a(1+x)}{2})+b(1-x)=\frac{1+x}{2}\\
			a(\overline{x}+\frac{(1-x)(1-\overline{x})}{4})-\frac{b(1-x)(1-\overline{x})}{2}=a'(\overline{x}+\frac{(1-x)(1-\overline{x})}{4})-\frac{c'}{2}\\
			-\frac{c}{2}+d=-\frac{b'(1-x)(1-\overline{x})}{2}+d'\\
			\frac{c(1+x)}{2}+d(1-x)=1-x\\
			c(\overline{x}+\frac{(1-x)(1-\overline{x})}{4})-\frac{d(1-x)(1-\overline{x})}{2}=-\frac{a'(1-x)(1-\overline{x})}{2}+c'\\
			-1=\frac{b'(x+1)(\overline{x}-1)}{2}-d'\\
			\frac{(x+1)(\overline{x}-1)}{2}=\frac{a'(x+1)(\overline{x}-1)}{2}-c'
		\end{cases}
	\end{equation}
	
	Now we need to solve \eqref{stab} to find the group $H_{\gamma(x)}(R).$ Note that
	\begin{align*}
		(-\frac{1}{2},1)\frac{(1-x)(1+\overline{x})}{2}+(\frac{1+x}{2},1-x)\overline{x}=(\frac{x\overline{x}+3\overline{x}-x-1}{4},-\frac{(1-x)(1-\overline{x})}{2}).
	\end{align*}
	Hence we obtain from \eqref{stab} the following equations on $(a',b',c',d'):$
	\begin{equation}\label{19}
		\begin{cases}
			(\frac{b'(x\overline{x}+3\overline{x}-x-1)}{4}-\frac{d'(x-1)(1+\overline{x})}{4}+\frac{(1+x)\overline{x}}{2}=\frac{a'(x\overline{x}+3\overline{x}-x-1)}{4}-\frac{c'}{2}\\
			(-\frac{b'(1-x)(1-\overline{x})}{2}+d')\frac{(x-1)(1+\overline{x})}{2}+\overline{x}(1-x)=-\frac{a'(1-x)(1-\overline{x})}{2}+c'\\
			-1=\frac{b'(x+1)(\overline{x}-1)}{2}-d'\\
			\frac{(x+1)(\overline{x}-1)}{2}=\frac{a'(x+1)(\overline{x}-1)}{2}-c'.
		\end{cases}
	\end{equation}
	A computation shows that the companion matrix is singular. By Gaussian elimination, we can further simplify the system of equations \eqref{19} to get
	\begin{equation}\label{20}
		\begin{cases}
			c'=\frac{b'(x+1)(\overline{x}-1)(x-1)(\overline{x}+1)}{4}\\
			d'=\frac{b'(x+1)(\overline{x}-1)}{2}+1\\
			c'=\frac{(a'-1)(x+1)(\overline{x}-1)}{2}.
		\end{cases}
	\end{equation}
	
	Denote by $y=\frac{(x-1)(\overline{x}+1)}{2}.$ Since $\begin{pmatrix}
		a'
		& b'\\
		c'&d'
	\end{pmatrix}\in G'(R),$ we have
	\begin{align*}
		\begin{pmatrix}
			\overline{b}'y+1& \overline{b}'\\
			\overline{b}'y\overline{y}&\overline{b}'\overline{y}+1
		\end{pmatrix}\begin{pmatrix}
			b'y+1& b'\\
			b'y\overline{y}&b'\overline{y}+1
		\end{pmatrix}=\Id,
	\end{align*}
	which turns out to be equivalent to $b'\overline{b}'(y+\overline{y})+b'+\overline{b}'=0,$ i.e.,
	\begin{equation}\label{14}
		b'\overline{b}'(x\overline{x}-1)+b'+\overline{b}'=0.
	\end{equation}

	Substituting \eqref{20} in \eqref{stab}, we then obtain
	\begin{align*}
		\begin{cases}
			-\frac{a}{2}+b=b'(\overline{x}+\frac{(1-x)(1-\overline{x})}{4})-b'(\frac{(1-x)(1-\overline{x})}{4}-\frac{1-\overline{x}}{2})-\frac{1}{2}=\frac{b'(1+\overline{x})}{2}-\frac{1}{2}\\
			\frac{a(1+x)}{2}+b(1-x)=\frac{1+x}{2}\\
			a(\overline{x}+\frac{(1-x)(1-\overline{x})}{4})-\frac{b(1-x)(1-\overline{x})}{2}=a'(\overline{x}+\frac{(1-x)(1-\overline{x})}{4})-(a'-1)(\frac{(1-x)(1-\overline{x})}{4}-\frac{1-\overline{x}}{2})\\
			-\frac{c}{2}+d=-\frac{b'(1-x)(1-\overline{x})}{2}+d'=b'(\frac{(1-x)(1-\overline{x})}{2}+\overline{x}-1)+1-\frac{b'(1-x)(1-\overline{x})}{2}\\
			\frac{c(1+x)}{2}+d(1-x)=1-x\\
			c(\overline{x}+\frac{(1-x)(1-\overline{x})}{4})-\frac{d(1-x)(1-\overline{x})}{2}=-\frac{a'(1-x)(1-\overline{x})}{2}+(a'-1)(\frac{(1-x)(1-\overline{x})}{2}-1+\overline{x}).
		\end{cases}
	\end{align*}
	
	Using Gaussian elimination method we then see the above system becomes
	\begin{equation}\label{*}
		\begin{cases}
			a=1-b'(1-x)\frac{1+\overline{x}}{2}\\
			b=\frac{b'(1+\overline{x})}{2}-\frac{b'(1-x)}{2}(1-\frac{\overline{x}}{2})=\frac{b'(1+x)(1+\overline{x})}{4}\\
			c=b'(1-x)(1-\overline{x})\\
			d=1-b'(1-\overline{x})+\frac{b'(1-x)(1-\overline{x})}{2}\\
			a(\overline{x}+\frac{(1-x)(1-\overline{x})}{4})-\frac{b(1-x)(1-\overline{x})}{2}=\frac{a'(1+\overline{x})}{2})+(\frac{(1-x)(1-\overline{x})}{4}-\frac{(1-\overline{x})}{2})\\
			c(\overline{x}+\frac{(1-x)(1-\overline{x})}{4})-\frac{d(1-x)(1-\overline{x})}{2}=-a'(1-\overline{x})-(\frac{(1-x)(1-\overline{x})}{2}-1+\overline{x}).
		\end{cases}
	\end{equation}
	
	A  calculation shows that the last two equations in \eqref{*} are redundant. That is, \eqref{*} is equivalent to
	\begin{equation}\label{16}
		\begin{cases}
			a=1-\frac{b'(1-x)(1+\overline{x})}{2}\\
			b=\frac{b'(1+\overline{x})}{2}-\frac{b'(1-x)(1+\overline{x})}{4}=\frac{b'(1+x)(1+\overline{x})}{4}\\
			c=b'(1-x)(1-\overline{x})\\
			d=1-\frac{b'(1+x)(1-\overline{x})}{2}.
		\end{cases}
	\end{equation}

	Since $\begin{pmatrix}
		a& b\\
		c&d
	\end{pmatrix}\in G'(\mathbb{Q}),$ we then have
	\begin{equation}\label{15}
		\begin{pmatrix}
			\overline{d}& \overline{b}\\
			\overline{c}&\overline{a}
		\end{pmatrix}\begin{pmatrix}
			a& b\\
			c&d
		\end{pmatrix}=\Id.
	\end{equation}
	Substituting \eqref{16} in \eqref{15} we then conclude together with \eqref{14} that
	\begin{equation}\label{17}
		\begin{cases}
			\big[b'+\overline{b}'+b'\overline{b}'(x\overline{x}-1)\big] \frac{(1-x)(1+\overline{x})}{2}=0\\
			b'+\overline{b}'+b'\overline{b}'(x\overline{x}-1)=0.
		\end{cases}
	\end{equation}
	
	Then it follows from \eqref{17} that $b'+\overline{b}'+b'\overline{b}'(x\overline{x}-x-\overline{x})=0.$ Hence Lemma \ref{23'} follows.
\end{proof}

We will now use that $x\notin E^1$, that is $x.\ov x-1\not=0$. It turns out that the conjugate stabilizer
of $\gamma(x)J$ (which is a conjugate to that of $\gamma(x)$ since $J\in G'(\Qq)$) has a nice parametrization and we will use it instead:
\begin{cor}\label{34}
Let notation be as before, for $x\in E^\times\!-E^1$ set $$s(x):=\frac{(x-1)(\overline{x}+1)}{2}\in E^\times.$$ For $R$ any $\mathbb{Q}$-algebra, the stabilizer of $\gamma(x)J$ is given  by $$H_{\gamma(x)J}(R)=
	\big\{(h_1(w),h_2(w))\in G'(R)\times G'(R):\ w\in E^\times(R),\ w\overline{w}=1\big\},$$ where
%
%
	
	\begin{equation}\label{23}
		\begin{cases}
			h_1(w):=\begin{pmatrix}
1+\frac{(w-1)s(x)}{s(x)+s(\overline{x})}& \frac{(w-1)(1+x)(\overline{x}+1)}{4(s(x)+s(\overline{x}))}\\
\frac{(w-1)(1-x)(1-\overline{x})}{s(x)+s(\overline{x})} &1+\frac{(w-1)s(\overline{x})}{s(x)+s(\overline{x})}
			\end{pmatrix}\\
			\\
			h_2(w):=\begin{pmatrix}
			1+\frac{(w-1)s(x)}{s(x)+s(\overline{x})} & \frac{w-1}{s(x)+s(\overline{x})} \\
\frac{(w-1)s(x)s(\overline{x})}{s(x)+s(\overline{x})}&1+\frac{(w-1)s(\overline{x})}{s(x)+s(\overline{x})}
		\end{pmatrix}.
		\end{cases}
	\end{equation}

	In particular, $h_1(w)$ and $h_2(w)$ are determined uniquely by their determinant: $$\det h_1(w)=\det h_2(w)=w.$$
\end{cor}
\begin{proof}
	Since $x\not\in E^1$ we have $x\overline{x}-1\neq 0$. Now we consider the equation $$y+\overline{y}+y\overline{y}(x\overline{x}-1)=0,$$ whose locus is a conic (since $x\overline{x}-1\neq 0$) and by the intersection method, its $R$-rational points are parametrized by $\mathbb{P}^{1}(R)$. Explicitly, we write $y=a+at\sqrt{-D},$ where $a, t\in R$. We have $$a(1+t^2D)(x\overline{x}-1)+2=0.$$ and we obtain
	\begin{align*}
		y=-\frac{2(1+t\sqrt{-D})}{(1+t^2D)(x\overline{x}-1)}=-\frac{2}{(1-t\sqrt{-D})(x\overline{x}-1)},\quad t\in\mathbb{P}^{1}(R).
	\end{align*}
	
	The parametrization \eqref{23} then follows from Lemma \ref{23'} and the fractional linear transform replacing $t$ with $w$ such that $(1-w)/2=(1-t\sqrt{-D})^{-1}.$ We have again used the fact that $J\in G'(\Qq).$ 
\end{proof}

\begin{remark}\label{closetounip}
	For $x_0\in E^1,$ the above computation  of $H_{\gamma(x)J}$ is consistent with what we obtained before for $H_{\gamma(x_0)}$ after taking limit $x\mapsto x_0$ in \eqref{23}.
\end{remark}

We will henceforth write $H_x$ for the image of $H_{\gamma(x)J}\subset G'\times G'$ under the projection to the first component, i.e., $$H_x(R)=\{h_1(u),\ u\in E^{\times}(R),\ u\ov u=1\}$$ for $h_1(u)$ defined in \eqref{23}.

As a consequence of Corollary \ref{34}, we obtain
\begin{cor}\label{35}
	Let $x\in E^\times\!-E^1.$ Then $H_x$ is isomorphic to $U(1),$ the unitary group of rank $1.$ In particular, $H_x(\mathbb{Q})\backslash H_x(\mathbb{A})$  has  volume $2$ for the Tamagawa measure.
\end{cor}
\begin{proof}
	A direct computation shows that $$h_1(u_1)h_1(u_2)=h_1(u_1u_2).$$
	This gives the group law of $H_x(R),$ where $R$ is a $\mathbb{Q}$-algebra. In particular, we have a natural isomorphism with the unitary group
	$$H_x\simeq U(1)=E^1.$$
\end{proof}

\subsection{Reduction to factorable integrals}\label{seccoarsebound} 

To simplify notations, we will write
$$\mcO_{x}(f^{\mathfrak{n}},\vphi'):=\mcO_{\gamma(x)}(f^{\mathfrak{n}},\vphi')$$
Observe that since $J\in G'(\Qq)$ and $\vphi'$ is $G'(\Qq)$-invariant we have by suitable changes of variables that for $x\in E^\times\!-E^1$
\begin{gather*}
	\mcO_{x}(f^{\mathfrak{n}},\vphi')=\mcO_{\gamma(x)}(f^{\mathfrak{n}},\vphi')=\mcO_{\gamma(x)J}(f^{\mathfrak{n}},\vphi')\\=\int_{H_{\gamma(x)J}(\mathbb{Q})\backslash G'(\mathbb{A})\times G'(\mathbb{A})}f^{\mathfrak{n}}(y_1^{-1}\gamma(x)J y_2)\ov\vphi'(y_1){\vphi}'(y_2)dy_1dy_2	\\
	=\int_{H_{\gamma(x)J}(\mathbb{A})\backslash G'(\mathbb{A})\times G'(\mathbb{A})}f^{\mathfrak{n}}(y_1^{-1}\gamma(x) Jy_2)\int_{[H_{\gamma(x)J}]}\ov\vphi'(h_1y_1){\vphi}'(h_2y_2)dudy_1dy_2
\end{gather*}
where $$[H_{\gamma(x)J}]:=H_{\gamma(x)J}(\mathbb{Q})\backslash H_{\gamma(x)J}(\mathbb{A}),$$ and $$(h_1,h_2)=(h_1(u),h_2(u))\in [H_{\gamma(x)J}].$$

For any $(y_1,y_2)$, the $u$-integral can be evaluated as an (infinite) sum of Waldspurger's period integrals (see \cite{Walds}). However, we will be much coarser and will content ourselves with the upper bound
\begin{equation}
	\label{supnormint}
	\int_{[H_{\gamma(x)J}]}\vphi'(h_1y_1)\overline{\vphi}'(h_2y_2)du\ll \|\vphi'\|_\infty^2.
\end{equation}
Therefore we obtain
\begin{equation}
	\label{supnormboundint}|\mathcal{O}_{x}(f^{\mathfrak{n}},\vphi')|\ll \|\vphi'\|_{\infty}^2\mcI(f^{\mathfrak{n}},x)
\end{equation}
where
$$\mcI(f^{\mathfrak{n}},x)=\int_{H_{\gamma(x)J}(\mathbb{A})\backslash G'(\mathbb{A})^2}|f^{\mathfrak{n}}(u^{-1}\gamma(x) Jv)|dudv,
$$
say.


%
The main benefit of this rather crude treatment is that the resulting integral is factorable: we have 
$$\mcI(f^{\mathfrak{n}},x)=\prod_v\mcI_v(f^{\mathfrak{n}}_v,x)$$
where
\begin{equation}\label{Ivdef}
\mcI_v(f^{\mathfrak{n}}_v,x):=\int_{H_{\gamma(x)J}(\Qv)\bash (G'\times G')(\Qv)}|f_v^{\mathfrak{n}}(u^{-1}\gamma(x) Jv)|dudv;
\end{equation}
indeed as we will verify in the next subsections below, for any given $x\in E-E_1$, we have $$\mcI_p(f^{\mathfrak{n}}_p,x)=1$$ for all but finitely many $p$.

In the subsequent subsections we analyse  these local integrals $\mcI_v(f^{\mathfrak{n}}_v,x)$ give criteria for vanishing and provide upper bounds.

In the sequel, to simplify notations we will write $\mcI(x), \mcI_v(x)$ in place of $\mcI(f^{\mathfrak{n}},x)$ and $\mcI_v(f^{\mathfrak{n}}_v,x)$.

\subsection{A vanishing criterion for the non-archimedean local integrals}
Let $p$ be a prime, $\mfp$ the place above $p$ we have fixed (we take $\varpi_p=p$ if $p$ is inert) and let $\nu$ be the associated valuation in the local field $E_\mfp$.

\begin{prop}\label{propvanish}
Let $x\in E^\times\!-E^1$ a non-zero regular element.
The following holds
\begin{itemize}
	\item If $p$ is ramified,  then $\mcI_p(x)=0$ unless $$x\in \mcO_{E,p}.$$
\item If $p$ is inert, then $\mcI_p(x)=0$ unless $$x\in p^{-\nu(\ell)}\mcO_{E,p}\ \hbox{ and } x\ov x\equiv 1\mods{(N,p)}.$$

	\item If $p$ is split then $\mcI_p(x)=0$ unless $$\nu(N'x),\nu(N'\ov x)\geq 0.$$
		
\end{itemize}	
\end{prop}
From this we deduce immediately the a global vanishing criterion:
\begin{thm}\label{thmvanish}
		With the same notation as above, let
		\begin{equation}\label{Xidef}
			\mathfrak{X}(N,N',\ell)=\big\{x\in E^\times\!-E^1,\ x\in (\ell N')^{-1}\mathcal{O}_E :\ x\overline{x}\equiv 1\Mod{N}\big\}.
		\end{equation}
		For $x\in E^\times\!-E^1$ and not contained in $\mathfrak{X}(N,N',\ell)$, we have 
		$$\mcI(f^{\mathfrak{n}},x)=0$$
		and by \eqref{supnormboundint} we have
		\begin{equation}
			\label{Orbregvanish}
			\mathcal{O}_{\gamma(x)}(f^{\mathfrak{n}},\vphi')=0
		\end{equation}
	\end{thm}

	\begin{remark}
		This criterion is similar to the phenomenon encountered in the case of the forms of $\GL(2)\times \GL(1)$ in \cite[\S 2.6]{RR05} and \cite{FeiWhi}. 
	\end{remark}

We will  prove this through case by case analysis: see Propositions \ref{nsplitkey}, \ref{withHecke}, \ref{keysplitcoprime}, In fact we will also provide upper bounds for $\mcI_p(x)$. The non-split case is discussed in \S \ref{secnonsplitreg} and the following subsection and the split case in \S \ref{secsplitreg}.

For any $\Qq$-algebra $R$ let
$$SG'(R):=SU(W)=\{g\in G'(R),\ \det g=1\}.$$
We start with the  observation that a fundamental domain for the quotient\linebreak $H_{\gamma(x)J}(\Qp)\bash (G'\times G')(\Qp)$ is  the subgroup
$$SG'(\Qp)\times G'(\Qp).$$
Indeed any $(u,v)\in G'(\Qp)\times G'(\Qp)$ can be written as
$$(u,v)=(h_1(w)u',h_2(w)v')$$
with $$w=\det(u)=\det(h_1(w)),\ u'=h_1(w)^{-1}u\in SG'(\Qp),\ v'= h_2(w)^{-1}v.$$
Moreover for $u_1,u'_1\in SG'(\Qp)$ and $v,v'\in G'(\Qp)$ we have
$$(u'_1,v')=(h_1(w)u_1,h_2(w)v)\Longrightarrow w=1,\ h_1(w)=h_2(w)=\Id_3,\ u_1=u'_1,\ v'=v.$$
We have therefore
$$\mcI_p(x):=\int_{SG'(\Qp)\times G'(\Qp)}|f_v^{\mathfrak{n}}(u^{-1}\gamma(x) Jv)|dudv.$$

\subsection{The non-split case}\label{secnonsplitreg}
In this section we evaluate the local integral $\mcI_p(x)$ when $p$ is non-split: this implies that $\ov\varpi_p=\pm\varpi_p$, that $\mfn_p=\Id_3$ and that  $f^{\mfn}_p=f_p$ is a scalar times the characteristic function of either $K_p(N)$ or of
$$G(\Zp)A_rG(\Zp)$$
for $A_r,\ r\geq 0$ the diagonal matrix defined in the Appendix \eqref{Andef} .

We use the Iwasawa decomposition for $(u,v)\in SG'(\Qp)\times G'(\mathbb{Q}_p)$. 

We have $\det u=1$ and from the description of $G'(\Qp)$ $\det v$ has  valuation  $0$ so that (here we represent $u$ and $v$ as $2\times 2$ matrices)
\begin{align}\label{uviwasawa}
u=\begin{pmatrix}
	\varpi_p^{i}&\\
	&\overline{\varpi}_p^{-i}
\end{pmatrix}\begin{pmatrix}
	1&b\\
	&1
\end{pmatrix}k_1=u'k_1\kappa_1^i,\ \ \ v=\begin{pmatrix}
	\varpi_p^{j}&\\
	&\overline{\varpi}_p^{-j}
\end{pmatrix}\begin{pmatrix}
	1&b'\\
	&1
\end{pmatrix}k_2=v'k_2,
\end{align}
with $i, j\in \mathbb{Z},$ $b,b'\in E^0_p$, $k_1, k_2\in G'(\mathbb{Z}_p)$, 
$\det k_1=1$ and $\kappa_1=\begin{pmatrix}
\ov\varpi_p&\\&\varpi_p
\end{pmatrix}$.

By definition, the function $f_p$ is bi-$K_p(N)$-invariant so that	 
$$f_p(u^{-1}\gamma(x)J v)=f_p(k_1^{-1}{u'}^{-1}\gamma(x)J v'k_2).$$
We have 
\begin{align*}
\gamma(x)J=\left(
\begin{array}{ccc}
	-\frac{1}{2}&\frac{1+x}{2}&\frac{x\overline{x}+3\overline{x}-x+1}{4}\\
	-1&x&\frac{(x+1)(\overline{x}-1)}{2}\\
	1&1-x&-\frac{(1-x)(1-\overline{x})}{2}
\end{array}
\right)
\end{align*}
and 
\begin{equation}\label{263}
{u'}^{-1}\gamma(x)J v'=	\begin{pmatrix}
	-\frac{\varpi_p^{j-i}}{2}-\overline{\varpi}_p^i\varpi_p^{j}b&e&z\\
	-\varpi_p^j&x&f\\
	\varpi_p^{i+j}&\overline{\varpi}_p^{i}(1-x)&g		\end{pmatrix},
\end{equation}
where 
\begin{gather}\nonumber	e=\frac{\varpi_p^{-i}(1+x)}{2}-\overline{\varpi}_p^{i}b(1-x),\\	
\label{efgnonsplitdef}	f=-\varpi_p^jb'+\overline{\varpi}_p^{-j}\frac{(x+1)(\overline{x}-1)}{2},\\
g=\overline{\varpi}_p^i\varpi_p^{j}b'-\overline{\varpi}_p^{i-j}\frac{(x-1)(\overline{x}-1)}{2},\nonumber
\end{gather}
and
\begin{gather}
\nonumber
z=-\frac{1}{2}\varpi_p^{j-i}b'+\varpi_p^{-i}\overline{\varpi}_p^{-j}\frac{x\overline{x}+3\overline{x}-x+1}{4}-\overline{\varpi}_p^i\varpi_p^{j}bb'+\overline{\varpi}_p^{i-j}b\frac{(x-1)(\overline{x}-1)}{2}.\nonumber\\
=\frac{1-x\overline{x}}{1-x}\varpi_p^{-i}\overline{\varpi}_p^{-j}+\frac{f}{1-x}\varpi_p^{-i}-\frac{1-\overline{x}}{1-x}e\overline{\varpi}_p^{-j}+\frac{ef}{1-x} 	\label{260}
\end{gather}

\subsubsection{The non-split case $f_p=1_{K_p(N)}$} 
\begin{prop}\label{nsplitkey} We assume here that $p$ is non-split and that $f_p=1_{K_p(N)}$.
The following holds:
\begin{itemize}
	\item If $\nu(x)<0$ we have $\mcI_p(x)=0$.
\end{itemize}
Assume that $\nu(x)\geq 0$ (i.e. $x\in\mcO_{E,p}$)
\begin{itemize}
	\item if $p\nmid 2D_EN$ and $\nu(x\overline{x}-1)=0$, we have
	$$\mcI_p(x)=1.$$ 
	\item if $p\mid N,$ and $\nu(x\ov x-1)\leq 0$ we have $f_{p}(u^{-1}\gamma(x)Jv)=0$.
\end{itemize}
\begin{itemize}
	\item In general, we have the bound 
	\begin{equation}\label{107}
		\mcI_p(x)\ll e_p(x)(N,p)N(\varpi_p)^{3\nu(x\overline{x}-1)};
	\end{equation}
	here we have set $$e_p(x)=(1+\nu(x-1))^2(1+\nu(x\overline{x}-1)),$$
	$$N(\varpi_p)=|\mcO_{E,\mfp}/\varpi_p\mcO_{E,\mfp}|=p^{f_p},\ f_p=\begin{cases}2&\hbox{ if $p$ is inert}\\1&\hbox{ if $p$ is ramified}	
	\end{cases}
	$$
	and the implicit constant is absolute.
\end{itemize}

\end{prop}

\begin{proof}
As noted previously $f_p(u^{-1}\gamma(x)J v)$ is non zero if and only if 
\begin{equation}
	\label{262}
	{u'}^{-1}\gamma(x)Jv'\in k_1K_p(N)k_2^{-1}\subset G(\Zp).
\end{equation}

This implies that
\begin{equation}\label{265}
	\begin{cases}
		\nu(x)\geq 0,\ \nu(z)\geq 0\\
		j\geq 0\\
		i+\nu(1-x)\geq 0\\
		e,f,g,z,-\frac{\varpi_p^{j-i}}{2}-\overline{\varpi}_p^i\varpi_p^{j}b\in \mathcal{O}_{E,p}\\
	\end{cases}
\end{equation}

In particular we have $f_p(u^{-1}\gamma(x)J v)=0$ unless
$\nu(x)\geq 0$. This proves the first part of Proposition \ref{nsplitkey}.

Since $j\geq 0,$ then it follows from the equality $$\varpi_p^je-[-\frac{1}{2}\varpi_p^{j-i}-\overline{\varpi}_p^i\varpi_p^{j}b](1-x)=\varpi_p^{j-i}$$ that $j\geq i.$ 

Note that $$\overline{\varpi}_p^if+g=\overline{\varpi}_p^{i-j}(\overline{x}-1).$$
Hence 
\begin{equation}\label{266}
	i-j+\nu(x-1)\geq \min \{0, i\}.
\end{equation}

From the condition $e,f,z\in\mcO_{E,p}$, the equality \eqref{260} and the fact that $\nu(1-x)=\nu(1-\ov x)$, we have 
\begin{equation}\label{267}
	-i-j-\nu(x-1)+\nu(x\overline{x}-1)\geq \min \{-i-\nu(x-1), -j\}. 
\end{equation}

Suppose now that $\nu(x\overline{x}-1)=0$; this implies that $\nu(x-1)=0$ and therefore by \eqref{265} $i\geq 0$ and eventually $i\geq j$ by \eqref{266}, therefore 
$i=j\geq 0.$
By \eqref{267} we  have $-2i\geq -i$ and therefore $$i=j=0$$
and from
\eqref{263} we conclude that
$$b, b'\in \mathcal{O}_{E,p}.$$ 

If $p$ does not divide $2N$ we have $K_p(N)=G(\Zp)$ so that if $\nu(x\overline{x}-1)=0$  we have $$|f_p(u^{-1}\gamma(x)J v)|=1$$
precisely when
$$i=j=0,\ k_1\in SK_p,\ k_2\in K_p,\ b, b'\in \mathcal{O}_{E,p}$$ and otherwise it is zero. It follows that	
\begin{align*}
	\int_{H_{\gamma(x)J}(\Qp)\bash (G'\times G')(\mathbb{Q}_p)}\big|f_p(u^{-1}\gamma(x) Jv)\big|dudv=\mu(\mathcal{O}_{E,p})^2=1.
\end{align*}
This proves the generic part of Proposition \ref{nsplitkey}.

We return to the general case: given two integers $i, j$ such that   $u,v$ have Iwasawa decomposition given by \eqref{uviwasawa} and such that $f_p(u^{-1}\gamma(x) Jv)\not=0$. From the previous discussion we see that
$$b=\frac{\varpi_p^{-i}\overline{\varpi}_p^{-i}(1+x)}{2(1-x)}-\frac{e}{\overline{\varpi}_p^i(1-x)}\in B_i,$$
$$
b'=\varpi_p^{-j}\overline{\varpi}_p^{-j}\frac{(x+1)(\overline{x}-1)}{2}-\varpi_p^{-j}f\in B'_j$$
where
$$
B_i:=\frac{\varpi_p^{-i}\overline{\varpi}_p^{-i}(1+x)}{2(1-x)}-\frac{1}{\overline{\varpi}_p^i(1-x)}\mathcal{O}_{E,p},$$
$$B'_j:= \varpi_p^{-j}\overline{\varpi}-p^{-j}\frac{(x+1)(\overline{x}-1)}{2}-\varpi_p^{-j}\mathcal{O}_{E,p}$$	
We have
\begin{equation}\label{269}
	\int_{B_i}db\int_{B_j'}db'\leq N(\varpi_p)^{i+\nu(1-x)+j}.
\end{equation}

Suppose $i\leq 0.$ Then by \eqref{266} we have $j\leq \nu(x-1)$ and $$-\nu(x-1)\leq i\leq 0\leq j\leq \nu(x-1).$$ Then by \eqref{269} 
\begin{equation}\label{270}
	\sumsum_{-\nu(x-1)\leq i\leq 0\leq j\leq \nu(x-1)}\int_{B_i}db\int_{B_j'}db'\ll N(\varpi_p)^{2\nu(x-1)}
\end{equation}
where the implicit constant is absolute. 

Suppose $i\geq 1.$ Then similar arguments as before show that $j-\nu(x-1)\leq i\leq j$ and $1\leq j\leq \nu(x\overline{x}-1).$ So in this range we have 
\begin{equation}\label{271}
	\sumsum_{i,j\cdots}\int_{B_i}db\int_{B_j'}db'\leq e_p(x) N(\varpi_p)^{2\nu(x\overline{x}-1)+\nu(x-1)}.
\end{equation}

Then \eqref{107} follows from \eqref{270} and \eqref{271}. When $p=2$, recall that $K_2(N)=K_2$ is the special maximal compact defined by Tits (see \S \ref{tits}); we need to replace $\mcO_{E,p}$ by a suitable fixed proper ideal of absolutely bounded index. A similar argument as above can be carried out and produces a different implicit constant depending on the index. 

This proves the last part of Proposition \ref{nsplitkey}.
\medskip 

Suppose now that $p\mid N.$ Then  $K_p(N)=K_p(p)\subset K_p$ is the Iwahori subgroup and its intersection with $G'(\mathbb{Z}_p)$ is $K_p'(N)=K'_p(p)$  the Iwahori subgroup of $G'(\mathbb{Z}_p).$ It follows that the function on $G'(\Zp)\times G'(\Zp)$ 
$$(u,v)\mapsto f_p(u^{-1}\gamma(x)J v)$$
is bi-$K'_p(p)$ invariant. 

Because of this, we will evaluate the integral $\mcI_p(x)$ using the Iwasawa decompositions of $u$ and $v$ \eqref{uviwasawa} and in particular decompose the two integrals over the $k_1\in SG'(\Zp)$ and the $k_2\in G'(\Zp)$ variables into a sum of $(p+1)^2$-integrals supported along $SK'_p(p)\times K'_p(p)$-cosets  using Lemma \ref{K'cosetlemma}. 

Given $u,v$  such that $f_{p}(u^{-1}\gamma(x)v)\neq0$ and whose Iwasawa decompositions are given by \eqref{uviwasawa} and let \begin{align*}
	Q:=\begin{pmatrix}
		1&&-b\\
		&1&\\
		&&1
	\end{pmatrix}\begin{pmatrix}
		\varpi_p^{-i}&\\
		&1&\\
		&	&\overline{\varpi}_p^{i}
	\end{pmatrix}\gamma(x)J\begin{pmatrix}
		\varpi_p^{j}&\\
		&1&\\
		&	&\overline{\varpi}_p^{-j}
	\end{pmatrix}
	\begin{pmatrix}
		1&&b'\\
		&1&\\
		&	&1
	\end{pmatrix},
\end{align*}

Since $f_p$ is supported on $K_p(p)$, for the integrand in one of these to be non-zero one of the following  hold	:
$$Q\in K_p(p),\hbox{ or }J\begin{pmatrix}
	1&&-\delta \\
	&1&\\
	&&1
\end{pmatrix}Q
\in K_p(p)\hbox{ or }Q
\begin{pmatrix}
	1&&\delta'\\
	&1&\\
	&	&1
\end{pmatrix}J\in K_p(p)$$
$$\hbox{ or }J\begin{pmatrix}
	1&&-\delta \\
	&1&\\
	&&1
\end{pmatrix}Q
\begin{pmatrix}
	1&&\delta'\\
	&1&\\
	&	&1
\end{pmatrix}J\in K_p(p)\hbox{ for some }\delta, \delta'\in \Zp/p\Zp.$$

We will show in each case that $\nu(x\overline{x}-1)\geq 1.$ 

If we assume instead that $\nu(x\overline{x}-1)=0$, we have $$\nu(x-1)=\nu(\overline{x}+1)=\nu(1-\overline{x})=0.$$ 
We will obtain a contradiction for each of the $1+2p+p^2$ possible locations of $g$:
\begin{enumerate}
	\item[(i)] Assume $Q\in K_p(p).$ Then by \eqref{263} we have $i\geq 1,$ $j\geq 1$ and $e, f, z\in \mathcal{O}_{E,p}.$ But this  contradicts the algebraic relation \eqref{260}, which forces that $z\not\in \mathcal{O}_{E,p}$ a contradiction and therefore $\nu(x\ov x-1)\geq 1$.
	
	\item[(ii)] Suppose $J\begin{pmatrix}
		1&&-\delta \\
		&1&\\
		&&1
	\end{pmatrix}Q
	\in K_p(p).$ This implies that
	\begin{equation}\label{348}
		\begin{pmatrix}
			-\frac{\varpi_p^{j-i}}{2}-\overline{\varpi}_p^i\varpi_p^{j}b&e&z\\
			-\varpi_p^j&x&f\\
			\varpi_p^{i+j}&\overline{\varpi}_p^{i}(1-x)&\overline{\varpi}_p^i\varpi_p^{j}b'-\overline{\varpi}_p^{i-j}\frac{(x-1)(\overline{x}-1)}{2}
		\end{pmatrix}\in JK_p(p).
	\end{equation}
	here we have made the change of variable $$b\mapsto b+\delta.$$
	Hence $i\geq 0,$ $j\geq 1,$ $\nu(e)\geq 1,$ $\nu(f)\geq 0$ and $\nu(z)\geq 0.$ However, this contradicts \eqref{260} as well. 
	
	Notice then, that for $\nu(x\overline{x}-1)\geq 1$,  \eqref{348} leads to $$-\frac{\varpi_p^{j-i}}{2}-\overline{\varpi}_p^i\varpi_p^{j}b\in \varpi_p\mathcal{O}_{E,p}$$ and $\varpi_p^{i+j}\in \mathcal{O}_{E,p}^{\times}.$ So there is only one possible choice for $b\mod{\varpi}_p.$ 
	
	\item[(iii)] Suppose that $Q
	\begin{pmatrix}
		1&&\delta'\\
		&1&\\
		&	&1
	\end{pmatrix}J\in K_p(p).$ This is similar to the preceding case (ii). Again we must have $\nu(x\overline{x}-1)\geq 1$ and there is only one choice for $b'\mod{\varpi}_p.$ 
	
	\item[(iv)] Suppose finally that $J\begin{pmatrix}
		1&&-\delta \\
		&1&\\
		&&1
	\end{pmatrix}Q
	\begin{pmatrix}
		1&\delta'\\
		&1&\\
		&	&1
	\end{pmatrix}J\in K_p(p).$ We then have again
	$$
	\begin{pmatrix}
		-\frac{\varpi_p^{j-i}}{2}-\overline{\varpi}_p^i\varpi_p^{j}b&e&z\\
		-\varpi_p^j&x&f\\
		\varpi_p^{i+j}&\overline{\varpi}_p^{i}(1-x)&\overline{\varpi}_p^i\varpi_p^{j}b'-\overline{\varpi}_p^{i-j}\frac{(x-1)(\overline{x}-1)}{2}
	\end{pmatrix}\in JK_p(p);
	$$
	where this time we have made the change of variables $$b\mapsto b+\delta,\ b'\mapsto b'+\delta'.$$ We have therefore  $i\geq 0,$ $j\geq 0,$ $\nu(e)\geq 1,$ $\nu(f)\geq 1$ and $\nu(z)\geq 1,$ which contradicts \eqref{260} and necessarily $\nu(x\overline{x}-1)\geq 1$. Since $\nu(e)\geq 1$ and $\nu(f)\geq 1,$ this implies that
	\begin{align*}
		&b\in C_i:=\frac{\varpi_p^{-i}\overline{\varpi}_p^{-i}(1+x)}{2(1-x)}-\frac{1}{\overline{\varpi}_p^i(1-x)}\varpi_p\mathcal{O}_{E,p},\\ 
		&b'\in C'_j:=\varpi_p^{-j}\overline{\varpi_p}-p^{-j}\frac{(x+1)(\overline{x}-1)}{2}-\varpi_p^{1-j}\mathcal{O}_{E,p}.
	\end{align*}
	Therefore, $$\int_{C_i}db\int_{C'_j}db'\ll N(\varpi_p)^{i+\nu(x-1)+j-2}.$$ Notice the extra saving $p^{-2}$ (compared with \eqref{269}): it comes to compensate the contribution of the $p^2$ choices of $\delta$ and $\delta'$ in $\Zp/p\Zp$.  
\end{enumerate}
In all case we have $\nu(x\overline{x}-1)\geq 1$. Integrating over these $SK'_p(p)\times K'_p(p)$ cosets, we conclude that when $p\mid N,$ one has
\begin{align*}
	\mcI_p(x)\ll \frac{e_p(x)N(\varpi_p)^{3\nu(x\overline{x}-1)}\mu(SK_p'(p))\mu( K_p'(p))}{\mu(K_p(N))}\ll \frac{(1/p)^{2}}{1/p^3}e_p(x)N(\varpi_p)^{3\nu(x\overline{x}-1)}.
\end{align*}

This completes the proof of the Proposition when $p|N$.
\end{proof}

\subsubsection{The inert Hecke case} We now evaluate $\mcI_{p}(x)$ for $p$ inert in the last remaining case.
\begin{prop}

\label{withHecke} Let $p$ be an inert prime, $r\geq 0$ and suppose that
$$f_{p}=1_{G(\Zp)A_{p^r}G(\Zp)}.$$
We have
$$\mcI_{p}(x)=0$$ unless $\nu(x)\geq -r$ in which case we have
\begin{equation}\label{eq9}
\mcI_{p}(x)\ll (r+|\nu(1-x\ov x)|+|\nu(1-x)|)(p^{3r+2\nu(1-x)}+p^{7r+2\nu(1-x\ov x)}).
\end{equation}
Here the implied constant is absolute. 
\end{prop}
\begin{proof} We use again the notation of \S \ref{secnonsplitreg}. Let $u,v$  such that
$$f_{p}(u^{-1}\gamma(x)Jv)=f_{p}({u'}^{-1}\gamma(x)Jv')\not=0,$$
which, by \eqref{263}, amounts to 
\begin{equation}\label{2}
{u'}^{-1}\gamma(x)J v'=	\begin{pmatrix}
	-\frac{p^{j-i}}{2}-p^{j+i}b&e&z\\
	-p^j&x&f\\
p^{i+j}&p^{i}(1-x)&g		\end{pmatrix}\in G(\Zp)A_rG(\Zp).
\end{equation}
where 
\begin{gather}\nonumber	e=\frac{p^{-i}(1+x)}{2}-p^{i}b(1-x),\\	
f=-p^jb'+p^{-j}\frac{(x+1)(\overline{x}-1)}{2},\\
	g=p^{i+j}b'-p^{i-j}\frac{(x-1)(\overline{x}-1)}{2}.\nonumber
\end{gather}
and
\begin{equation}\label{eq3}
	\nonumber
z=\frac{1-x\overline{x}}{1-x}p^{-i-j}+\frac{f}{1-x}p^{-i}-\frac{1-\overline{x}}{1-x}ep^{-j}+\frac{ef}{1-x}.
\end{equation}

Analyzing the entries in \eqref{2} we then derive that $z, e, f\in p^{-r}\Zp.$ Hence, 
\begin{equation}\label{eq7}
\begin{cases}
\nu(x)\geq -r,\ \nu(g)\geq -r,\\
j\geq -r,\ i\geq -r-\nu(1-x),\\
f=-p^jb'+p^{-j}\frac{(x+1)(\overline{x}-1)}{2}\in p^{-r}\Zp,\\
e=\frac{p^{-i}(1+x)}{2}-p^{i}b(1-x)\in p^{-r}\Zp,\\
z=\frac{1-x\overline{x}}{1-x}p^{-i-j}+\frac{f}{1-x}p^{-i}-\frac{1-\overline{x}}{1-x}ep^{-j}+\frac{ef}{1-x}\in p^{-r}\Zp,
\end{cases}
\end{equation}
where the last constraint yields that 
\begin{equation}\label{eq8}
-i-j+\nu(1-x\ov x)\geq \min\big\{-r-i, -r-j+\nu(1-x), -2r, -r+\nu(1-x)\big\}.
\end{equation}

Since $p^if+g=p^{i-j}(\ov x-1),$ then 
\begin{equation}\label{eq4}
i-j+\nu(1-x)= \min\{i+\nu(f),\nu(g)\}\geq \min\{i-r,-r\}.
\end{equation}

Since $p^je-(1-x)a=p^{j-i},$ then 
\begin{equation}\label{eq5}
j-i=\min\{j+\nu(e),\nu(1-x)+\nu(a)\}\geq \min\{j-r,\nu(1-x)-r\}.
\end{equation}

We now separate the cases to derive the ranges of $i$ and $j$ as follows.
\begin{enumerate}
\item Suppose $i\leq 0.$ Then by \eqref{eq4} we obtain that $i-j+\nu(1-x)\geq i-r,$ implying that $j\leq r+\nu(1-x).$ Thus in this case we have $-r\leq i\leq 0$ and $-r\leq j\leq r+\nu(1-x).$

\item Suppose that $i\geq 1.$ Then \eqref{eq4} gives 
\begin{equation}\label{eq6}
i-j+\nu(1-x)\geq -r.
\end{equation}
\begin{itemize}
\item Suppose further that $j\leq \nu(1-x).$ Then it follows from \eqref{eq5} that $j-i\geq j-r,$ i.e., $i\leq r.$ Hence, in this case we have $1\leq i\leq r$ and $-r\leq j\leq \nu(1-x).$

\item Suppose that $j>\nu(1-x).$ Then by \eqref{eq5} we have $j-i\geq \nu(1-x)-r$, i.e., $-r-i\geq -2r-j+\nu(1-x).$ Substituting this in \eqref{eq8} we obtain 
\begin{align*}
	-i-j+\nu(1-x\ov x)\geq& \min\big\{-2r-j+\nu(1-x), -2r,-r+\nu(1-x)\big\}\\
	=&\min\big\{-2r-j+\nu(1-x), -2r\big\}=-2r-j+\nu(1-x).
\end{align*}
Here we note the fact that $-2r\leq -r+\nu(1-x)$ and $j>\nu(1-x).$ Therefore, we have
$1\leq i\leq 2r+\nu(1-x\ov x)-\nu(1-x).$ By \eqref{eq6},
$$
\nu(1-x)<j\leq i+r+\nu(1-x)\leq 3r+\nu(1-x\ov x).
$$ 
\end{itemize} 
\end{enumerate}

Denote by $\mathcal{S}$ the support of $(i,j)\in\mathbb{Z}^2$ determined by \eqref{eq7}. From the above discussion we see that $\#\mathcal{S}\ll r+|\nu(1-x\ov x)|+|\nu(1-x)|$ and for $(i,j)\in\mathcal{S},$ 
$$
i+j\leq \max\{r+\nu(1-x), 5r+2\nu(1-x\ov x)-\nu(1-x)\}.
$$

Note that \eqref{eq7} also implies that $b$ (resp. $b'$) ranges over a translate of $p^{-i-r-\nu(1-x)}\Zp$ (resp. $p^{-j-r}\Zp$). Therefore, 
$$
\mcI_p(x)\ll\sum_{(i,j)\in \mathcal{S}}\int_{p^{-i-r-\nu(1-x)}\Zp}db\int_{p^{-j-r}\Zp}db'\ll \sum_{(i,j)\in \mathcal{S}}p^{i+j}\cdot p^{\nu(1-x)+2r}.
$$
Hence, \eqref{eq9} follows.
\end{proof}

\subsection{The split case} \label{secsplitreg}
Let $p$ be a  prime split in $E$, the corresponding ideal of $K$ decomposes as $\mathfrak{p}\omfp=(p)$ and we have $E\otimes_\Qq\Qp\simeq E_\mfp\times E_\omfp \simeq \Qp\times\Qp$. Let $\varpi$ in $E_\mfp\simeq \Qp$ be a uniformizer.

In the split case, we have $G'(\mathbb{Q}_p)\simeq\GL(2,\mathbb{Q}_p)$ and by Corollary \ref{34} one has $$H_x(\mathbb{Q}_p)\simeq \GL(1,\mathbb{Q}_p).$$ 

let $u,v\in G'(\mathbb{Q}_p)$ with Iwasawa decompositions given by
\begin{equation}\label{Iwa22}
u=\begin{pmatrix}
	p^{i_1}&\\
	&p^{i_2}
\end{pmatrix}\begin{pmatrix}
	1&b\\
	&1
\end{pmatrix}k_1,\ \ 
v=\begin{pmatrix}
	p^{j_1}&\\
	&p^{j_2}
\end{pmatrix}\begin{pmatrix}
	1&b'\\
	&1
\end{pmatrix}k_2,
\end{equation} 
where $i_1,i_2, j_1, j_2\in\mathbb{Z},$ $b, b'\in\mathbb{Q}_p$ and $k_1, k_2\in G'(\mathbb{Z}_p)$. In particular if $\nu(\det u)=0$ we have $i_2=-i_1$. 


Applying this decomposition to the variable $u$ in the integral $\mcI_p(x)$ we have
\begin{align*}
	\mcI_{p}(x)=\sum_{i\in\mathbb{Z}}\int_{SG'(\Zp)}\int_{\mathbb{Q}_p}\int_{G'(\mathbb{Q}_p)}\Big|f^{\mathfrak n}_p\left(k_1^{-1}\begin{pmatrix}
		p^{-i}&&-p^ib\\
		&1&\\
		&		&p^{i}
	\end{pmatrix}\gamma(x)Jv	\right)\Big|dbdk_1dv.
\end{align*}

\begin{prop}\label{keysplitcoprime}
	Let $p$ be a  prime, split in $E$  and $x\in E^\times\!-E^1.$
	
	\begin{enumerate}
		\item 	If $p\nmid N'$ we have $$\mcI_p(x)=0$$ unless $\nu(x),\ \nu(\ov x)\geq 0$; in that case we have the bound  
	\begin{equation}\label{101}
		\mcI_{p}(x)\ll (1+\nu(P(x,\ov x))p^{3\nu(x\overline{x}-1)},
	\end{equation}
	where $P(X,Y)\in \Zz[X,Y]$ is a polynomial independent of $p$ whose degree and coefficients are absolutely bounded; moreover the implicit constant is absolute.
	
		-- In addition, if $p\neq 2$ and $$\nu(x\overline{x}-1)=\nu(1-x)=\nu(1-\overline{x})=0$$
we have
	$\mcI_{p}(x)=1.$

\item If $p|N'$, we have \begin{equation}
	\label{supportsplitN'}\mcI_{p}(x)=0
\end{equation}
	unless $x\in p^{-1}\mathcal{O}_{E,p},$ i.e., $\nu(x)\geq -1$ and $\nu(\overline{x})\geq -1.$ 

-- Moreover if $$\nu(x),\nu(\ov x)\geq 0,$$	
one has
	\begin{equation}
	\label{301}
		\mcI_{p}(x)\ll  (1+\nu(P(x,\ov x)))^2 \\
		\big( p^{\nu(x\ov x(1-x)^2(1-\ov x)^2)}+p^{\nu(1-x\ov{x})-1}
		+p^{2\nu(1-x\ov{x})-1}\big)
	\end{equation}
	where $P(X,Y)\in \Zz[X,Y]$ is a polynomial independent of $p$ whose degree and coefficients are absolutely bounded; moreover the implicit constant is absolute.

\end{enumerate}

\end{prop}

\begin{proof}
We start with the case $p\nmid  N'$.	We have  $f_p^{\mathfrak{n}}=f_p$ and $$f_{p}(u^{-1}\gamma(x)Jv)=0$$ if and only if $u^{-1}\gamma(x)Jv\not\in G(\Zp)$. Since the determinant of $v$ has valuation $0$, its Iwasawa coordinates \eqref{Iwa22} satisfy $j_1+j_2=0$ . Hence $\mcI_{p}(x)$
	is bounded from above by 
	\begin{align*}
		\sum_{i\in\mathbb{Z}}\sum_{j\in\mathbb{Z}}\int_{\mathbb{Q}_p}\int_{\mathbb{Q}_p}|f_p\left(\begin{pmatrix}
			p^{i}&&p^ib\\
			&1&\\
			&&p^{-i}
		\end{pmatrix}^{-1}\gamma(x)J\begin{pmatrix}
			p^{j}&&p^jb'\\
			&1&\\
			&&p^{-j}
		\end{pmatrix}\right)|dbdb'.
	\end{align*}

	The proof then follows the same lines as for the nonsplit case considered in Proposition \ref{nsplitkey}. We also note that at a split place $p$, we always have $G'(\mathbb{Z}_p)\subset G(\mathbb{Z}_p)$ if $p\nmid N'$.

We now consider the case of a split prime $p\mid N'.$  Let $b, b'\in\mathbb{Q}_p$ and $i, j\in\mathbb{Z}.$ We have 
\begin{equation}\label{279}
	\begin{pmatrix}
		p^{i}&&p^{i}b\\
		&1&\\
		&&p^{-i}
	\end{pmatrix}^{-1}\gamma(x)J \begin{pmatrix}
		p^{j}&&p^{j}b'\\
		&1&\\
		&&p^{-j}
	\end{pmatrix}=\begin{pmatrix}
		a&e&z\\
		-p^j&x&f\\
		p^{i+j}&p^{i}(1-x)&g
	\end{pmatrix},
\end{equation}
where 
\begin{equation}\label{lettervalues}
	\begin{cases}
		a=-\frac{1}{2}p^{j-i}-p^{i+j}b\\
		e=\frac{p^{-i}(1+x)}{2}-p^{i}b(1-x)\\
		f=-p^jb'+p^{-j}.\frac{(x+1)(\overline{x}-1)}{2}\\
		g=p^{i+j}b'-\frac{(x-1)(\overline{x}-1)}{2}p^{i-j}\\
		z=-\frac{1}{2}p^{j-i}b'+p^{-i-j}y-p^{i+j}bb'+p^{i-j}b\frac{(x-1)(\overline{x}-1)}{2}.
	\end{cases}
\end{equation}
where $$y=\frac{x\overline{x}+3\overline{x}-x+1}{4}.$$
Then one has an explicit algebraic relation 
\begin{equation}\label{zalgebraic}
	z=\frac{1-x\overline{x}}{1-x}p^{-i-j}+\frac{f}{1-x}p^{-i}-\frac{1-\overline{x}}{1-x}ep^{-j}-\frac{ef}{1-x}.
\end{equation}

Taking the inverse of \eqref{279} to obtain 
\begin{equation}\label{317}
	\begin{pmatrix}
		p^{j}&&p^{j}b'\\
		&1&\\
		&&p^{-j}
	\end{pmatrix}^{-1}J\gamma(x)^{-1}\begin{pmatrix}
		p^{i}&&p^{i}b\\
		&1&\\
		&&p^{-i}
	\end{pmatrix}=\begin{pmatrix}
		a'&e'&z'\\
		p^i(1-\overline{x})&\overline{x}&f'\\
		p^{i+j}&-p^{j}&g'
	\end{pmatrix},
\end{equation}
where 
\begin{equation}\label{letter'values}
	\begin{cases}
		a'=-\frac{(1-x)(1-\overline{x})}{2}p^{i-j}-p^{i+j}b'\\
		e'=\frac{(x-1)(\overline{x}+1)}{2}p^{-j}+p^{j}b'\\
		f'=(1-\overline{x})p^ib+\frac{1+\overline{x}}{2}p^{-i}\\
		g'=p^{i+j}b-\frac{1}{2}p^{j-i}\\
		z'=-p^{i-j}b\frac{(x-1)(\overline{x}-1)}{2}+p^{-i-j}\overline{y}-p^{i+j}bb'+\frac{1}{2}p^{j-i}b'
	\end{cases}
\end{equation}
and one notes the algebraic relation
\begin{equation}\label{z'algebraic}
	z'
	=\frac{1-x\overline{x}}{1-\overline{x}}p^{-i-j}+\frac{e'}{1-\overline{x}}p^{-i}-\frac{1-x}{1-\overline{x}}f'p^{-j}-\frac{e'f'}{1-\overline{x}}.
\end{equation}
We also note that
\begin{gather}\label{agsum}
	a+g'=-p^{j-i},\ a'+g=-{(1-x)(1-\overline{x})}p^{i-j},\\ p^i.f+g=p^{i-j}(\ov x-1),\ p^jf'-(1-\overline{x})g'=p^{j-i}\nonumber
\end{gather}

	By definition \eqref{Ivdef}, we have 
	\begin{align*}
		\mcI_{p}(x)=\int_{SG'(\mathbb{Q}_p)\times G'(\mathbb{Q}_p)}\big|f_p(\widetilde{\mathfrak{n}}_p^{-1}u^{-1}\gamma(x)J v\widetilde{\mathfrak{n}}_p)\big|dudv.
	\end{align*}
	where we recall that $f_p=\textbf{1}_{K_p}/\mu(K_p),\ K_p=G(\Zp)$ and
	$$	\widetilde{\mathfrak{n}}_p=
	w'\mathfrak{n}_p w'= \begin{pmatrix}
		1&p^{-1}&\\
		&1&\\
		&&1
	\end{pmatrix}$$
	We will apply the Iwasawa decomposition as in the beginning of the proof of Proposition \ref{keysplitcoprime}. 
	
	Due to the conjugation by $\widetilde{\mathfrak{n}}_p,$ the function
	$$(u,v)\mapsto f_p(\widetilde{\mathfrak{n}}_p^{-1}u^{-1}\gamma(x)J v\widetilde{\mathfrak{n}}_p)$$ is bi $I'_p(1)$-invariant (since $\widetilde{\mathfrak{n}}_pI'_p(1)\widetilde{\mathfrak{n}}_p^{-1}\subset K_p$) so we shall further decompose $G'(\mathbb{Z}_p)$ into a union of disjoint cosets of $I_p'(1)$: we start with the Iwahori decomposition for $G'$ (see \eqref{EqG'1})
	$$G'(\Zp)=I_p'\sqcup I_p'J
I_p'=I_p'\sqcup\bigsqcup_{\delta\in(\Zz/p\Zz)^\times} I_p'J\begin{pmatrix}\delta&&\\&1&\\&&1	
	\end{pmatrix}
	I_p'(1)$$ We then check four cases as in Proposition \ref{nsplitkey}. 
	
	\subsection*{Case I}
	We start with the most complicated case:
	\begin{align}\label{uform1}
		u=&\begin{pmatrix}
			p^{i}&&p^{i}b\\
			&1&\\
			&&p^{-i}
		\end{pmatrix}\begin{pmatrix}
			1&&\mu\\
			&1&\\
			&&1
		\end{pmatrix}J\begin{pmatrix}
			\tau^{-1}&&\\
			&1&\\
			&&1
		\end{pmatrix}k_1,\\
		v=&\begin{pmatrix}
			p^{j}&&p^{j}b'\\
			&1&\\
			&&p^{-j}
		\end{pmatrix}\begin{pmatrix}
			1&&\mu'\\
			&1&\\
			&&1
		\end{pmatrix}J\begin{pmatrix}
			\delta&&\\
			&1&\\
			&&1
		\end{pmatrix}k_2,\label{vform1}
	\end{align}
	with $\mu,\mu'\in \Zp$ run over representatives of $\Zp/p\Zp$  and  $\tau, \delta \in \Zpt$ run over representatives of $(\Zp/p\Zp)^{\times},$ and $k_1, k_2\in  I_p'(1).$
	\begin{defn}\label{defIp1} Let $\mcI_{p}(x;1)$ be the contribution  to $\mcI_{p}(x)$ of all the $u, v$ whose Iwasawa decomposition is of the  form \eqref{uform1} and \eqref{vform1}.	
	\end{defn}
	
	We first look for some necessary condition for $\mcI_{p}(x;1)$ to be non zero. Later (in Proposition \ref{supportsplitN''}) we will provide more precise bounds when it does not vanish.

	For $\delta\in \Zpt,$ we denote by 
	$$
	\mathfrak{u}_{\delta}=\begin{pmatrix}
		1&&\\
		&1&\\
		&\delta p^{-1}&1
	\end{pmatrix}=J\begin{pmatrix}
		\delta&&\\
		&1&\\
		&&1
	\end{pmatrix}\widetilde{\mathfrak{n}}_p J.
	$$
	Then, since $J\in K_p$, $f_p(\widetilde{\mathfrak{n}}_p^{-1}u^{-1}\gamma(x)J v\widetilde{\mathfrak{n}}_p)\neq 0$ if and only if 
	\begin{equation}\label{278} 
		\mathfrak{u}_{\tau}^{-1}\begin{pmatrix}
			p^{i}&&p^{i}(b+\mu)\\
			&1&\\
			&&p^{-i}
		\end{pmatrix}^{-1}\gamma(x)J \begin{pmatrix}
			p^{j}&&p^{j}(b'+\mu')\\
			&1&\\
			&&p^{-j}
		\end{pmatrix}\mathfrak{u}_{\delta}\in K_p
	\end{equation}
	or equivalently
	\begin{equation}\label{matrix1caseI}
		\begin{pmatrix}
			a&e+\frac{\delta}pz&z\\
			-p^j&x+\frac{\delta}pf&f\\
			p^{i+j}+\tau p^{j-1}&p^{i}(1-x)+\frac{\delta}pg-\frac{\tau}{p}(x+\delta p^{-1}f)&g-\tau p^{-1}f
		\end{pmatrix}\in K_p,
	\end{equation}
	and taking the inverse of \eqref{matrix1caseI} we also obtain the condition
	\begin{equation}\label{matrix2caseI}
		\begin{pmatrix}
			a'&e'+\frac{\tau}{p}z'&z'\\
			p^i(1-\overline{x})&\overline{x}+\frac{\tau}{p}f'&f'\\
			p^{i+j}-\delta p^{i-1}(1-\overline{x})&-p^{j}+\frac{\tau}{p}g'-\frac{\delta}p(\overline{x}+\frac{\tau}{p}f')&g'-\frac{\delta}pf'
		\end{pmatrix}\in K_p.
	\end{equation}
	Here $a,e,f,g,z$ are defined in \eqref{lettervalues}	and $a',e',f',g',z'$ as defined in \eqref{letter'values}.
	
	These conditions imply that $$\nu(f),\ \nu(f'),\ \nu(x+\delta p^{-1}f),\ \nu(\ov x+\tau p^{-1}f')\geq 0$$ which in turn imply that
	$$\nu(x),\ \nu(\overline{x})\geq -1.$$
	This proves Proposition \ref{keysplitcoprime} in Case I.
	
	Note also that \eqref{matrix1caseI} and \eqref{matrix2caseI} imply that
	\begin{gather}\label{generallowerbounds}
		i':=i+\nu(1-\ov x)\geq 0,\ j\geq 0,\\
		\nu(z),\ \nu(z')\geq 0,\ \nu(e),\ \nu(e')\geq -1.	\nonumber
	\end{gather}

	We have
	\begin{equation}\label{281}
		\mcI_{p}(x;1)\leq \mu(I_p'(1))^2p^2\sum_{i\geq-\nu(1-\ov x)}\sum_{j\geq 0}\int_{\mathbb{Q}_p}\int_{\mathbb{Q}_p}\sum_{\tau}\sum_{\delta}\textbf{1}_{K_p}(\cdots)dbdb',
	\end{equation}
	where the ``$\cdots$'' in the parenthesis indicates the left hand side of \eqref{matrix1caseI} with $$\mu=\mu'=0.$$ Indeed, up to changing variables in the $b$ and $b'$ integrals, we may assume this is the case. 
	
	Define 
	\begin{align*}
		\mcI_{p}^{j=0}(x;1):=\mu(I_p'(1))^2p^2\sum_{i\geq -\nu(1-\ov x)}\sum_{j= 0}\int_{\mathbb{Q}_p}\int_{\mathbb{Q}_p}\sum_{\tau}\sum_{\delta}\textbf{1}_{K_p}(\cdots)dbdb',\\
		\mcI_{p}^{j> 0}(x;1):=\mu(I_p'(1))^2p^2\sum_{i\geq -\nu(1-\ov x)}\sum_{j> 0}\int_{\mathbb{Q}_p}\int_{\mathbb{Q}_p}\sum_{\tau}\sum_{\delta}\textbf{1}_{K_p}(\cdots)dbdb'.
	\end{align*}

	We now prove Proposition \ref{keysplitcoprime} in case I and when assuming that $$\nu(x),\ \nu(\overline{x})\geq 0$$
	and therefore $$\nu(1\pm x),\ \nu(1\pm \ov x)\geq 0.$$
	
	By \eqref{matrix1caseI} and \eqref{matrix2caseI}, we have
	\begin{gather}\label{294}
		\nu(f)\geq 1,\ \nu(f')\geq 1\\
		\nonumber
		g=p^{i+j}b'-\frac{(x-1)(\overline{x}-1)}{2}p^{i-j}\in\Zp,\ g'=-\frac{p^{j-i}}{2}-p^{i+j}b\in\mathbb{Z}_p.
	\end{gather}
	
	Also since (see \eqref{agsum})
	$$p^{j-i}=a+g'\in \Zp\hbox{ and }-(x-1)(\ov x-1)p^{i-j}=a'+g\in \Zp$$
	we have 
	\begin{equation}\label{j-iboundcaseI}
		0\leq j- i\leq\ \nu(1-x)+\nu(1-\ov x).
	\end{equation}

	From $\nu(e')\geq -1$ and $\nu(f)\geq 1$ we then have 
	\begin{equation}\label{299}
		j\leq 1+\nu(x(\overline{x}-1)).
	\end{equation}
	and likewise
	$$j\leq 1+\nu(\ov x({x}-1))$$
	so that
	\begin{equation}\label{jupperboundcaseI}
		j\leq 1+\frac{\nu(x\ov x({x}-1)(\ov x-1))}{2}.	
	\end{equation}

	\subsubsection*{Localization of $b,b'$}
	Another general observation in Case I is that, since
	\begin{equation}\label{fintegral}
		f=-p^jb'+p^{-j}.\frac{(x+1)(\overline{x}-1)}{2}\in p\Zp,\ f'=(1-\overline{x})p^ib+\frac{1+\overline{x}}{2}p^{-i}\in p\Zp,	
	\end{equation}
	$b$ and $b'$ are contained, respectively, in translates of $p^{-i-\nu(1-\ov x)+1}\Zp$ and $p^{-j+1}\Zp$ (depending only on $x$) whose volumes are $p^{i+\nu(1-\ov x)-1}$ and $p^{j-1}$. 
	
	Similarly since
	\begin{equation}\label{gintegral}
		g=p^{i+j}b'-\frac{(x-1)(\overline{x}-1)}{2}p^{i-j}\in\Zp,\ g'=-\frac{p^{j-i}}{2}-p^{i+j}b\in\mathbb{Z}_p,	
	\end{equation}
	$b$ and $b'$ are contained in translates of $p^{-(i+j)}\Zp$ (depending only on $x$) whose volumes are $p^{i+j}$.
	
	We will use either of these conditions depending on the values of  $$\min(i+\nu(1-\ov x)-1,i+j)\hbox{ and }\min(j-1,i+j).$$

	For this we need to split the discussion into further cases:
	
	\subsubsection*{The case $j=0$} Since $p^{i}+\tau p^{-1}\in\mathbb{Z}_p$  we have $i=-1$ and $\tau\equiv -1\pmod{p}$. Since $i+j=-1$ we have that  $b$ and $b'$ each belong to some translates of $p\mathbb{Z}_p$. Hence we have  
	\begin{equation}\label{296}
		\mcI_{p}^{0}(x;1)\leq \frac{\mu(I_p'(1))^2p^2(p-1)}{\mu(K_p)}\int_{p\mathbb{Z}_p}\int_{p\mathbb{Z}_p}dbdb'\leq \frac{1}{p^2(p-1)}.
	\end{equation}
	
	\subsubsection*{The case $j\geq 1$: $\nu(x)=\nu(\overline{x}-1)=0$} 
	Suppose first that $\nu(x)=\nu(\overline{x}-1)=0.$ From \eqref{generallowerbounds}, \eqref{299} and \eqref{j-iboundcaseI}, we have $$0\leq i\leq j=1.$$
	Finally, since
	$$p^{i+1}-\delta p^{i-1}(1-\overline{x})\in\Zp$$ 
	we conclude that $$i=j=1.$$
	In particular by \eqref{fintegral}, $b$ and $b'$ are contained in single translates of $ \mathbb{Z}_p$.
	
	We need to split the discussion into two further cases:
	
	-- Suppose that, given $x, e, f, z$, there exists at most one $\delta\Mod{p}$ satisfying \eqref{matrix1caseI} and \eqref{matrix2caseI}	we then have (under the assumption that $\nu(x)=\nu(\overline{x}-1)=0$)
	\begin{equation}\label{297}
		\mcI_{p}^{> 0}(x;1)\leq \frac{\mu(I_p'(1))^2(p-1)}{\mu(K_p)}\int_{\mathbb{Z}_p}\int_{\mathbb{Z}_p}dbdb'\leq \frac{1}{p^2(p-1)}.
	\end{equation}	 			
	
	-- Suppose instead that, given $x, e, f, z$, there exists $\delta_1\not\equiv\delta_2\Mod{p}$ satisfying \eqref{matrix1caseI} and \eqref{matrix2caseI}; we then have
	$(\delta_1-\delta_2)p^{-1}z\in\mathbb{Z}_p$; this implies that $\nu(e)\geq 0$ and
	$$(1-x\overline{x})p^{-2}+{f}p^{-1}-
	(1-\overline{x})ep^{-1}-{ef}\in (1-x)p\Zp.$$
	Since the last three terms belong to $p^{-1}\Zp$ we must have
	$$\nu(x\ov x-1)\geq 1.$$
	In all cases, we conclude that
	\begin{equation}\label{2ndcaseIjgeq1}
		\mcI_{p}^{> 0}(x;1)\leq \frac{\mu(I_p'(1))^2(p-1)^2}{\mu(K_p)}\int_{\mathbb{Z}_p}\int_{\mathbb{Z}_p}dbdb'\leq\frac{1}{p^2}\leq  \frac{p^{\nu(x\ov x-1)}}{p^2(p-1)}.
	\end{equation}

	\subsubsection*{The case $j\geq 1$: $\nu(x)+\nu(\overline{x}-1)\geq 1$} 
	Suppose now that $\nu(\overline{x}-1)+\nu(x)\geq 1.$ From \eqref{generallowerbounds}, \eqref{j-iboundcaseI} and \eqref{jupperboundcaseI}  we find that $\mcI_{p}^{> 0}(x;1)$ is bounded by
	\begin{gather*}
		\frac{\mu(I_p'(1))^2p^2(p-1)^2}{\mu(K_p)}\sum_{j=1}^{\frac{\nu(x\overline{x}(x-1)(\overline{x}-1))}{2}+1}\sum_{i=-\nu({x}-1)}^{j}\int_{p^{-(i+\nu(1-\ov x)-1)}\Zp}\int_{p^{-(j-1)}\Zp}dbdb'\\
		\ll \sum_{j=1}^{\frac{\nu(x\overline{x}(x-1)(\overline{x}-1))}{2}+1}\sum_{i=-\nu({x}-1)}^{j}p^{i+j+\nu(1-\ov x)-2}\\
		\ll p^{\nu(x\overline{x}(x-1)(\overline{x}-1))+\nu(1-\ov x)}
	\end{gather*}

	
	Since $\nu(1-x)\geq 0$ we can replace the exponent in $p$ above by the more symmetric expression
	$$\nu(x\overline{x})+2\nu((x-1)(\overline{x}-1)),$$ so that
	combining this with \eqref{296}, \eqref{297}, \eqref{2ndcaseIjgeq1}
	\begin{equation}\label{Ip1}
		\mcI_{p}(x;1)\ll 	p^{\nu(x\overline{x})+2\nu((x-1)(\overline{x}-1))}+\frac{p^{\nu(1-x\ov x)}}{p^3}.
	\end{equation}

	\subsection*{Case II} We assume now that the Iwasawa decomposition of $u$ and $v$ are of the form
	\begin{align} \label{uform2}
		u=\begin{pmatrix}
			p^{i}&&p^{i}b\\
			&1&\\
			&&p^{-i}
		\end{pmatrix}\begin{pmatrix}
			\tau^{-1}&&\\
			&1&\\
			&&1
		\end{pmatrix}k_1,\\
		v=\begin{pmatrix}
			p^{j}&&p^{j}b'\\
			&1&\\
			&&p^{-j}
		\end{pmatrix}\begin{pmatrix}
			\delta&&\\
			&1&\\
			&&1
		\end{pmatrix}k_2,\label{vform2}
	\end{align}
	where $\tau, \delta \in (\mathbb{Z}/p\mathbb{Z})^{\times},$ and $k_1, k_2\in I_p'(1).$  
	
	\begin{defn}\label{defIp2} Let $\mcI_{p}(x;2)$ be  the contribution  to $\mathcal{I}_p(x)$
 of the $u, v$ of the  form \eqref{uform2} and \eqref{vform2}.		
	\end{defn}
	  We first look for some necessary condition for $\mcI_{p}(x;2)$ to be non zero. Later (in Proposition \ref{supportsplitN''}) we will provide more precise bounds when it does not vanish.

	Computing $u^{-1}\gamma(x)Jv$ we see that $f_p^{\mathfrak{n}_p}(u^{-1}\gamma(x)Jv)\neq 0$ if and only if
	\begin{equation}\label{matrix1caseII}
		\begin{pmatrix}
			a+\tau p^{j-1}&e+\delta p^{-1}a-\tau p^{-1}(x-\delta p^{j-1})&z-\tau p^{-1}f\\
			-p^j&x-\delta p^{j-1}&f\\
			p^{i+j}& p^i(1-x)+\delta p^{i+j-1}& g
		\end{pmatrix}\in K_p.
	\end{equation}
	Taking the inverse, we obtain
	\begin{equation}\label{matrix2caseII}
		\begin{pmatrix}
			a'-\delta p^{i-1}(1-\overline{x})&e'+\frac{\tau}{p} a'-\frac{\delta}p(\overline{x}+\tau p^{i-1}(1-\overline{x}))&z'-\frac{\delta}pf'\\
			p^i(1-\overline{x})&\overline{x}+\tau p^{i-1}(1-\overline{x})&f'\\
			p^{i+j}& -p^j+\tau p^{i+j-1}& g'
		\end{pmatrix}\in K_p.
	\end{equation}
	
	Here, as in Case I, $a,e,f,g,z$ are defined in \eqref{lettervalues}	and $a',e',f',g',z'$ as defined in \eqref{letter'values}.

	These conditions imply in particular that	 
	\begin{equation}\label{caseIIbasiclowerbounds}
		i+\nu(1-\ov x)\geq 0,\ j\geq 0,\ i+j\geq 0	
	\end{equation}
	which together with $\nu(x-\delta p^{j-1}),\ \nu(\ov x- \tau p^{i}(1-\overline{x})p^{-1})\geq 0$  implies that
	$$\nu(x),\ \nu(\overline{x})\geq -1.$$
	This proves Proposition \ref{keysplitcoprime} in Case II.
	
	We also note that since $$f, f', g, g',\ z-\tau p^{-1}f,\ z'-\delta p^{-1}f'\in \Zp$$ we have
	$$\nu(z),\nu(z')\geq -1$$
	and
	\begin{equation}\label{zcaseII}
		\frac{1-x\overline{x}}{1-x}p^{-i-j}+\frac{f}{1-x}p^{-i}-\frac{1-\overline{x}}{1-x}ep^{-j}-\frac{ef}{1-x}=z\in p^{-1}\Zp.	
	\end{equation}
	
	We consider   Case II when $p|N'$ and assume that $$\nu(x),\ \nu(\overline{x})\geq 0.$$
	In particular $\nu(1\pm x),\ \nu(1\pm \ov x)\geq 0.$
	
	A few general remarks:
	
	\begin{itemize}
		\item Since $\nu(x)\geq 0$ and $\nu(x-\delta p^{j-1})\geq 0$ we have $j\geq 1$ and therefore since $a+\tau p^{j-1}\in\Zp$ we have
		$a\in\Zp$.
		
		\item  This together with $e+\delta p^{-1}a-\tau p^{-1}(x-\delta p^{j-1})\in\Zp$ implies that
		$$\nu(e)\geq -1.$$ 
		
		\item  Since $f,f',g,g'\in\Zp$, we see that $b,b'$ are contained in translates of respectively
		$$p^{-\min(i+\nu(1-\ov x),i+j)}\Zp\hbox{ and }p^{-\min(j,i+j)}\Zp$$
		which have volumes
		$$p^{\min(i+\nu(1-\ov x),i+j)}\Zp\hbox{ and }p^{\min(j,i+j)}\Zp.$$	
		
	\end{itemize}

	\subsubsection*{The case $\nu(\overline{x}-1)=0$} 

	Since $\overline{x}+\tau p^{i-1}(1-\overline{x})\in\Zp$ we have
	$i\geq 1$ and  since
	\begin{equation}\label{caseII1-ovx0}
		p^i.f+p^{i-j}(1-\ov x)=-g\in \Zp,\ p^jf'-(1-\overline{x})g'=p^{j-i}\in \Zp	
	\end{equation}
	we have $$j=i\geq 1.$$ 
	
	
	Since $$x-\delta p^{j-1},a,e+\delta p^{-1}a-\tau p^{-1}(x-\delta p^{j-1})\in\Zp$$ we have $\nu(e)\geq -1$
	%

	We now look at the variable $z$: under our current assumptions \eqref{zcaseII} becomes
	$$\frac{1-x\overline{x}}{1-x}p^{-2i}+\frac{f}{1-x}p^{-i}-\frac{1-\overline{x}}{1-x}ep^{-i}-\frac{ef}{1-x}=z\in p^{-1}\Zp.$$
	The valuation of the first term is $\geq$ of the minimum of the valuations of the three other terms and of $z$:
	$$\nu(1-x\ov x)-2i\geq \min(-i+\nu(f),\nu(1-\ov x)+\nu(e)-i,\nu(e)+\nu(f),\nu(z)+\nu(1-x))$$
	which yields (since $\nu(1-\ov x), \nu(f)\geq 0$ and $\nu(e),\ \nu(z)\geq -1$)
	$$\nu(1-x\ov x)-2i\geq -i-1$$
	so that
	$$1\leq i=j\leq \nu(1-x\ov x)+1.$$

	Notice also that if $\nu(f)=0$ or $\nu(f')=0$ (which is the generic case) the relations
	$$z-\tau p^{-1}f,\ z'-\delta p^{-1}f' \in\mathbb{Z}_p$$
	uniquely determine $\delta\mods p$ and $\tau\mods p$. On the other hand, if either $f\in p\Zp$ or $f'\in p\Zp$, then $b'$ or $b$ belong to fixed translates of $p^{-i+1}\Zp$ (whose volume is smaller by a factor $p$): 
	Therefore, we have
	\begin{align}\nonumber
		\mcI_{p}(x;2)\ll&\ \frac{\mu(I_p'(1))^2}{\mu(K_p)}\sum_{1\leq i\leq 1+\nu(1-x\overline{x})}\iint_{b\in p^{-i+1}\Zp,b'\in p^{-i+1}\Zp}\sumsum_{\delta,\tau\mods p}dbdb'\\ \nonumber
		&\ +2\frac{\mu(I_p'(1))^2}{\mu(K_p)}\sum_{1\leq i\leq 1+\nu(1-x\overline{x})}\iint_{b\in p^{-i}\Zp,b'\in p^{-i+1}\Zp}\sum_{\tau\mods p}dbdb'\\ \nonumber
		&+\frac{\mu(I_p'(1))^2}{\mu(K_p)}\sum_{1\leq i\leq 1+\nu(1-x\overline{x})}\iint_{b\in p^{-i}\Zp,b'\in p^{-i}\Zp}dbdb'\\
		\ll&\frac{1}{p^4}(1+\nu(1-x\overline{x}))p^{2\nu(1-x\overline{x})+2}\ll 	(1+\nu(1-x\overline{x}))\frac{p^{2\nu(1-x\overline{x})}}{p^2}\label{Ip20}	
	\end{align}
	if $\nu(1-\ov x)=0$.			
		
		\subsubsection*{The case $\nu(1-\overline{x})\geq 1$}	   Since $\overline{x}+\tau p^{i-1}(1-\overline{x})\in\Zp$, $j\geq 1$, $f',g'\in\Zp$ and $p^{j-i}=p^jf'-(1-\overline{x})g'$ we have
		$$i\geq 1-\nu(1-\overline{x}),\ j\geq i+1.$$
		%
		By \eqref{zcaseII} we have (since $\nu(f)\geq 0$, $\nu(e),\nu(z),\nu(1-x)\geq 0$)			
		$$\nu(1-x\ov x)-i-j\geq \min(-i,-j+\nu(1-\ov x)-1,-1)$$
		or equivalently 
		\begin{equation}\label{i+jcaseII}
			i+j\leq \nu(1-x\ov x)+\max(i,j+1-\nu(1-\ov x),1).	
		\end{equation}
		
		Also since
		$p^i.f+p^{i-j}(1-\ov x)=-g\in \Zp$
		we have $i-j+\nu(1-\ov x)\geq \min(i,0)$ or equivalently 
		\begin{equation}\label{244}
			j\leq \max(i, 0)+\nu(1-\overline{x}).
		\end{equation}
		If $i\leq 0,$ this gives $j\leq \nu(1-\overline{x})$ and by \eqref{i+jcaseII}
		\begin{equation}\label{ijbound1}i+j\leq \nu(1-x\ov x)+1.	
		\end{equation}
		
		We have therefore
		$$1-\nu(1-\ov x)\leq i\leq 0,\ 1\leq j\leq \nu(1-x\overline{x})+1-i.$$	
		The contribution of this configuration is bounded by 
		\begin{align}\nonumber
			\ll&\ \frac{\mu(I_p'(1))^2}{\mu(K_p)}\sumsum_\stacksum{1-\nu(1-\ov x) \leq i\leq 0}{1\leq j\leq \nu(1-x\overline{x})+1-i}\iint_\stacksum{b\in p^{-i-\nu(1-\ov x)+1}\Zp}{b'\in p^{-j+1}\Zp}\sumsum_{\delta,\tau\mods p}1dbdb'\\ \nonumber
			&\ +\frac{\mu(I_p'(1))^2}{\mu(K_p)}\sumsum_\stacksum{1-\nu(1-\ov x) \leq i\leq 0}{1\leq j\leq \nu(1-x\overline{x})+1-i}\iint_\stacksum{b\in p^{-i-\nu(1-\ov x)+1}\Zp}{b'\in p^{-j}\Zp}\sum_{\tau\mods p}1dbdb'\\ \nonumber
			&\ +\frac{\mu(I_p'(1))^2}{\mu(K_p)}\sumsum_\stacksum{1-\nu(1-\ov x) \leq i\leq 0}{1\leq j\leq \nu(1-x\overline{x})+1-i}\iint_\stacksum{b\in p^{-i-\nu(1-\ov x)}\Zp}{b'\in p^{-j+1}\Zp}\sum_{\delta\mods p}1dbdb'\\ \nonumber
			&\ +\frac{\mu(I_p'(1))^2}{\mu(K_p)}\sumsum_\stacksum{1-\nu(1-\ov x) \leq i\leq 0}{1\leq j\leq \nu(1-x\overline{x})+1-i}\iint_\stacksum{b\in p^{-i-\nu(1-\ov x)}\Zp}{b'\in p^{-j}\Zp}dbdb'\end{align}
		and using \eqref{ijbound1} this is bounded by	
		\begin{equation}\ll 	(1+\nu(1-x\overline{x}))(1+\nu(1-\overline{x}))\frac{p^{\nu(1-x\overline{x})+\nu(1-\ov x)}}{p^3}.\label{Ip20caseII1}
		\end{equation}
		
		If $i\geq 1$ then by \eqref{244} we have
		$j\leq \nu(1-\ov x)+i$ and by \eqref{i+jcaseII}
		$$j\leq \nu(1-x\ov x)+1\hbox{ and  }i\leq j-1\leq \nu(1-x\ov x).$$	
		
		\begin{align}\nonumber
			\ll&\ \frac{\mu(I_p'(1))^2}{\mu(K_p)}\sumsum_\stacksum{1 \leq i\leq \nu(1-x\ov x)}{i+1\leq j\leq \nu(1-x\overline{x})+1}\iint_\stacksum{b\in p^{-i-\nu(1-\ov x)+1}\Zp}{b'\in p^{-j+1}\Zp}\sumsum_{\delta,\tau\mods p}1dbdb'\\ \nonumber
			&\ +\frac{\mu(I_p'(1))^2}{\mu(K_p)}\sumsum_\stacksum{1 \leq i\leq \nu(1-x\ov x)}{i+1\leq j\leq \nu(1-x\overline{x})+1}\iint_\stacksum{b\in p^{-i-\nu(1-\ov x)+1}\Zp}{b'\in p^{-j}\Zp}\sum_{\tau\mods p}1dbdb'\\ \nonumber
			&\ +\frac{\mu(I_p'(1))^2}{\mu(K_p)}\sumsum_\stacksum{1 \leq i\leq \nu(1-x\ov x)}{i+1\leq j\leq \nu(1-x\overline{x})+1}\iint_\stacksum{b\in p^{-i-\nu(1-\ov x)}\Zp}{b'\in p^{-j+1}\Zp}\sum_{\delta\mods p}1dbdb'\\ \nonumber
			&\ +\frac{\mu(I_p'(1))^2}{\mu(K_p)}\sumsum_\stacksum{1 \leq i\leq \nu(1-x\ov x)}{i+1\leq j\leq \nu(1-x\overline{x})+1}\iint_\stacksum{b\in p^{-i-\nu(1-\ov x)}\Zp}{b'\in p^{-j}\Zp}dbdb'\end{align}
		which is bounded by	
		\begin{equation}\ll 	(1+\nu(1-x\overline{x}))^2\frac{p^{2\nu(1-x\overline{x})+\nu(1-\ov x)}}{p^3}.\label{Ip20caseII2}
		\end{equation}

		Combining \eqref{Ip20} with the bounds \eqref{Ip20caseII1} and \eqref{Ip20caseII2} for $\nu(1-\ov x)\geq 1$ we obtain 
		\begin{equation}\label{Ip2bound}
			\mcI_{p}(x;2)\ll (1+\nu(1-x\overline{x}))(1+\nu(1-x\overline{x})+\nu(1-\ov x))\frac{p^{2\nu(1-x\overline{x})+\max(1,\nu(1-\ov x))}}{p^3}.
		\end{equation}

		\subsection*{Case III} We assume now that the Iwasawa decomposition of $u$ and $v$ are of the form
		\begin{align} \label{uform3}
			u=&\begin{pmatrix}
				p^{i}&&p^ib\\
				&1&\\
				&&p^{-i}
			\end{pmatrix}\begin{pmatrix}
				1&&\mu\\
				&1&\\
				&&1
			\end{pmatrix}J\begin{pmatrix}
				\tau&&\\
				&1&\\
				&&1
			\end{pmatrix}\gamma_1,\\
			v=&\begin{pmatrix}
				p^{j}&&p^{j}b'\\
				&1&\\
				&&p^{-j}
			\end{pmatrix}\begin{pmatrix}
				\delta&&\\
				&1&\\
				&&1
			\end{pmatrix}\gamma_2,\label{vform3}
		\end{align}
		where $\tau, \delta \in (\mathbb{Z}/p\mathbb{Z})^{\times},$ $\mu\in \mathbb{Z}/p\mathbb{Z},$ and $\gamma_1, \gamma_2\in I_p'(1).$ 
		\begin{defn}\label{defIp3}
			Let $\mcI_{p}(x;3)$ be  the contribution  to $\mathcal{I}_p(x)$ of the $u, v$ of the  form \eqref{uform3}, \eqref{vform3}.
		\end{defn}
		
		  We first look for some necessary condition for $\mcI_{p}(x;3)$ to be non zero. Later (in Proposition \ref{supportsplitN''}) we will provide more precise bounds when it does not vanish.

		We have $f_p^{\mathfrak{n}_p}(u^{-1}\gamma(x)Jv)\neq 0$ if and only if
		\begin{equation}\label{matrix1caseIII'}
			\begin{pmatrix}
				a&e+\delta p^{-1}a&z\\
				-p^j&x-\delta p^{j-1}& f\\
				p^{i+j}-\tau p^{j-1}& p^i(1-x)+\delta p^{i+j-1}-\frac{\tau}{p}(x-\delta p^{j-1})&g-\tau p^{-1}f
			\end{pmatrix}\in K_p.
		\end{equation}
		
		Taking the inverse of \eqref{matrix1caseIII'} we then obtain that
		\begin{equation}\label{matrix2caseIII'}
			\begin{pmatrix}
				a'-\delta p^{i-1}(1-\overline{x})&e'-\frac{\delta}p\overline{x}+\frac{\tau}{p}(z'-\delta p^{-1}f')&z'-\delta p^{-1}f'\\
				p^i(1-\overline{x})&\overline{x}+\tau p^{-1}f'& f'\\
				p^{i+j}& -p^j+\tau p^{-1}g'&g'
			\end{pmatrix}\in K_p.
		\end{equation}
		
		These conditions imply  in particular that	 
		\begin{equation}\label{caseIIIbasiclowerbounds}
			i+j\geq 0,\ j\geq 1	, i+\nu(1-\ov x)\geq 0
		\end{equation}
		and therefore
		$\nu(x)\geq 0$. Also  $f',\overline{x}+\tau p^{-1}f'\in\Zp$ implies that $\nu(\ov x)\geq -1$. This proves  \eqref{keysplitcoprime} (in a stronger form) in Case III.

		We assume now that $x,\ov x\in\Zp$. We have
		$$f',g'\in p\Zp$$

		So $j\geq 1,$ $\nu(f')\geq 1$ and $\nu(g')\geq 1.$ Consequently, since $$p^{j-i}=p^jf'-(1-\overline{x})g'\in p\Zp$$ (see \eqref{agsum}) we have
		\begin{equation}\label{caseIIIijbound}i\leq j-1	
		\end{equation}
		and $a\in p\Zp$ which implies that $e\in\Zp.$
		
		Since $z\in\Zp$ by \eqref{zalgebraic} we have 
		$$\frac{1-x\overline{x}}{1-x}p^{-i-j}+\frac{f}{1-x}p^{-i}-\frac{1-\overline{x}}{1-x}ep^{-j}-\frac{ef}{1-x}=z\in\Zp$$
		and since $f,e\in\Zp$ we have
		$$	\nu(1-x\overline{x})-i-j\geq \min\{\nu(1-x),-i,-j+\nu(1-\overline{x}), 0\}=\min\{-i, -j+\nu(1-\overline{x}), 0\}.
		$$
		or
		\begin{equation}\label{boundi+jcaseIII}
			i+j\leq \nu(1-x\overline{x})+\max\{i, j-\nu(1-\overline{x}), 0\}.	
		\end{equation}
		Also, we have $\nu(g)\geq -1$ so from the relation $p^if+g=p^{i-j}(1-\overline{x})$ we conclude that 
		\begin{equation}\label{j-ibound}
			j-i-\nu(1-\overline{x})\leq \max\{1,-i\}	
		\end{equation}
		and subtracting $i$ from \eqref{boundi+jcaseIII} we obtain
		\begin{equation}\label{caseIIIjbound}
			1\leq j\leq \nu(1-x\overline{x})+\max\{1, -i\}.	
		\end{equation}
		
		\subsubsection*{Localization of $b$ and $b'$}	Finally we observe that since $f\in\Zp,\ f'\in p\Zp$, we see from the expression of $f$ and $f'$ in \eqref{lettervalues} and \eqref{letter'values} that $b$ and $b'$ are contained in translates of
		$$p^{1-i-\nu(1-\ov x)}\Zp\hbox{ and }p^{-j}\Zp\hbox{ respectively}$$
		which have volumes
		$$p^{i+\nu(1-\ov x)-1}\hbox{ and }p^j.$$
		
		Moreover we notice that if $\nu(f)=0$ then since $g-\tau p^{-1}f\in\Zp$ the congruence class $\tau$ is uniquely determined by $g$ and $f$ (which depend on $x$ and $b'$) while for $\nu(f)\geq 1$ there is no constraint on $\tau$ but $b'$ varies over a translate of $p^{-j+1}\Zp$.
		
		\subsubsection*{The case $\nu(1-\ov x)=0$} We have $1\leq j\leq \max\{i+1,0\}\leq j$ so $j\leq i+1$ and 
		$$j=i+1.$$
		Therefore, the contribution from this case to $\mcI_{p}(x;3)$ is bounded by 
		\begin{align*}
			&\frac{\mu(I_p'(1))^2}{\mu(K_p)}\sum_{\mu\in \mathbb{Z}/p\mathbb{Z}}\sum_{0\leq i\leq \nu(1-x\overline{x})}\sum_{\tau}\sum_{\delta}\int_{p^{1-i}\mathbb{Z}_p}db\int_{p^{-i-1}\mathbb{Z}_p}db'\\
			\ll&{p\mu(I_p'(1))^2(1+\nu(1-x\overline{x}))p^{2\nu(1-x\overline{x})}}\ll (1+\nu(1-x\overline{x}))\frac{p^{2\nu(1-x\overline{x})}}{p^3}.
		\end{align*}
		
		\subsubsection*{The case $\nu(1-\ov x)\geq 1$} We have $j-\nu(1-\ov x)\leq j-1$ and (since $i\leq j-1$) we have
		$$i+j\leq \nu(1-x\overline{x})+j-1\Longleftrightarrow i\leq \nu(1-x\overline{x})-1.$$

		-- If $i\geq 0$ we have by \eqref{caseIIIjbound}
		$$j\leq \nu(1-x\overline{x})+1$$
		and by \eqref{caseIIIijbound} 
		$$i+j\leq 2\nu(1-x\overline{x})+1.$$
		
		-- If $i\leq -1$ we have by \eqref{j-ibound} $$-i,j\leq \nu(1-\ov x)$$ and
		$$i+j\leq \nu(1-x\ov x)$$

		We conclude that the contribution from the case $\nu(1-\ov x)\geq 1$ to
		$\mcI_{p}(x;3)$ is bounded by	(see the paragraph on the localization of $b$ and $b'$)
		\begin{align*}
			\frac{\mu(I_p'(1))^2}{\mu(K_p)}&\sum_{\mu\in \mathbb{Z}/p\mathbb{Z}}
			\sumsum_\stacksum{-\nu(1-\ov x)\leq i\leq j-1\leq \nu(1-x\overline{x})}
			{i+j\leq 2\nu(1-x\overline{x})+1} \int_{p^{1-i-\nu(1-\ov x)}\mathbb{Z}_p}\int_{p^{-j+1}\mathbb{Z}_p}\sum_{\tau}\sum_{\delta}1dbdb'\\
			+	\frac{\mu(I_p'(1))^2}{\mu(K_p)}&\sum_{\mu\in \mathbb{Z}/p\mathbb{Z}}
			\sumsum_\stacksum{-\nu(1-\ov x)\leq i\leq j-1\leq \nu(1-x\overline{x})}
			{i+j\leq 2\nu(1-x\overline{x})+1} \int_{p^{1-i-\nu(1-\ov x)}\mathbb{Z}_p}\int_{p^{-j}\mathbb{Z}_p}\sum_{\delta}1dbdb'\\
			&\ll (1+\nu(1-\ov x)+\nu(1-x\ov x))^2\frac{p^{2\nu(1-x\overline{x})+\nu(1-\ov x)}}{p^2}.
		\end{align*}
		
		%
			%
			%
			%

		Putting the above discussion together we then obtain 
		\begin{equation}\label{Ip3}
			\mcI_{p}(x;3)\ll (1+\nu(1-\ov x)+\nu(1-x\ov x))^2\frac{p^{2\nu(1-x\overline{x})+\nu(1-\ov x)}}{p^2}.
		\end{equation}

		\subsection*{Case IV}
		Consider $u$ and $v$ in their Iwasawa forms:
		\begin{align} \label{uform4}
			u=&\begin{pmatrix}
				p^{i}&&p^{i}b\\
				&1&\\
				&&p^{-i}
			\end{pmatrix}\begin{pmatrix}
				\tau&&\\
				&1&\\
				&&1
			\end{pmatrix}\gamma_1,\\
			v=&\begin{pmatrix}
				p^{j}&&p^{j}b'\\
				&1&\\
				&&p^{-j}
			\end{pmatrix}\begin{pmatrix}
				1&&\mu\\
				&1&\\
				&&1
			\end{pmatrix}J\begin{pmatrix}
				\delta&&\\
				&1&\\
				&&1
			\end{pmatrix}\gamma_2,\label{vform4}
		\end{align}
		where $\tau, \delta \in (\mathbb{Z}/p\mathbb{Z})^{\times},$ $\mu\in \mathbb{Z}/p\mathbb{Z},$ and $\gamma_1, \gamma_2\in I_p'(1).$ 
		
		\begin{defn}\label{defIp4} Let $\mcI_{p}(x;4)$ be  the contribution  to $\mathcal{I}_p(x)$
 of the $u, v$ of the  form \eqref{uform4} and \eqref{vform4}.		
	\end{defn}

 We first look for some necessary condition for $\mcI_{p}(x;4)$ to be non zero. Later (in Proposition \ref{supportsplitN''}) we will provide more precise bounds when it does not vanish.

		Then $f_p^{\mathfrak{n}_p}(u^{-1}\gamma(x)Jv)\neq 0$ if and only if
		\begin{equation}\label{matrix1caseIV}
			\begin{pmatrix}
				a+\tau p^{j-1}&e+\frac{\delta}{p}z-\frac{\tau}{p}(x+\frac{\delta}{p}f)&z-\frac{\tau}{p}f\\
				-p^j&x+\frac{\delta}{p}f& f\\
				p^{i+j}& p^i(1-x)+\frac{\delta}{p}g&g
			\end{pmatrix}\in K_p.
		\end{equation}
		
		Taking the inverse of \eqref{matrix1caseIV} we then obtain that
		\begin{equation}\label{matrix2caseIV}
			\begin{pmatrix}
				a'&e'+\frac{\tau}{p}a'&z'\\
				p^i(1-\overline{x})&\overline{x}+\tau p^{i-1}(1-\overline{x})& f'\\
				p^{i+j}-\frac{\delta}{p} p^{i}(1-\overline{x})& -p^j+p^{i+j}\frac{\tau}{p}-\frac{\delta}{p}(\overline{x}+\frac{\tau}{p}p^{i}(1-\overline{x}))&g'-\frac{\delta}{p}f'
			\end{pmatrix}\in K_p.
		\end{equation}
		These conditions imply in particular that
		$$j\geq 0, i+j\geq 0,\ i+\nu(1-\ov x)\geq 1.$$
		
		Moreover since $f,x+\delta f/p\in\Zp$ we have $\nu(x)\geq -1$. Since $$p^i(1-\ov x)\in p\Zp,\ov x+\tau p^i(1-\ov x)/p\in\Zp$$ we have $\nu(\ov x)\geq 0$. This proves  Proposition  \ref{keysplitcoprime} in case IV (in a stronger form).

		We assume now that $x,\ov x\in\Zp$. We have
		$$f\in p\Zp, z\in \Zp.$$
		Since $z\in\Zp$ by \eqref{zalgebraic} we have 
		$$\frac{1-x\overline{x}}{1-x}p^{-i-j}+\frac{f}{1-x}p^{-i}-\frac{1-\overline{x}}{1-x}ep^{-j}-\frac{ef}{1-x}=z\in\Zp$$
		and since $f\in p\Zp,\ e\in p^{-1}\Zp$ we have
		$$	\nu(1-x\overline{x})-i-j\geq \min\{\nu(1-x),-i+1,-j+\nu(1-\overline{x})-1, 0\}=\min\{-i+1, -j+\nu(1-\overline{x})-1, 0\}.
		$$
		or
		$$	i+j\leq \nu(1-x\overline{x})+\max\{i-1,j-\nu(1-\overline{x})+1, 0\}.
		$$
		Also from $g\in\Zp$ and $p^if+g=p^{i-j}(1-\overline{x})$
		we have
		$$j-\nu(1-\ov x)-i\leq \max(-i-1,0)$$
		which implies that
		$$j-\nu(1-\ov x)\leq \max(-1,i)$$
		and
		$$	i+j\leq \nu(1-x\overline{x})+\max\{i+1, 0\}.
		$$
		We also have $f'\in\Zp, g'\in p^{-1}\Zp$ and since
		$$		p^{j-i}=p^jf'-(1-\overline{x})g'$$
		we have
		$$j-i\geq \min(j,\nu(1-\ov x)-1)$$
		and therefore
		$$1-\nu(1-\ov x)\leq i\leq \max(0,j+1-\nu(1-\ov x)).$$

		\subsection*{The case $i\geq 0$} In that case we have
		$$j-\nu(1-\ov x)\leq i,\  j\leq \nu(1-x\ov x)+1$$
		and therefore
		$$0\leq i\leq \nu(1-x\ov x)+2$$	
		so that
		$$i+j\leq 2\nu(1-x\ov x)+3$$
		
		\subsection*{The case $i<0$} We have
		$$1-\nu(1-\ov x)\leq i \leq 0\leq j\leq \nu(1-\ov x)-1$$
		and
		$$	i+j\leq \nu(1-x\overline{x})\leq 2\nu(1-x\ov x)+3.
		$$
		
		\subsubsection*{Localization of $b$ and $b'$}
		we observe that since $f\in p\Zp,\ f'\in \Zp$, we see from the expression of $f$ and $f'$ in \eqref{lettervalues} and \eqref{letter'values} that $b$ and $b'$ are contained in translates of
		$$p^{-i-\nu(1-\ov x)}\Zp\hbox{ and }p^{-j+1}\Zp\hbox{ respectively}$$
		which have volumes
		$$p^{i+\nu(1-\ov x)}\hbox{ and }p^{j-1}.$$
		
		Moreover we notice that if $\nu(f')=0$, since $g'-\frac{\delta}{p}f'\in\Zp$ the congruence class $\delta$ is uniquely determined by $g'$ and $f'$ (which depend on $x$ and $b$) while for $\nu(f')\geq 1$ there is no constraint on $\delta$ but $b$ varies over a translate of $p^{i+\nu(1-\ov x)+1}\Zp$ (whose volume is smaller by a factor $p$).
		
		Arguing as before, we deduce that
		\begin{align*}
			\mcI_{p}(x;4)&\ll  (1+\nu(1-\ov x)+\nu(1-x\ov x))^2\frac{p}{p^4}p^{2\nu(1-x\ov x)+3+\nu(1-\ov x)-1+1}\\
			&\ll (1+\nu(1-\ov x)+\nu(1-x\ov x))^2p^{2\nu(1-x\ov x)+\nu(1-\ov x)}	
		\end{align*}

		Putting the above discussion together we then obtain 
		\begin{equation}\label{Ip4}
			\mcI_{p}(x;4)\leq {\mu(I_p'(1))\big[4\nu(1-\overline{x})+\nu(1-x\overline{x})^2\big]p^{2\nu(1-x\overline{x})}}. 
		\end{equation}

		Combining \eqref{Ip1}, \eqref{Ip2bound}, \eqref{Ip3} with \eqref{Ip4} we then obtain 
		\begin{align*}
			\mcI_{p}(x)\leq & \frac{2+2p^{2\nu(1-x\overline{x})}}{p}+{\nu(x\overline{x}(x-1)(\overline{x}-1))^2p^{\nu(x\overline{x})+2\nu((\overline{x}-1)(x-1))}}\\
			&+\frac{8\nu((1-x\overline{x})(1-x\overline{x}))^2p^{2\nu(1-x\overline{x})}}{p^2}+\frac{4\nu(1-\overline{x})\nu(1-x\overline{x})^2p^{2\nu(1-x\overline{x})}}{p}.
		\end{align*}
		Consequently, \eqref{301} follows readily. 
	\end{proof}

	\subsection{Analysis of the non-integral cases}
	We now deal with the remaining cases when $x$ or $\ov x$ have negative valuation and obtain sharper estimates for the sums $\mcI_{p}(x;i),\ i\leq 4$.
	\begin{prop}\label{supportsplitN''}
		Let the notation be as before. Let $p$ be a prime divisor of $N'.$ Let $x\in E^\times\!-E^1$ be such that $$\nu(x), \nu(\ov x)\geq -1,\hbox{ and  } \nu(x)\hbox{ or }\nu(\ov x)= -1.$$ 
		Then 
		\begin{enumerate}
			\item[(I)] Let $\mcI_{p}(x;1)$ be as in Definition \ref{defIp1}; we have $\mcI_{p}(x;1)=0$ unless   
			$\nu(x)=-1$ and $\nu(1-\overline{x})\geq 1$, in which case 
			$$
			\mcI_{p}(x;1)\ll
			\begin{cases}
			p^{-4},& \text{if $\nu(x)=-1$ and $\nu(1-\ov x)=1$},\\
			p^{-2},& \text{if $\nu(x)=-1$ and $\nu(1-\ov x)\geq 2$}.
		   \end{cases} 
			$$  
			\item[(II):] Let $\mcI_{p}(x;2)$ be as in Definition \ref{defIp2}; we have $\mcI_{p}(x;2)= 0$ unless either $\nu(x)=-1,$ $\nu(\overline{x})=-1,$ or $\nu(x)=-1,$ $\nu(\overline{x})=\nu(1-\ov x)=0$, in which case 
			\begin{align*}
				\mcI_{p}(x;2)\ll
				\begin{cases}
					p^{-4},& \text{if $\nu(x)=-1 $ and $\nu(\overline{x})=-1$},\\
					p^{-3},& \text{if $\nu(x)=-1 $ and $\nu(\overline{x})=\nu(1-\ov x)=0$}.
				\end{cases}
			\end{align*}
			\item[(III):] Let $\mcI_{p}(x;3)$ be as in Definition \ref{defIp3}; we have $\mcI_{p}(x;3)= 0$ unless $\nu(x)\geq 0$ and $\nu(\overline{x})=-1$, in which case 
			$$
			\mcI_{p}(x;3)\ll 	
			\begin{cases}
				p^{-3},& \text{if $\nu(x)=0$ and $\nu(\overline{x})=-1$},\\
				p^{2\nu(1-x\ov x)-1},& \text{if $\nu(x)\geq 1$ and $\nu(\overline{x})=-1$}.
			\end{cases} 
			$$  
			\item[(IV):] Let $\mcI_{p}(x;4)$ be as in Definition \ref{defIp4}; we have $\mcI_{p}(x;4)= 0$ unless $\nu(x)=-1$ and $\nu(\overline{x})\geq 0$, in which case 
							\begin{align*}
					\mcI_{p}(x;4)\ll 
							\begin{cases}
							p^{-3}, &\text{if $\nu(x)=-1$ and $\nu(\overline{x})=\nu(1-\overline{x})=0;$}\\
							p^{2\nu(1-x\ov x)-1}, &\text{if $\nu(x)=-1$ and $\nu(\overline{x})\geq 1;$}\\
							p^{-3}, &\text{if $\nu(x)=-1$ and $\nu(1-\overline{x})=1;$}\\
				\nu(1-\ov x)p^{2\nu(1-\ov x)-5},& \text{if $\nu(x)=-1$ and $\nu(1-\overline{x})\geq 2$}.
							\end{cases}
				\end{align*}
		\end{enumerate}
	\end{prop}

	\begin{proof}{} 	We shall keep the notation of the proof of Proposition \ref{keysplitcoprime}  and investigate the four cases therein.

		\subsection*{Case I} We start by recalling the integrality conditions in that case:
		\begin{equation}\label{matrix1caseIbis}
			\begin{pmatrix}
				a&e+\frac{\delta}pz&z\\
				-p^j&x+\frac{\delta}pf&f\\
				p^{i+j}+\tau p^{j-1}&p^{i}(1-x)+\frac{\delta}pg-\frac{\tau}{p}(x+\delta p^{-1}f)&g-\tau p^{-1}f
			\end{pmatrix}\in K_p,
		\end{equation}
		\begin{equation}\label{matrix2caseI2bis}
			\begin{pmatrix}
				a'&e'+\frac{\tau}{p}z'&z'\\
				p^i(1-\overline{x})&\overline{x}+\frac{\tau}{p}f'&f'\\
				p^{i+j}-\delta p^{i-1}(1-\overline{x})&-p^{j}+\frac{\tau}{p}g'-\frac{\delta}p(\overline{x}+\frac{\tau}{p}f')&g'-\frac{\delta}pf'
			\end{pmatrix}\in K_p.
		\end{equation}

		Suppose that $\nu(\overline{x})=-1$. The inclusions$$f',\ov x+\tau p^{-1}f'\in\Zp$$ imply that $\nu(f')=0$ and since $\nu(g'-\delta p^{-1}f')\in\Zp$, we have $\nu(g')=-1$ which in turn would imply (since $j\geq 0$) $$\nu(-p^{j}+\tau p^{-1}g'-\delta p^{-1}(\overline{x}+\tau p^{-1}f'))=\nu(\tau p^{-1}g')=-2$$ a contradiction. 
		
		We have therefore $\nu(\ov x)\geq 0$ and $\nu(x)=-1$. It follows from $$\nu(x+\delta p^{-1}f),\ \nu(f)\geq 0$$ that
		$\nu(f)=0$ and since  $\nu(g-\tau p^{-1}f)\geq 0$ we then have $\nu(g)=-1$.

		We have also
		$$p^{i}(1-x)+\delta p^{-1}g-\tau p^{-1}(x+\delta p^{-1}f)\in \Zp.$$
		In the above expression, the three terms on the lefthand side have respective valuations $i-1,\ -2, \geq -1$; this forces $i=-1$ and since $p^i(1-\overline{x})\in\mathbb{Z}_p$ we have the congruence $$\nu(1-\overline{x})\geq 1;$$
		this implies that $\nu(\ov x)=0$ and $\nu(x\overline{x}-1)=-1$.
		This implies also that $\nu(f')\geq 1$ and $\nu(g')\geq 0$.
		
		If $\nu(1-\ov x)=1$, then since
		$$p^{-1+j}-\delta p^{-2}(1-\overline{x})\in\Zp$$
		and the second term has valuation $-1$ we must have $j=0$.
		
		If $\nu(1-\ov x)\geq 2$ then since
		$$(1-x\overline{x})p^{1-j}+{f}p-(1-\overline{x})ep^{-j}-{ef}=(1-x)z\in p^{-1}\Zp$$
		the terms above have respective valuations
		$$=-j,\ \geq 1,\ \geq 1-j,\ \geq -1$$
		we conclude that $0\leq j\leq 1$. Moreover, looking at the $(3,1)$-th entry of \eqref{matrix2caseI2bis}, we derive that $p^{-1+j}\in \Zp,$ implying that $j\geq 1.$ So the assumption that $\nu(1-\ov x)\geq 2$ forces that $j=1.$ 
		
		\subsubsection*{Localization of $b$ and $b'$}
		Since $g'\in\Zp$, we see that $b$ belongs to a translate of $p^{1-j}\Zp$ and since $a'\in \Zp$ we see that $b'$ belongs to a translate of $p^{-i-j}\Zp=p^{1-j}\Zp.$ In particular, the translations depend \textit{only} on $x.$

		\subsubsection*{Localization of $\delta$ and $\tau$}
		For fixed $b$ and $b',$ we show that $\delta$ and $\tau$ are determined uniquely. Since $\nu(x)=-1$, $\nu(f)=0$ and $\delta\mods p$ is determined by $f$.
		
		If $j=0$ then since $p^{-1}+\tau p^{-1}$ we have $\tau\equiv -1\mods p$.
		Suppose now that $j=1$; we have, by considering the $(3,2)$-th entry of \eqref{matrix2caseI2bis}, that 
		$${\tau}(g'-\delta p^{-1}f')-{\delta}\overline{x}\in p\Zp$$
		so if $\nu(g'-\delta p^{-1}f')=0$, $\tau\mods p$ is determined. 
		
		Otherwise $\nu(z')=0$ because the last column of \eqref{matrix2caseI2bis} cannot be divisible by $p$ and the last two entries are. In that case, the condition 
		$e'+\frac{\tau}{p}z'\in\Zp$ determines $\tau\mods p$ in terms of $e'$ and $z'$.
		
		Hence the corresponding contribution in the case $\nu(1-\ov x)=1$ to $\mcI_p(x;1)$ is 
		$$
		\ll \frac{1}{p^4}\sum_{\mu\in \mathbb{Z}_p/p\mathbb{Z}_p}\sum_{\mu'\in \mathbb{Z}_p/p\mathbb{Z}_p}\int_{p\Zp}\int_{p\mathbb{Z}_p}db'db\ll p^{-4}.
		$$
		 and the corresponding contribution in the case$\nu(1-\ov x)\geq 2$ to $\mcI_p(x;1)$ is 
		 $$
		 \ll \frac{1}{p^4}\sum_{\mu\in \mathbb{Z}_p/p\mathbb{Z}_p}\sum_{\mu'\in \mathbb{Z}_p/p\mathbb{Z}_p}\int_{\Zp}\int_{\mathbb{Z}_p}db'db\ll p^{-2}.
		 $$

		In conclusion we obtain that
		$$\mcI_{p}(x;1)\ll \delta_{\nu(1-\ov x)=1}p^{-4}+\delta_{\nu(1-\ov x)\geq 2}p^{-2}.$$

		\subsection*{Case II}
		We suppose we are in case II and assume that $\nu(x)\geq 0$ and $\nu(\overline{x})=-1.$ 
		Considering \eqref{matrix1caseII} and \eqref{matrix2caseII} we see that
		$$i=1,\ j\geq 1.$$
		
		Also from the above equations, we have $$p^if+g=(\overline{x}-1)p^{i-j},\ f, g\in \mathbb{Z}_p.$$
		So $i-j\geq 1,$ and $j\leq 0.$ A contradiction! 
		
		Hence, we only have the following two possible cases:
		\begin{enumerate}
			\item[(i)]  $\nu(x)=\nu(\overline{x})=-1.$ We then have $j=0$, $i=1$, $\tau=1\mods p$ and $\delta= px\mods{p}.$ Moreover, since $f, f'\in\mathbb{Z}_p$ we see that $b, b'$ belong to translates of $\Zp$ determined by $x$ and $\overline{x}.$ We then have			$$
			\mcI_{p}(x;2)\ll p^{-4}.
			$$
			\item[(ii)] $\nu(x)=-1 $ and $\nu(\overline{x})\geq 0.$ We have $j=0$ and $\delta\equiv px\mods{p}.$ Also, since $f\in\mathbb{Z}_p$ we see that $b'$  belongs to a translate of $\Zp$ determined by $x$ and $\overline{x}.$ 
			Since $$-p^j+\tau p^{i+j-1}\in\mathbb{Z}_p$$ we obtain $i\geq 1$, and since $$a\in p^{-1}\mathbb{Z}_p,\ g'\in \mathbb{Z}_p\hbox{ and }a+g'=-p^{j-i}=-p^{-i}\in p^{-1}\mathbb{Z}_p,$$ we have $i\leq 1$ and therefore $i=1$.
			
			Since $a(1-\overline{x})+f'=\overline{x}p^{-1}$ and $a+\tau p^{-1}\in \mathbb{Z}_p$ we have 
			$$
			f'\in \overline{x}p^{-1}+\tau (1-\overline{x})p^{-1}+\mathbb{Z}_p;
			$$
			but $f'\in \mathbb{Z}_p$, therefore  $$\overline{x}p^{-1}+\tau (1-\overline{x})p^{-1}\in \mathbb{Z}_p,$$ 
			which implies that $\nu(\ov x)=\nu(1-\ov x)=0$ and $\tau\mods p$ is uniquely determined by $\overline{x}.$ Finally, since $g'\in \mathbb{Z}_p$, one has $b\in \frac{1}{2}p^{-2}+p^{-1}\mathbb{Z}_p.$ 
			
			It follows that in this case we have
			$$
			\mcI_{p}(x)\ll \frac{1}{p^4} \int_{\mathbb{Z}_p}db'\int_{p^{-1}\mathbb{Z}_p}db=\frac{1}{p^3}.
			$$
		\end{enumerate}

		\subsection*{Case III}
		Consider now Case III; we recall the two integrality conditions \eqref{matrix1caseIII'} and \eqref{matrix2caseIII'}:
		\begin{equation}\label{matrix1caseIV'2}
			\begin{pmatrix}
				a&e+\delta p^{-1}a&z\\
				-p^j&x-\delta p^{j-1}& f\\
				p^{i+j}-\tau p^{j-1}& p^i(1-x)+\delta p^{i+j-1}-\frac{\tau}{p}(x-\delta p^{j-1})&g-\tau p^{-1}f
			\end{pmatrix}\in K_p.
		\end{equation}
		\begin{equation}\label{matrix2caseIV'2}
			\begin{pmatrix}
				a'-\delta p^{i-1}(1-\overline{x})&e'-\frac{\delta}p\overline{x}+\frac{\tau}{p}(z'-\delta p^{-1}f')&z'-\delta p^{-1}f'\\
				p^i(1-\overline{x})&\overline{x}+\tau p^{-1}f'& f'\\
				p^{i+j}& -p^j+\tau p^{-1}g'&g'
			\end{pmatrix}\in K_p.
		\end{equation}
		
$$
			\begin{cases}
				a=-\frac{1}{2}p^{j-i}-p^{i+j}b\\
				e=\frac{p^{-i}(1+x)}{2}-p^{i}b(1-x)\\
				f=-p^jb'+p^{-j}.\frac{(x+1)(\overline{x}-1)}{2}\\
				g=p^{i+j}b'-\frac{(x-1)(\overline{x}-1)}{2}p^{i-j}\\
				z=-\frac{1}{2}p^{j-i}b'+p^{-i-j}y-p^{i+j}bb'+p^{i-j}b\frac{(x-1)(\overline{x}-1)}{2}.
			\end{cases}
$$
		where $$y=\frac{x\overline{x}+3\overline{x}-x+1}{4}.$$
		Then one has an explicit algebraic relation 
		\begin{equation}\label{zalgebraicbis}
			z=\frac{1-x\overline{x}}{1-x}p^{-i-j}+\frac{f}{1-x}p^{-i}-\frac{1-\overline{x}}{1-x}ep^{-j}-\frac{ef}{1-x}.
		\end{equation}
		
		\begin{equation}\label{letter'valuesbis}
			\begin{cases}
				a'=-\frac{(1-x)(1-\overline{x})}{2}p^{i-j}-p^{i+j}b'\\
				e'=\frac{(x-1)(\overline{x}+1)}{2}p^{-j}+p^{j}b'\\
				f'=(1-\overline{x})p^ib+\frac{1+\overline{x}}{2}p^{-i}\\
				g'=p^{i+j}b-\frac{1}{2}p^{j-i}\\
				z'=-p^{i-j}b\frac{(x-1)(\overline{x}-1)}{2}+p^{-i-j}\overline{y}-p^{i+j}bb'+\frac{1}{2}p^{j-i}b'
			\end{cases}
		\end{equation}
		and one notes the algebraic relation
	$$
			z'
			=\frac{1-x\overline{x}}{1-\overline{x}}p^{-i-j}+\frac{e'}{1-\overline{x}}p^{-i}-\frac{1-x}{1-\overline{x}}f'p^{-j}-\frac{e'f'}{1-\overline{x}}.
	$$
		We also recall that
		\begin{gather*}
			a+g'=-p^{j-i},\ a'+g=-{(1-x)(1-\overline{x})}p^{i-j},\\ p^if+g=p^{i-j}(\ov x-1),\ p^jf'-(1-\overline{x})g'=p^{j-i}\nonumber
		\end{gather*}

		Suppose that $\nu(x)=-1.$ We have $j=0$ and since $p^{i+j}-\tau p^{j-1}\in\mathbb{Z}_p$ we have $i=-1,$ which contradicts the condition $p^{i+j}\in\mathbb{Z}_p.$ 
		
		So we must have $\nu(x)\geq 0$ and therefore $\nu(\overline{x})=-1.$
		
		This implies that $i,j\geq 1$ and since $$p^{i-j}(\overline{x}-1)=p^if+g\in p^{-1}\mathbb{Z}_p$$ we have $i\geq j\geq 1.$ We also have $f',p^{-1}g'\in \Zp$ and since
		$$p^{j-i}=p^jf'-(1-\overline{x})g'\in \Zp$$
		we have $j\geq i$ and
		$$i=j\geq 1.$$
		\subsubsection*{Localization of $b$ and $b'$}
		We  observe that the conditions $\nu(\ov x)=-1$  and
		$\ov x+\tau f'/p\in \Zp$ imply that $f'\in \Zpt$ and that 
		$$f'\in \tau^{-1}p\ov x+p\Zp.$$		
		We also note that since
		$$f'=(1-\overline{x})p^ib+\frac{1+\overline{x}}{2}p^{-i}$$  we have
		$$(1-\overline{x})p^ib+\frac{1+\overline{x}}{2}p^{-i}-\tau^{-1}p\ov x\in p\Zp$$
		which implies that $b$	belongs to a translate (depending on $x$ and $\tau$) of
		$$(1-\overline{x})^{-1}p^{1-i}\Zp=p^{2-i}\Zp.$$
		We also have
		$$p^if+g=\ov x-1,\ g-\tau p^{-1}f\in\Zp$$
		so that
		$$(\tau p^{-1}+p^i)f+\ov x\in\Zp.$$
		Since
		$$f=-p^ib'+p^{-i}.\frac{(x+1)(\overline{x}-1)}{2}$$
		we conclude that
		$b'$ belongs to a translate (depending on $x$ and $\tau$) of $$(\tau p^{-1}+p^i)^{-1}p^{-i}\Zp=p^{1-i}\Zp.$$	
		
		We now consider the possible values of $i=j\geq 1$.
		
		\subsubsection{The case $i=1$} Suppose that $i=j=1$. The inclusion 
		$$p(1-x)+\delta p-\frac{\tau}{p}(x-\delta)\in \Zp$$
		implies that $x-\delta \in p\Zp$. This implies that $\nu(x)=0$ and the congruence class $\delta\mods p$ is determined by $x$. 
		
		Remembering that for $i=j=1$, $b$ and $b'$ belong respectively to additive translates of $p\Zp$ and $\Zp$, we conclude that for $\nu(\ov x)=-1$ and $\nu(x)=0$, the contribution to $\mcI_{p}(x;3)$ of the case $i=1$ is bounded by
		$$\mcI^{i=1}_p(x;3)\ll \frac{1}{p^4}p^{1+1-1+0}\leq \frac{1}{p^3}.$$

		\subsubsection{The case $i\geq 2$}
		Suppose that $i=j\geq 2$. We observe that $a$ is a unit because the two other terms in the first column of \eqref{matrix1caseIV'2} are divisible by $p$; this implies that  $\nu(e)=-1$ and that the congruence  $\delta\mods{\Zp}$ is determined by $e$ and $a$ (so by $x$ and $b$). 
		
		By \eqref{zalgebraicbis} we have 
		\begin{equation}\label{alg2}
			(1-x\overline{x})+{f}p^i-(1-\overline{x})ep^{i}-{ef}p^{2i}=(1-x)zp^{2i}\in p^{2i}\mathbb{Z}_p,	
		\end{equation}
		and since $\nu(f)\geq 0$, $\nu(e)=-1$ we obtain that 
		$$\nu((1-\ov x)ep^i)=i-2.$$ Since the second and last term of \eqref{alg2} have valuation $> i-2$ we have
		$$\nu(1-x\ov x)= i-2\geq 0.$$
		This is only possible if $\nu(x)\geq 1$.
		
		Therefore, the contribution of this case to $
		\mcI_{p}(x;3)$ is bounded by
		$$
		\mcI^{i\geq 2}_p(x;3)\ll \frac{1}{p^4}\sum_{\mu\in \mathbb{Z}_p/p\mathbb{Z}_p}\sum_{\tau}\int_{p^{2-i}\Zp}\int_{p^{1-i}\mathbb{Z}_p}db'db\ll p^{-4+2+2i-3}=p^{2\nu(1-x\ov x)-1}.
		$$
		and when $\nu(\ov x)=-1,\ \nu(x)\geq 0$ we have
		$$\mcI_{p}(x;3)\ll p^{2\nu(1-x\ov x)}.$$
		
		\subsection*{Case IV}		
		Consider Case IV, and recall \eqref{matrix1caseIV} and \eqref{matrix2caseIV}: 
		\begin{equation}\label{matrix1caseIV2}
			\begin{pmatrix}
				a+\tau p^{j-1}&e+\frac{\delta}{p}z-\frac{\tau}{p}(x+\frac{\delta}{p}f)&z-\frac{\tau}{p}f\\
				-p^j&x+\frac{\delta}{p}f& f\\
				p^{i+j}& p^i(1-x)+\frac{\delta}{p}g&g
			\end{pmatrix}\in K_p.
		\end{equation}
		
		\begin{equation}\label{matrix2caseIV2}
			\begin{pmatrix}
				a'&e'+\frac{\tau}{p}a'&z'\\
				p^i(1-\overline{x})&\overline{x}+\tau p^{i-1}(1-\overline{x})& f'\\
				p^{i+j}-\frac{\delta}{p} p^{i}(1-\overline{x})& -p^j+p^{i+j}\frac{\tau}{p}-\frac{\delta}{p}(\overline{x}+\frac{\tau}{p}p^{i}(1-\overline{x}))&g'-\frac{\delta}{p}f'
			\end{pmatrix}\in K_p.
		\end{equation}
		
$$
			\begin{cases}
				a=-\frac{1}{2}p^{j-i}-p^{i+j}b\\
				e=\frac{p^{-i}(1+x)}{2}-p^{i}b(1-x)\\
				f=-p^jb'+p^{-j}.\frac{(x+1)(\overline{x}-1)}{2}\\
				g=p^{i+j}b'-\frac{(x-1)(\overline{x}-1)}{2}p^{i-j}\\
				z=-\frac{1}{2}p^{j-i}b'+p^{-i-j}y-p^{i+j}bb'+p^{i-j}b\frac{(x-1)(\overline{x}-1)}{2}.
			\end{cases}
	$$
		where $$y=\frac{x\overline{x}+3\overline{x}-x+1}{4}.$$
		Then one has an explicit algebraic relation 
		\begin{equation}\label{zalgebraic2}
			z=\frac{1-x\overline{x}}{1-x}p^{-i-j}+\frac{f}{1-x}p^{-i}-\frac{1-\overline{x}}{1-x}ep^{-j}-\frac{ef}{1-x}.
		\end{equation}
		
		\begin{equation}\label{letter'values2}
			\begin{cases}
				a'=-\frac{(1-x)(1-\overline{x})}{2}p^{i-j}-p^{i+j}b'\\
				e'=\frac{(x-1)(\overline{x}+1)}{2}p^{-j}+p^{j}b'\\
				f'=(1-\overline{x})p^ib+\frac{1+\overline{x}}{2}p^{-i}\\
				g'=p^{i+j}b-\frac{1}{2}p^{j-i}\\
				z'=-p^{i-j}b\frac{(x-1)(\overline{x}-1)}{2}+p^{-i-j}\overline{y}-p^{i+j}bb'+\frac{1}{2}p^{j-i}b'
			\end{cases}
		\end{equation}
		and one notes the algebraic relation
		\begin{equation}\label{z'algebraic2}
			z'
			=\frac{1-x\overline{x}}{1-\overline{x}}p^{-i-j}+\frac{e'}{1-\overline{x}}p^{-i}-\frac{1-x}{1-\overline{x}}f'p^{-j}-\frac{e'f'}{1-\overline{x}}.
		\end{equation}
		We also note that
		\begin{gather*}
			a+g'=-p^{j-i},\ a'+g=-{(1-x)(1-\overline{x})}p^{i-j},\\ p^i.f+g=p^{i-j}(\ov x-1),\ p^jf'-(1-\overline{x})g'=p^{j-i}\nonumber
		\end{gather*}

		We assume now that $\nu(x),\nu(\ov x)\geq -1$ and that one of the two equals $-1$.
		These conditions imply first that
		$$j\geq 0, i+j\geq 0,\ i+\nu(1-\ov x)\geq 0.$$
		
		Suppose $\nu(\overline{x})=-1.$ Then $i=1$ (since $\overline{x}+\tau p^{i-1}(1-\overline{x})\in\Zp$) and we get a contradiction from $p^{i+j}-\delta p^{i-1}(1-\overline{x})\in\mathbb{Z}_p$.
		
		So we must have $\nu(\overline{x})\geq 0$ and $\nu(x)=-1.$ Since $p^{i}(1-x)+\delta p^{-1}g\in\Zp$ we have $i\geq 0.$ We now distinguish two subcases.
		
		\subsection*{The case $\nu(1-\overline{x})=0$}
		Note that $$-(1-x)(1-\overline{x})p^{i-j}=a'+g\in\mathbb{Z}_p,$$
		which implies that $i-j\geq 1.$ Since $j\geq 0,$ then $i\geq 1.$

	\subsubsection*{The case $i=1$} Suppose $i=1.$ Then $j=0.$ Since the $(3,2)$-th entry of  \eqref{matrix2caseIV2} is in $\Zp$ we have the $(2,2)$-th entry of \eqref{matrix2caseIV2} lies in $p\Zp,$ which implies that $\tau$ is determined by $\ov x.$ From the $(2,2)$-th entry of \eqref{matrix1caseIV2} we see that $ f+\delta^{-1}px\in p\mathbb{Z}_p.$ So $b'$ lies in a translate of $p\mathbb{Z}_p.$ 
	
	The identity $g'=(1-\ov x)^{-1}f'+(1-\ov x)^{-1}p^{-1}$ together with $g'-\delta p^{-1}f'\in \Zp$ implies that $f'$ belongs to a translate of $p\Zp$ and that $b$ belongs to a translate of $\Zp$ (depending on $x$ and $\delta$). Therefore, the contribution in this case to $\mcI_{p}(x;4)$ is 
	\begin{align*}
	\ll \frac{1}{p^4}\sum_{\mu\in \mathbb{Z}_p/p\mathbb{Z}_p}\sum_{\delta\in (\Zp/p\Zp)^{\times}}\int_{\Zp}db\int_{p\Zp}db'\ll p^{-3}.
	\end{align*}
	
	\subsubsection*{The case $i\geq 2$} Now we suppose $i\geq 2.$ Looking at the first column of \eqref{matrix2caseIV2} we see that $\nu(a')=0$ (because the whole column cannot be $0$ modulo $p$). This also implies that $\nu(e')=-1$.
		
		Since the  $(3,2)$-th entry of \eqref{matrix1caseIV2} is integral, we conclude that $\nu(g)\geq 1$ and that $a'+g=-(1-x)(1-\overline{x})p^{i-j}$ is a unit so that $$i-j=1$$ and $j\geq 1.$ 
		
		Next we recall that
		$$
		(1-x\overline{x})+{e'}p^{j}-({1-x})f'p^{i}-{e'f'}p^{i+j}=(1-\ov x)z'p^{i+j}\in p^{i+j}\Zp$$
		Observe that  $$\nu({e'}p^{j})=	i-2,\ \nu(({1-x})f'p^{i})\geq i-1,\ \nu({e'f'}p^{i+j})\geq 2i-2$$
		the first valuation is therefore the smallest as $i\geq 2$; from this we conclude that
		$$\nu(1-x\ov x)=i-2\geq 0.$$
		In particular $\nu(\ov x)\geq 1$.
		
		\subsubsection*{Localization of $b,b'$}
		Finally since $x+\delta p^{-1}f\in\Zp$ we see that $b'$ belongs to a translate of $p^{1-j}\Zp=p^{2-i}\Zp$.
		
		The identity $g'=(1-\ov x)^{-1}p^{i-1}f'+(1-\ov x)^{-1}p^{-1}$ together with $g'-\delta p^{-1}f'\in \Zp$ implies that $f'$ belongs to a translate of $p\Zp$ and that $b$ belongs to a translate of $ p^{1-i}\Zp$ (depending on $x$ and $\delta$).

	\subsubsection*{Localization of $\tau$}
Comparing evaluations on both sides of \eqref{zalgebraic2} we derive that 
$$
\frac{1-x\overline{x}}{1-x}p^{-i-j}-\frac{1-\overline{x}}{1-x}ep^{-j}\in p^{-i+1}\Zp.
$$

As a consequence, we have 
\begin{equation}\label{eq1}
(1-x\ov x) p^{-i}-(1-\ov x) e\in p^{-1}\Zp.
\end{equation}
	
On the other hand, considering the $(1,2),$ $(1,3)$ and $(2,2)$-th entries of \eqref{matrix1caseIV2}, one has 
$$
e+\delta\tau p^{-2} f\in p^{-1}\Zp,\ \ z-\tau p^{-1}f\in \Zp,\ \ x+\delta p^{-1}f\in \Zp.
$$	

Hence $e-\tau p^{-1}x\in p^{-1}\Zp.$ In conjunction with \eqref{eq1} we deduce that 
$$
(1-x\ov x) p^{-i}-\tau p^{-1}x(1-\ov x) \in p^{-1}\Zp.
$$
So $\tau$ is determined by $x.$ 

		Hence, in this case, i.e., $i\geq 2,$ the corresponding contribution to $\mcI_{p}(x;4)$ is 
	$$
	\ll \frac{1}{p^4}\sum_{\mu\in \mathbb{Z}_p/p\mathbb{Z}_p}\sum_{\delta\in (\Zp/p\Zp)^{\times}}\int_{p^{1-i}\Zp}db\int_{p^{2-i}\Zp}db'\ll p^{2\nu(1-x\ov x)-1}.
	$$
		
In conclusion, for $\nu(x)=-1$ and $\nu(1-\ov x)=0$ we have
$$
\mcI_{p}(x;4)\ll \delta_{\nu(\ov x)=\nu(1-\ov x)=0}p^{-3}+\delta_{\nu(\ov x)\geq 1}p^{2\nu(1-x\ov x)}.
$$

		\subsection*{Case $\nu(1-\ov x)\geq 1$} We then have $\nu(\ov x)=0$ and $\nu(1-x\ov x)=-1$.

		Suppose that $i\geq 1$ then since $j\geq 0,i+j-1\geq 0$ and $i+\nu(1-\ov x)-1\geq 1$ and
		$\ov x$ is a unit, the condition
		$$ -p^j+p^{i+j}\frac{\tau}{p}-\frac{\delta}{p}(\overline{x}+\frac{\tau}{p}p^{i}(1-\overline{x}))\in\Zp$$
		leads to $\delta\ov x p^{-1}\in\Zp$ a contradiction.
		
		Therefore $i\leq 0$. Moreover since $\nu(g)\geq 0$ and $\nu(1-x)=-1$ the condition
		$$p^i(1-x)+\frac{\delta}{p}g\in\Zp$$
		forces $i\geq 0$  so we have $$i=0.$$
		
		The condition
		$$-(1-x)(1-\overline{x})p^{-j}=a'+g\in\mathbb{Z}_p$$
		gives the bound
		$$0\leq j\leq \nu(1-\ov x)-1.$$
	
\subsubsection*{The case $\nu(\ov x-1)=1$} Suppose $\nu(\ov x-1)=1.$ Then $j=0.$ From $x+\delta p^{-1}f\in\Zp$ we derive that $b'$ is in a translate of $p\Zp.$ Looking at the $(1,1)$-th entry of \eqref{matrix1caseIV2}, we obtain that $a+\tau p^{-1}=-b-1/2+\tau p^{-1}\in \Zp.$ So $b$ belongs to a translate of $\Zp.$ 
		
Considering the $(3,2)$-th entry of \eqref{matrix2caseIV2}, we obtain 
$$
\tau -\delta (\ov x+\tau p^{-1}(1-\ov x))\in p\Zp.
$$
Hence $\delta$ is determined by $\tau.$ Therefore, the contribution in this case to $\mcI_{p}(x;4)$ is 
\begin{align*}
\ll \frac{1}{p^4}\sum_{\mu\in \mathbb{Z}_p/p\mathbb{Z}_p}\sum_{\tau\in (\Zp/p\Zp)^{\times}}\int_{\Zp}db\int_{p\Zp}db'\ll p^{-3}.
\end{align*}
	
\subsubsection*{The case $\nu(\ov x-1)\geq 2$} Now we consider the situation where $\nu(\ov x-1)\geq 2.$
		\subsubsection*{Localization of $\delta,\tau$}
		
		Since $\nu(x)=-1$, we see that $f$ is a unit and that $\delta\mods p$ is determined by $x$ and $f$ (which depend on $b'$). In addition we see that $\nu(z)=-1$ and $\nu(\delta z/p)=-2$ ans since  $\nu(\frac{\tau}{p}(x+\frac{\delta}{p}f))\geq -1$ we must have $\nu(e)=-2$. Since 
		$$e+\frac{\delta}{p}z-\frac{\tau}{p}(x+\frac{\delta}{p}f)\in\Zp$$
		we see that $e+\frac{\delta}{p}z$ has valuation $-1$ which implies that $\delta\mods p$ is determined by $e$ and $z$ which depend on $b$ and $b'$.
		
		\subsubsection*{Localization of $b,b'$}
		
Since $x+\delta p^{-1}f\in\Zp$ we derive that $b'$ is in a translate of $p^{1-j}\Zp,$ and since $a+\tau p^{j-1}\in \Zp,$ $b$ belongs to a translate of $p^{-j}\Zp.$ Therefore, the contribution in this case to $\mcI_{p}(x;4)$ is 
\begin{align*}
	\ll \frac{1}{p^4}\sum_{0\leq j\leq \nu(1-\ov x)-1}\sum_{\mu\in \mathbb{Z}_p/p\mathbb{Z}_p}\sum_{\tau\in (\Zp/p\Zp)^{\times}}\int_{p^{-j}\Zp}db\int_{p^{1-j}\Zp}db'\ll \nu(1-\ov x)p^{2\nu(1-\ov x)-5}.
\end{align*}

In conclusion, for $\nu(x)=-1,\ \nu(\ov x)\geq 0$ and $\nu(1-\overline{x})\geq 1$ we have
$$
\mcI_{p}(x;4)\ll \delta_{\nu(1-\ov x)=1}p^{-3}+\delta_{\nu(1-\ov x)\geq 2} \nu(1-\ov x)p^{2\nu(1-\ov x)-5}.
$$

		This concludes Proposition \ref{supportsplitN''}.
		
	\end{proof}

	\subsubsection{The archimedean place}\label{secarchreg}
	In this subsection we study the archimedean local orbital integral $\mcI_\infty(x)$. Recall the definition (given in \eqref{Ivdef}):
	\begin{align*}
		\mcI_{\infty}(x)=\int_{H_{x}(\mathbb{R})\backslash G'(\mathbb{R})\times G'(\mathbb{R})}\big|f_{\infty}(y_{1}^{-1}\gamma(x)J y_{2})\big|dy_{1}dy_{2},\ x\in E^\times\!-E^1.
	\end{align*}

	\begin{prop}\label{Iinftybound}
		Let the notation be as before. Let $x\in E^\times\!-E^1.$ Define 
		\begin{equation}
			\label{peterxdef}
			\peter{x}:=\begin{cases}
					|x|^2+1&\hbox{ if }|x|<1,\\
					|x|^2&\hbox{ if }|x|>1.
			\end{cases}
		\end{equation}
		 Then for $k\geq \kmin$
		\begin{equation}\label{Iinftyboundeq}
			\mcI_{\infty}(x)\ll \frac{1}{k\peter{x}^{\frac{k}{4}-2}(|x|^2-1)^2},
		\end{equation}
	 where the implied constant is absolute, and the absolute value $|\cdot|$ is the usual norm in $\mathbb{C}.$
	\end{prop}

	Before engaging the proof we will need two elementary lemmas
	
	\begin{lemma}\label{lem1stintegral}
		Let $A, B, C>0.$ Let $m\geq 2.$ Then
		\begin{equation}\label{229}
			\int_{0}^{\infty}\frac{1}{\big[A+(Ba-Ca^{-1})^2\big]^m}\frac{da}{a^2}\ll \frac{1}{A^{m-\frac{1}{2}}C},
		\end{equation}
		where the implied constant is absolute. 
	\end{lemma}
	\begin{proof}
		Denote by $\LHS$ the left hand side of \eqref{229}. Then 
		\begin{align*}
			\LHS=&\int_{\sqrt{\frac{C}{B}}}^{\infty}\frac{1}{\big[A+(Ba-Ca^{-1})^2\big]^m}\frac{da}{a^2}+\int_0^{\sqrt{\frac{C}{B}}}\frac{1}{\big[A+(Ba-Ca^{-1})^2\big]^m}\frac{da}{a^2}\\
			\leq& \int_{\sqrt{\frac{C}{B}}}^{\infty}\frac{1}{\big[A+(Ba-\sqrt{BC})^2\big]^m}\frac{da}{a^2}+\int_0^{\sqrt{\frac{C}{B}}}\frac{1}{\big[A+(Ca^{-1}-\sqrt{BC})^2\big]^m}\frac{da}{a^2}\\
			=&\int_{0}^{\infty}\frac{da}{(a+\sqrt{CB^{-1}})^2(A+B^2a^2)^m}+\int_0^{\infty}\frac{1}{(A+C^2a^2)^m}da\\
			\leq& \int_{0}^{\infty}\frac{da}{(\sqrt{CB^{-1}})^2(A+B^2a^2)^m}+\frac{1}{A^{m-\frac{1}{2}}C}
			\ll \frac{1}{A^{m-\frac{1}{2}}C},
		\end{align*}
		where the implied constant is absolute. Then \eqref{229} holds.
	\end{proof}
	\begin{remark} One has
		\begin{equation}
			\label{Tmbound}
			\int_0^\infty \frac{da}{(1+a^2)^m}=\frac{2\pi}{2^{2m}}\frac{(2m-2)!}{(m-1)!^2}=\frac{2\pi}{2^{2m}}\frac{2^{2(m-1)}}{(\pi m)^{1/2}}(1+o(1))\ll \frac{1}{m^{1/2}}
		\end{equation}
		from which one can extract slightly better bounds.
	\end{remark}
		
	Similarly we have 
	\begin{lemma}\label{lem2ndintegral}
		Let $A, B, C>0.$ Let $m\geq 2.$ Then
		\begin{equation}\label{234}
			\int_{0}^{\infty}\frac{1}{\big[A+(Ba+Ca^{-1})^2\big]^m}\frac{da}{a^2}\ll \frac{1}{(A+2BC)^{m-\frac{1}{2}}C},
		\end{equation}
		where the implied constant is absolute (independent of $m$).
	\end{lemma}

	\begin{proof}
		Denote by $\LHS$ the left hand side of \eqref{234}. Then 
		\begin{align*}
			\LHS=&\int_{\sqrt{\frac{C}{B}}}^{\infty}\frac{1}{\big[A+(Ba+Ca^{-1})^2\big]^m}\frac{da}{a^2}+\int_0^{\sqrt{\frac{C}{B}}}\frac{1}{\big[A+(Ba+Ca^{-1})^2\big]^m}\frac{da}{a^2}\\
			\leq& \int_{\sqrt{\frac{C}{B}}}^{\infty}\frac{1}{\big[A+2BC+(Ba)^2\big]^m}\frac{da}{a^2}+\int_{\sqrt{\frac{B}{C}}}^{\infty}\frac{1}{\big[A+2BC+(Ca)^2\big]^m}da.
		\end{align*}
For the first term, we note that in the range of integration we have $$(Ba)^2\geq (Ba-C/a)^2$$ so that the first integral is bounded using Lemma \ref{lem1stintegral} and the second is bounded using a linear change of variable. This yields \eqref{234}.
	\end{proof}

	\begin{proof} (of Proposition \ref{Iinftybound})
		Recall in \S \ref{3.1.5} the notation $$g_E=\diag(|D_E|^{1/4}, 1, |D_E|^{-1/4}).$$ 
		Express $y_{1,\infty}$ and $y_{2,\infty}$ in terms of their Iwasawa coordinates:
		\begin{align*}
			y_{1,\infty}=g_E\begin{pmatrix}
				a&\\
				&{a}^{-1}
			\end{pmatrix}\begin{pmatrix}
				1&-ib\\
				&1
			\end{pmatrix}k_1g_E^{-1},\  y_{2,\infty}=g_E\begin{pmatrix}
				a'&\\
				&{a}'^{-1}
			\end{pmatrix}\begin{pmatrix}
				1&-ib'\\
				&1
			\end{pmatrix}k_2g_E^{-1},
		\end{align*}
		where $a, a'\in\mathbb{R}_{+}^{\times}$ and $k_1, k_2$ lie in the maximal compact subgroup.
		
		Then $g_E^{-1}y_{1,\infty}^{-1}\gamma(x)y_{2,\infty}g_E$ is equal to 
		\begin{align*}
			k_1^{-1}\begin{pmatrix}
				a^{-1}&&iab\\
				&1\\
				&&a
			\end{pmatrix}g_E^{-1}\gamma(x)Jg_E\begin{pmatrix}
				a'&&-ia'b'\\
				&1&\\
				&&{a}'^{-1}
			\end{pmatrix}k_2.
		\end{align*}
		
Let $f_\infty$ be as in \eqref{55}. Noting the $K$-type, we then obtain 
\begin{align*}
|f_{\infty}(y_{1,\infty}^{-1}\gamma(x)y_{2,\infty})|=\Bigg|M\left(
\begin{pmatrix}
a^{-1}&&iab\\
&1\\
&&a
\end{pmatrix}g_E^{-1}\gamma(x)Jg_E\begin{pmatrix}
a'&&ia'b'\\
&1&\\
&&{a}'^{-1}
\end{pmatrix}
\right)\Bigg|,
\end{align*}
where $M(g):=\langle D^{\Lambda}(g)\phi_{\circ},\phi_{\circ}\rangle_{\Lambda}$ is defined in \eqref{10} (cf. Lemma \ref{26}). A direct computation shows that 
$\begin{pmatrix}
a^{-1}&&iab\\
&1\\
&&a
\end{pmatrix}g_E^{-1}\gamma(x)Jg_E\begin{pmatrix}
a'&&ia'b'\\
&1&\\
&&{a}'^{-1}
\end{pmatrix}$ is equal to 
\begin{equation}\label{302}
\begin{pmatrix}
*&\frac{1+x}{2a|D_E|^{\frac{1}{4}}}+iab(1-x)|D_E|^{\frac{1}{4}}&*\\
-a'|D_E|^{\frac{1}{4}}&x&ia'b'|D_E|^{\frac{1}{4}}+\frac{(x+1)(\overline{x}-1)}{2a'|D_E|^{\frac{1}{4}}}\\
aa'|D_E|^{1/2}&(1-x)a|D_E|^{\frac{1}{4}}&*
\end{pmatrix}.
\end{equation}
Denote by $(g_{ij})$ the matrix given in \eqref{302}. Then by Lemma \ref{26}
\begin{equation}\label{303}
|f_{\infty}(y_{1,\infty}^{-1}\gamma(x)y_{2,\infty})|=|M((g_{ij}))|=\frac{2^k|g_{22}|^{k/2}}{|{g}_{11}-{g}_{13}-{g}_{31}+{g}_{33}|^k}.
\end{equation}
		
Since $g=(g_{ij})$ is unitary with respect to $J$, that is, $\transp{\overline{g}}Jg=J$, conjugating $g$ by the matrix $\rmB$ defined in \eqref{b} gives
\begin{align*}
g'=(g'_{ij})
:=
\rmB^{-1}g\rmB
\in G_{J_{2,1}}(\mathbb{R}),
\qquad
J_{2,1}=\diag(1,1,-1).
\end{align*}
Thus,
\begin{align*}
\transp{\overline{g'}}J_{2,1}g'=J_{2,1}.
\end{align*}
Since $J_{2,1}^2=1$, it follows that $g'^{-1}=J_{2,1}\transp{\overline{g'}}J_{2,1}$. Multiplying this identity on the left by $g'$ and on the right by
$J_{2,1}$, we also obtain
\begin{align*}
g'J_{2,1}\transp{\overline{g'}}=J_{2,1}.
\end{align*}

Taking the $(2,2)$-entry of the identity
$\transp{\overline{g'}}J_{2,1}g'=J_{2,1}$, we obtain
\begin{equation}\label{9.114}
|g'_{12}|^2+|g'_{22}|^2-|g'_{32}|^2=1,
\quad\text{and hence}\quad
|g'_{32}|^2
=
|g'_{12}|^2+|g'_{22}|^2-1.
\end{equation}
On the other hand, taking the $(3,3)$-entry of
$g'J_{2,1}\transp{\overline{g'}}=J_{2,1}$ gives
\begin{equation}\label{9.115}
-|g'_{31}|^2-|g'_{32}|^2+|g'_{33}|^2=1,
\quad\text{and hence}\quad
|g'_{33}|^2
=
1+|g'_{31}|^2+|g'_{32}|^2.
\end{equation}
Combining \eqref{9.114} and \eqref{9.115}, we obtain
\begin{equation}\label{9.116}
|g'_{33}|^2
=
|g'_{22}|^2+|g'_{31}|^2+|g'_{12}|^2.
\end{equation}

Similarly, taking the $(2,2)$-entry of
$g'J_{2,1}\transp{\overline{g'}}=J_{2,1}$ gives
\begin{equation}\label{9.117}
|g'_{21}|^2+|g'_{22}|^2-|g'_{23}|^2=1,
\quad\text{and hence}\quad
|g'_{23}|^2
=
|g'_{21}|^2+|g'_{22}|^2-1.
\end{equation}
Taking the $(3,3)$-entry of
$\transp{\overline{g'}}J_{2,1}g'=J_{2,1}$ gives
\begin{equation}\label{9.118}
-|g'_{13}|^2-|g'_{23}|^2+|g'_{33}|^2=1,
\quad\text{and hence}\quad
|g'_{33}|^2
=
1+|g'_{13}|^2+|g'_{23}|^2.
\end{equation}
Combining \eqref{9.117} and \eqref{9.118}, we obtain
\begin{equation}\label{9.119}
|g'_{33}|^2
=
|g'_{22}|^2+|g'_{13}|^2+|g'_{21}|^2.
\end{equation}

Consequently, averaging \eqref{9.116} and \eqref{9.119}, we obtain
\begin{equation}\label{307}
|g'_{33}|^2=
|g'_{22}|^2
+\frac{
|g'_{13}|^2+|g'_{31}|^2+|g'_{12}|^2+|g'_{21}|^2
}{2}\geq
|g'_{22}|^2
+\frac{|g'_{12}|^2+|g'_{21}|^2}{2}.
\end{equation}
By 
\begin{align*}
\rmB\left(\begin{array}{ccc}
g_{11}&{g}_{12}&{g}_{13}\\
g_{21}&{g}_{22}&{g}_{23}\\
g_{31}&{g}_{32}&{g}_{33}
\end{array}
\right)\rmB=\left(\begin{array}{ccc}
\frac{g_{11}+g_{13}+g_{31}+g_{33}}{2}&\frac{g_{12}+g_{32}}{\sqrt{2}}&\frac{g_{11}-g_{13}+g_{31}-g_{33}}{2}\\
\frac{g_{21}+g_{23}}{{\sqrt{2}}}&g_{22}&\frac{g_{21}-g_{23}}{\sqrt{2}}\\
\frac{g_{11}+g_{13}-g_{31}-g_{33}}{{2}}&\frac{g_{12}-g_{32}}{\sqrt{2}}&\frac{g_{11}-g_{13}-g_{31}+g_{33}}{{2}}
\end{array}\right).
\end{align*}
we then have from \eqref{307} that 
\begin{equation}\label{304}
|{g}_{11}-{g}_{13}-{g}_{31}+{g}_{33}|^2\geq 4|g_{22}|^2+2|g_{12}+g_{32}|^2+2|g_{21}+g_{23}|^2.
\end{equation}
		
Substituting \eqref{304} in \eqref{303} we obtain  
\begin{equation}\label{308}
|f_{\infty}(y_{1,\infty}^{-1}\gamma(x)y_{2,\infty})|\leq \left(\frac{2|x|}{2|g_{22}|^2+|g_{12}+g_{32}|^2+|g_{21}+g_{23}|^2}\right)^{k/2}.
\end{equation}
		
Let us write $x=m+ni\sqrt{|D_E|}$ with $m, n\in \mathbb{Q}.$ Using \eqref{302} we see that		\begin{gather*}
			2|g_{22}|^2+|g_{12}+g_{32}|^2+|g_{21}+g_{23}|^2\\=2|x|^2+
			h_1(a|D_E|^{1/4},b,x)^2+h_2(|D_E|^{1/4}a,b,x)^2\\
			+h_1'(|D_E|^{1/4}a',b',x)^2+h_2'(|D_E|^{1/4}a',b',x)^2	
		\end{gather*}
		where 
		\begin{equation}\label{315}
			\begin{cases}
				h_1(a,b,x)=\frac{m+1}{2a}+abn|D_E|^{\frac{1}{2}}-(m-1)a\\
				h_2(a,b,x)=ab(m-1)+an|D_E|^{\frac{1}{2}}-\frac{n|D_E|^{\frac{1}{2}}}{2a}\\
				h_1'(a',b',x)=a'-\frac{m^2+n^2|D_E|-1}{2a'}\\
				h_2'(a',b',x)=a'b'-\frac{n|D_E|^{\frac{1}{2}}}{a'}.
			\end{cases}
		\end{equation}
		
		Hence combining \eqref{308} with the change of variables $$ a|D_E|^{1/4}\longleftrightarrow a,\ a'|D_E|^{1/4}\longleftrightarrow a',$$ 
		we obtain that $\mcI_{\infty}(x)$ is bounded by 
		\begin{equation}\label{310}
			\int_{\Rr_{>0}}\int_{\Rr_{>0}}\int_{\Rr}\int_{\Rr}\Bigg[\frac{2|x|}{
				2|x|^2+
				\sum_{j=1}^2\big[h_j(a,b,x)^2+h'_j(a',b',x)^2\big]
			}\Bigg]^{\frac{k}{2}}\frac{dbdb'dada'}{aa'}.
		\end{equation}
		
		Note that $x\not\in E^1,$ i.e., $|x|^2=m^2+n^2|D_E|\neq 1.$ So $(m,n)\neq (1,0)$ and if furthermore $|D_E|=1,$ then $(m,n)\neq (\pm 1, 0)$ or $(0,\pm 1)$. 
		
		We suppose first that $|x|^2>1$.
		\begin{enumerate}
			\item[1.] Suppose ${n\neq 0}$. We make the linear change of variables
			$$ h_1(a,b,x) \longleftrightarrow {b},\ h'_2(a',b',x)\longleftrightarrow {b'} $$ and find that 
			\eqref{310} is bounded by 
			\begin{equation}\label{312'}
				\int_{\Rr_{>0}}\int_{\Rr_{>0}}\int_{\Rr}\int_{\Rr}\Bigg[\frac{2|x|}{
					2|x|^2+b^2+h(a,b,x)^2+b'^2+h'(a',x)^2
				}\Bigg]^{\frac{k}{2}}\frac{dbdb'dada'}{|n|{|D_E|^{1/2}}a^2a'^2},
			\end{equation}
			where $$h'(a',x)=h_1'(a',b',x)=a'-\frac{|x|^2-1}{2}\frac{1}{a'}$$ is defined in \eqref{315} and 
			\begin{align*}
				h(a,b,x)=\frac{m-1}{n|D_E|^{\frac{1}{2}}}b+\frac{|x-1|^2}{n|D_E|^{\frac{1}{2}}}a-\frac{|x|^2-1}{2n|D_E|^{\frac{1}{2}}}\frac{1}a=\alpha.b+\beta,
			\end{align*}
			say.	Making the linear change of variable 
			$$(1+\alpha^2)^{1/2}.b+\frac{\alpha\beta}{(1+\alpha^2)^{1/2}}\longleftrightarrow {b},$$
			and noting that
			$$1+\alpha^2=\frac{|x-1|^2}{n^2|D_E|},$$
			\eqref{312'} becomes
			\begin{equation}\label{312}
				\int_{\Rr_{>0}}\int_{\Rr_{>0}}\int_{\Rr}\int_{\Rr}\Bigg[\frac{2|x|}{
					2|x|^2+b^2+h(a,x)^2+b'^2+h'(a',x)^2
				}\Bigg]^{\frac{k}{2}}\frac{dbdb'dada'}{|x-1|a^2a'^2},
			\end{equation}
			where 
			\begin{align*}
				h(a,x)=\frac{\beta}{(1+\alpha^2)^{1/2}}=|x-1|a-\frac{|x|^2-1}{2|x-1|}\frac{1}{a}.
			\end{align*}

			By two changes of variable \eqref{312} is equal to 
			\begin{equation}\label{313}
				T_k.T_{k-1}\frac{(2|x|)^{k/2}}{|x-1|}\int_{\Rr_{>0}}\int_{\Rr_{>0}}\Bigg[\frac{1}{
					2|x|^2+h(a,x)^2+h'(a',x)^2
				}\Bigg]^{\frac{k}{2}-1}\frac{dada'}{a^2a'^2}.
			\end{equation}
			where
			$$T_k=\int_{-\infty}^{\infty}\frac{db}{(1+b^2)^{k/2}}\ll \frac{1}{k^{1/2}}$$
			by \eqref{Tmbound}.	Applying twice the computational Lemma \ref{lem1stintegral} above, with
			$$A=2|x|^2+h'(a',x)^2,\ C=\frac{|x|^2-1}{2|x-1|}>0,\ A'=2|x|^2, C'=\frac{|x|^2-1}{2},$$ we have 
			\begin{align*}
				\int_{\Rr_{>0}}\int_{\Rr_{>0}}\Bigg[\frac{1}{
					2|x|^2+h(a,x)^2+h'(a',x)^2
				}\Bigg]^{\frac{k}{2}-1}\frac{dada'}{a^2a'^2}\ll \frac{1}k\frac{|x-1|}{(2|x|^2)^{k/2-2}(|x|^2-1)^2}
			\end{align*}
			where the implicit constant is absolute.
			Therefore, 
			\begin{equation}\label{230}
				\mcI_{\infty}(x)
				\ll\frac{1}{k}\frac{1}{|x|^{\frac{k}{2}-4}(|x|^2-1)^2},
			\end{equation}
			where the implicit constant is absolute. 
			
			\item[2.] Suppose $n=0$. We then have $x=m\neq 1$ and $\mcI_{\infty}(x)$ is bounded by \eqref{310} with
			\begin{equation}\label{315bis}
				\begin{cases}
					h_1(a,b,x)=(m-1)a-\frac{m+1}{2}\frac{1}a\\
					h_2(a,b,x)=a(m-1)b\\
					h_1'(a',b',x)=a'-\frac{m^2-1}{2}\frac{1}{a'}\\
					h_2'(a',b',x)=a'b'.
				\end{cases}
			\end{equation}
			In particular both $h_1(a,b,x)$ and $h'_1(a,b,x)$ do not depend on $b$ and $b'$ and we note them $h(a,x)$ and $h'(a,x)$ respectively. Changing variables this expression equals
			\begin{equation}\label{231}
				\frac{1}{|x-1|}\int_{\Rr_{>0}}\int_{\Rr_{>0}}\int_{\Rr}\int_{\Rr}\Bigg[\frac{2|x|}{
					2|x|^2+b^2+h(a,x)+b'^2+h'(a',x)
				}\Bigg]^{\frac{k}{2}}\frac{dbdb'dada'}{a^2a'^2}.
			\end{equation}
			By Lemma \ref{lem1stintegral}, \eqref{231} is majorized by 
			\begin{align*}
				T_k.T_{k-1}.\frac{(2|x|)^{k/2}}{|x-1|}\int_{0}^{\infty}\int_0^{\infty}\Bigg[\frac{1}{
					2|x|^2+h(a,x)+h'(a',x)
				}\Bigg]^{\frac{k}{2}-1}\frac{dada'}{a^2a'^2}\ll\frac{1}{k}\frac{1}{|x|^{\frac{k}{2}-4}(|x|^2-1)^2},
			\end{align*}
			where the implied constant is absolute. We have also used that $$m^2-1=|x|^2-1.$$
		\end{enumerate}
		
		Therefore, assuming that $|x|>1,$ it follows from \eqref{230} and the above estimates that 
		$$
		\mcI_{\infty}(x)\ll \frac{1}{k}\frac{1}{|x|^{\frac{k}{2}-4}(|x|^2-1)^2}.
		$$
		
		If we assume that $|x|<1,$ i.e., $1-|x|^2>0$. As above  we find (using Lemma \ref{lem2ndintegral} instead of Lemma \ref{lem1stintegral})
		$$
		\mcI_{\infty}(x)\ll \frac{1}{k} \frac{|x|^{\frac{k}{2}}}{(|x|^2+1)^{\frac{k}{2}-2}(|x|^2-1)^2}\ll \frac{1}{k}\frac{1}{(|x|^2+1)^{\frac{k}{4}-2}(|x|^2-1)^2}.
		$$
		
		Thus \eqref{Iinftyboundeq} holds.
	\end{proof}

\begin{remark}
Recall that in the regular orbital integral case we have $x\notin E^1,$ so it does not lead to the divergence of the integral \eqref{310}. That is why the bound \eqref{supnormint} was too coarse to deal with the unipotent orbital integrals and we have had to make explicit computations instead. 
\end{remark}

	\section{\bf Bounds for the Global Sum of Regular Orbital Integrals}\label{sec8.3}
In this section we collect the local bounds from the previous section to bound the sum of global regular orbital integrals in \eqref{Jsimple}
\begin{equation}
	\label{regOrbsum}
	\sum_{x\in E^\times\!-E^1}\mathcal{O}_{\gamma(x)}(f^{\mathfrak{n}},\vphi')
\end{equation}
	(and in particular establish absolute convergence when $k$ is large enough).	
	
	We recall that to the test function $f^{\mathfrak{n}}$ is associated an integer $\ell\geq 1$ given by \eqref{elldefinition} and that in Theorem \ref{thmvanish} we have introduced the set
		$$	\mathfrak{X}(N,N',\ell)=\big\{x\in E^\times\!-E^1,\ x\in (\ell N')^{-1}\mathcal{O}_E,\ x\overline{x}\equiv 1\Mod{N}\big\},
	$$
	and have proved that for $x\in E^\times\!-E^1$ not contained in $\mathfrak{X}(N,N',\ell)$ we have 
	$$\mcI(f^{\mfn};x)=\mathcal{O}_{\gamma(x)}(f^{\mathfrak{n}},\vphi')=0.$$

	The main result of this section is the following upper bound for the sum	\eqref{regOrbsum}:
	\begin{thm}\label{regularglobalbound}
		With the notation and assumptions as above, let $\vphi'$ be the new form of weight $k\geq \kmin$, level $N'$ described in \S \ref{U(W)choice} subject to the normalization \eqref{L2norm} and set 
		\begin{equation}\label{kappadef}\kappa=\frac{k}{4}-2.
		\end{equation}

We have
		\begin{equation}
	\label{334stable1}
		\sum_{x\in E^\times\!-E^1}\frac{\mathcal{O}_{\gamma(x)}(f^{\mfn},\vphi')}{\peter{\vphi',\vphi'}}\ll_{E} (k\ell NN')^{o(1)}k^{-\frac{1}{2}}{\ell^{15}}N'^{\frac{14}{3}}N^{2}(1+\frac{\ell^2{N'}^2}N)^{2}  \mathcal{E}
\end{equation}
		where $$\mathcal{E}:=e^{-\frac{\kappa}{(\ell N')^2+1}}+2^{-\kappa}$$ 
		Moreover if we assume that
		\begin{equation}\label{ellNN'boundfirst}
	\ell^2{N'}^2<N
\end{equation}
		we have
		\begin{equation}\label{334stable}
			\sum_{x\in E^\times\!-E^1}\mathcal{O}_{\gamma(x)}(f^{\mathfrak{n}},\vphi')\ll (k\ell NN')^{o(1)}
k^{-\frac{1}{2}}{\ell^{11}}N^4{N'}^{2+2/3}(\frac{\ell^2{N'}^2}{N})^\kappa.
		\end{equation}
Here the implicit constants depend only on $E$.
	\end{thm}
	\begin{remark} 
The above estimate could be improved by carefully analyzing each situation determined by Proposition \ref{supportsplitN''}.
\end{remark}

	\subsection{Decomposition of Regular Orbital Integrals} Let us recall that we have made the following reductions in \S \ref{seccoarsebound}: 
	\begin{align}\nonumber
		\bigl|\sum_{x\in E^\times\!-E^1}\mathcal{O}_{\gamma(x)}(f^{\mathfrak{n}},\vphi')\bigr|&\leq \sum_{x\in E^\times\!-E^1}|\mathcal{O}_{\gamma(x)}(f^{\mathfrak{n}},\vphi')|
		\\
		&\ll \|\vphi'\|_\infty^2\sum_{x\in \mathfrak{X}(N,N',\ell)}\mcI(x)\nonumber\\
		&\ll\|\vphi'\|_\infty^2\sum_{x\in \mathfrak{X}(N,N',\ell)}\mcI_\infty(x)\prod_p\mcI_p(x)
		\nonumber \\
		&\leq (kN')^{o(1)}k^{1/2}{N'}^{2/3}\sum_{x\in \mathfrak{X}(N,N',\ell)}\mcI_\infty(x)\prod_p\mcI_p(x)\label{reductionstep}
	\end{align}
	here for the last step we have used the following bound from \cite[Thm. 1]{HT13} for $\vphi'$ satisfying \eqref{L2norm} 
	$$
	\|\vphi'\|^2_{\infty}\leq (kN')^{o(1)}k^{1/2}N'^{\frac{2}{3}}\langle\vphi',\vphi'\rangle.
	$$

	Set
	$$S(N,N',\ell):=\sum_{x\in \mathfrak{X}(N,N',\ell)}\mcI_\infty(x)\prod_p\mcI_p(x).$$
	By Proposition \ref{Iinftybound} we have
\begin{equation}\label{eq10}
\mcI_\infty(x)\ll \frac{1}{k}\frac{1}{\peter{x} ^{\kappa}(|x|^2-1)^2},
\end{equation}
where $$\peter{x}=\begin{cases}
	|x|^2+1&\hbox{ if $|x|<1$,}\\
	|x|^2&\hbox{ if $|x|>1.$}
\end{cases}	$$

	We consider the decomposition
	$$\prod_p\mcI_p(x)=\mcI_{n-sp}(x)\mcI_{sp}(x)\mcI_{N'}(x)\mcI_{\ell}(x)$$
	where $\mcI_{n-sp}(x)$, $\mcI_{sp}(x)$ , $\mcI_{N'}(x)$ and $\mcI_{\ell}(x)$ denote respectively the product of the $\mcI_p(x)$ over  all the non-split primes (coprime to $\ell$), over the split primes (coprime to $N'$) and over the primes dividing $N'$ (if any) and $\ell$. These integrals have been bounded in \S \ref{secnonsplitreg} and in \S \ref{secsplitreg}. 
	
	To implement these bounds, we need some extra notation: for any $z\in E^\times$ and $p$ a prime we set 
	$$\Nr(z)_p=\prod_{\mfp |p}p^{e_pf_p\nu_\mfp(z)}$$
	where $e_p,f_p$  and $\nu_\mfp$ are respectively the ramification index, residual degree and valuation at $\mfp$; for $S$ a subset of the prime numbers we set
	$$\Nr(z)_S:=\prod_{p\in S}\Nr(z)_p.$$
	We have
	\begin{equation}\Nr(z)=z\ov z=|z|^2=\prod_p\Nr(z)_p=\Nr(z)_{\ell N'}\Nr(z)_{sp}\Nr(z)_{n-sp}
		\label{normproduct}	
	\end{equation}
	where $sp$ (resp. $n-sp$) denotes the product over the split (resp. non-split) primes.
	
	By Propositions \ref{nsplitkey} (for the non-split case) and  \ref{keysplitcoprime} (for the split case) and the fact that $\omega(D_ENN')\ll_E1$ , we have for any $\eps>0$ and $x\in \mathfrak{X}(N,N',\ell)$ (in particular $x\not\in E^1$)
	\begin{equation}\label{nspbound}
		\mcI_{n-sp}(x)\ll_\eps \delta_{x\ov x\equiv 1\mods N}N\Nr(x\ov x-1)_{n-sp}^{3+\eps},	
	\end{equation}
	
	\begin{equation}\label{spbound}\mcI_{sp}(x)\ll_\eps \Nr(x\ov x-1)_{sp}^{3/2+\eps}\leq \Nr(x\ov x-1)_{sp}^{3+\eps}	
	\end{equation}
	(the last inequality because $x\ov x-1$ has non-negative valuation at every prime $p\nmid N'\ell$).
	
To bound $\mcI_\ell(x)=\prod_{p|\ell}\mcI_p(x)$ we apply Proposition \ref{withHecke}: for $p\mid \ell$ and $r=\nu_p(\ell)$, we have
\begin{equation}\label{eq11}
\mcI_p(x)\ll \delta_{\nu(x)\geq -r}(1+|\nu(1-x))|+|\nu(1-x\ov x)|)p^{7r}\bigl(p^{2\nu(1-x)}+p^{2\nu(1-x\ov x)}\bigr).
\end{equation}
	
To bound  $\mcI_{N'}(x)$ (granted $N'>1$) we use  Propositions \ref{keysplitcoprime} and \ref{supportsplitN''}, which depend on the valuations $\nu(x)$, $\nu(\ov x)$, $\nu(1-\ov x)$ at the prime $N'$. Define
\begin{align}
 \nonumber\mathfrak{X}(N,N',\ell)_{0}:=&\big\{x\in \mathfrak{X}(N,N',\ell):\ \nu(x)\geq 0,\ \nu(\ov x)\geq 0\big\},\\
\nonumber \mathfrak{X}(N,N',\ell)_{1}:=&\big\{x\in \mathfrak{X}(N,N',\ell):\ \nu(x)=-1,\ \nu(1-\ov x)\geq 1\big\},\\
\label{Xiidef} \mathfrak{X}(N,N',\ell)_{2}:=&\big\{x\in \mathfrak{X}(N,N',\ell):\ \nu(x)=-1,\ \nu(\ov x)=-1\big\},\\
\nonumber \mathfrak{X}(N,N',\ell)_{3}:=&\big\{x\in \mathfrak{X}(N,N',\ell):\ \nu(x)\geq 0,\ \nu(\ov x)=-1\big\},\\
\nonumber \mathfrak{X}(N,N',\ell)_{4}:=&\big\{x\in \mathfrak{X}(N,N',\ell):\ \nu(x)=-1,\ \nu(\ov x)\geq 0\big\}.
\end{align}

 Then \eqref{eq10}, together with Propositions \ref{keysplitcoprime} and \ref{supportsplitN''}, implies that 
\begin{equation}\label{13}
\sum_{x\in E^\times\!-E^1}\mathcal{O}_{\gamma(x)}(f^{\mathfrak{n}},\vphi')\ll (kN')^{o(1)}k^{-1/2}{N'}^{2/3}\sum_{i=0}^4S_i(N,N',\ell),
\end{equation}
where 
$$
S_i(N,N',\ell):=\sum_{x\in \mathfrak{X}(N,N',\ell)_i}\frac{1}{\peter{x}^{\kappa}(|x|^2-1)^2}\prod_p\mcI_p(x)
$$
where we recall that
$\kappa=k/4-2$.
By \eqref{nspbound}, \eqref{spbound} and \eqref{eq11}, for $0\leq i\leq 4,$ $S_i(N,N',\ell)$ is majorized by 
$$
\ell^{7}\sum_{x\in \mathfrak{X}(N,N',\ell)_i}\frac{N\Nr(x\ov x-1)_{n-sp}^{3+\eps} \Nr(x\ov x-1)_{sp}^{3+\eps}\mcI_{N'}(x)
}{\peter{x}^{\kappa}(|x|^2-1)^2}\prod_{p\mid \ell}(\Nr(1-x)_p+\Nr(1-x\ov x)_p)^{1+\eps}.
$$

\subsection{Bounding $S_0(N,N',\ell)$}\label{10.3}
We first bound the simplest term $S_0(N,N',\ell),$ which is the ``generic'' case if $N'$ is not too large. In fact, since we are in the stable range, one can conceptually think $N'=1,$ then $S_i(N,N',\ell)=0$ for $1\leq i\leq 4.$
\begin{lemma}\label{lem10.4}
Assuming that $\kappa>6,$ we have for all $\eps>0$ and $\ell\geq 1$  
\begin{align*}
S_0(N,N',\ell)\ll_\eps (\ell NN')^\eps\ell^{{13}}N^{2}(1+\frac{\ell^{2}{N'}^2}{N})^{2}\left(e^{-\frac{\kappa}{\ell^2+1}}+2^{-\kappa}\right).
\end{align*}
If in addition we have
\begin{equation}
	\label{ellNbound}
	\ell^2<N
\end{equation}
we have
\begin{align*}
S_0(N,N',\ell)\ll_\eps {(\ell N)^\eps\ell^{17}}(\frac{\ell^2}{N})^{\kappa-4}.
\end{align*}
\end{lemma}

\begin{proof}
By Proposition \ref{keysplitcoprime}, for $x\in \mathfrak{X}(N,N',\ell)_0,$ we have 
\begin{equation}\label{case0N'bound}\mcI_{N'}(x)\ll_\eps \Nr(x\ov x)^\eps\bigl(\Nr(X)_{N'}+\frac{\Nr(x\ov x-1)_{N'}}{N'}\bigr)	
\end{equation}
where $X=X(x)=x\ov x(1-x)(1-\ov x).$ 

For $x\in \mathfrak{X}(N,N',\ell)_0,$ we may write $$x=z\ell^{-1},\ z\in\mathcal{O}_E-\{0\},\ \Nr(z)\neq \ell^2,\ z\ov z\equiv \ell^2\Mod{N}.$$ We have therefore
\begin{equation}
	\label{zcondition}
	|z|^2=qN+\ell^2>0,\  q\in\mathbb{Z}-\{0\}.
\end{equation}
 Then  
$S_0(N,N',\ell)$ is majorized by 
\begin{multline*}
	\sum_{\substack{q> -\ell^2N^{-1}}}\frac{r(q)\ell^{{4}}N (|q|N)^{3+\eps}\ell^{7}}{\peter{x} ^{\kappa}q^2N^2}\\\times \Bigg[\left((qN+\ell^2)\big|1-z/\ell\big|^4+\frac{qN+\ell^2}{N'}\right) \left(\big|1-z/\ell\big|^2+\frac{|q|N}{\ell^2}\right)\Bigg]^{1+\eps},
\end{multline*}

where $$r(q)=|\{z\in\mcO_E-\{0\}:\ z\ov z=Nq+ \ell^2	\}|\ll_\eps (N\ell)^\eps$$
for any $\eps>0$. 

By triangle inequality, 
$$\big|1-z/\ell\big|^2\leq 2(1+z\ov z/\ell^2)=2(2+qN\ell^{-2}).$$
 Hence, $S_0(N,N',\ell)$ is bounded by 
\begin{multline*}
	\ll (\ell N)^\eps  \sum_{\substack{q> -\ell^2/N}}\frac{r(q)\ell^{{4}} |q|^{1+\eps}N^{2} \ell^{7}}{\peter{x}^{\kappa}}\\
	\times\left(1+\frac{|q|N}{\ell^2}\right) \left((qN+\ell^2)(1+\frac{qN}{\ell^{2}})^2+\frac{qN+\ell^2}{N'}\right)
	\end{multline*}
$$\ll(\ell N)^\eps \ell^{{13}}N^{2}\sum_{\substack{q> -\ell^2/N}}\frac{|q|^{1+\eps}}{\peter{x}^{\kappa}} (1+\frac{|q|N}{\ell^{2}})^{4+\eps}\ll (\ell N)^\eps\ell^{{13}}N^{2}(S_{01}+S_{02}),
$$
say, where 
\begin{align*}
S_{01}:=&\sum_{-\ell^2/N<q<0}\frac{|q|^{1+\eps}}{\big(\frac{qN+\ell^2}{\ell^2}+1\big)^{\kappa}} {(1+\frac{|q|N}{\ell^{2}})^{4+\eps}
},
\end{align*}
and 
\begin{align*}
	S_{02}:=&\sum_{\substack{q\geq 1 }}\frac{|q|^{1+\eps}}{\big(\frac{qN+\ell^2}{\ell^2}\big)^{\kappa}} (1+\frac{qN}{\ell^{2}})^{4+\eps}.
\end{align*}

Note that for $A>0$,
\begin{equation}\label{inequality}
	\left(A+1\right)^{-\kappa}=\exp\left(-\kappa\log\left(1+A\right)\right)\leq \exp\left(-\frac{\kappa }{A^{-1}+1}\right);
\end{equation}
this implies that for   all $q> -\ell^2/N$ one has (since $qN+\ell^2\geq 1$)
$$
\exp\left(-\frac{\kappa }{\left(\frac{qN+\ell^2}{\ell^2}\right)^{-1}+1}\right)\leq e^{-\frac{\kappa}{\ell^2+1}}.
$$
 This implies that 
$$
S_{01}
	\ll \sum_{{-\ell^2N^{-1}<q<0}}|q|^{1+\eps} e^{-\frac{\kappa}{\ell^2+1}}\ll (1+\ell^2/N)^{2+\eps}  e^{-\frac{\kappa}{\ell^2+1}}.
$$

To estimate $S_{02}$ we break it into two further pieces: 
$$S_{02}=\sum_{1\leq q\leq \ell^2/N}\cdots+\sum_{q>\ell^2/N}\cdots.$$
The first piece  is bounded by
\begin{multline*}
	\sum_{1\leq q\leq \ell^2/N}\cdots\ll \sum_{\substack{1\leq q\leq \ell^2/N}}\frac{|q|^{1+\eps}}{(\frac{qN}{\ell^{2}}+1)^{\kappa}}\\
	\ll (\ell N)^\eps\frac{\ell^{4}}{N^{2}} e^{-\frac{\kappa}{\ell^2/N+1}}\ll (\ell N)^\eps(1+\frac{\ell^2}{N})^{2}  e^{-\frac{\kappa}{\ell^2+1}}.
\end{multline*}
The second piece is bounded by
\begin{multline*}
	\sum_{q>\ell^2/N}\cdots\ll \frac{\ell^2}{N}\sum_{\substack{q> \ell^2/N }}\frac{(qN\ell^{-2})^{5+\eps}}{\big(\frac{qN}{\ell^2}+1\big)^{\kappa}}\\
	\ll (1+\ell^2/N)^2\int_{1}^{\infty}\frac{t^{5+\eps}}{(t+1)^{\kappa}}dt\ll \frac{(1+\ell^2/N)^2}{2^{\kappa}}.
\end{multline*}

Consequently 
$$
S_{02}\ll (\ell N)^\eps(1+\ell^{2}/N)^{2} (e^{-\frac{\kappa}{\ell^2+1}}+2^{-\kappa}).
$$

Let us now assume that \eqref{ellNbound} holds, then \eqref{zcondition} implies that $q\geq 1$ and in the discussion above, the terms $S_{01}$ and the first piece of $S_{02}$ are empty and we have
\begin{align*}
S_0(N,N',\ell)&={(\ell N)^\eps\ell^{13}N^{2}}\sum_{q\geq 1}\cdots\ll {(\ell N)^\eps\ell^{15}N}\sum_{\substack{q\geq 1 }}\frac{(qN\ell^{-2})^{5+\eps}}{\big(\frac{qN}{\ell^2}+1\big)^{\kappa}}\\
&\ll {(\ell N)^\eps\ell^{15}N}(\frac{\ell^2}{N})^{\kappa-5-\varepsilon}\sum_{\substack{q\geq 1 }}\frac{q^{5+\eps}}{\big(q+\ell^2/N\big)^{\kappa}}\\
&\ll {(\ell N)^\eps\ell^{15}N}(\frac{\ell^2}{N})^{\kappa-5-\varepsilon}\sum_{\substack{q\geq 1 }}\frac{q^{5+\eps}}{q^{\kappa}}\ll {(\ell N)^{10\eps}\ell^{17}}(\frac{\ell^2}{N})^{\kappa-4}.
\end{align*}

This concludes the proof of Lemma \ref{lem10.4}.
\end{proof}
	
	\begin{remark}\label{kappalarge}
	In the above, the series are absolutely convergent since $\kappa-5>1$. This is indeed our treatment of  $S_0(N,N',\ell)$ which is responsible for the constraint $k\geq \kmin$. The following sums will be absolutely convergent for smaller values of $k$.\end{remark}

\subsection{Bounding $S_2(N,N',\ell)$}
The worst case scenario happens when $x\in \mathfrak{X}(N,N',\ell)_{2}$ (see \eqref{Xiidef}). In this section we bound $S_2(N,N',\ell).$ The approach is similar to that in \textsection \ref{10.3}, with a mild modification.
\begin{lemma}\label{lem10.5}
Let the notation be as before. We have for any $\eps>0$
\begin{equation}\label{s2}
S_2(N,N',\ell)\ll (\ell NN')^\eps\ell^{{11}}N'^{{2}}N^{2}(1+\ell^2{N'}^2/N)^{2}\left(e^{-\frac{\kappa}{(\ell N')^2+1}}+2^{-\kappa}\right).
\end{equation}
If we assume in addition that
\begin{equation}
\label{ellNN'bound}
	\ell^2{N'}^2<N
\end{equation}
we have
\begin{equation}\label{s2stable}
{S_2(N,N',\ell)\ll(\ell NN')^\eps\ell^{9}N^{3}(\frac{\ell^2 {N'}^2}{N})^{\kappa}.}
\end{equation}
\end{lemma}
\begin{proof}
For $x\in \mathfrak{X}(N,N',\ell)_2,$ we may write $$x=z/(\ell N'),\ z\in\mathcal{O}_E-\{0\},\ \Nr(x)\neq 1,\ z\ov z\equiv (\ell N')^2\mods{N}$$ 
and we can write
\begin{equation}
	\label{zcondition2}
	0<z\ov z=qN+(\ell N')^2,\ q\in\Zz-\{0\}.
\end{equation}
We have
 $$\prod_{p\mid \ell}\Nr(1-x)_p\leq {\big|(1-z/(\ell N'))\big|^2\cdot N'^2\ll N'^2\Big[1+\frac{\Nr(z)}{\ell^{2}N'^2}\Big]}.$$

{By Proposition \ref{supportsplitN''} Case (II) we have 
\begin{align*}
\mcI_{N'}(x)\ll N'^{-4}. 
\end{align*}}
Therefore, 
$$
S_2(N,N',\ell)\ll S_2^{\heartsuit}(N,N',\ell),
$$ 
where $S_2^{\heartsuit}(N,N',\ell)$ is defined by 
\begin{align*}
\sum_{q> -(\ell N')^2N^{-1}}\frac{r(q)(\ell N')^{{4}}N (|q|N)^{3+\eps} \ell^{7}}{\peter{x}^{\kappa}q^2N^2} \left({N'^2\Big[1+\frac{\Nr(z)}{\ell^{2}N'^2}\Big]}+\frac{|q|N}{\ell^2}\right)^{1+\eps} N'^{{-4}}.
\end{align*}
We break the sum into two pieces as above: 
$$S_2^{\heartsuit}(N,N',\ell)=\sum_{-\frac{(\ell N')^2}{N}<q<0}\cdots\ +\  \sum_{q\geq 1}\cdots$$
For the first sum we use the trivial bound
 $$ \left({N'^2\Big[1+\frac{\Nr(z)}{\ell^{2}N'^2}\Big]}+\frac{|q|N}{\ell^2}\right)^{1+\eps}\ll N'^{{2}+\eps},$$ and use \eqref{inequality} to bound the remaining terms as in the treatment of $S_{01}$ in the proof of Lemma \ref{lem10.4}. 
 
 In the range $q\geq 1$ the sum decays exponentially as in the treatment of $S_{02}$ in the proof of Lemma \ref{lem10.4}. Here we provide an explicit calculation (the value of $\eps$ may change from line to line)
\begin{align*}
\sum_{q\geq 1}\cdots&\ll (\ell NN')^\eps \ell^{{11}}N^{2}\sum_{q\geq 1}\frac{q^{1+\eps}}{\big(\frac{qN}{(\ell N')^2}+1\big)^{\kappa}} \left(N'^2+\frac{qN}{\ell^2}\right)\\
&\ll (\ell NN')^\eps \ell^{{11}}{N'}^{{2}}N^{2}\sum_{1\leq q\leq \frac{(\ell N')^2}N} q^{1+\eps}\exp\left(-\frac{\kappa}{\frac{(\ell N')^2}{qN}+1}\right)\\
&\quad\quad +(\ell NN')^\eps\ell^{{11}}N^{2}\sum_{q> \frac{(\ell N')^2}{N}}\frac{q^{1+\eps}}{\big(\frac{qN}{(\ell N')^2}+1\big)^{\kappa}} \left(\frac{qN}{\ell^2}\right)\\
&\ll  (\ell NN')^\eps\ell^{{11}}N'^{{2}}N^{2}(1+\frac{\ell^{2}N'^2}{N})^{2} e^{-\frac{\kappa}{(\ell N')^2/N+1}}\\
&\quad\quad +(\ell NN')^\eps\ell^{{9}}N^{{3}}\sum_{q> \frac{\ell^{2}N'^2}{N}}\frac{q^{2+\eps}}{\big(\frac{qN}{(\ell N')^2}+1\big)^{\kappa}}\\
&\ll (\ell NN')^\eps \ell^{{11}}N'^{{2}}N^{2}(1+\frac{(\ell N')^2}{N})^{2}\left(e^{-\frac{\kappa}{{(\ell N')^2}/{N}+1}}+\int_1^{\infty}\frac{t^{2+\eps}}{(t+1)^{\kappa}}dt\right)\\
&\ll (\ell NN')^\eps\ell^{{11}}N'^{{2}}N^{2}(1+\frac{\ell^{2}N'^2}{N})^{2} \left(e^{-\frac{\kappa}{(\ell N')^2+1}}+2^{-\kappa}\right).
\end{align*}
Then \eqref{s2} follows.

If we moreover assume \eqref{ellNN'bound} then \eqref{zcondition2} implies that $q\geq 1$ and we have
$$
{S_2(N,N',\ell)\ll (\ell NN')^\eps\ell^{11}N^{2}\sum_{q\geq 1}\frac{q^{1+\eps}}{\big(\frac{qN}{(\ell N')^2}+1\big)^{\kappa}} \left(\frac{qN}{\ell^2}\right)\ll (\ell NN')^\eps\ell^{9}N^{3}(\frac{\ell^2 {N'}^2}{N})^{\kappa}.}
$$
{Hence, \eqref{s2stable} follows.} 
\end{proof}

\subsection{Bounding $S_i(N,N',\ell):$ $i=1, 3, 4$}

\begin{lemma}\label{lem10.6}
With the same notation as above, we have for $i=1, 3, 4,$  and any $\eps>0$
\begin{equation}\label{si}
{S_i(N,N',\ell)\ll (\ell NN')^\eps\ell^{15}N'^{4}N^{2}(1+\ell^2{N'}^2/N)^{2}\left(e^{-\frac{\kappa}{(\ell N')^2+1}}+2^{-\kappa}\right).}
\end{equation}
If we assume in addition that
\begin{equation}
\label{ellNN'bound3}
	\ell^2{N'}^2<N,
\end{equation}
we have
\begin{equation}\label{s2stablebis}
{S_i(N,N',\ell)\ll (\ell NN')^\epsilon \ell^{11}N^4{N'}^{2}(\frac{\ell^2{N'}^2}{N})^\kappa.}
\end{equation}
\end{lemma}

\begin{proof}
Investigating the situations in Proposition \ref{supportsplitN''} we see that $\mcI_{N'}(x)$ is majorized by $N'^{-2}$ or $N'^{2\nu(1-x\ov x)-1}$ or $\nu(1-x)N'^{2\nu(1-\ov x)-5},$ depending on $x\in \mathfrak{X}(N,N',\ell)_{i},$ $1\leq i\leq 4.$ In these cases, we may still write $$x=z/(\ell N'),\ z\in\mathcal{O}_E,\ \Nr(x)\neq 1,\ z\ov z\equiv (\ell N')^2\mods{N}$$ and we can write
 \begin{equation}
	\label{zcondition3}
	0<z\ov z=qN+(\ell N')^2,\ q\in\Zz-\{0\}.
\end{equation}

We call $x$ is \textit{good} if $$\mcI_{N'}\ll N'^{-2},$$ i.e., $x\in \mathfrak{X}(N,N',\ell)_{1}$ and some subsets of $\mathfrak{X}(N,N',\ell)_{3}$ and $\mathfrak{X}(N,N',\ell)_{4}.$ We don't need to cover $\mathfrak{X}(N,N',\ell)_{2}$ here since it has been handled in Lemma \ref{lem10.5} already. 

The same arguments as in the proof of Lemma \ref{lem10.5} yields the following upper bound for $x$ good:
\begin{align*}
S_i(N,N',\ell)\ll \sum_{q>-(\ell N')^2/N}\frac{r(q)(\ell N')^{{4}}N (|q|N)^{3+\eps} \ell^{7}}{\peter{x}^{\kappa}q^2N^2} \left({N'^2\Big[1+\frac{\Nr(z)}{\ell^{2}N'^2}\Big]}+\frac{|q|N}{\ell^2}\right)^{1+\eps}\frac{1}{{N'}^{2}}.
\end{align*}

Note that the above bound is obtained by replacing $\mcI_{N'}(x)\ll N'^{{-4}}$ for $x\in \mathfrak{X}(N,N',\ell)_{2}$ with $\mcI_{N'}(x)\ll N'^{-2}$ when $x$ is good. By  Lemma \ref{lem10.5} the corresponding contribution from good $x$ is 
\begin{equation}\label{e10.26}
\ll (\ell NN')^\eps\ell^{{11}}N'^{4}N^{2}(1+\ell^2{N'}^2/N)^{2}\left(e^{-\frac{\kappa}{(\ell N')^2+1}}+2^{-\kappa}\right). 
\end{equation}

Now we consider the remaining cases where one has only $$\mcI_{N'}\ll N'^{2\nu(1-x\ov x)-1}\ \hbox{ or }\ \mcI_{N'}\ll \nu(1-x)N'^{2\nu(1-\ov x)-5}.$$ 
 The prime decomposition of $N'\mcO_E$ take the form
$$N'\mcO_E=\mathfrak{p}\overline{\mathfrak{p}}.$$

\subsubsection{First case} If
$$\mcI_{N'}\ll N'^{2\nu(1-x\ov x)-1},$$ by Proposition \ref{supportsplitN''} we see that $z\in \mathfrak{p}^2\mathcal{O}_E$ or $z\in \overline{\mathfrak{p}}^2\mathcal{O}_E$. So 
$$
\mcI_{N'}(x)\ll N'^{2\nu(1-x\ov x)-1}\ll  (|x|^2-1)^2\ell^4/N'.
$$
We can write 
$$0<x\ov x=qN/\ell^{2}+1,\ q\in\Zz-\{0\}.$$ Therefore, the contribution from these $x$'s is bounded by
\begin{multline*}
\ll \ell^{7}\sum_{x}\frac{N\Nr(x\ov x-1)_{n-sp}^{3+\eps} \Nr(x\ov x-1)_{sp}^{3+\eps}\mcI_{N'}(x)
}{\peter{x}^{\kappa}(|x|^2-1)^2}\\
\times\prod_{p\mid \ell }(\Nr(1-x)_p+\Nr(1-x\ov x)_p)^{1+\eps}
\end{multline*}
\begin{gather*}
\ll (\ell NN')^\eps \ell^7 \sum_{q> -\ell^2/N}\frac{r(q)(|q|N)^{3+\eps}}{\peter{x}^{\kappa}}\frac{\ell^4}{N'} (1+\frac{|q|N}{\ell^{2}})\\
{\ll (\ell NN')^\eps \ell^{15}\frac{N}{N'}\Big[\sum_{-\ell^2/N<q<0}\frac{|q|^{1+\eps}}{\peter{x}^{\kappa}} \frac{(qN)^2}{\ell^{4}}(1+\frac{|q|N}{\ell^{2}})+\sum_{q\geq 1}\frac{|q|^{1+\eps}}{\peter{x}^{\kappa}} (1+\frac{qN}{\ell^{2}})^3\Big]}. 
\end{gather*}
This sum is bounded by
$$\ll(\ell NN')^\eps \ell^{15}\frac{N}{N'}(S_{01}+S_{02})$$
where $S_{01}$ and $S_{02}$ were introduced in the proof of Lemma \ref{lem10.4}. Therefore, the contribution from the $x$'s under the consideration is bounded by
\begin{equation}\label{10.26}
{\ll (\ell NN')^\eps\ell^{15}\frac{N}{N'}(1+\frac{\ell^2}{N})^{2} \left(e^{-\frac{\kappa}{\ell^2+1}}+2^{-\kappa}\right).}
\end{equation}

\subsubsection{Second case} In the case  $$\mcI_{N'}(x)\ll \nu(1-x)N'^{2\nu(1-x)-5},$$ by Proposition \ref{supportsplitN''} we see that $\nu(x)=-1$ and $\nu(1-\ov x)\geq 2.$ So $z\in N'\ell+\overline{\mathfrak{p}}^3\mathcal{O}_E.$ Write 
$$x=1+\frac{\overline{\varpi}^3u}{N'\ell},\ u\in\mathcal{O}_E,\ \overline{\varpi}\in \overline{\mathfrak{p}}\mathcal{O}_E,\ \Nr(\overline{\varpi})=N'.$$
 So 
$$
\mcI_{N'}(x)\ll \nu(1-x)N'^{2\nu(1-x)-5}\ll |u|^{\eps} N'^{-5}\ell^2 \frac{N'^3|u|^2}{N'^2\ell^2}=\frac{|u|^{2+\eps}}{{N'}^{4}}.
$$

The congruence condition $x\ov x\equiv 1\Mod{N}$ becomes 
\begin{equation}\label{eq12}
\Nr(N'\ell+\overline{\varpi}^3u)\equiv (N'\ell)^2\Mod{N}.
\end{equation}
Write \begin{equation}\label{xcondition3} 
 {x\ov x=\frac{qN}{(\ell N')^2}+1},\quad q\in\Zz-\{0\}.	
 \end{equation}
The contribution from these $x$'s is bounded by
\begin{gather*}
\ll  \ell^{7}\sum_{x}\frac{N\Nr(x\ov x-1)_{n-sp}^{3+\eps} \Nr(x\ov x-1)_{sp}^{3+\eps}\mcI_{N'}(x)
}{\peter{x}^{\kappa}(|x|^2-1)^2}\prod_{p\mid l}(\Nr(1-x)_p+\Nr(1-x\ov x)_p)^{1+\eps}\\
\ll (\ell NN')^\eps \sum_{q> -(\ell N')^2N^{-1}}\frac{(\ell N')^{{4}}N (|q|N)^{3+\eps} \ell^{7}}{\peter{x}^{\kappa}q^2N^2} \left({N'^2\Big[1+\frac{\Nr(z)}{\ell^{2}N'^2}\Big]}+\frac{|q|N}{\ell^2}\right)\frac{|u|^{2}}{{N'}^{4}}\\
\ll (\ell NN')^\eps(S_{1}+S_2),
\end{gather*}
where $$S_1=\sum_{\substack{q> -(\ell N')^2N^{-1}\\ |u|\leq 2\ell N'^{-2}}}\cdots,\ S_2=\sum_{\substack{q> -(\ell N')^2N^{-1}\\ |u|> 2\ell N'^{-2}}}\cdots.$$
By definition, we have
\begin{align*}
{S_1\ll (\ell NN')^\eps (\frac{\ell}{N'^{2}})^{2}S_2^{\heartsuit}(N,N',\ell)= (\ell NN')^\eps \frac{\ell^{2}}{{N'^4}}S_2^{\heartsuit}(N,N',\ell),}
\end{align*}
where $S_2^{\heartsuit}(N,N',\ell)$ was defined in the proof of Lemma \ref{lem10.5}.

To handle $S_2,$ we observe that \eqref{eq12}  in the range $|u|> 2\ell N'^{-2}$ implies that
$$
1+\frac{qN}{(\ell N')^2}=\Big|1+\frac{\overline{\varpi}^3u}{\ell N'}\Big|^2\gg \frac{N'^4|u|^2}{\ell^2},
$$
i.e., $$|u|^2\ll \frac{\ell^2}{N'^{4}}\left(1+\frac{|q|N}{(\ell N')^{2}}\right).$$ Therefore, $S_2$ is bounded by 
\begin{align*}
\ll (\ell NN')^\eps\sum_{\substack{q> -(\ell N')^2N^{-1}}}\frac{(\ell N')^{{4}}N (|q|N)^{3+\eps} \ell^{7}}{\peter{x}^{\kappa}q^2N^2} \left({N'^2\Big[1+\frac{\Nr(z)}{\ell^{2}N'^2}\Big]}+\frac{|q|N}{\ell^2}\right){\left(1+\frac{|q|N}{(\ell N')^{2}}\right)}\frac{\ell^{2}}{N'^{8}}.
\end{align*}

Consequently, we obtain 
\begin{align*}
S_2\ll (\ell NN')^\eps \frac{\ell^{2}}{{N'^{{4}}}}S_2^{\heartsuit}(N,N',\ell){+S^{\heartsuit}(N,N',\ell)}.
\end{align*}
{where $S^{\heartsuit}(N,N',\ell)$ is defined by 
\begin{align*}
(\ell NN')^\eps\sum_{q> \frac{\ell^{2}N'^2}{N}}\frac{(\ell N')^4N (|q|N)^{3+\eps} \ell^{7}}{\peter{x}^{\kappa}q^2N^2} \left(N'^2\Big[1+\frac{\Nr(z)}{\ell^{2}N'^2}\Big]+\frac{qN}{\ell^2}\right)\frac{qN}{(\ell N')^{2}}\frac{\ell^{2}}{N'^{8}}.
\end{align*}}

{Proceeding as in the treatment of   $S_2^{\heartsuit}(N,N',\ell)$, we obtain 
\begin{align*}
S^{\heartsuit}(N,N',\ell)& \ll (\ell NN')^\eps\frac{N^{3+\eps}\ell^{11}}{N'^{6}}\sum_{q> \frac{\ell^{2}N'^2}{N}}\frac{q^{2+\eps}}{\big(\frac{qN}{(\ell N')^2}+1\big)^{\kappa}} \left(N'^2+\frac{qN}{\ell^2}\right)\\
& \ll (\ell NN')^\eps\frac{N^{3+\eps}\ell^{11}}{N'^{6}}(\frac{\ell^{2}N'^2}{N})^3N'^22^{-\kappa}\ll (\ell NN')^\eps\ell^{17}N'^22^{-\kappa}.
\end{align*}}

{As a consequence, by Lemma \ref{lem10.5}, the contribution from $x$'s in the second case is bounded by
\begin{multline}\label{10.29}
(\ell NN')^\eps\ell^{13}N'^{-2}N^{2}(1+\frac{\ell^2{N'}^2}{N})^{2}\left(e^{-\frac{\kappa}{(\ell N')^2+1}}+2^{-\kappa}\right)
+(\ell NN')^\eps\ell^{17}N'^22^{-\kappa}\\
\ll (\ell NN')^\eps\ell^{13}N'^{-2}N^{2}(1+\ell^2{N'}^2/N)^{2}\left(e^{-\frac{\kappa}{(\ell N')^2+1}}+2^{-\kappa}\right)
\end{multline}}

{Therefore, \eqref{si} follows from \eqref{e10.26}, \eqref{10.26} and \eqref{10.29}.}

Suppose now that in addition \eqref{ellNN'bound3} holds. For the good $x$'s we have $q\geq 1$ in \eqref{zcondition3} and that contribution is bounded by
\begin{multline*}
\ll (\ell NN')^\eps \sum_{q\geq 1}\frac{(\ell N')^{{4}}N (|q|N)^{3+\eps} \ell^{7}}{(\frac{qN}{(\ell N')^2})^{\kappa}q^2N^2} \left(N'^2+\frac{qN}{\ell^{2}}\right)\frac{1}{{N'}^{2}}\frac{1}{{N'}^{2}}\\
\ll  (\ell NN')^\epsilon \ell^{{9}}N^3{N'}^{2}(\frac{\ell^2{N'}^2}{N})^\kappa.
\end{multline*}
Next the contribution of the non good $x$'s in the first case is bounded by
\begin{align*}
{\ll (\ell NN')^\eps \ell^7 \sum_{q\geq 1}\frac{r(q)(|q|N)^{3+\eps}}{(qN/\ell^2)^{\kappa}}\frac{\ell^4}{N'}\cdot \frac{qN}{\ell^{2}}\ll (\ell NN')^\eps\frac{\ell^{9}N^{4}}{N'}(\frac{\ell^2}N)^{\kappa},}
\end{align*}
and in the second case, their contribution is bounded by
\begin{align*}
&{\ll (\ell NN')^\eps \frac{\ell^{2}}{{N'^{4}}}S_2^{\heartsuit}(N,N',\ell)+(\ell NN')^\eps\frac{N^{3+\eps}\ell^{11}}{N'^{6}}\sum_{q\geq 1}\frac{q^{2+\eps}}{\big(\frac{qN}{(\ell N')^2}\big)^{\kappa}} \left(N'^2+\frac{qN}{\ell^2}\right)}\\
&{\ll (\ell NN')^\eps \frac{\ell^{11}N^{3}}{{N'^{4}}}(\frac{\ell^2 {N'}^2}{N})^{\kappa}+(\ell NN')^\eps\frac{N^{4}\ell^{9}}{N'^{6}}(\frac{\ell^2 {N'}^2}{N})^{\kappa}.}
\end{align*}

{Therefore, \eqref{s2stablebis} follows from the above estimates.}
\end{proof}

\subsection{Proof of Theorem \ref{regularglobalbound}} 
Combining Lemma \ref{lem10.4}, \ref{lem10.5} and \ref{lem10.6} we obtain 
\begin{align*}
\sum_{i=0}^4S_i(N,N',\ell)\ll (\ell NN')^{\eps}\ell^{{15}}N^{2}N'^{4}(1+\frac{\ell^2{N'}^2}{N})^{2} \left(e^{-\frac{\kappa}{(\ell N')^2+1}}+2^{-\kappa}\right).
\end{align*}
and if in addition we have
$$\ell^2N'^2<N$$ we have
$$\sum_{i=0}^4S_i(N,N',\ell)\ll (\ell NN')^{\eps}\ell^{{11}}N^4{N'}^{2}(\frac{\ell^2{N'}^2}{N})^\kappa.$$

Substituting the above estimates in \eqref{13} yields 
\begin{align*}
\sum_{x\in E^\times\!-E^1}\mathcal{O}_{\gamma(x)}(f^{\mathfrak{n}},\vphi')\ll (k\ell NN')^{o(1)}\frac{\ell^{{15}}N'^{\frac{14}{3}}N^{2}}{k^{1/2}}(1+\frac{\ell^2{N'}^2}N)^{2} \left(e^{-\frac{\kappa}{(\ell N')^2+1}}+2^{-\kappa}\right).
\end{align*}
and, assuming that $\ell^2N'^2<N$,
$$\sum_{x\in E^\times\!-E^1}\mathcal{O}_{\gamma(x)}(f^{\mathfrak{n}},\vphi')\ll (k\ell NN')^{o(1)}
\frac{\ell^{{11}}N^4{N'}^{8/3}}{k^{1/2}}(\frac{\ell^2{N'}^2}{N})^\kappa.$$
So Theorem \ref{regularglobalbound} follows.


\section{\bf Twisted Moments of Bessel Periods}\label{SecMainThmPf}

In this section, we establish  an asymptotic formula for the average of the Bessel periods $|\mcP(\vphi,\vphi')|^2$ twisted by  eigenvalues of Hecke operators supported at inert primes. Theorem  \ref{thmB} will follow as a consequence. 

\begin{thm}
	\label{firstmomentwithl}
	
	We use the notation and assumptions of Theorem \ref{thmB}; in particular we recall that $$k> \kmin,\ \kappa=\frac{k}{4}-2> \kappamin,$$
	$$d_\Lambda=\dLambda, d_k=k-1$$
	and
$$\Psi(N)=\prod_{p\mid N}\left(1-\frac{1}p+\frac1{p^{2}}\right),\ \mathfrak{S}({N'})=\prod_{p|N'}(1-\frac{1}{p^2})^{-1}.$$ 

Given $\ell\geq 1$ an integer  coprime to $N$ and divisible only by primes which are inert in $E$, let $\lambda_\pi(\ell)$ be the eigenvalue of the corresponding Hecke operator for the representation $\pi$ (see \eqref{lambdapidef} below). We have
\begin{multline}
	\label{eqmomenttwisted}
\frac{1}{d_{{\Lambda}}}\sum_{\pi\in \mcA_k(N)}\lambda_\pi(\ell)\sum_{\vphi\in \mcB_k^{\widetilde{n}}(\pi)}\frac{\big|\mathcal{P}(\vphi,\vphi')\big|^2}{\peter{\vphi,\vphi}\peter{\vphi',\vphi'}}=\frac{w_E }{d_k}(\frac{N}{{N'}})^2\Psi(N)\mathfrak{S}(N')\frac{\lambda_{\pi'}(\ell)}{\ell}\\
	+O({(\ell{NN'})^{o(1)}}\frac{1}{2^{4k}k^2}\frac{ N}{{N'}^3}\frac{1}{\ell}+(k\ell NN')^{o(1)}\frac{{\ell^{13}}N'^{\frac{14}{3}}N^{2}}{k^{1/2}}(1+\frac{\ell^2{N'}^2}N)^{2} \mcE)
\end{multline}
where
$$\mcE=e^{-\frac{\kappa}{(\ell N')^2+1}}+2^{-\kappa}.$$
Moreover, if we assume that $$\ell^2{N'}^2<N,$$ then the third term on the right-hand side of \eqref{eqmomenttwisted} can be replaced by
$$(k\ell NN')^{o(1)}
\frac{{\ell^{9}}N^4{N'}^{8/3}}{k^{1/2}}(\frac{\ell^2{N'}^2}{N})^\kappa.$$
\end{thm}

\subsection{The Hecke algebra at inert primes}\label{secinertHecke}
We refer to \cite[\S 2.1]{Nowland} for proofs of the well known facts listed below. 

Given  $p$ a prime inert in $E$ and coprime with $N$ and $r\geq 0$ an integer, the (normalized) $p^r$-th Hecke operator is the convolution operator by the function
\begin{equation}
	\label{Heckenormalized}
	T(p^r)=\frac{1}{p^{2r}}\mathrm{1}_{G(\Zp)A_rG(\Zp)}.
\end{equation}
These  satisfy  the recurrence relation 
\begin{equation}
\label{Heckeinert}	
T(p^r)T(p)=T(p^{r+1})+\frac{1}pT(p^r)+T(p^{r-1}),\ r\geq 1.
\end{equation}
Given $\pi\in\mcA_k(N)$, the $G(\Zp)$-invariant vectors of $\pi$ are eigenvectors of the $T(p^r)$ and share the same eigenvalue which we denote by $\lambda_\pi(p^r)$. From \eqref{Heckeinert} we have therefore
\begin{equation}
	\label{Heckeinertlambda}
	\lambda_\pi({p^r})\lambda_\pi(p)=\lambda_\pi(p^{r+1})+\frac{1}p\lambda_\pi({p^r})+\lambda_\pi(p^{r-1}),\ r\geq 1
\end{equation}
or in other terms
\begin{equation}
	\label{lambdapiLserie}
	\sum_{r\geq 0}\frac{\lambda_\pi(p^r)}{p^{rs}}=(1+\frac{1}{p^{1+s}}){(1-\frac{\alpha_\pi(p)}{p^s})^{-1}(1-\frac{\alpha^{-1}_\pi(p)}{p^s})^{-1}}
\end{equation}
where
$$
\lambda_\pi(p)=\alpha_\pi(p)+1/p+\alpha^{-1}_\pi(p)
$$
for some $\alpha_\pi(p)\in\Ct$. 
 
For any integer $\ell=\prod_pp^{r_p}$ coprime to $N$ and divisible only by  primes inert in $E$ we set
\begin{equation}\label{lambdapidef}
	\lambda_\pi(\ell):=\prod_p\lambda_\pi(p^r)
\end{equation}
Finally, since the representation $\pi$ is cohomological, it satisfies the Ramanujan-Petersson conjectures \cite{LR}) and one has
$$|\alpha_\pi(p)|=1;$$ therefore for $r\geq 1$, one has
\begin{equation}
	\label{RPboundprimes}
	|\lambda_\pi(p^r)|\leq r+2
\end{equation}
and for $\ell$ as above one has
\begin{equation}
	\label{RPboundl}\lambda_\pi(\ell)\ll \ell^{o(1)}.\end{equation}

\begin{remark} The Satake parameters at the prime $p$ of the base change $\pi_{E}$ of $\pi$ are given by $\{\alpha_\pi(p),1,\alpha^{-1}_\pi(p)\}$. 
\end{remark}

\subsection{Proof of Theorem \ref{firstmomentwithl}}\label{secThmBcollect}	Let $\ell=\prod_{p}p^{r_p}$ be as above and let
$$f^{\mfn}=f^{\mfn}_\infty\prod_pf_p^{\mfn}$$ be the smooth function (which depends on $\ell$) that was constructed in \S \ref{secglobalf}; in particular for $p|\ell$, one has
$$f_p^\mfn=\mathrm{1}_{G(\Zp)A_{r_p}G(\Zp)}.
$$

By Lemma \ref{lem34} and our normalization for the Hecke operators \eqref{Heckenormalized}, we have
\begin{multline}\label{momentperioddecomp}
	\frac{1}{d_\Lambda}\sum_{\varphi\in \mcB^{\tfn}_{k}(N)}\frac{\big|\mathcal{P}(\varphi,\varphi')\big|^2}{\langle \varphi,\varphi\rangle \langle\varphi',\varphi'\rangle}\lambda_\pi(\ell)\ell^2=
	w_E\frac{\mathcal{O}_{\gamma_1}(f^{\mathfrak{n}},\varphi')}{\peter{\vphi',\vphi'}}\\+\sum_{x\in E^1}\frac{\mathcal{O}_{\gamma(x)}(f^{\mathfrak{n}},\varphi')}{\peter{\vphi',\vphi'}}+\sum_{x\in E^\times\!-E^1}\frac{\mathcal{O}_{\gamma(x)}(f^{\mathfrak{n}},\varphi')}{\peter{\vphi',\vphi'}}.
\end{multline}
The proof then follows immediately from  Proposition \ref{propIdentity}, Corollary \ref{corunipotentorb} and Theorem \ref{regularglobalbound} after dividing by $\ell^2$.\qed

\subsection{Proof of Theorem \ref{thmB}} This is a direct consequence of Theorem \ref{firstmomentwithl} after observing that  there exists a suitable absolute constant $C\geq 32$, such that given any $\delta>0$, if either of the two following conditions is satisfied
$$
\ell {N'}^2\leq N^{1-\delta},\ N>16, k\geq C(1+1/\delta)$$
or
$$
\ell {N'}^2\leq k^{1-\delta},\ N\leq 2^{4k},
$$
then the second and third terms on the right-hand side of \eqref{eqmomenttwisted} are negligible compared to the first term.
\qed


\section{\bf Weighted Vertical Sato--Tate Distribution}
\label{STsec}

In this section,  we interpret Theorem \ref{firstmomentwithl} as a vertical Sato--Tate type joint equidistribution result for products of Hecke eigenvalues $\lambda_{\pi}(p_i)$  at a finite set of inert prime $p_i$'s, for $\pi$ varying over $\mcA_k(N)$ and with the Hecke eigenvalues weighted by the Bessel periods $\big|\mathcal{P}(\vphi,\vphi')\big|^2$. For $\GL(2)$, a result of that kind goes back to Royer \cite{Royer}.

\subsection{The Measure}\label{secmeasure}

The Sato--Tate measure is the measure on $[-2,2]$ with density
\begin{align*}
d\mu_{\ST}(x):=\begin{cases}
\frac{1}{\pi}\sqrt{1-\frac{x^2}{4}}dx,\ \ & \text{if $-2\leq x\leq 2,$}\\
0,\ \ & \text{otherwise}.
\end{cases}
\end{align*}
We recall that an orthonormal basis for $\mu_{ST}$ is provided by the Chebyshev polynomials $C_r(X),\ r\geq 1$ where $C_r(X)$ (of degree $r$) is
defined by $$C_r(2\cos\theta)=\frac{\sin(r+1)\theta}{\sin\theta}.$$

Let $x\in [-2,2]$ and $p$ be a prime inert in $E$ and such that $p\nmid NN'$. Let $\sigma_{p,x}$ be the unramified unitary representation of $G'(\Qp)$ with Satake parameters $(\alpha_{x}(p),\alpha_{x}(p)^{-1})$ satisfying $$\alpha_{x}(p)+\alpha_{x}(p)^{-1}=x$$
and $\sigma_{E_p,x}$ its base change. 
Let $L(1/2,\sigma_{E_p,x}\times \pi_{E_p}')$ be the local Rankin-Selberg $L$-factor of the base change representations. 
We define the measure on $\Rr$ supported on $[-2,2]$
$$d\mu_p(x):=L(1/2,\sigma_{E_{p_i},x}\times \pi_{E_{p_i}}')d\mu_{\ST}(x).$$

Given $\bfp=(p_1,\cdots,p_m)$   a tuple of inert primes which are coprime with $NN'$, we define a measure $\mu_\bfp$ on $\mathbb{R}^m$ by 
\begin{equation}\label{measure}
d\mu_\bfp(x_1,\cdots,x_m):=d\mu_{p_1}(x_1)\otimes\cdots \otimes d\mu_{p_m}(x_m).
\end{equation}

\begin{remark}
The measure $\mu_\bfp$ is a positive measure since, by temperedness, the local factors  satisfy $$(1-1/p)^6\leq L(1/2,\sigma_{E_{p_i},x}\times \pi_{E_{p_i}}')\leq (1+1/p)^6$$	
\end{remark}

\subsection{Weighted Equidistribution of Joint Hecke Eigenvalues}

\begin{thm}\label{equidistribution}
Let the notation be as in Theorem \ref{firstmomentwithl}. Let $\bfp=(p_1,\cdots,p_m)$ be a tuple of inert primes coprime with $NN'$ and for any $\pi\in \mcA_k(N)$ set
$$\tilde\lambda_{\pi}(\mathbf{p}):=(\lambda_\pi(p_1)-p_1^{-1},\cdots, \lambda_\pi(p_m)-p_m^{-1})\in \mathbb{R}^m$$
where $\lambda_\pi(p)$ denotes the $p$-th Hecke eigenvalue. For any continuous function $\phi$ on $\mathbb{R}^m,$ we have, as $k+N\ra\infty$
\begin{align*}
\frac{N'^2}{w_E\mathfrak{S}(N')}\frac{d_k}{d_{\Lambda}}\frac{1}{N^2\Psi(N)}\sum_{\substack{\pi\in \mcA_k(N)\\ \vphi\in \mcB_k^{\widetilde{n}}(\pi)}}\frac{\big|\mathcal{P}(\vphi,\vphi')\big|^2}{\peter{\vphi,\vphi}\peter{\vphi',\vphi'}}\phi(\tilde\lambda_{\pi}(\mathbf{p}))=\mu_\bfp(\phi)+o(1),
\end{align*}
where   $\mu_\bfp$ is defined by \eqref{measure} and the error term depends on $E,N',\bfp$ and $\phi$. 
\end{thm}
\begin{remark}\label{unifSTrem} Regarding uniformity, it will be clear from the proof that this asymptotic formula  is valid as long as $N'\prod_{i=1}^mp_i$ is bounded by some absolute positive power of $kN$ but we will ignore this aspect here.
\end{remark}

\proof As a reminder (cf. \textsection\ref{secinertHecke}), we recall that $\lambda_{\pi}(p)$ can be expressed as $$\lambda_{\pi}(p)=\alpha_{\pi}(p)+1/p+\alpha^{-1}_{\pi}(p),$$ where $|\alpha_{\pi}(p)|=1$ since $\pi_p$ is tempered. In particular, $\lambda_{\pi}(p)-p^{-1}\in [-2,2]$.

  For $r\geq 0$ define  $\tilde{\lambda}_{\pi}(p^r)$ via the formula
\begin{equation}
	\sum_{r\geq 0}\frac{\tilde{\lambda}_{\pi}(p^r)}{p^{rs}}={(1-\frac{\alpha_\pi(p)}{p^s})^{-1}(1-\frac{\alpha^{-1}_\pi(p)}{p^s})^{-1}}.
	\label{lambdapitildeLserie}
\end{equation}
In particular we have
$$\tilde{\lambda}_{\pi}(p)=\alpha_{\pi}(p)+\alpha^{-1}_{\pi}(p)=\lambda_{\pi}(p)-1/p$$
and more generally
\begin{equation}\label{chebyprimepi}
\tilde{\lambda}_{\pi}(p^r)=C_r(\tilde{\lambda}_{\pi}(p)).
\end{equation}
In view of \eqref{lambdapiLserie} we also have
\begin{align*}
\tilde{\lambda}_{\pi}(p^r)=\frac{1}{p^r}\sum_{l=0}^r(-1)^{r-l}p^{l}{\lambda}_{\pi}(p^l).
\end{align*}
The identity above can be rewritten
$$\tilde{\lambda}_{\pi}(p^r)=\big(\frac{(-1)^{\Omega(\bullet)}}{\Id}\star {\lambda}_{\pi}\bigr)(p^r)$$
where $\star$ is the Dirichlet convolution and $\frac{(-1)^{\Omega(\bullet)}}{\Id}$ is the multiplicative function
$$n\mapsto \frac{(-1)^{\Omega(n)}}{n}.$$
In particular, for $\bfp=(p_1,\cdots,p_m)$ a tuple of inert primes  coprime with $NN'$ we have
$$\tilde\lambda_{\pi}(\bfp)=(\tilde\lambda_{\pi}(p_1),\cdots,\tilde\lambda_{\pi}(p_m)).$$
Moreover if, for a tuple of integers $(r_1,\cdots,r_m)\in\Nn^m$ and $\ell=p_1^{r_1}\cdots p_m^{r_m}$, we define 
$$
\tilde{\lambda}_{\pi}(\ell):=\prod_{i=1}^m\tilde{\lambda}_{\pi}(p_i^{r_i}),\ \tilde{\lambda}_{\pi}(1)=1,$$
we obtain a multiplicative function which can be expressed as a Dirichlet convolution:
\begin{equation}
	\tilde{\lambda}_{\pi}(\ell)=\sum_{\ell_1\ell_2=\ell}\frac{(-1)^{\Omega(\ell_1)}}{\ell_1}\lambda_\pi(\ell_2).
	\label{lambdapidirichletconv}
\end{equation}

We now turn to the combinatorics of the Hecke eigenvalues $\lambda_{\pi'}(p^r),\ r\geq 0$:
the product of Cartan cells 
\begin{align*}
K_p'\begin{pmatrix}
p&\\
&p^{-1}
\end{pmatrix}K_p'\cdot K_p'\begin{pmatrix}
p^{r}&\\
&p^{-r}
\end{pmatrix}K_p'
\end{align*}
decomposes as the disjoint union
\begin{align*}
K_p'\begin{pmatrix}
p^{r+1}&\\
&p^{-r-1}
\end{pmatrix}K_p'\bigsqcup K_p'\begin{pmatrix}
p^{r}&\\
&p^{-r}
\end{pmatrix}K_p'\bigsqcup K_p'\begin{pmatrix}
p^{r-1}&\\
&p^{-r+1}
\end{pmatrix}K_p',
\end{align*}
which implies the Hecke relation
\begin{align*} 
\lambda_{\pi'}(p)\lambda_{\pi'}(p^r)=\lambda_{\pi'}(p^{r+1})+\lambda_{\pi'}(p^r)+\lambda_{\pi'}(p^{r-1}).
\end{align*} 

So if we set
\begin{align*}
{\tilde{\lambda}_{\pi'}(p^r):=\sum_{l=0}^r(-1)^{r-l}{\lambda}_{\pi'}(p^l),}
\end{align*}
we see, by substituting this definition in the above relation, that
\begin{align*}
{\tilde{\lambda}_{\pi'}(p)\tilde{\lambda}_{\pi'}(p^r)=\tilde{\lambda}_{\pi'}(p^{r+1})+\tilde{\lambda}_{\pi'}(p^{r-1}).}
\end{align*}
This in turn implies that
\begin{equation}\label{chebyprimepi'}
{\tilde{\lambda}_{\pi'}(p^r)=C_r(\tilde{\lambda}_{\pi'}(p)).} 
\end{equation}

\begin{remark}
	Unlike the case of $\tilde{\lambda}_{\pi}(p^r)$, there is no factor $p^l$ included in the definition of $\tilde{\lambda}_{\pi'}(p^r)$.
\end{remark}

If we set for $\ell=p_1^{r_1}\cdots p_m^{r_m}$
$$\tilde{\lambda}_{\pi'}(\ell):=\prod_{i=1}^m\tilde{\lambda}_{\pi'}(p_i^{r_i}),\ \tilde{\lambda}_{\pi'}(1)=1,$$
we see that $\tilde{\lambda}_{\pi'}(\ell)$ can be expressed as a  Dirichlet convolution
\begin{equation}
	\label{lambdapi'dirichlet}
{\tilde{\lambda}_{\pi'}(\ell)=\big((-1)^{\Omega(\bullet)}\star \lambda_{\pi'}\big)(\ell)=\sum_{\ell_1\ell_2=\ell}(-1)^{\Omega(\ell_1)}\lambda_{\pi'}(\ell_2).}
\end{equation}

Suppose $N>\ell^2{N'}^2$. Combining Theorem \ref{firstmomentwithl} with \eqref{lambdapidirichletconv} and \eqref{lambdapi'dirichlet}, we have
\begin{multline}\label{eq13.2}
	\frac{N'^2}{w_E\mathfrak{S}(N')}\frac{d_k}{d_{\Lambda}}\frac{1}{N^2\Psi(N)}\sum_{\substack{\pi\in \mcA_k(N)\\ \vphi\in \mcB_k^{\widetilde{n}}(\pi)}}\tilde{\lambda}_{\pi}(\ell)\frac{\big|\mathcal{P}(\vphi,\vphi')\big|^2}{\peter{\vphi,\vphi}\peter{\vphi',\vphi'}}=\\
	{\frac{\tilde{\lambda}_{\pi'}(\ell)}{\ell}}
	+\mcR(k,\ell,N,N').
\end{multline}
where
$$\mcR(k,\ell,N,N')=\frac{(\ell{NN'})^{o(1)}}{2^{4k}kN{N'}\ell}+(k\ell NN')^{o(1)}{{\ell^{17}}N'^{\frac{20}{3}}k^{1/2}}(1+\frac{\ell^2{N'}^2}N)^{2} \mcE$$
with
$$\mcE=e^{-\frac{\kappa}{(\ell N')^2+1}}+2^{-\kappa}$$
and if we assume that 
\begin{equation}
	\label{lN'upperbound}
	\ell^2{N'}^2<N,
\end{equation}  the second term  can be replaced by
$$(k\ell NN')^{o(1)}
{{\ell^9}N^2{N'}^{14/3}k^{1/2}}(\frac{\ell^2{N'}^2}{N})^\kappa.$$

The now interpret this formula in terms of the measure discussed above.

For $x=(x_1,\cdots,x_m)\in[-2,2]^m$ let $$\phi(x)=\prod_{i=1}^mC_{r_i}(x_i).$$ 
From \eqref{chebyprimepi} we have
\begin{align*}
\tilde{\lambda}_{\pi}(\ell):=\prod_{i=1}^m\tilde{\lambda}_{\pi}(p_i^{r_i})=\prod_{i=0}^mC_{r_i}(\lambda_{\pi}(p_i)-p_i^{-1})=\phi(\lambda_{\pi}(\mathbf{p})),
\end{align*}

Also for $i=1,\cdots,m$ we have 
$$
L(1/2,\sigma_{E_{p_i},x_i}\times \pi_{E_{p_i}}')=\sum_{r=0}^{\infty}\frac{C_{r}(x_i)\tilde\lambda_{\pi'}(p_i^{r})}{p_i^{r}}$$
and by \eqref{chebyprimepi'} this is equal to
$$
{\sum_{r=0}^{\infty}\frac{C_{r}(x_i)C_{r_i}(\tilde\lambda_{\pi'}(p_i))}{p_i^r}.}
$$
Since Chebyshev polynomials are orthonormal relative to $d\mu_{\ST},$ we have
\begin{align*}
{\frac{\tilde{\lambda}_{\pi'}(p_i^{r_i})}{{p_i^{r_i}}}=\int_{\mathbb{R}}L(1/2,\sigma_{E_{p_i},x_i}\times \pi_{E_{p_i}}')C_{r_i}(x_i)d\mu_{\ST}(x_i)=\mu_{p_i}(C_{r_i}).}
\end{align*}

We have therefore
$$
{\frac{\tilde{\lambda}_{\pi'}(\ell)}{\ell}}=\prod_{i=1}^m{\frac{\tilde{\lambda}_{\pi'}(p_i^{r_i})}{{p_i^{r_i}}}}=\mu_{\bfp}(\phi)$$

So \eqref{eq13.2} becomes 
\begin{multline}
	\frac{N'^2}{w_E\mathfrak{S}(N')}\frac{d_k}{d_{\Lambda}}\frac{1}{N^2\Psi(N)}\sum_{\substack{\pi\in \mcA_k(N)\\ \vphi\in \mcB_k^{\widetilde{n}}(\pi)}}\frac{\big|\mathcal{P}(\vphi,\vphi')\big|^2}{\peter{\vphi,\vphi}\peter{\vphi',\vphi'}}\phi(\lambda_{\pi}(\mathbf{p}))=\\\mu_{\bfp}(\phi)
	+\mcR(k,\ell,N,N').\end{multline}
	
Suppose that $k+N\ra\infty$. If $k\geq N$ we see that $$\mcR(k,\ell,N,N')=o_{\phi,N'}(1)$$
since $\mcE$ converges exponentially fast to $0$ while the dependency in $N$ as at most polynomial.
If $N\geq k$ then for $N$ large enough \eqref{lN'upperbound} is satisfied and
$$\mcR(k,\ell,N,N')=\frac{(\ell{NN'})^{o(1)}}{2^{4k}kN{N'}\ell}+(k\ell NN')^{o(1)}
{{\ell^9}N^{2+1/2}{N'}^{14/3}}(\frac{\ell^2{N'}^2}{N})^\kappa.$$
The first term in the expression above is always $o_{\phi,N'}(1)$ while the second term is because $\kappa>6$.

Theorem \ref{equidistribution} for general $\phi$ follows from the Stone-Weierstrass theorem.
\qed



\section{\bf Averaging over Forms of Exact Level $N$}\label{secoldnew}

Suppose $N>1$ is an inert prime. With the choice of the test function made in \S  \ref{secglobalf} the spectral side of the relative trace formula picks up
 both newforms and oldforms of level $N$. In this section, we show that when $N$ is large enough the contribution from the oldforms is smaller than from the new forms; from this, we will eventually  deduce \eqref{firstmomentn}.
 
 We use the notation of \S \ref{secspectralexp}. The set $\mcA_k(N)$
is the disjoint union of the two subsets
 $\mcAkn(N)$ and $\mcA_k(1)$
 where $$\mcA_k(1)=\{\pi=\pi_{\infty}\otimes\pi_{f}\in\mathcal{A}(G), \omega_{\pi}=\textbf{1},\ \pi_\infty\simeq D^{\Lambda},\ \pi_{f}^{K_f(1)}\not=\{0\}\}$$ is the space of "level $
 1$" automorphic representations  and $\mcAkn(N)$ is the space of  "new" at $N$ automorphic representations . 
 
 Consequently the space of automorphic forms $\mcV_{k}(N)$  admits  an orthogonal decomposition
 $$\mcV_{k}(N)=\mcV^{new}_{k}(N)\oplus \mcV^{old}_{k}(N)$$
 (here $\mcV^{old}_{k}(N)$ is the subspace generated by the forms that belong to the elements of $\mcA_k(1)$). We choose a corresponding orthogonal basis
 $$\mcB_k(N)=\mcB^{new}_k(N)\sqcup \mcB^{old}_k(N)$$ whose elements belong to the $\pi$ contained in either $\mcAkn(N)$ or $\mcA_k(1)$ and are factorable vectors. Accordingly we have a corresponding decomposition
 $$\mcB^{\tilde\mfn}_k(N)=\mcB^{\tilde\mfn,new}_k(N)\sqcup \mcB^{\tilde\mfn,old}_k(N)$$
 and the spectral side of the relative trace formula decomposes as
\begin{align*}
J(f^{\mathfrak{n}})=J^{new}(f^{\mathfrak{n}})+J^{old}(f^{\mathfrak{n}}),
\end{align*}
where 
\begin{align*}
J^{new}(f^{\mathfrak{n}})=&\frac{1}{d_{\Lambda}}\sum_{\vphi\in \mathcal{B}_k^{\tilde{\mathfrak{n}},new}(N)
}\frac{\big|\mcP(\vphi,\vphi')\big|^2}{\peter{\vphi,\vphi}\peter{\vphi',\vphi'}},\\
J^{old}(f^{\mathfrak{n}})=&\frac{1}{d_{\Lambda}}\sum_{\vphi\in \mathcal{B}_k^{\tilde{\mathfrak{n}},old}(N)
}\frac{\big|\mcP(\vphi,\vphi')\big|^2}{\peter{\vphi,\vphi}\peter{\vphi',\vphi'}}.
\end{align*}

We show that the contribution from oldforms are negligible. The main result of this section is the following. 
\begin{prop}\label{prop1}\label{oldformcontrib}
With the same notations and assumptions as in Theorem \ref{thmA}, we have 
\begin{equation}\label{0}
J^{old}(f^{\mathfrak{n}})\ll_{N'} \frac{1}{k}
\end{equation}
\end{prop}

\subsection{Proof of Theorem \ref{thmA}}\label{secproofthmA} Assuming this Proposition let us prove \eqref{firstmomentn}.

We have by Theorem \ref{firstmomentwithl}  
\begin{align}
\sum_{\vphi\in \mathcal{B}_k^{\tilde{\mathfrak{n}}}(N)
}\frac{\big|\mcP(\vphi,\vphi')\big|^2}{\peter{\vphi,\vphi}\peter{\vphi',\vphi'}}&=\sum_{\vphi\in \mathcal{B}_k^{\tilde{\mathfrak{n}},new}(N)}
\frac{\big|\mcP(\vphi,\vphi')\big|^2}{\peter{\vphi,\vphi}\peter{\vphi',\vphi'}}+O_{N'}(\frac{d_\Lambda}{d_k})\nonumber\\
&=w_E\frac{ d_{\Lambda}}{d_k}(\frac{N}{N'})^2\Psi(N)\mathfrak{S}(N')\label{newoldmoment}
\\
&\quad+
(k{N})^{o(1)}\frac{Nk}{2^{4k}} +(kN)^{o(1)}\frac{k^{5/2}}{N^{2}} (e^{-\frac{\kappa}{(\ell N')^2+1}}+2^{-\kappa})\nonumber
\end{align}
For $N$ sufficiently large (depending on $N'$) the main term
\begin{equation}
	\label{main13}
	w_E\frac{ d_{\Lambda}}{d_k}(\frac{N}{N'})^2\Psi(N)\mathfrak{S}(N')\asymp \frac{ d_{\Lambda}}{d_k}(\frac{N}{N'})^2
\end{equation}
will be twice the size of the term $O_{N'}({ d_{\Lambda}}/{d_k})$ above; moreover as $k+N\ra\infty$ the second and third terms on the righthand side of  \eqref{newoldmoment} are negligible compared to \eqref{main13}. Therefore under the assumptions of Theorem \ref{thmA} we have
\begin{equation}
	\label{eqnewformsperiod}
	\sum_{\vphi\in \mathcal{B}_k^{\tilde{\mathfrak{n}},new}(N)}
\frac{\big|\mcP(\vphi,\vphi')\big|^2}{\peter{\vphi,\vphi}\peter{\vphi',\vphi'}}\asymp_{N'} \frac{ d_{\Lambda}}{d_k}N^2.
\end{equation}

By Proposition \ref{prop33} we have for any $\vphi\in \mcB^{\tfn}_{\pi}(N),\ \pi\in\mcAkn(N)$
\begin{equation}\label{224}
\frac{\big|\mcP(\vphi,\vphi')\big|^2}{\langle\vphi,\vphi\rangle\langle\vphi',\vphi'\rangle}\asymp \frac{L(1/2,\pi_{E}\times\pi'_E)}{L(1,\pi,\Ad)L(1,\pi',\Ad)}\cdot \frac{1}{d_k}\frac{1}{N{N'}^2},
\end{equation}
which  implies  that
$$\sum_{\pi\in \mcAkn(N)}\frac{L(1/2,\pi_{E}\otimes\pi'_E)}{L(1,\pi,\Ad)L(1,\pi',\Ad)}\asymp_{N'} d_\Lambda N^3\asymp_{N'}|\mcAkn(N)|$$
by Weyl's law \eqref{Weyllaw}.
\qed

\subsection{Local Analysis: an elucidation of oldforms}\label{secold}
Let $p$ be a prime inert in $E$. Given $\pi\in \mcA_k(1)$, in this section we shall describe explicitly the space, $\pi_p^{I_p}$, of Iwahori-fixed vectors at $p$. This is certainly well known but we could not find a reference for it.

As this subsection is purely local (at the place $p$), we will often, to simplify the notation, omit the index $p$: we will write $E$ for the local field $E_p$ $\nu$ for $\nu_p$ its valuation, $G$ for $G(\Qp)$, $\pi$ for the local component $\pi_p$, $I$ for $I_p$, $\pi^I$ for $\pi_p^{I_p}$ etc...

Let $\phi^{\circ}\in\pi$ be the spherical vector normalized such that $\phi^{\circ}(e)=1,$ where $e$ is the identity matrix in $G$.

It is well known that the subspace $\pi^{I}$ has dimension two : $\pi$ is induced from an unramified character, hence trivial on $G(\Zp)$ and $$B(\Zp)\bash G(\Zp)/I\simeq B(\Fp)\bash G(\Fp)/B(\Fp)$$ has two elements. Obviously $\pi^{I}$ contains $\phi^{\circ}$. Our first goal  is to construct an explicit vector $\phi^*\in\pi^I$ which is not multiple of $\phi^{\circ}$. By the Gram-Schmidt process, we will obtain an orthonormal basis of $\pi^I.$

\subsection{Construction of $\phi^*$}
Let $$t=A^{-1}_{1}=\diag(p^{-1},1,p);$$ we set
\begin{align*}
\phi^*(g):=\frac{1}{\vol(I)}\int_{I}\phi^{\circ}(gkt)dk.
\end{align*}

\begin{lemma}\label{lem1}
We have 	
\begin{equation}\label{eqphi*}
\phi^*=p^{-1}\pi(t)\phi^{\circ}+p^{-1}\sum_{\alpha\in \mathbb{F}_p^{\times}}\pi\left(\begin{pmatrix}
	1&&i\alpha p^{-1}\\
	&1\\
	&&1
\end{pmatrix}\right)\phi^{\circ}\in \pi^I-\{0\}.
\end{equation}
\end{lemma}
\begin{proof}
Define
\begin{align*}
I_{p^2}:=G(\Zp)\cap \begin{pmatrix}
	\mcO_p&\mcO_p&\mcO_p\\ p^2\mcO_p&\mcO_p&\mcO_p\\
p^2\mcO_p & p^2\mcO_p &\mcO_p\\
\end{pmatrix}.
\end{align*}

Let $\mathcal{S}$ be a set of representatives for $I/I_{p^2}$. We may assume that the elements of $\mathcal{S}$ have the form $\begin{pmatrix}
1&& \\
\alpha_1 &1&\\
\alpha_2 &\alpha_3 &1
\end{pmatrix}$, where $\alpha_1, \alpha_2, \alpha_3\in p\mcO_p/p^2\mcO_p$. In particular, the unitary group structure requires $\alpha_2\in ip\mathbb{Z}_p/ip^2\mathbb{Z}_p$.

For $\alpha\in \mathbb{F}_p$, let $\mathcal{S}(\alpha)$ denote the subset of $\mathcal{S}$ consisting of elements whose $(3,1)$-entry is equal to $ip\alpha$. Then
\begin{equation}\label{13.7}
\mathcal{S}=\bigsqcup_{\alpha\in \mathbb{F}_p}\mathcal{S}(\alpha)=\bigsqcup_{\alpha\in \mathbb{F}_p^{\times}}\mathcal{S}(\alpha)\bigsqcup \mathcal{S}(0).
\end{equation}

Observe that, for each $\alpha\in \mathbb{F}_p$, the map
\begin{equation}\label{e13.8}
\mathcal{S}(0)\longrightarrow\mathcal{S}(\alpha),\ \ \ \delta\mapsto \begin{pmatrix}
1\\
& 1\\
ip\alpha &  &1
\end{pmatrix}\delta
\end{equation}
is bijective. Consequently,
\begin{equation}\label{13.8}
\#\mathcal{S}(0)=\#\mathcal{S}(\alpha)	\ \   \text{for all $\alpha\in \mathbb{F}_p$}.
\end{equation}

Note that $t^{-1}I_{p^2}t\subseteq K_p$. Therefore,
\begin{equation}\label{13.9}
\phi^*(g)=\frac{1}{\vol(I)}\sum_{\delta\in \mathcal{S}}\int_{I_{p^2}}\phi^{\circ}(g\delta kt)dk
=\frac{\vol(I_{p^2})}{\vol(I)}\sum_{\delta\in \mathcal{S}}\phi^{\circ}(g\delta t)
=\frac{1}{\#\mathcal{S}}\sum_{\delta\in \mathcal{S}}\phi^{\circ}(g\delta t).
\end{equation}

Substituting \eqref{13.7} and \eqref{e13.8} into \eqref{13.9} yields
\begin{equation}\label{13.11}
\phi^*(g)=\frac{1}{\#\mathcal{S}}\Big[\sum_{\delta\in \mathcal{S}(0)}\phi^{\circ}(g\delta t)+\sum_{\alpha\in \mathbb{F}_p^{\times}}\sum_{\delta\in \mathcal{S}(0)}\phi^{\circ}\left(g\begin{pmatrix}
1\\
& 1\\
ip\alpha &  &1
\end{pmatrix}\delta t\right)\Big].
\end{equation}

Since $t^{-1}\mathcal{S}(0)t\subset K_p$, it follows from \eqref{13.11} that
\begin{align*}
\phi^*(g)=\frac{\#\mathcal{S}(0)}{\#\mathcal{S}}\Big[\phi^{\circ}(gt)+\sum_{\alpha\in \mathbb{F}_p^{\times}}\phi^{\circ}\left(g\begin{pmatrix}
1\\
& 1\\
ip\alpha &  &1
\end{pmatrix}t\right)\Big],
\end{align*}
which, in conjunction with \eqref{13.8}, yields
\begin{equation}\label{13.12}
\phi^*(g)=p^{-1}\phi^{\circ}(gt)+p^{-1}\sum_{\alpha\in \mathbb{F}_p^{\times}}\phi^{\circ}\left(g\begin{pmatrix}
1\\
& 1\\
ip\alpha &  &1
\end{pmatrix}t\right).
\end{equation}

We have the identity  
\begin{equation}\label{eqalphaproduct}
\begin{pmatrix}
	1&&i\alpha^{-1} p^{-1}\\
	&1&\\
	&&1
\end{pmatrix}\begin{pmatrix}
p^{-1}\\
&1\\
i\alpha& & p
\end{pmatrix}=\begin{pmatrix}
&&i\alpha^{-1}\\
&1\\
i\alpha& & p
\end{pmatrix}\in K_p.
\end{equation}
Then the equality \eqref{eqphi*} follows from \eqref{13.12} and \eqref{eqalphaproduct}, together with the change of variable $\alpha\mapsto \alpha^{-1}$.

Since $JtJ=t^{-1}$ and $\phi^{\circ}$ is spherical, it follows from \eqref{eqphi*} that
\begin{equation}\label{13.14}
\phi^*(J)=p^{-1}\phi^{\circ}(t^{-1})+p^{-1}\sum_{\alpha\in \mathbb{F}_p^{\times}}\phi^{\circ}\left(\begin{pmatrix}
	&&1\\
	&1\\
	1&&i\alpha p^{-1}
\end{pmatrix}\right).
\end{equation}
Substituting the decomposition
\begin{align*}
\begin{pmatrix}
	&&1\\
	&1\\
	1&&i\alpha p^{-1}
\end{pmatrix}=\begin{pmatrix}
i\alpha^{-1}p&&1\\
	&1\\
	&&i\alpha p^{-1}
\end{pmatrix}
\begin{pmatrix}
1&&\\
&1\\
-i\alpha^{-1}p & & 1
\end{pmatrix}
\end{align*}
into \eqref{13.14}, we obtain
\begin{align*}
\phi^*(J)&=p^{-1}\phi^{\circ}(t^{-1})+p^{-1}\sum_{\alpha\in \mathbb{F}_p^{\times}}\phi^{\circ}\left(\begin{pmatrix}
i\alpha^{-1}p&&1\\
	&1\\
	&&i\alpha p^{-1}
\end{pmatrix}\right)\\
&=p^{-1}\phi^{\circ}(t^{-1})+p^{-1}\sum_{\alpha\in \mathbb{F}_p^{\times}}\phi^{\circ}(t^{-1})=\phi^{\circ}(t^{-1})=\delta(t^{-1})^{\frac{1}{2}}\chi^2(p)=p^{-2}\chi^2(p)\neq 0.
\end{align*}
Hence, $\phi^*\in \pi^I-\{0\}$.
\end{proof}

\begin{lemma}\label{lem2}
The vector $\phi^*$ is not a scalar multiple of $\phi^{\circ}.$
\end{lemma}
\begin{proof}
Since $\phi^{\circ}$ is spherical, we have
$$
\phi^{\circ}(e)=\phi^{\circ}(J)=1.
$$

On the other hand, by Lemma \ref{lem1} we have  
\begin{align*}
\phi^*(e)=p^{-1}\phi^{\circ}(t)+p^{-1}\sum_{\alpha\in \mathbb{F}_p^{\times}}\phi^{\circ}\left(\begin{pmatrix}
	1&&i\alpha p^{-1}\\
	&1\\
	&&1
\end{pmatrix}\right)=p^{-1}(\phi^{\circ}(t)+p-1).
\end{align*}

Since $\phi^{\circ}(t)=\delta(t)^{\frac{1}{2}}\overline{\chi}^2(p)=p^2\overline{\chi}^2(p)$, it follows that
\begin{align*}
\phi^*(e)=p^{-1}(p^2\overline{\chi}^2(p)+p-1)=p\overline{\chi}^2(p)+1-p^{-1}\neq \phi^*(J).
\end{align*}

Therefore, $\phi^*$ cannot be a scalar multiple of $\phi^{\circ}$.
\end{proof}

\subsubsection*{The Gram-Schmidt Process}\label{sec2.4}
Let
\begin{align*}
\phi^{\dagger}=\frac{\phi^*-\frac{\peter{\phi^*, \phi^{\circ}}}{\peter{\phi^{\circ},\phi^{\circ}}}\phi^{\circ}}{\sqrt{\peter{\phi^*,\phi^*} -\frac{\peter{\phi^*, \phi^{\circ}}^2}{\peter{\phi^\circ,\phi^\circ}}}}.
\end{align*}
Then by construction$$\{\frac{\phi^\circ}{\sqrt{\peter{\phi^\circ,\phi^\circ}}},\ \phi^{\dagger}\}$$
is an orthonormal basis of $\pi^I$.

\subsubsection*{Norm computations} Given $\alpha\in\Zpt$ we set
\begin{equation}
	\label{nalphadef}
	n_{\alpha}=\begin{pmatrix}
	1&&i\alpha p^{-1}\\
	&1\\
	&&1
\end{pmatrix}.
\end{equation}
We have
\begin{equation}\label{6.}
\langle \pi(n_{\alpha})\phi^{\circ},\phi^{\circ}\rangle =\langle \pi(t^{-1})\phi^{\circ},\phi^{\circ}\rangle
\end{equation}
as \eqref{eqalphaproduct} implies that $\begin{pmatrix}
1&i\alpha p^{-1}\\
&1
\end{pmatrix}\in K't^{-1}K'.$ Thus by \eqref{eqphi*} we have
\begin{align*}
\langle\phi^*,\phi^{\circ}\rangle=&\langle p^{-1}\pi(t)\phi^{\circ}+p^{-1}\sum_{\alpha\in \mathbb{F}_p^{\times}}\pi(n_{\alpha})\phi^{\circ},\phi^{\circ}\rangle\\
=&p^{-1}\langle \pi(t)\phi^{\circ},\phi^{\circ}\rangle+p^{-1}(p-1)\langle \pi(t^{-1})\phi^{\circ},\phi^{\circ}\rangle,
\end{align*}
and
\begin{align*}
\langle\phi^*,\phi^*\rangle=&\langle p^{-1}\pi(t)\phi^{\circ}+p^{-1}\sum_{\alpha\in \mathbb{F}_p^{\times}}\pi(n_{\alpha})\phi^{\circ},p^{-1}\pi(t)\phi^{\circ}+p^{-1}\sum_{\alpha\in \mathbb{F}_p^{\times}}\pi(n_{\alpha})\phi^{\circ}\rangle\\
=&p^{-2}\langle \phi^{\circ},\phi^{\circ}\rangle+p^{-2}\sum_{\alpha\in \mathbb{F}_p^{\times}}\langle \pi(n_{\alpha})\phi^{\circ},\pi(t)\phi^{\circ}\rangle\\
&+p^{-2}\sum_{\alpha\in \mathbb{F}_p^{\times}}\langle \pi(t)\phi^{\circ},\pi(n_{\alpha})\phi^{\circ}\rangle+p^{-2}\sum_{\alpha\in \mathbb{F}_p^{\times}}\sum_{\beta\in \mathbb{F}_p^{\times}}\langle \pi(n_{\beta}^{-1}n_{\alpha})\phi^{\circ},\phi^{\circ}\rangle.
\end{align*}

Note that
\begin{align*}
t^{-1}n_{\alpha}=\begin{pmatrix}
p&&i\alpha \\
& 1\\
&&p^{-1}
\end{pmatrix}=\begin{pmatrix}
1&& i\alpha p \\
&1\\
&& 1
\end{pmatrix}\begin{pmatrix}
p&\\
& 1\\
&&p^{-1}
\end{pmatrix}.
\end{align*}
Therefore, we have
\begin{align*}
\langle \pi(n_{\alpha})\phi^{\circ},\pi(t)\phi^{\circ}\rangle=\langle \pi(t^{-1}n_{\alpha})\phi^{\circ},\phi^{\circ}\rangle=\langle \pi(t^{-1})\phi^{\circ},\phi^{\circ}\rangle=\langle \pi(t)\phi^{\circ},\phi^{\circ}\rangle,
\end{align*}
where we use the fact that $t^{-1}=JtJ$ and $\phi^{\circ}$ is right $J$-invariant.
By \eqref{eqalphaproduct}, we have
\begin{align*}
\langle \pi(n_{\beta}^{-1}n_{\alpha})\phi^{\circ},\phi^{\circ}\rangle=\begin{cases}
	\langle \phi^{\circ},\phi^{\circ}\rangle,&\ \text{if $\alpha=\beta$ in $\mathbb{F}_p^{\times}$}\\
	\langle \pi(t)\phi^{\circ},\phi^{\circ}\rangle,&\ \text{otherwise}.
\end{cases}
\end{align*}

Therefore, we obtain
\begin{align*}
\langle\phi^*,\phi^{\circ}\rangle=\langle \pi(t)\phi^{\circ},\phi^{\circ}\rangle,
\end{align*}
and
\begin{align*}
\langle\phi^*,\phi^*\rangle=&p^{-1}\langle \phi^{\circ},\phi^{\circ}\rangle+(1-p^{-1})\langle \pi(t)\phi^{\circ},\phi^{\circ}\rangle.
\end{align*}

By  MacDonald's formula for spherical vectors and the temperedness of $\pi$ we have
$$\frac{\peter{\pi(t)\phi^\circ,\phi^\circ}}{\peter{\phi^\circ,\phi^\circ}}=O(\frac{1}{p^2}).
$$

Consequently,
\begin{align}\nonumber
\frac{\langle\phi^*,\phi^*\rangle}{\peter{\phi^\circ,\phi^\circ}}
=&p^{-1}+(1-p^{-1})\frac{\langle \pi(t)\phi^{\circ},\phi^{\circ}\rangle}{\peter{\phi^\circ,\phi^\circ}}=p^{-1}+O(\frac{1}{p^2}),\\
\frac{\peter{\phi^*,\phi^\circ}}{\peter{\phi^\circ,\phi^\circ}}=&\frac{\peter{t.\phi^\circ,\phi^\circ}}{\peter{\phi^\circ,\phi^\circ}}=O(\frac{1}{p^2}),\label{phi*phioinner}
\end{align}
where the implied constants are absolute. Consequently we have
\begin{equation}\label{eqdifference}
\frac{1}{\peter{\phi^\circ,\phi^\circ}}\bigl({\peter{\phi^*,\phi^*}} -\frac{\peter{\phi^*, \phi^{\circ}}^2}{\peter{\phi^\circ,\phi^\circ}}\bigr)=p^{-1}+O(\frac{1}{p^2}).
\end{equation}

\subsection{Global Analysis: Proof of Proposition \ref{prop1}}
In this section we are back to the global setting and return to the notation in force at the beginning of section \ref{secoldnew}.

Given $\pi\simeq \otimes_{p\leq\infty}\pi_p\in\mcA_k(1)$, by the previous section we may assume that
$$\mcB_{\pi,k}(N)=\pi\cap \mcB_k(N)=\{\vphi_\pi^\circ,\vphi_\pi^\dagger\}$$
is made of two factorable vectors such that
$$\vphi_\pi^\circ\simeq \otimes_{v}\phi_v^\circ,\ \vphi_\pi^\dagger\simeq \phi_N^\dagger\otimes(\otimes_{v\not= N}\phi_v^\circ)$$ where
\begin{itemize}
	\item $\phi_v^\circ\in\pi_v$ is either spherical for $v<\infty$ or a highest weight vector of the minimal $K$-type of $D^\Lambda$  for $v=\infty$ and
	\item $\phi_N^{\dagger}\in\pi_N^{I_N}$ is the vector denoted $\phi^{\dagger}$ in \textsection\ref{sec2.4}.
\end{itemize}
We have
\begin{equation}\label{6}
J^{old}(f^{\mathfrak{n}})=\frac{1}{d_{\Lambda}}\sum_{\pi\in\mcAk(1)
}\frac{\big|\mcP(\vphi_\pi^{\circ},\vphi')\big|^2}{\langle\vphi_\pi^{\circ},\vphi_\pi^{\circ}\rangle}+\frac{1}{d_{\Lambda}}\sum_{\pi\in\mcAk(1)
}\frac{\big|\mcP(\vphi_\pi^{\dagger},\vphi')\big|^2}{\langle\vphi_\pi^{\dagger},\vphi_\pi^{\dagger}\rangle}.
\end{equation}

We handle the second sum on the RHS of \eqref{6}. Set $$p=N,\ t=\diag(p^{-1},1,p)\in G(\mathbb{Q}_p)\hookrightarrow G(\mathbb{A}).
$$
We have for $\pi\in\mcAk(1)$ (in the sequel we drop the index $\pi$ to simplify the notation)
\begin{equation}
	\label{pdagboundCS}
	\frac{|\mcP(\vphi_\pi^{\dagger},\vphi')|^2}{\peter{\vphi_\pi^{\dagger},\vphi_\pi^{\dagger}}}\leq 2\frac{|\mcP(\phi^*,\vphi')|^2+|\frac{\peter{\phi^*, \phi^{\circ}}}{\peter{\phi^{\circ},\phi^{\circ}}}|^2|\mcP(\phi^\circ,\vphi')|^2}{\peter{\phi^*,\phi^*} -\frac{\peter{\phi^*, \phi^{\circ}}^2}{\peter{\phi^\circ,\phi^\circ}}}.
\end{equation}
We recall that
\begin{align*}
\phi^*=p^{-1}\pi_p(t)\phi^{\circ}+p^{-1}\sum_{\alpha\in \mathbb{F}_p^{\times}}\pi_p(n_{\alpha})\phi^{\circ}\in \pi^I,
\end{align*}
and $n_{\alpha}$ as in \eqref{nalphadef}. By the change of variables $$x\mapsto xt^{-1},\ xn_{\alpha}\mapsto xn_{\alpha}^{-1},\ \alpha\mapsto -\alpha,$$ we derive that  
\begin{align*}
\mcP(\phi^*,\vphi')=&p^{-1}\int_{[G']}\phi^{\circ}(x)\vphi'(xt^{-1})dx+p^{-1}\sum_{\alpha\in \mathbb{F}_p^{\times}}\int_{[G']}\phi^{\circ}(x)\vphi'(xn_{\alpha})dx.
\end{align*}
with $[G']$ is defined in \eqref{bGdef}.

Similar to \eqref{6.}, or making use of \eqref{eqalphaproduct}, we have for any $\alpha\in\Fp^\times$
\begin{align*}
\int_{[G']}\phi^{\circ}(x)\vphi'(xn_{\alpha})dx=\int_{[G']}\phi^{\circ}(x)\vphi'(xt^{-1})dx
\end{align*}
so that
\begin{equation}
	\label{phi*inner}
	\mcP(\phi^*,\vphi')=\int_{[G']}\phi^{\circ}(x)\vphi'(xt^{-1})dx.
\end{equation}
Let $K'$ be the maximal compact subgroup of $G'(\mathbb{A}).$ Since $\phi^{\circ}$ is $K'$-invariant, we have by a change of variable,
\begin{align*}
\int_{[G']}\phi^{\circ}(x)\vphi'(xt^{-1})dx=\frac{1}{\vol(K')}\int_{[G']}\phi^{\circ}(x)\int_{K'}\vphi'(xk't^{-1})dk'dx,
\end{align*}
where the inner integral defines a spherical function.

By multiplicity one and MacDonald's formula, we have
\begin{align*}
\int_{[G']}\phi^{\circ}(x)\vphi'(xt^{-1})dx=c_{\pi'_N}(t^{-1})\int_{[G']}\phi^{\circ}(x)\vphi'(x)dx=c_{\pi'_N}(t^{-1})\mcP(\phi^{\circ},\vphi')
\end{align*}
for some scalar function $$c_{\pi'_N}(t^{-1})\ll \delta'(t^{-1})\ll p^{-1}.$$
Here the implied constant is absolute since $\pi_p'$ is tempered. Therefore, by \eqref{phi*inner} we have
\begin{equation}\label{13.23}
\mcP(\phi^*,\vphi')\ll p^{-1}\mcP(\phi^{\circ},\vphi').
\end{equation}
By \eqref{phi*phioinner}, \eqref{eqdifference}, \eqref{pdagboundCS}, and \eqref{13.23}, we obtain
\begin{align*}
\frac{\big|\mcP(\vphi^{\dagger},\vphi')\big|^2}{\peter{\vphi^{\dagger},\vphi^{\dagger}}}\ll \frac{1}p\frac{|\mcP(\phi^{\circ},\vphi')|^2}{\peter{\phi^{\circ},\phi^{\circ}}},
\end{align*}
which implies
\begin{equation}\label{ident9}
\frac{1}{d_{\Lambda}}\sum_{\pi\in\mcAk(1)
}\frac{\big|\mcP(\vphi_\pi^{\dagger},\vphi')\big|^2}{\langle\vphi_\pi^{\dagger},\vphi_\pi^{\dagger}\rangle}\ll \frac{1}{N} \frac{1}{d_{\Lambda}}\sum_{\pi\in\mcAk(1)
}\frac{\big|\mcP(\vphi_\pi^{\circ},\vphi')\big|^2}{\langle\vphi_\pi^{\circ},\vphi_\pi^{\circ}\rangle}.
\end{equation}

Combining \eqref{6} with \eqref{ident9} we obtain
$$J^{old}(f^{\mathfrak{n}})\ll \frac{1}{d_{\Lambda}}\sum_{\pi\in\mcAk(1)
}\frac{\big|\mcP(\vphi_\pi^{\circ},\vphi')\big|^2}{\langle\vphi_\pi^{\circ},\vphi_\pi^{\circ}\rangle}.$$
The argument of \S \ref{secThmBcollect} for $N=\ell=1$ yields
\begin{align*}
J^{old}(f^{\mathfrak{n}})
\ll_{N'} \frac{1}{k}	+\frac{k^{o(1)}}{2^{4k}k^{2}}{ }+k^{1/2+o(1)} (e^{-\frac{\kappa}{{N'}^2+1}}+2^{-\kappa})\ll_{N'}\frac{1}{k},
\end{align*}
proving \eqref{0}.



\section{\bf Amplification and Non-vanishing}\label{non-vanish}

In this section, we prove Theorem \ref{upperboundperiodthm} and Theorem \ref{thmnonvanishpower} . We assume that
$$k,N\geq C(E,N')$$
for a suitable constant depending on $E$ and $N'$.

\subsection{The Amplifier}
Let $\sigma\in \mcA_k(N).$ Let $L>1.$ Denote by 
$$
\mathcal{L}:=\{L/2<\ell<L:\ \text{$\ell$ is an inert prime in $E$, and $(\ell,NN')=1$}\}.
$$
By the Prime Number Theorem in arithmetic progressions, one has $|\mathcal{L}|\asymp_E L/\log L,$ where the implied constant depends on $E.$

Recall that, for $r\geq 1,$ one has $\lambda_{\sigma}({\ell^r})\in\mathbb{R},$ and that
by \eqref{Heckeinertlambda}, one has
\begin{align*}
\lambda_{\sigma}(\ell)^2=\lambda_{\sigma}(\ell^{2})+\ell^{-1}\lambda_{\sigma}({\ell})+1.
\end{align*}

Suppose $|\lambda_{\sigma}({\ell})|< 1/2$ and $|\lambda_{\sigma}(\ell^2)|< 1/2$. By triangle inequality we obtain 
\begin{align*}
1\leq \lambda_{\sigma}(\ell)^2+|\lambda_{\sigma}(\ell^{2})|+\ell^{-1}|\lambda_{\sigma}(\ell)|<\frac{1}{4}+\frac{1}{2}+\ell^{-1}\cdot \frac{1}{2}<1,
\end{align*}
a contradiction! Hence, there exists $r_\ell\in \{1,2\}$ such that $|\lambda_{\sigma}({\ell^{r_\ell}})|\geq 1/2$. 

Let
\begin{align*}
J_{\Spec}(\sigma,L):=\frac{1}{d_{{\Lambda}}}\sum_{\pi\in \mcA_k(N)}\Big|\sum_{p\in\mathcal{L}}\lambda_{\sigma}(p^{r_p})\lambda_{\pi}(p^{r_p})\Big|^2\sum_{\vphi\in \mcB_{\pi,k}^{\widetilde{\mfn}}\!(N)}\frac{\big|\mathcal{P}(\vphi,\vphi')\big|^2}{\peter{\vphi,\vphi}\peter{\vphi',\vphi'}}.
\end{align*}

\subsection{Spectral Side: A lower bound}\label{lowerbound}
By dropping all $\pi$'s that are not equal to $\sigma$ we have

\begin{align}\label{lowerboundJspec}
	J_{\Spec}(\sigma,L)&\geq \frac{1}{d_{{\Lambda}}}\Big|\sum_{p\in\mathcal{L}}\lambda_{\sigma}(p^{r_p})^2\Big|^2\sum_{\vphi\in \mcB_{\sigma,k}^{\widetilde{\mfn}}\!(N)}\frac{\big|\mathcal{P}(\vphi,\vphi')\big|^2}{\peter{\vphi,\vphi}\peter{\vphi',\vphi'}}\nonumber
	\\&
	\gg \frac{L^2}{d_{{\Lambda}}\log^2 L} \sum_{\vphi\in \mcB_{\sigma,k}^{\widetilde{\mfn}}\!(N)}\frac{\big|\mathcal{P}(\vphi,\vphi')\big|^2}{\peter{\vphi,\vphi}\peter{\vphi',\vphi'}}.
\end{align}

\subsection{Spectral Side: Decompostion} 
Squaring out of the sum over the primes $\ell\in\mathcal{L},$ we obtain 
$$J_{\Spec}(\sigma,L)=J^{=}_{\Spec}(\sigma,L)+J^{\not=}_{\Spec}(\sigma,L)$$
where we have set
\begin{gather*}
	J^{=}_{\Spec}(\sigma,L)=\sum_{\substack{\ell\in\mathcal{L}}}\lambda_{\sigma}(\ell^{r_{\ell}})^2J_{\Spec}(\ell),\\ 
	J_{\Spec}^{\not=}(\sigma,L)=\sum_{\substack{\ell_1, \ell_2\in\mathcal{L}\\ \ell_1\neq \ell_2}}\lambda_{\sigma}(\ell_1^{r_{\ell_1}})\lambda_{\sigma}(\ell_2^{r_{\ell_2}})J_{\Spec}(\ell_1,\ell_2)
\end{gather*}
with
\begin{gather*}
	J_{\Spec}(\ell)=\frac{1}{d_{{\Lambda}}}\sum_{\pi\in \mcA_k(N)}\lambda_{\pi}(\ell^{r_{\ell}})^2\sum_{\vphi\in \mcB_{\pi,k}^{\widetilde{\mfn}}\!(N)}\frac{\big|\mathcal{P}(\vphi,\vphi')\big|^2}{\peter{\vphi,\vphi}\peter{\vphi',\vphi'}},\\
	J_{\Spec}(\ell_1,\ell_2)=\frac{1}{d_{{\Lambda}}}\sum_{\pi\in \mcA_k(N)}\lambda_{\pi}(\ell_1^{r_{\ell_1}}\ell_2^{r_{\ell_2}})\sum_{\vphi\in \mcB_{\pi,k}^{\widetilde{\mfn}}\!(N)}\frac{\big|\mathcal{P}(\vphi,\vphi')\big|^2}{\peter{\vphi,\vphi}\peter{\vphi',\vphi'}}.
\end{gather*}

We will now bound $J_{\Spec}(\ell)$ and $J_{\Spec}(\ell_1,\ell_2)$ using Theorem \ref{firstmomentwithl}.

We notice first that by the Hecke relations \eqref{Heckeinertlambda} we have
\begin{equation}\label{hecke}
\begin{cases}
\lambda_{\pi}(\ell)^2=\lambda_{\pi}(\ell^2)+\ell^{-1}\lambda_{\pi}(\ell)+1,\\
\lambda_{\pi}(\ell^2)^2=\lambda_{\pi}(\ell^4)+\ell^{-1}\lambda_{\pi}(\ell^3)+\lambda_{\pi}(\ell^2)+\ell^{-1}\lambda_{\pi}(\ell)+1.
\end{cases}
\end{equation}
Il follows that 
$$|J_{\Spec}(\ell)|\leq 4\max_{0\leq \alpha\leq 4}\big|\frac{1}{d_{{\Lambda}}}\sum_{\pi\in \mcA_k(N)}\lambda_{\pi}(\ell^{\alpha})\sum_{\vphi\in \mcB_{\pi,k}^{\widetilde{\mfn}}\!(N)}\frac{\big|\mathcal{P}(\vphi,\vphi')\big|^2}{\peter{\vphi,\vphi}\peter{\vphi',\vphi'}}\big|.$$
Since $\ell^\alpha\in [1, L^4]$, by Theorem \ref{firstmomentwithl} and Deligne's bound, $|\lambda_{\pi'}(\ell^\alpha)|\leq \alpha+1$, we obtain
\begin{equation}
	\label{Jspeclbound}
	J_{\Spec}(\ell)\ll_{N'} (kLN)^{o(1)}\Bigl(\frac{N^2}{k}+\frac{{L^{52}}N^2}{k^{1/2}}(1+\frac{L^8}{N})(e^{-\kappa/{{N'}^2L^8}}+2^{-\kappa})\Bigr)
\end{equation}
and if 
\begin{equation}
	\label{N'Lbound}
	{N'}^2L^8<N
\end{equation} the second term on the righthand side above can be replaced by
\begin{equation}
	\label{replacebound}
	\frac{{L^{36}}N^4}{k^{1/2}}(\frac{L^8{N'}^{2}}{N})^{\kappa}
\end{equation}
Averaging over $\ell\in\mcL$ and using the bound $\lambda_\sigma(\ell^{r_\ell})^2\ll 1$ we obtain that
\begin{equation}
	\label{Jspec=bound}J^{=}_{\Spec}(\sigma,L)\ll_{N'} (kLN)^{o(1)}\Bigl(\frac{LN^2}{k}+\frac{{L^{53}}N^2}{k^{1/2}}(1+\frac{L^8}{N})(e^{-\kappa/{{N'}^2L^8}}+2^{-\kappa})\Bigr)
		\end{equation}
and if  \eqref{N'Lbound} holds, the second term on the righthand side of the above bound can be replaced by
$$\frac{{L^{37}}N^4}{k^{1/2}}(\frac{L^8{N'}^{2}}{N})^{\kappa}.$$
We treat $J^{\not=}_{\Spec}(\sigma,L)$ is the same way. Since $\ell_1^{r_{\ell_1}}\ell_2^{r_{\ell_2}}\in[L^2/4,L^4]$, using again Theorem \ref{firstmomentwithl}, we obtain the bound
$$J_{\Spec}(\ell_1,\ell_2)\ll_{N'} (kLN)^{o(1)}\Bigl(\frac{N^2}{kL^2}+\frac{{L^{52}}N^2}{k^{1/2}}(1+\frac{L^8}{N})(e^{-\kappa/{{N'}^2L^8}}+2^{-\kappa})\Bigr)
$$
and if  \eqref{N'Lbound} holds, the second term on the righthand side of \eqref{Jspec=bound} can be replaced by \eqref{replacebound}. 
Averaging over $\ell_1\not=\ell_2\in\mcL$ we obtain
\begin{equation}
	\label{Jspecnot=bound}J^{\not=}_{\Spec}(\sigma,L)\ll_{N'} (kLN)^{o(1)}\Bigl(\frac{N^2}{k}+\frac{{L^{54}}N^2}{k^{1/2}}(1+\frac{L^8}{N})(e^{-\kappa/{{N'}^2L^8}}+2^{-\kappa})\Bigr)
		\end{equation}
and if  \eqref{N'Lbound} holds, the second term on the righthand side of \eqref{Jspeclbound} can be replaced by
$$\frac{{L^{38}}N^4}{k^{1/2}}(\frac{L^8{N'}^{2}}{N})^{\kappa}.$$
In conclusion we obtain that
$$J_{\Spec}(\sigma,L)\ll_{N'} (kLN)^{o(1)}\Bigl(\frac{LN^2}{k}+\frac{{L^{54}}N^2}{k^{1/2}}(1+\frac{L^8}{N})(e^{-\kappa/{{N'}^2L^8}}+2^{-\kappa})\Bigr)$$
and if in addition \eqref{N'Lbound} holds we have
$$J_{\Spec}(\sigma,L)\ll_{N'} (kLN)^{o(1)}\Bigl(\frac{LN^2}{k}+\frac{{L^{38}}N^4}{k^{1/2}}(\frac{L^8{N'}^{2}}{N})^{\kappa}\Bigr)$$
combining this with \eqref{lowerboundJspec} we obtain that
\begin{multline}
	\label{Pbound1}
	\sum_{\vphi\in \mcB_{\sigma,k}^{\widetilde{\mfn}}\!(N)}\frac{\big|\mathcal{P}(\vphi,\vphi')\big|^2}{\peter{\vphi,\vphi}\peter{\vphi',\vphi'}}\\
	\ll_{N'} (kNL)^{o(1)}(\frac{k^2N^2}L+{{L^{52}}k^{3/2}N^2}(1+\frac{L^8}{N})(e^{-\kappa/{{N'}^2L^8}}+2^{-\kappa}))
\end{multline}
and if  \eqref{N'Lbound} holds, we have
\begin{equation}
	\label{Pbound2}
	\sum_{\vphi\in \mcB_{\sigma,k}^{\widetilde{\mfn}}\!(N)}\frac{\big|\mathcal{P}(\vphi,\vphi')\big|^2}{\peter{\vphi,\vphi}\peter{\vphi',\vphi'}}\ll_{N'} (kLN)^{o(1)}\Bigl(\frac{k^2N^2}{L}+{{L^{36}}k^{3/2}N^4}(\frac{L^8{N'}^{2}}{N})^{\kappa}\Bigr).
\end{equation}

\subsubsection{The case $k\geq N$.} We choose $L\gg_E 1$ such that $\mcL$ is not empty and
$$L^8=\frac{(kN)^{1/2}}{{N'}^2\log^2(kN)}$$
(which requires that ${N'}\ll _E\frac{(kN)^{1/4}}{\log(kN)}$). Since $k\geq (kN)^{1/2}$ we conclude that the second term on the left-hand side of \eqref{Pbound1} is negligible and that
$$\sum_{\vphi\in \mcB_{\sigma,k}^{\widetilde{\mfn}}\!(N)}\frac{\big|\mathcal{P}(\vphi,\vphi')\big|^2}{\peter{\vphi,\vphi}\peter{\vphi',\vphi'}}
	\ll_{N'} (kN)^{o(1)}(kN)^{2-1/16}.$$
	
	\subsubsection{The case $k\leq N$.} We choose $L\gg_E 1$ such that $\mcL$ is not empty and
\begin{equation}
	\label{firstchoice}
	L^8\leq\frac12\frac{N^{1/2}}{{N'}^2};
\end{equation}
this implies that \eqref{N'Lbound} is satisfied and requires ${N'}\ll _E{N^{1/4}}$.

The second term on the left-hand side of \eqref{Pbound2} is bounded by
$$
	(kN)^{o(1)}(kN)^{2}{L^{36}}k^{-1/2}N^{2-\kappa/2}
	\leq (kN)^{o(1)}(kN)^{3/2}{L^{36}}
$$
since  $\kappa/2-2\geq 1/2$.
Choosing 
$$L={(kN)^{1/74}}$$ (this is compatible with \eqref{firstchoice}) so that 
$(kN)^2/L=(kN)^{3/2}{L^{36}}$
we obtain for $\sigma\in\mcA_k(N)$ the bound
\begin{equation}
	\sum_{\vphi\in \mcB_{\sigma,k}^{\widetilde{\mfn}}\!(N)}\frac{\big|\mathcal{P}(\vphi,\vphi')\big|^2}{\peter{\vphi,\vphi}\peter{\vphi',\vphi'}}
	\ll_{N'} (kN)^{o(1)}{(kN)^{2-1/74}}.
	\label{eqperiodupperboundfinal}
\end{equation}
	\qed
	\subsection{Proof of Theorem \ref{thmnonvanishpower}}\label{secThm12}
	If $N>1$, by \eqref{eqnewformsperiod} we have
	$$\sum_{\pi\in\mcA_k^{\mathrm{n}}(N)}\sum_{\vphi\in \mathcal{B}_{\pi,k}^{\tilde{\mathfrak{n}}}\!(N)}
\frac{\big|\mcP(\vphi,\vphi')\big|^2}{\peter{\vphi,\vphi}\peter{\vphi',\vphi'}}\asymp_{N'} (kN)^2$$
(note that since $\pi\in \mcA_k^{\mathrm{n}}(N)$, $\mathcal{B}_{\pi,k}^{\tilde{\mathfrak{n}}}\!(N)$ is a singleton) and from \eqref{eqperiodupperboundfinal} (for $k>32$ and $N'\ll_E (Nk)^{1/8}$) we have 
$$\sum_{\pi\in\mcA_k^{\mathrm{n}}(N)}\sum_{\vphi\in \mathcal{B}_{\pi,k}^{\tilde{\mathfrak{n}}}\!(N)}\delta_{\mcP(\vphi,\vphi')\not=0}\gg_{N'} {(kN)^{1/74+o(1)}}$$
since $|\mathcal{B}_{\pi,k}^{\tilde{\mathfrak{n}}}\!(N)|\leq 1$ and $L(1/2,\pi_E\times\pi'_E)$ is proportional to $\big|\mcP(\vphi,\vphi')\big|^2$ we obtain Theorem \ref{thmnonvanishpower} for $N>1$.

The case $N=1$ follows the same principle by using \eqref{m1} for $N=1$.

\renewcommand{\appendixname}{\bf Appendix}
\appendix

\section{\bf Explicit Double Coset Decompositions}

In this appendix we record several consequences of the Bruhat, Iwahori and Cartan decompositions for the open compact groups $G(\Zp)$ and $G'(\Zp)$ which are used in the evaluation of the local period integrals in \S \ref{S}.

\subsection{Decompositions for $U(W)$}\label{4.3}

In this section we discuss the case of $G'(\Zp)$. For this it will be useful to represent the elements of $G'$ by their  $2\times 2$ matrices in the basis $\{e_{-1},e_1\}$; moreover if $p$ is split we will identify $G'(\Qp)$ with $\GL_2(\Qp)$.

 The following lemma is a consequence of the Bruhat decomposition for $G'(\Ff_p)$.

\begin{lemma}\label{K'cosetlemma}
We have the disjoint union decompositions
	\begin{align}
	\label{K'pcosetsplit}
		G'(\mathbb{Z}_p)=I'_p\sqcup\bigsqcup_{\delta\in \Zp/p\Zp} \begin{pmatrix}
	\delta&1\\
	1&
	\end{pmatrix}I'_p&\ \hbox{ if $p$ is split;} \\	
\label{K'pcosetinert}
		G'(\mathbb{Z}_p)=I'_p\sqcup\bigsqcup_\stacksum{\delta\in \mcO_{E_p}/p\mcO_{E_p}}{\delta+\ov\delta=0} \begin{pmatrix}
	\delta&1\\
	1&
	\end{pmatrix}I'_p&\	\hbox{ if $p$ is inert.}
	\end{align}
	 In particular
	$$|G'(\mathbb{Z}_p)/I'_p|=p+1=\mu(I'_p)^{-1}.$$
\end{lemma}

\subsubsection{Bruhat-Iwahori-Cartan Decomposition on $U(W)$}
We set $$J'=\begin{pmatrix}&1\\1&	
\end{pmatrix},\ 
 A_n=\begin{pmatrix}p^n&\\&p^{-n}	
\end{pmatrix},\ n\geq 1.$$

We have the following double cosets decomposition:
\begin{lemma}\label{178}
For $p$ inert in $E$, we have the disjoint unions
	\begin{align*}
	&I_p'A_nI_p'=\bigsqcup_{\substack{
			\tau\in \mathcal{O}_{p}/p^{2n}\mathcal{O}_{p}\\ \tau+\overline{\tau}=0}}\begin{pmatrix}
	1&\tau\\
	&1
	\end{pmatrix}A_nI_p',\\
	&I_p'J'A_nI_p'=\bigsqcup_{\substack{
			\tau\in p\mathcal{O}_{p}/p^{2n}\mathcal{O}_{p}\\ \tau+\overline{\tau}=0}}\begin{pmatrix}
	1&\\
	\tau&1
	\end{pmatrix}J'A_nI_p',\\
	&I_p'A_nJ'I_p'=\bigsqcup_{\substack{
			\tau\in \mathcal{O}_{p}/p^{2n+1}\mathcal{O}_{p}\\ \tau+\overline{\tau}=0}}\begin{pmatrix}
	1&\tau\\
	&1
	\end{pmatrix}A_nJ'I_p',\\
	&I_p'J'A_nJ'I_p'=\bigsqcup_{\substack{
			\tau\in p\mathcal{O}_{p}/p^{2n+1}\mathcal{O}_{p}\\ \tau+\overline{\tau}=0}}\begin{pmatrix}
	1&\\
	\tau&1
	\end{pmatrix}J'A_nJ'I_p'.
	\end{align*}
	
	For $p$ split in $E$ these decompositions holds upon replacing $\mathcal{O}_{p}$ by $\Zp$ and by removing the condition $\tau+\ov \tau=0$.
	\end{lemma}
	
For the proof we refer to \S \ref{secBICG}  where we discuss the more complicated case of $G(\Zp)$. 

\begin{lemma}\label{162'}
	With the same notation as in the previous lemma, we have 
	\begin{equation}\label{156}
	G'(\mathbb{Z}_p)A_nG'(\mathbb{Z}_p)=I_p'A_nI_p'\bigsqcup I_p'A_nJ'I_p'\bigsqcup I_p'J'A_nI_p'\bigsqcup I_p'J'A_nJ'I_p'.
	\end{equation}
\end{lemma}
\begin{proof} We discuss again only the case $p$ inert.

Taking inverses in the identity \eqref{K'pcosetinert} we  have
\begin{equation}\label{158}
G'(\mathbb{Z}_p)=I_p'\sqcup\bigsqcup_{\substack{\delta\in \mathcal{O}_{p}/N'\mathcal{O}_{p}\\ \delta+\overline{\delta}=0}} I_p'\begin{pmatrix}
&1\\
1&\overline{\delta}
\end{pmatrix}.
\end{equation}

We thus have, by \eqref{K'pcosetinert} and \eqref{158}, that $G'(\mathbb{Z}_p)A_nG'(\mathbb{Z}_p)=U_1'\bigcup U_2',$ where
\begin{align*}
U_1':=&\bigcup_{\substack{\delta\in \mathcal{O}_{p}/N'\mathcal{O}_{p}\\ \delta+\overline{\delta}=0}}I_p'A_n\begin{pmatrix}
\delta&1\\
1&
\end{pmatrix}I_p\cup\bigcup_{\substack{\delta\in \mathcal{O}_{p}/N'\mathcal{O}_{p}\\
		\delta+\overline{\delta}=0}}I_p'\begin{pmatrix}
&1\\
1&\overline{\delta}
\end{pmatrix}A_nI_p'\\
U_2':=&I_p'A_nI_p'\cup\bigcup_{\substack{\delta_1, \delta_2\in \mathcal{O}_{p}/N'\mathcal{O}_{p}\\
		\delta_1+\overline{\delta}_1=0\\ \delta_2+\overline{\delta}_2=0}}I_p'\begin{pmatrix}
&1\\
1&\overline{\delta}_1
\end{pmatrix}A_n\begin{pmatrix}
\delta_2&1\\
&1
\end{pmatrix}I_p'.
\end{align*}
	
Suppose $n\geq 1.$ Then $G'(\mathbb{Z}_p)A_nG'(\mathbb{Z}_p)=I_p'A_nI_p'\bigcup I_p'A_nJ'I_p'\bigcup U_3',$ where
\begin{align*}
U_3':=\bigcup_{\substack{\delta\in \mathcal{O}_{p}/N'\mathcal{O}_{p}\\
			\delta+\overline{\delta}=0}}I_p'\begin{pmatrix}
&1\\
	1&\overline{\delta}
	\end{pmatrix}A_nI_p'\cup\bigcup_{\substack{\delta\in \mathcal{O}_{p}/N'\mathcal{O}_{p}\\
			\delta+\overline{\delta}=0}}I_p'\begin{pmatrix}
	&1\\
	1&\overline{\delta}
	\end{pmatrix}A_nJ'I_p'.
	\end{align*}
	
Note that under the assumption $\delta\in \mathcal{O}_{p}^{\times},$ we have
	\begin{align*}
	\begin{pmatrix}
	&1\\
	1&\overline{\delta}
	\end{pmatrix}A_n=\begin{pmatrix}
	&p^{-n}\\
	p^n&\overline{\delta}p^{-n}
	\end{pmatrix}=\begin{pmatrix}
	1&\overline{\delta}^{-1}\\
	&1
	\end{pmatrix}A_n\begin{pmatrix}
	\delta^{-1}&\\
	p^{2n}&\overline{\delta}
	\end{pmatrix}\in I_p'A_nI_p';
	\end{align*}
and $J'A_n=\begin{pmatrix}
	&1\\
	1&
	\end{pmatrix}A_n\in I_p'J'A_nI_p'.$ Hence, we obtain
	\begin{equation}\label{149}
	\bigcup_{\substack{\delta\in \mathcal{O}_{p}/N'\mathcal{O}_{p}\\
			\delta+\overline{\delta}=0}}I_p'\begin{pmatrix}
	&1\\
	1&\overline{\delta}
	\end{pmatrix}A_nI_p'\subseteq I_p'A_nI_p'\bigcup I_p'J'A_nI_p'.
	\end{equation}
	
	Note that, for $\delta\in \mathcal{O}_{p}^{\times},$ a straightforward computation shows
	\begin{align*}
	\begin{pmatrix}
	&1\\
	1&\overline{\delta}
	\end{pmatrix}A_nJ'=\begin{pmatrix}
	p^{-n}&\\
	\overline{\delta}p^{-n}&p^n
	\end{pmatrix}=\begin{pmatrix}
	{\delta}^{-1}&1\\
	&\overline{\delta}
	\end{pmatrix}A_nJ'\begin{pmatrix}
	1&p^{2n}\overline{\delta}^{-1}\\
	&1
	\end{pmatrix}\in I_p'A_nJ'I_p'.
	\end{align*}
	Also, $J'A_nJ'=\begin{pmatrix}
	&1\\
	1&
	\end{pmatrix}A_nJ'\in I_p'J'A_nJ'I_p'.$ Hence, similar to \eqref{149} we have,
	\begin{equation}\label{150}
	\bigcup_{\substack{\delta\in \mathcal{O}_{p}/N'\mathcal{O}_{p}\\
			\delta+\overline{\delta}=0}}I_p'\begin{pmatrix}
	&1\\
	1&\overline{\delta}
	\end{pmatrix}A_nJ'I_p'\subseteq I_p'A_nI_p'\bigcup I_p'J'A_nJ'I_p'.
	\end{equation}
	
Substituting the relations \eqref{149} and \eqref{150} in the definition of $U'_3$ and the decomposition $G'(\mathbb{Z}_p)A_nG'(\mathbb{Z}_p)=I_p'A_nI_p'\bigcup I_p'A_nJ'I_p'\bigcup U_3'$ we then conclude
\begin{equation}\label{159}
G'(\mathbb{Z}_p)A_nG'(\mathbb{Z}_p)=I_p'A_nI_p'\bigcup I_p'A_nJ'I_p'\bigcup I_p'J'A_nI_p'\bigcup I_p'J'A_nJ'I_p'.
\end{equation}

Then \eqref{156} follows from the fact that the union in \eqref{159} is actually disjoint.
\end{proof}

\subsection{Decompositions for $U(V)$}\label{secBICG}

Let $p$ be a prime which is inert in $E$ (for instance  $p=N$); let $E_p=E\otimes_\Qq \Qp$ be the corresponding local field, $\overline\bullet:z\mapsto \overline z$ the complex conjugation on $E_p$ and $\mathcal{O}_{p}$ be its ring of integers and $p$ is a uniformizer. We denote by $\nu:E_p\mapsto \Zz$ the normalized valuation. 

 We recall that, by definition of the unitary group $G(\Qp)$, we have for 
$g=\begin{pmatrix}
	g_{11}&g_{12}&g_{13}\\
	g_{21}&g_{22}&g_{23}\\
	g_{31}&g_{32}&g_{33}
	\end{pmatrix}\in G(\mathbb{Q}_p)$,
the relations	
	\begin{equation}\label{eqGunitary}
	\begin{pmatrix}
	g_{11}&g_{12}&g_{13}\\
	g_{21}&g_{22}&g_{23}\\
	g_{31}&g_{32}&g_{33}
	\end{pmatrix}\begin{pmatrix}
	\overline{g}_{33}&\overline{g}_{23}&\overline{g}_{13}\\
	\overline{g}_{32}&\overline{g}_{22}&\overline{g}_{12}\\
	\overline{g}_{31}&\overline{g}_{21}&\overline{g}_{11}
	\end{pmatrix}=I_3
	\end{equation}
 and
 \begin{equation}\label{eqGunitary2}
		\begin{pmatrix}
		\overline{g}_{33}&\overline{g}_{23}&\overline{g}_{13}\\
		\overline{g}_{32}&\overline{g}_{22}&\overline{g}_{12}\\
		\overline{g}_{31}&\overline{g}_{21}&\overline{g}_{11}
		\end{pmatrix}\begin{pmatrix}
		g_{11}&g_{12}&g_{13}\\
		g_{21}&g_{22}&g_{23}\\
		g_{31}&g_{32}&g_{33}
		\end{pmatrix}=I_3.
		\end{equation}

We denote by
$I_p$ the Iwahori subgroup
\begin{equation}
		\label{Iwahoridef}I_p:=G(\Zp)\cap \begin{pmatrix}
	\mcO_p&\mcO_p&\mcO_p\\ p\mcO_p&\mcO_p&\mcO_p\\
p\mcO_p&p\mcO_p&\mcO_p\\
\end{pmatrix}.
\end{equation}
In particular if $p=N$, $K_p(N)=I_p$.

Like in Lemma \ref{K'cosetlemma} we have the following is a consequence of the Bruhat decomposition for $G(\Ff_p)$:
$$G(\Ff_p)=P(\Ff_p)\sqcup P(\Ff_p)J N(\Ff_p)$$
where $P\subset G$ is the Borel subgroup, $N$ its unipotent radical, and 
\begin{align*}
N(\Ff_p)=\Bigg\{n(\delta,\tau):=\begin{pmatrix}
	1&\delta&\tau\\&1&-\ov \delta\\&&1
\end{pmatrix},\ \delta,\tau\in\Ff_{p^2}\Bigg\}.
\end{align*}

\begin{lemma}\label{lemIwahoriCartanU(V)} Let $p$ be a prime inert in $E$.	 We have a disjoint coset decomposition,
	\begin{equation}\label{120}
	G(\mathbb{Z}_p)=I_p\sqsqcup_{\substack{\tau\in \mathcal{O}_{p}/p\mathcal{O}_{p}\\ \delta\in \mathcal{O}_{p}/p\mathcal{O}_{p}\\
			\tau+\overline{\tau}+\delta\overline{\delta}=0}}\begin{pmatrix}
	\tau&\delta&1\\
	-\overline{\delta}&1&\\
	1&&
	\end{pmatrix}I_p.
	\end{equation}
	In particular
	$$|G(\Zp)/I_p|=p^3+1=(p+1)(p^2-p+1)=\mu(I_p)^{-1}$$
\end{lemma}

For $n\in\Zz$ we set
\begin{equation}
	\label{Andef}
	A_n=\begin{pmatrix}
	p^{n}&&\\&1&\\&&p^{-n}
\end{pmatrix}.
\end{equation}

\begin{lemma}\label{160} Assume that $p$ is inert in $E$.
	For $n\geq 1$, we have the disjoint decompositions
	\begin{align*}
	&I_pA_nI_p=\sqsqcup_{\substack{\delta\in \mathcal{O}_{p}/p^{n}\mathcal{O}_{p}\\
			\tau\in \mathcal{O}_{p}/p^{2n}\mathcal{O}_{p}\\ \tau+\overline{\tau}+\delta\overline{\delta}=0}}\begin{pmatrix}
	1&\delta&\tau\\
	&1&-\overline{\delta}\\
	&&1
	\end{pmatrix}A_nI_p,\\
	&I_pJA_nI_p=\sqsqcup_{\substack{\delta\in p\mathcal{O}_{p}/p^{n}\mathcal{O}_{p}\\
	\tau\in p\mathcal{O}_{p}/p^{2n}\mathcal{O}_{p}\\ \tau+\overline{\tau}+\delta\overline{\delta}=0}}\begin{pmatrix}
	1&&\\
	-\overline{\delta}&1&\\
	\tau&\delta&1
	\end{pmatrix}JA_nI_p,\\
	&I_pA_nJI_p=\sqsqcup_{\substack{\delta\in \mathcal{O}_{p}/p^{n+1}\mathcal{O}_{p}\\
			\tau\in \mathcal{O}_{p}/p^{2n+1}\mathcal{O}_{p}\\ \tau+\overline{\tau}+\delta\overline{\delta}=0}}\begin{pmatrix}
	1&\delta&\tau\\
	&1&-\overline{\delta}\\
	&&1
	\end{pmatrix}A_nJI_p,\\
	&I_pJA_nJI_p=\sqsqcup_{\substack{\delta\in p\mathcal{O}_{p}/p^{n+1}\mathcal{O}_{p}\\
			\tau\in p\mathcal{O}_{p}/p^{2n+1}\mathcal{O}_{p}\\ \tau+\overline{\tau}+\delta\overline{\delta}=0}}\begin{pmatrix}
	1&&\\
	-\overline{\delta}&1&\\
	\tau&\delta&1
	\end{pmatrix}JA_nJI_p.
	\end{align*}

\end{lemma}

\begin{proof}
Let
\begin{align*}
X=\begin{pmatrix}
g_{11}&g_{12}&g_{13}\\
g_{21}&g_{22}&g_{23}\\
g_{31}&g_{32}&g_{33}
\end{pmatrix}\in I_p
\end{align*}
and let
\begin{align*}
Y=\begin{pmatrix}
1&&\\
p^n\overline{g}_{32}&1&\\
p^{2n}\overline{g}_{33}& -p^ng_{32}&1
\end{pmatrix}\in I_p.
\end{align*}
A priori we have $g_{33}\in\mcO_p^\times$ but we may assume as well that $g_{33}=1$.

Using \eqref{eqGunitary} one has $$g_{12}-g_{13}g_{32}=-\overline{g}_{23}(g_{22}-g_{23}g_{32})\hbox{ and }(g_{22}-g_{23}g_{32})(\overline{g}_{22}-\overline{g}_{23}\overline{g}_{32})=1.$$ 

We have  
\begin{multline*}XA_nY=
\begin{pmatrix}
1&g_{12}-g_{13}g_{32}&g_{13}\\
&g_{22}-g_{23}g_{32}&g_{23}\\
&&1
\end{pmatrix}A_n\\=\begin{pmatrix}
1&-\overline{g}_{23}&g_{13}\\
&1&g_{23}\\
&&1
\end{pmatrix}A_n\begin{pmatrix}
1&&\\
&g_{22}-g_{23}g_{32}&\\
&&1
\end{pmatrix}.
\end{multline*}
Since $\diag (1, g_{22}-g_{23}g_{32}, 1)\in I_p,$ we then obtain
\begin{equation}\label{IpAnIp}
I_pA_nI_p=N(\mathbb{Z}_p)A_nI_p=\ccup_{\substack{\delta, \tau\in \mathcal{O}_{p}\\ \tau+\overline{\tau}+\delta\overline{\delta}=0}}n(\delta,\tau)A_nI_p.
\end{equation}

Since $A_n^{-1}n(\delta,\tau)A_n=n(p^{-n}\delta,p^{2n}\tau),$ we then have
\begin{equation}\label{161}
I_pA_nI_p=\ccup_{\substack{\delta, \tau\in \mathcal{O}_{p}\\ \tau+\overline{\tau}+\delta\overline{\delta}=0}}n(\delta,\tau)A_nI_p=\sqsqcup_{\substack{\delta\in p\mathcal{O}_{p}/p^{n}\mathcal{O}_{p}\\
		\tau\in p\mathcal{O}_{p}/p^{2n}\mathcal{O}_{p}\\ \tau+\overline{\tau}+\delta\overline{\delta}=0}}\begin{pmatrix}
1&\delta&\tau\\
&1&-\overline{\delta}\\
&&1
\end{pmatrix}A_nI_p.
\end{equation}

Similarly, there are some $\delta_1, \tau_1\in p\mathcal{O}_{p}$ such that
\begin{align*}
\begin{pmatrix}
1&g_{12}&g_{13}\\
g_{21}&g_{22}&g_{23}\\
g_{31}&g_{32}&g_{33}
\end{pmatrix}JA_nJ\begin{pmatrix}
1&-p^ng_{12}&p^{2n}\overline{g}_{13}\\
&1&p^n\overline{g}_{12}\\
&&1
\end{pmatrix}=Jn(\delta_1,\tau_1)JA_n^{-1}.
\end{align*}

Note $A_n^{-1}=JA_nJ.$ Similar to \eqref{161}, one has
\begin{align*}
I_pJA_nJI_p=\bigcup_{\substack{\delta, \tau\in p\mathcal{O}_{p}\\ \tau+\overline{\tau}+\delta\overline{\delta}=0}}Jn(\delta,\tau)JA_n^{-1}I_p=\bigsqcup_{\substack{\delta\in p\mathcal{O}_{p}/p^{n+1}\mathcal{O}_{p}\\
\tau\in p\mathcal{O}_{p}/p^{2n+1}\mathcal{O}_{p}\\ \tau+\overline{\tau}+\delta\overline{\delta}=0}}\begin{pmatrix}
1&&\\
-\overline{\delta}&1&\\
\tau&\delta&1
\end{pmatrix}JA_nJI_p.
\end{align*}

By a straightforward computation there are some $\delta_2, \tau_3\in \mathcal{O}_{p}$ such that
\begin{align*}
\begin{pmatrix}
g_{11}&g_{12}&g_{13}\\
g_{21}&g_{22}&g_{23}\\
g_{31}&g_{32}&1
\end{pmatrix}A_nJ\begin{pmatrix}
1&-p^ng_{32}&p^{2n}\overline{g}_{31}\\
&1&p^n\overline{g}_{32}\\
&&1
\end{pmatrix}=n(\delta_2,\tau_2)A_nJ.
\end{align*}

Note $A_n^{-1}n(\delta_2,\tau_2)A_n=n(p^{-n}\delta_2,p^{-2n}\tau_2).$ Therefore, we have
\begin{align*}
I_pA_nJI_p=\bigcup_{\substack{\delta, \tau\in \mathcal{O}_{p}\\ \tau+\overline{\tau}+\delta\overline{\delta}=0}}n(\delta,\tau)A_nJI_p=\bigsqcup_{\substack{\delta\in \mathcal{O}_{p}/p^{n+1}\mathcal{O}_{p}\\
\tau\in \mathcal{O}_{p}/p^{2n+1}\mathcal{O}_{p}\\ \tau+\overline{\tau}+\delta\overline{\delta}=0}}\begin{pmatrix}
1&\delta&\tau\\
&1&-\overline{\delta}\\
&&1
\end{pmatrix}A_nJI_p.
\end{align*}

Likewise, there are some $\delta_3, \tau_3\in p\mathcal{O}_{p}$ such that
\begin{align*}
\begin{pmatrix}
1&g_{12}&g_{13}\\
g_{21}&g_{22}&g_{23}\\
g_{31}&g_{32}&g_{33}
\end{pmatrix}JA_n\begin{pmatrix}
1&&\\
p^n\overline{g}_{12}&1&\\
p^{2n}\overline{g}_{13}&-p^ng_{12}&1
\end{pmatrix}=Jn(\delta_3,\tau_3)A_n.
\end{align*}

Again, by $A_n^{-1}n(\delta_3,\tau_3)A_n=n(p^{-n}\delta_3,p^{-2n}\tau_3),$ one has
\begin{align*}
I_pJA_nI_p=\bigcup_{\substack{\delta, \tau\in p\mathcal{O}_{p}\\ \tau+\overline{\tau}+\delta\overline{\delta}=0}}Jn(\delta,\tau)A_nI_p=\bigsqcup_{\substack{\delta\in p\mathcal{O}_{p}/p^{n}\mathcal{O}_{p}\\
		\tau\in p\mathcal{O}_{p}/p^{2n}\mathcal{O}_{p}\\ \tau+\overline{\tau}+\delta\overline{\delta}=0}}\begin{pmatrix}
1&&\\
-\overline{\delta}&1&\\
\tau&\delta&1
\end{pmatrix}JA_nI_p.
\end{align*}
Lemma \ref{160} follows.
\end{proof}

\begin{lemma}\label{162} Let $p$ be inert in $E$.
	 We have
	\begin{equation}\label{153}
	G(\mathbb{Z}_p)A_nG(\mathbb{Z}_p)=I_pA_nI_p\bigsqcup I_pA_nJI_p\bigsqcup I_pJA_nI_p\bigsqcup I_pJA_nJI_p.
	\end{equation}
	Moreover, $$G(\mathbb{Z}_p)A_nG(\mathbb{Z}_p)=G(\mathbb{Z}_p)A_{-n}G(\mathbb{Z}_p).$$
\end{lemma}
\begin{proof}
	Appealing to Lemma \ref{lemIwahoriCartanU(V)} one has the decomposition
	\begin{equation}\label{154}
	G(\mathbb{Z}_p)=I_p\sqsqcup_{\substack{\tau\in \mathcal{O}_{p}/p\mathcal{O}_{p}\\ \delta\in \mathcal{O}_{p}/p\mathcal{O}_{p}\\
			\tau+\overline{\tau}+\delta\overline{\delta}=0}}\begin{pmatrix}
	\tau&\delta&1\\
	-\overline{\delta}&1&\\
	1&&
	\end{pmatrix}I_p.
	\end{equation}
	
	Taking  inverses of the above identity we then obtain
	\begin{equation}\label{155}
	G(\mathbb{Z}_p)=I_p\sqsqcup_{\substack{\tau\in \mathcal{O}_{p}/p\mathcal{O}_{p}\\ \delta\in \mathcal{O}_{p}/p\mathcal{O}_{p}\\
			\tau+\overline{\tau}+\delta\overline{\delta}=0}}I_p\begin{pmatrix}
	&&1\\
	&1&\overline{\delta}\\
	1&-\delta&\overline{\tau}
	\end{pmatrix}.
	\end{equation}

	We thus have, by \eqref{154} and \eqref{155}, that $G(\mathbb{Z}_p)A_nG(\mathbb{Z}_p)=U_1\bigcup U_2,$ where
	\begin{align*}
	U_1:=&\bigcup_{\substack{\tau\in \mathcal{O}_{p}/p\mathcal{O}_{p}\\ \delta\in \mathcal{O}_{p}/p\mathcal{O}_{p}\\
			\tau+\overline{\tau}+\delta\overline{\delta}=0}}I_pA_n\begin{pmatrix}
	\tau&\delta&1\\
	-\overline{\delta}&1&\\
	1&&
	\end{pmatrix}I_p\ccup_{\substack{\tau\in \mathcal{O}_{p}/p\mathcal{O}_{p}\\ \delta\in \mathcal{O}_{p}/p\mathcal{O}_{p}\\
			\tau+\overline{\tau}+\delta\overline{\delta}=0}}I_p\begin{pmatrix}
	&&1\\
	&1&\overline{\delta}\\
	1&-\delta&\overline{\tau}
	\end{pmatrix}A_nI_p\\
	U_2:=&I_pA_nI_p\ccup_{\substack{\tau_1, \tau_2\in \mathcal{O}_{p}/p\mathcal{O}_{p}\\ \delta_1, \delta_2\in \mathcal{O}_{p}/p\mathcal{O}_{p}\\
			\tau_1+\overline{\tau}_1+\delta_1\overline{\delta}_1=0\\ \tau_2+\overline{\tau}_2+\delta\overline{\delta}_2=0}}I_p\begin{pmatrix}
	&&1\\
	&1&\overline{\delta}_1\\
	1&-\delta_1&\overline{\tau}_1
	\end{pmatrix}A_n\begin{pmatrix}
	\tau_2&\delta_2&1\\
	-\overline{\delta}_2&1&\\
	1&&
	\end{pmatrix}I_p.
	\end{align*}
	
	Since $$G(\mathbb{Z}_p)A_nG(\mathbb{Z}_p)=G(\mathbb{Z}_p)JA_nJG(\mathbb{Z}_p)=G(\mathbb{Z}_p)A_{-n}G(\mathbb{Z}_p),$$
	 we may suppose $n\geq 1$ without loss of generality. Therefore, with a straightforward computation we have  $$G(\mathbb{Z}_p)A_nG(\mathbb{Z}_p)=I_pA_nI_p\bigcup I_pA_nJI_p\bigcup U_3,$$ where
	\begin{align*}
	U_3:= \bigcup_{\substack{\tau\in \mathcal{O}_{p}/p\mathcal{O}_{p}\\ \delta\in \mathcal{O}_{p}/p\mathcal{O}_{p}\\
			\tau+\overline{\tau}+\delta\overline{\delta}=0}}I_p\begin{pmatrix}
	&&1\\
	&1&\overline{\delta}\\
	1&-\delta&{\tau}
	\end{pmatrix}A_nI_p\ccup_{\substack{\tau\in \mathcal{O}_{p}/p\mathcal{O}_{p}\\ \delta\in \mathcal{O}_{p}/p\mathcal{O}_{p}\\
			\tau+\overline{\tau}+\delta\overline{\delta}=0}}I_p\begin{pmatrix}
	&&1\\
	&1&\overline{\delta}\\
	1&-\delta&{\tau}
	\end{pmatrix}A_nJI_p.
	\end{align*}
	
	Let $\delta\in\mathcal{O}_{p}^{\times}.$ Then for $\tau\in\mathcal{O}_{p}$ such that $\tau+\overline{\tau}+\delta\overline{\delta}=0,$ one has $\tau\in \mathcal{O}_{p}^{\times}.$ Then
	\begin{align*}
	\begin{pmatrix}
	&&1\\
	&1&\overline{\delta}\\
	1&-\delta&\tau
	\end{pmatrix}A_n=\begin{pmatrix}
	-1&-\delta\overline{\tau}^{-1}&-\tau^{-1}\\
	&1&-\overline{\delta}\tau^{-1}\\
	&&-1
	\end{pmatrix}A_n\begin{pmatrix}
	-\overline{\tau}^{-1}&&\\
	-p^n\overline{\delta}\tau^{-1}&-\overline{\tau}\tau^{-1}&\\
	-p^{2n}&p^n\delta&-\tau
	\end{pmatrix}.
	\end{align*}
	Denote by $\LHS_{\delta,\tau}^{(1)}$ the left hand side of the above identity. Note that
	\begin{align*}
	\begin{pmatrix}
	-1&-\delta\overline{\tau}^{-1}&-\tau^{-1}\\
	&1&-\overline{\delta}\tau^{-1}\\
	&&-1
	\end{pmatrix}\in I_p,\quad \begin{pmatrix}
	-\overline{\tau}^{-1}&&\\
	-p^n\overline{\delta}\tau^{-1}&-\overline{\tau}\tau^{-1}&\\
	-p^{2n}&p^n\delta&-\tau
	\end{pmatrix}\in I_p.
	\end{align*}
	Then we have $\LHS_{\delta,\tau}^{(1)}\in I_pA_nI_p.$ Suppose, on the other hand, $\delta=0.$ Then $\tau+\overline{\tau}=0.$ When $\tau\in\mathcal{O}_{p}^{\times},$ we then have
	\begin{align*}
	\begin{pmatrix}
	&&1\\
	&1&\\
	1&&{\tau}
	\end{pmatrix}A_n=\begin{pmatrix}
	1&&\tau^{-1}\\
	&1&\\
	&&1
	\end{pmatrix}A_n\begin{pmatrix}
	\overline{\tau}^{-1}&&\\
	&1&\\
	p^{2n}&&\tau
	\end{pmatrix}\in I_pA_nI_p;
	\end{align*}
	when $\tau=0,$ we have $JA_n\in I_pJA_nI_p.$ Combining these discussions, we obtain
	\begin{equation}\label{151}
	\bigcup_{\substack{\tau\in \mathcal{O}_{p}/N\mathcal{O}_{p}\\ \delta\in \mathcal{O}_{p}/N\mathcal{O}_{p}\\
			\tau+\overline{\tau}+\delta\overline{\delta}=0}}I_p\begin{pmatrix}
	&&1\\
	&1&\overline{\delta}\\
	1&-\delta&{\tau}
	\end{pmatrix}A_nI_p\subseteq I_pA_nI_p\bigcup I_pJA_nI_p.
	\end{equation}
	
	Let $\delta\in\mathcal{O}_{p}^{\times}.$ Then for $\tau\in\mathcal{O}_{p}$ such that $\tau+\overline{\tau}+\delta\overline{\delta}=0,$ one has $\tau\in \mathcal{O}_{p}^{\times}.$ Then
	\begin{align*}
	\begin{pmatrix}
	&&1\\
	&1&\overline{\delta}\\
	1&-\delta&\tau
	\end{pmatrix}A_nJ=\begin{pmatrix}
	\overline{\tau}^{-1}&-\delta\overline{\tau}^{-1}&1\\
	&-1&-\overline{\delta}\\
	&&\tau
	\end{pmatrix}A_n\begin{pmatrix}
	&&1\\
	&-\overline{\tau}\tau^{-1}&-p^n\overline{\delta}\tau^{-1}\\
	1&-p^n{\delta}\tau^{-1}&p^{2n}\tau^{-1}
	\end{pmatrix}.
	\end{align*}
	
	Denote by $\LHS_{\delta,\tau}^{(2)}$ the left hand side of the above identity. Note that
	\begin{align*}
	\begin{pmatrix}
	\overline{\tau}^{-1}&-\delta\overline{\tau}^{-1}&1\\
	&-1&-\overline{\delta}\\
	&&\tau
	\end{pmatrix}\in I_p,\quad \begin{pmatrix}
	&&1\\
	&-\overline{\tau}\tau^{-1}&-p^n\overline{\delta}\tau^{-1}\\
	1&-p^n{\delta}\tau^{-1}&p^{2n}\tau^{-1}
	\end{pmatrix}\in JI_p.
	\end{align*}
	
	Then we have $\LHS_{\delta,\tau}^{(2)}\in I_pA_nJI_p.$ Suppose, on the other hand, $\delta=0.$ Then $\tau+\overline{\tau}=0.$ When $\tau\in\mathcal{O}_{p}^{\times},$ we then have
	\begin{align*}
	\begin{pmatrix}
	&&1\\
	&1&\\
	1&&{\tau}
	\end{pmatrix}A_nJ=\begin{pmatrix}
	\overline{\tau}^{-1}&&1\\
	&1&\\
	&&\tau
	\end{pmatrix}A_nJ\begin{pmatrix}
	1&&p^{2n}\tau^{-1}\\
	&1&\\
	&&1
	\end{pmatrix}\in I_pA_nJI_p;
	\end{align*}
	when $\tau=0,$ we have $JA_nJ\in I_pJA_nJI_p.$ Combining these discussions, we obtain
	\begin{equation}\label{152}
	\bigcup_{\substack{\tau\in \mathcal{O}_{p}/N\mathcal{O}_{p}\\ \delta\in \mathcal{O}_{p}/N\mathcal{O}_{p}\\
			\tau+\overline{\tau}+\delta\overline{\delta}=0}}I_p\begin{pmatrix}
	&&1\\
	&1&\overline{\delta}\\
	1&-\delta&{\tau}
	\end{pmatrix}A_nJI_p\subseteq I_pA_nJI_p\bigcup I_pJA_nJI_p.
	\end{equation}
	
	It the follows from \eqref{151}, \eqref{152} and definition of $U_3$ that
	\begin{multline*}
	G(\mathbb{Z}_p)A_nG(\mathbb{Z}_p)\subseteq I_pA_nI_p\bigcup I_pA_nJI_p\bigcup I_pJA_nI_p\bigcup I_pJA_nJI_p\subseteq G(\mathbb{Z}_p)A_nG(\mathbb{Z}_p),
	\end{multline*}
	
	namely, $$G(\mathbb{Z}_p)A_nG(\mathbb{Z}_p)=I_pA_nI_p\bigcup I_pA_nJI_p\bigcup I_pJA_nI_p\bigcup I_pJA_nJI_p.$$
	 Moreover, by Lemma \ref{160}, the union is in fact disjoint. As a consequence, we obtain \eqref{153}.
\end{proof}

For some inert primes not dividing $N$, we will also need another closely related double coset decomposition.

\begin{lemma}\label{Heckelemma} Let $p$ be an inert prime. We have for $n\geq 1$ we have
$$G(\Zp)A_nG(\Zp)=A_nG(\Zp)\sqcup\bigsqcup_\stacksum{\delta\mods{p^n}}{\tau\mods {p^{2n}},\ \tau+\ov\tau=0}\gamma(\delta,\tau)A_nG(\Zp)$$
with
$$\gamma(\delta,\tau)=\begin{pmatrix}
	\tau&\delta&1\\-\ov\delta&1\\
	1
\end{pmatrix}$$
	
\end{lemma}

\begin{proof} 
We have
	$$G(\Zp)A_nG(\Zp)=G(\Zp)A_{-n}G(\Zp)=A_{-n}A_n.G(\Zp).A_{-n}.G(\Zp).$$
	Let $K_{2,1}(p^n)$ be the intersection
	$$A_n.G(\Zp).A_{-n}\cap G(\Zp).$$
	We have
	$$K_{2,1}(p^n)=G(\Zp)\cap\begin{pmatrix}
	\mcO_p&\mcO_p&\mcO_p\\ p^n\mcO_p&\mcO_p&\mcO_p\\p^{2n}\mcO^0_p&p^n\mcO_p&\mcO_p
\end{pmatrix}$$
where
$$\mcO^0_p=\{z\in \mcO_p,\ \tr(z)=0\}.$$
We have the following decomposition
\begin{equation}
		\label{K21coset}
		G(\Zp)=K_{2,1}(p^n)\sqcup\bigsqcup_\stacksum{\delta\mods{p^n}}{\tau\mods {p^{2n}},\ \tau+\ov\tau=0}\gamma(\delta,\tau)K_{2,1}(p^n).
	\end{equation}
	Let $N,\ov N, A\subset G(\Qp)$ be respectively the upper triangular nilpotent subgroup, the lower triangular nilpotent subgroup	and the diagonal torus.
	From the Iwahori decomposition we have
\begin{align*}
	K_{2,1}(p^n)&=(K_{2,1}(p^n)\cap \ov N).(K_{2,1}(p^n)\cap A).(K_{2,1}(p^n)\cap N)\\
	&=(K_{2,1}(p^n)\cap \ov N).(G(\Zp)\cap A).(G(\Zp)\cap N)
\end{align*}
Let
$$G(\Zp)A_nG(\Zp)=K_{2,1}(p^n)A_nG(\Zp)\cup\bigcup_\stacksum{\delta\mods{p^n}}{\tau\mods {p^{2n}},\ \tau+\ov\tau=0}\gamma(\delta,\tau)K_{2,1}(p^n)A_nG(\Zp).$$
since
$$K_{2,1}(p^n)A_n=A_nA_{-n}K_{2,1}(p^n)A_n\subset A_nG(\Zp)$$ we have
$$G(\Zp)A_nG(\Zp)=A_nG(\Zp)\cup\bigcup_\stacksum{\delta\mods{p^n}}{\tau\mods {p^{2n}},\ \tau+\ov\tau=0}\gamma(\delta,\tau)A_nG(\Zp)$$	
and the disjointness in \eqref{K21coset} implies the disjointness of this union.
\end{proof}

\section*{\bf Acknowledgements}

Significant parts of this work were carried out while Ph.M. was visiting the Department of Mathematics at Caltech in particular while on sabbatical in the spring semester 2022. Ph.M. gratefully acknowledges its hospitality and the wonderful working conditions provided  in the newly renovated Linde Hall.

We are very grateful to Paul Nelson for encouragement and helpful comments and especially for  detecting a serious error that we had copied from \cite{Wal76}. We would also like to thank Laurent Clozel, Dipendra Prasad,  Yiannis Sakellaridis and Xinwen Zhu for their interest and comments.

Many thanks are due to the referee for making very helpful comments and suggestions for corrections, and pointing out places where further explanations were necessary.

Ph.M. was partially supported by the SNF grant $197045$ and  the SNF-ANR ``ETIENE'' grant $10003145$. D. R.  was supported by a grant from the Simons Foundation (award Number: 523557).
\renewcommand\refname{\bf References} 

\newcommand{\etalchar}[1]{$^{#1}$}

\end{document}